\documentclass[11pt, letterpaper]{article}

\usepackage[letterpaper, margin=1in]{geometry} 
\usepackage{layout} 
\usepackage{ifthen}

\usepackage[bookmarksnumbered,bookmarksopen,colorlinks,citecolor=blue,linkcolor=blue]{hyperref}

\usepackage{natbib}

\usepackage{amsfonts}
\usepackage{mathrsfs}
\usepackage{amsmath}
\usepackage{amssymb}
\usepackage{amsthm}

\usepackage{color}
\usepackage{graphicx}

\usepackage{longtable}

\usepackage{verbatim}
\usepackage{listings}

\usepackage{versions}

\excludeversion{proofLong}

\excludeversion{proofName}

\excludeversion{paperStructure}

\usepackage{ifthen}

\newboolean{FlagHideLineProof}
\setboolean{FlagHideLineProof}{true}

\newboolean{FlagHideMyNewPage}
\setboolean{FlagHideMyNewPage}{true}

\renewcommand{\baselinestretch}{1.5}
\definecolor{light}{gray}{.8} 

\makeatletter
\def\singlespace{\def\baselinestretch{1}\@normalsize}

\ifthenelse{\boolean{FlagHideLineProof}}
{

\newcommand{\lineProof}{}
}
{

\newcommand{\lineProof}
{\begin{center}
\line(1,0){250}
\end{center}}
}

\ifthenelse{\boolean{FlagHideMyNewPage}}
{

}
{

}

 \theoremstyle{plain}
 \newtheorem{thm}{Theorem}[section]

 \newtheorem{lem}[thm]{Lemma}
 \newtheorem{prop}[thm]{Proposition}

 \theoremstyle{remark}
 \newtheorem{rem}{\textbf{Remark}}

 \theoremstyle{definition}

 \numberwithin{equation}{section}

\newcommand{\Keywords}[1]{\par\noindent{\small{\em Keywords\/}: #1}}

\newcommand{\argmin}{\text{argmin }}

\newcommand{\sign}{\text{Sign}}

\newcommand{\diag}{\text{diag}}

\newcommand{\E}{\mathbb{E}}

\newcommand{\Var}{\text{Var}}
\newcommand{\Cov}{\text{Cov}}

\newcommand{\bZ}{\mathbf{Z}}

\newcommand{\SC}{\overset{a.s.}{\longrightarrow}}

\newcommand{\ConvDist}{\overset{d}{\longrightarrow}}

\newcommand{\ConvProb}{\overset{P}{\longrightarrow}}

\newcommand{\abs}[1]{\left\lvert#1\right\rvert}
\newcommand{\norm}[1]{\left\lVert#1\right\rVert}
\newcommand{\normF}[1]{\left\lVert#1\right\rVert_{F}}  
\newcommand{\normFd}[1]{\left\lVert#1\right\rVert_{F,d}}  
\newcommand{\normE}[1]{\left\lVert#1\right\rVert_{2}}  
\newcommand{\normEd}[1]{\left\lVert#1\right\rVert_{2,d}}  

\renewcommand{\abs}[1]{\lvert#1\rvert}
\renewcommand{\norm}[1]{\lVert#1\rVert}
\renewcommand{\normF}[1]{\lVert#1\rVert_{F}}  
\renewcommand{\normFd}[1]{\lVert#1\rVert_{F,d}}  
\renewcommand{\normE}[1]{\lVert#1\rVert_{2}}  
\renewcommand{\normEd}[1]{\lVert#1\rVert_{2,d}}  

\newcommand{\bd}[1]{\boldsymbol{#1}}

\newcommand{\OnItem}[1]{\textbf{#1}}

\newcommand{\revTang}[1]{\textcolor[rgb]{0.78,0.39,0.00}{#1}}
\renewcommand{\revTang}[1]{#1}

\usepackage{authblk}

\begin{document}

\title {Partial Consistency  with Sparse Incidental Parameters}


\author[1]{Jianqing Fan}
\author[2]{Runlong Tang}
\author[1]{Xiaofeng Shi}
\affil[1]{Princeton University}
\affil[2]{Johns Hopkins University}
\maketitle

\setcounter{page}{1}   
\pagenumbering{roman}   

\setcounter{tocdepth}{3} 


\begin{abstract}
Penalized estimation principle is fundamental to high-dimensional problems.
In the literature, it has been extensively and successfully applied
to various models with only structural parameters.
As a contrast,
in this paper,
we apply this penalization principle
to a linear regression model
with
a finite-dimensional vector of structural parameters and
a high-dimensional vector of sparse incidental parameters.
For the estimators of the structural parameters,
we derive their consistency and
asymptotic normality,
which reveals an oracle property.
However,
the penalized estimators for the incidental parameters
possess only partial selection consistency but not consistency.
This is an interesting partial consistency phenomenon:
the structural parameters are consistently estimated
while the incidental ones cannot.
For the structural parameters, also considered is
an alternative two-step penalized estimator,
which has fewer possible asymptotic distributions
and thus is more suitable for statistical inferences.
We further extend the methods and results to the case
where the dimension of the structural parameter vector
diverges with but slower than the sample size.
A data-driven approach for selecting a penalty regularization parameter
is provided.
The finite-sample performance of
the penalized estimators for the structural parameters
is evaluated by simulations
and a real data set is analyzed.
\end{abstract}


\setcounter{page}{1}    
\pagenumbering{arabic}  


\section{Introduction}
\label{sec1}
Since the pioneering papers by
\cite{Tibshirani1996a}
and \cite{Fan2001a},
the penalized estimation methodology for
exploiting sparsity has been
studied extensively.
For example,
\cite{Zhao2006} provides
an almost necessary and sufficient condition,
namely Irrepresentable Condition,
for the LASSO estimator to be strong sign consistent.
\cite{Fan2011b} shows that an oracle property
holds for
the folded concave penalized estimator
with ultrahigh dimensionality.
For an overview on this topic, see
\cite{Fan2010}.

All the aforementioned papers consider models
only with the so-called structural parameters,
which are related to every data point.
By contrast,
we consider in this paper another type of models where
there are not only the structural parameters
but also the so-called incidental parameters,
each of which is related to only one data point.
Specifically, suppose data $\{\bd{X}_{i}, Y_{i}\}_{i=1}^{n}$ are from the following linear model:
\begin{equation}\label{model:basic}
    Y_{i} = \mu_{i}^{\star} + \bd{X}_{i}^{T}\bd{\beta}^{\star} + \epsilon_{i},
\end{equation}
where
the vector of \emph{incidental} parameters
$\bd{\mu}^{\star}=(\mu_{1}^{\star}, \cdots, \mu_{n}^{\star})^{T}$
is sparse,
the vector of \emph{structural} parameters
$\bd{\beta}^{\star} = (\beta_{1}^{\star}, \cdots, \beta_{d}^{\star})^{T}$
is of main interest, and for each $i$,
$\bd{X}_{i}$ is a $d$-dimensional covariate vector,
and $\epsilon_{i}$ is a random error.
Let $\bd{\nu}=(\bd{\mu}^{\star T}, \bd{\beta}^{\star T})^{T}$.
Then, in model (\ref{model:basic}),
a different data point $(\bd{X}_{i}, Y_{i})$ depends on a different subset of $\bd{\nu}$,
that is, $\mu_{i}^{\star}$ and $\bd{\beta}^{\star}$.

Model (\ref{model:basic}) arises as a working model for estimation
from \cite{Fan2012a}, which considers
a large-scale hypothesis testing problem
under arbitrary dependence of test statistics.
By Principal Factor Approximation, a method proposed by \cite{Fan2012a},
the dependent test statistics $\bZ = (Z_1, \cdots, Z_p)^{T} \sim N(\bd{\mu}, \bd{\Sigma})$ can be decomposed as
$Z_{i} = \mu_{i} + \bd{b}_{i}^{T}\bd{W} + K_{i}$,
where $\bd{\mu}=(\mu_{1}, \cdots, \mu_{p})^{T}$
and $\bd{b}_{i}$ is the $i$th row of the first $k$ unstandardized principal components,
denoted by $\bd{B}$, of $\bd{\Sigma}$ and  $\bd{K}=(K_{1}, \cdots, K_{p})^{T}\sim N(0, \bd{A})$  with $\bd{A} = \bd{\Sigma} - \bd{B} \bd{B}^{T}$.
The common factor $\bd{W}$ drives the dependence among the test statistics.
This realized but unobserved factor is critical for False Discovery Proportion (FDP) estimation
and power improvements by removing the common factor $\{\bd{b}_{i}^{T}\bd{W}\}$ from the test statistics.
Hence, an important goal is to estimate $\bd W$ with given $\{\bd{b}_i\}_{i=1}^n$.
In many applications on large-scale hypothesis testing,
the parameters $\{\mu_i\}_{i=1}^p$ are sparse.
For example, genome-wide association studies show that
the expression level of gene CCT8
is highly related to the phenotype of Down Syndrome.
It is of interest to test the association
between each of millions of SNP's and the CCT8 gene expression level.
In the framework of \cite{Fan2012a},
each $\mu_{i}$ stands for such an association.
That is,
if $\mu_{i}=0$, the $i$th SNP is not associated with the CCT8 gene expression level;
otherwise, it is associated.
Since most of the SNP's are not associated the CCT8 gene expression level,
it is reasonable to assume
$\{\mu_i\}_{i=1}^p$ are sparse.
Replacing
$Z_{i}$,
$\mu_{i}$,
$\bd{b}_{i}$,
$\bd{W}$,
$k$, $p$,
and $K_{i}$
with
$Y_{i}$,
$\mu_{i}^{\star}$,
$\bd{X}_{i}$,
$\bd{\beta}^{\star}$,
$d$, $n$,
and $\epsilon_{i}$
respectively,
we obtain model (\ref{model:basic}) formally.
It is of interest to study this model independently
with simplifications.

Although model (\ref{model:basic})
emerges from a critical component of estimating FDP in \cite{Fan2012a},
it stands with its own interest.
For example,
in some applications,
there are only few signals (nonzero $\mu_{i}^{\star}$'s)
and what is interesting is to learn about $\bd{\beta}^{\star}$,
which reflects the relationship
between the covariates and response.
For another example,
those few nonzero $\mu_{i}^{\star}$'s
might be some measurement or recording errors
of the responses $\{Y_{i}\}$.
In these cases,
model (\ref{model:basic})
is suitable for modeling data with contaminated responses
and a method producing a reliable estimator for $\bd{\beta}^{\star}$
is essentially a robust replacement for ordinary least squares estimate,
which is sensitive to outliers.

Several models with structural and incidental parameters
have first been studied
in a seminal paper by \cite{Neyman1948},
which points out the inconsistency of the maximum likelihood estimators (MLE)
of structural parameters
in the presence of a large number of incidental parameters and
provides a
modified
MLE.
However, their method does not work for model (\ref{model:basic})  due to no exploration
of the sparsity of incidental parameters.
\cite{Kiefer1956} shows
the consistency of the MLE of the structural parameters
when the incidental parameters are assumed to be
from a common distribution.
That is, they eliminate the essential high-dimensional issue
of the incidental parameters by randomizing them.
In contrast,
this paper considers deterministic incidental parameters
and handles the high-dimensional issue by penalization with a sparsity assumption.
\cite{Basu1977}
considers the elimination of nuisance parameters
via marginalizing and conditioning methods
and
\cite{Moreira2009}
solves the incidental parameter problem
with an invariance principle. For a review of the incidental parameter problems
in statistics and economics, see
\citet{Lancaster2000}.

Without loss of generality,
suppose the first $s$ incidental parameters $\{\mu_{i}^{\star}\}_{i=1}^{s}$ are nonvanishing
and the remaining are zero.
Then, model (\ref{model:basic}) can be written in a matrix form as
$\bd{Y} = \bd{X}\bd{\nu} + \bd{\epsilon}$,
where
\begin{equation*}
    \bd{X} =
    \begin{pmatrix}
  \bd{I}_{s} & \bd{X}_{1,s}^{T} & \bd{0} \\
  \bd{0} & \bd{X}_{s+1,n}^{T} & \bd{I}_{n-s} \\
 \end{pmatrix},
\end{equation*}
$\bd{X}_{i,j}^{T}=(\bd{X}_{i}, \bd{X}_{i+1}, \cdots, \bd{X}_{j})^{T}$,
$\bd{I}_{k}$ is a $k\times k$ identity matrix,
$\bd{0}$ is a generic block of zeros
and
$ \bd{\nu} = (\mu_{1}^{\star}, \cdots, \mu_{s}^{\star}, \bd{\beta}^{T}, \mu_{s+1}^{\star}, \cdots, \mu_{n}^{\star})^{T}$.
Although this is a sparse high-dimensional problem,
the matrix $\bd{X}$ does not satisfy
the sufficient conditions of the theoretical results in \cite{Zhao2006} and \cite{Fan2011b}
due to the inconsistency of the estimation of the incidental parameters in $\bd{\nu}$
and the penalty should not be simply placed on all parameters.
For details, see Supplement \ref{PC.Paper.Supplement.Introduction}.

In this paper,
we investigate a penalized estimator \revTang{of $(\bd{\mu}^{\star}, \bd{\beta}^{\star})$}
defined by
\begin{equation} \label{eq1.2}
(\hat{\bd{\mu}}, \hat{\bd{\beta}})
=
\underset{(\bd{\mu},\bd{\beta})\in\mathbb{R}^{n+d}}{\argmin}
\sum_{i=1}^{n}(Y_{i}-\mu_{i} - \bd{X}_{i}^{T}\bd{\beta})^{2}
+ \sum_{i=1}^{n}p_{\lambda}(|\mu_{i}|),
\end{equation}
where $p_{\lambda}$ is a penalty function
with a regularization parameter $\lambda$.
Since only the incidental parameters are sparse,
the penalty is 
imposed on them.
An iterative algorithm is proposed
to numerically compute the estimators.
The estimator $\hat{\bd{\beta}}$
possesses consistency, asymptotic normality and an oracle property.
On the other hand, the nonvanishing elements of $\bd{\mu}^{\star}$ cannot be consistently estimated
even if $\bd{\beta}^{\star}$ were known.
So, there is a partial consistency phenomenon.

Penalized estimation (\ref{eq1.2}) is a one-step method.
For the estimation of $\bd{\beta}^{\star}$, We also propose a two-step method
whose first step is designed to eliminate the influence of the data
with large incidental parameters.
The estimator $\tilde{\bd{\beta}}$ from the two-step method
has fewer possible asymptotic distributions
than $\hat{\bd{\beta}}$
and thus is more suitable for constructing confidence regions for $\bd{\beta}^{\star}$.
It is asymptotically equivalent to
the one-step estimator $\hat{\bd{\beta}}$ when the sizes of the nonzero incidental parameters are small enough,
that is, when the incidental parameters are really sparse.
Also, the two-step method improves the convergence rate and efficiency over the one-step method
for challenging situations
where large nonzero incidental parameters
increase the asymptotic covariance or even reduce the convergence rate
for the one-step method.

The rest of the paper is organized as follows.
In Section \ref{sec2},
the model and penalized estimation method
are formally introduced
and the corresponding penalized estimators are
characterized.
In Section \ref{sec3},
asymptotic properties
of the penalized estimators
are derived;
a penalized two-step estimator is proposed and its theoretical properties are obtained;
we also 
provide a data-driven approach for selecting the regularization parameter.
In Section \ref{sec:Diverging number of structural parameters},
we consider the case where the number of covariates grows with but slower than the sample size.
In Section \ref{sec6}, we present simulation results
and analyze a read data set.
Section \ref{PC.Paper.Conclusion} concludes this paper
with a discussion
and all the proofs and some theoretical results are relegated to
the appendix and supplements.

\section{Model and Method}
\label{sec2}

\revTang{The matrix form of model (\ref{model:basic}) is} given by
\begin{equation}
\label{eq2.1}
    \bd{Y} = \bd{\mu}^{\star} + \bd{X}\bd{\beta}^{\star} + \bd{\epsilon},
\end{equation}
where $\bd{Y}=(Y_{1}, Y_{2}, \cdots, Y_{n})^{T}$,
$\bd{X}=(\bd{X}_{1}, \bd{X}_{2}, \cdots, \bd{X}_{n})^{T}$, and
$\bd{\epsilon}=(\epsilon_{1},\epsilon_{2},\cdots,\epsilon_{n})^{T}$.
The covariates
$\{\bd{X}_{i}\}_{i=1}^{n}$ are independent and identically distributed (i.i.d.) copies of $\bd{X}_{0}\in\mathbb{R}^{d}$,
which is a random vector
with mean zero and a covariance matrix $\bd{\Sigma}_{X}>0$.
They are independent of the random errors $\{\epsilon_{i}\}$,
which are i.i.d. copies of $\epsilon_{0}$,
which is a random variable with mean zero and variance $\sigma^{2}>0$.
Denote $a_{n} \ll b_{n}$ and $a_{n} \gg b_{n}$
if $a_{n}=o(b_{n})$ and $b_{n}=o(a_{n})$,
respectively.
There is an assumption on the covariates and random errors.
\begin{description}
  \item[Assumption (A):]
  There exist positive sequences $\kappa_{n}\ll  \sqrt{n}, \gamma_{n} \ll  \sqrt{n}$
such that
\begin{equation}
\label{eq2.2}
P(\max_{1\leq i\leq n} \normE{\bd{X}_{i}} > \kappa_{n}) \rightarrow 0
\text{ and }
P(\max_{1\leq i\leq n} |\epsilon_{i}| > \gamma_{n}) \rightarrow 0,
\text{ as } n\rightarrow \infty,
\end{equation}
where $\|\cdot\|_{2}$ stands for the $l_{2}$ norm of $\mathbb{R}^{d}$.
\end{description}

Suppose there are three types of incidental parameters in model
(\ref{model:basic}) or (\ref{eq2.1}):
for simplicity on the indexes,
the first $s_{1}$ incidental parameters $\{\mu_{i}^{\star}\}_{i=1}^{s_{1}}$ are large
in the sense that
$|\mu_{i}^{\star}| \gg  \max\{\kappa_{n}, \gamma_{n}\}$ for ${1\leq i \leq s_{1}}$;
the next $s_{2}$ ones $\{\mu_{i}^{\star}\}_{i=s_{1}+1}^{s}$ are nonzero and bounded by $\gamma_{n}$
with $s=s_{1}+s_{2}$;
the last $n-s$ ones $\{\mu_{i}^{\star}\}_{i=s+1}^{n}$ are zero.
Note that it is unknown to us
which $\mu_{i}^{\star}$'s are large, bounded or zero.
The sparsity of $\bd{\mu}^{\star}$ is understood by
$s_{1} + s_{2} \ll n$, i.e. $s_{1}+s_{2}=o(n)$.
Denote the vectors of the three types of incidental parameters $\bd{\mu}_{1}^{\star}$,
$\bd{\mu}_{2}^{\star}$, and $\bd{\mu}_{3}^{\star}$, respectively.

The penalized estimation (\ref{eq1.2}) can be written as
\begin{equation}
\label{eq2.3}
    (\hat{\bd{\mu}}, \hat{\bd{\beta}})
    =
    \underset{(\bd{\mu}, \bd{\beta})}{\argmin}
    L(\bd{\mu},\bd{\beta}), \qquad
    L(\bd{\mu},\bd{\beta})
    =
    \normE{\bd{Y} - \bd{\mu} - \bd{X}\bd{\beta}}^{2}
    +
    \sum_{i=1}^{n} p_{\lambda}(|\mu_{i}|).
\end{equation}
The penalty function $p_{\lambda}$ can be the soft (i.e. $L_{1}$ or LASSO),
hard, SCAD
or a general folded concave penalty function \citep{Fan2001a}.
For simplicity,
we next consider only the soft penalty function,
that is, $p_{\lambda}(|\mu_{i}|) = 2\lambda|\mu_{i}|$.
The cases with
the hard and SCAD penalties
can be considered
in a similar way.

By Lemma \ref{lem2.1} in Supplement \ref{PC.Paper.Supplement.Model.Data.Method},
A necessary and sufficient condition for
$(\hat{\bd{\mu}}, \hat{\bd{\beta}})$
to be a minimizer of $L(\bd{\mu},\bd{\beta})$ is that
$\hat{\bd{\beta}} = (\bd{X}^{T}\bd{X})^{-1}\bd{X}^{T}(\bd{Y}-\hat{\bd{\mu}})$,
$Y_{i} - \hat{\mu}_{i} - \bd{X}_{i}^{T}\hat{\bd{\beta}} = \lambda  \sign(\hat{\mu}_{i})$ for $i \in \hat{I}_{0}^{c}$
and
$|Y_{i} - \bd{X}_{i}^{T}\hat{\bd{\beta}}| \leq \lambda$ for $i \in \hat{I}_{0}$,
where $\sign(\cdot)$ is a sign function and $\hat{I}_{0}=\{1\leq i\leq n: \hat{\mu}_{i} = 0\}$.

Numerically, the special structure of $L(\bd{\mu}, \bd{\beta})$
suggests a marginal decent algorithm
for the minimization problem in (\ref{eq2.3}),
which iteratively computes
$
    \bd{\mu}^{(k)}
    =
    \underset{\bd{\mu}\in\mathbb{R}^{n}}{\argmin}
    L(\bd{\mu},\bd{\beta}^{(k-1)})$ and
$
    \bd{\beta}^{(k)}
    =
    \underset{\bd{\beta}\in\mathbb{R}^{d}}{\argmin}
    L(\bd{\mu}^{(k)},\bd{\beta})
$
until convergence. The advantage of this algorithm is that
there exist analytic solutions to the \revTang{above two minimization problems}.
They are respectively the soft-threshold estimators with residuals
$\{Y_i - \bd{X}_{i}^{T}\bd{\beta}^{(k-1)}\}$
and ordinary least-squares estimator with responses $\bd{Y}-\bd{\mu}^{(k)}$.
In this section and the next,
the number of the covariates $d$ is assumed to be a \emph{fixed} integer.
A case where $d$ diverging to infinity
will be considered in Section \ref{sec:Diverging number of structural parameters}.
For the case where $d$ is finite,
we make the following assumption on $\lambda$.
\begin{description}
  \item[Assumption (B):]
  The regularization parameter $\lambda$ satisfies
\begin{equation}
\label{eq3.1}
\kappa_{n} \ll  \lambda,~  \alpha\gamma_{n} \leq \lambda, \text{ and } \lambda \ll  \min\{\mu^{\star},\sqrt{n}\},
\end{equation}
where
$\kappa_{n}$ and $\gamma_{n}$ are defined in (\ref{eq2.2}),
$\alpha$ is a constant greater than 2,
and $\mu^{\star} = \min_{1\leq i \leq s_1} |\mu_{i}^{\star}|$.
\end{description}
For simplicity,
abbreviate ``with probability going to one" to ``wpg1".
A stopping rule for the above algorithm
is based on the successive difference
$\normE{\bd{\beta}^{(k+1)} - \bd{\beta}^{(k)}}$.
By Proposition \ref{prop2.2} in Supplement \ref{PC.Paper.Supplement.Model.Data.Method},
wpg1,
the iterative algorithm stops at the
the second iteration,
given the initial estimator is bounded wpg1.

Suppose $\{\bd{\beta}^{(k)}\}$ has a theoretical limit $\bd{\beta}^{(\infty)}$,
corresponding to which,
there is a limit estimator $\bd{\mu}^{(\infty)}$.
Then, $(\bd{\mu}^{(\infty)}, \bd{\beta}^{(\infty)})$ is a solution of the following system of nonlinear equations
\begin{equation}
\label{eq2.4}
    \bd{\beta}
    =
    (\bd{X}^{T}\bd{X})^{-1}\bd{X}^{T}(\bd{Y}-\bd{\mu}),
\end{equation}
and, with soft-threshold applied to each component, it follows
\begin{equation}
\label{eq2.5}
\bd{\mu} = (|\bd{Y} - \bd{X}^T\bd{\beta}|- \lambda)_{+} \sign(\bd{Y} - \bd{X}^T\bd{\beta}),
\end{equation}
where $(\cdot)_{+}$ returns the maximum value of the input and zero.
By Lemma \ref{lem2.3} in Supplement \ref{PC.Paper.Supplement.Model.Data.Method},
a necessary and sufficient condition for
$(\hat{\bd{\mu}}, \hat{\bd{\beta}})$
to be a minimizer of $L(\bd{\mu},\bd{\beta})$ is that
it is a solution to equations
(\ref{eq2.4}) and (\ref{eq2.5}).
Hence,
$(\bd{\mu}^{(\infty)}, \bd{\beta}^{(\infty)})$
is a minimizer of $L(\bd{\mu},\bd{\beta})$
and
can also be denoted
as $(\hat{\bd{\mu}},\hat{\bd{\beta}})$.

Note that
$\hat{\bd{\beta}}$ is also the minimizer of the profiled loss function
$\tilde{L}(\bd{\beta})=L(\bd{\mu(\bd{\beta})},\bd{\beta})$,
where $\bd{\mu(\bd{\beta})}$
as a function of $\bd{\beta}$
is given by (\ref{eq2.5}) .
Interestingly, this profiled loss function
is a criterion function equipped with the famous Huber loss function
(see \cite{Huber1964} and \cite{Huber1973}).
Specifically, the profiled loss function can be expressed as
$
\tilde{L}(\bd{\beta})=
\sum_{i=1}^{n}\rho(Y_{i}-\bd{X}_{i}^{T}\bd{\beta})
$,
where $\rho(x)= x^2 I(|x| \leq \lambda) + (2\lambda x - \lambda^2) I(|x|> \lambda)$
is \emph{exactly}  the Huber loss function,
which is optimal in a minimax sense.
This equivalence
between the penalized estimation and Huber's robust estimation
indicates that the penalization principle is versatile and
can naturally produce an important loss function in robust statistics.
This equivalence
also provides a formal endorsement of
the least absolute deviation
robust regression (LAD) in \cite{Fan2012a}
and indicates that
it is better to use all data with LAD regression
rather than 90\% of them.
It is worthwhile to note that
the penalized estimation is only formally equal to the Huber's.
Our model (\ref{eq2.1})
considers \emph{deterministic} sparse incidental parameters $\mu_{i}^{\star}$'s,
while the model in Huber's works assumes \emph{random} contamination as in \cite{Kiefer1956}.
Recently,
there appear a few papers on
robust regression in high-dimensional settings,
see, for example, \cite{Chen2010}, \cite{Lambert-Lacroix2011}, \cite{Fan2014} and \cite{Bean2012}.
\cite{Portnoy2000} provide a high level review of literature on robust statistics.

From the equations  (\ref{eq2.4}) and (\ref{eq2.5}),
$\hat{\bd{\beta}}$ is a solution to
\begin{equation}
\label{eq2.6}
\varphi_{n}(\bd{\beta})=0,
\text{ where }
    \varphi_{n}(\bd{\beta})
    =
    \bd{\beta} -
    (\bd{X}^{T}\bd{X})^{-1}\bd{X}^{T}(\bd{Y}-\bd{\mu}(\bd{\beta})).
\end{equation}
In general, this is a Z-estimation problem.
The following theoretical analysis
is based on this characterization of
$\hat{\bd{\beta}}$.

At the end of this section, we provide for further analysis some notations and an expansion of $\varphi_{n}(\bd{\beta})$.
Let
$\mathbb{S}=\sum_{i=1}^{n}\bd{X}_{i}\bd{X}_{i}^{T}$,
$\mathbb{S}_{S}=\sum_{i\in S}\bd{X}_{i}\bd{X}_{i}^{T}$,
$\mathbb{S}_{S}^{\mu}=\sum_{i\in S}\bd{X}_{i}\mu_{i}^{\star}$,
$\mathbb{S}_{S}^{\epsilon}=\sum_{i\in S}\bd{X}_{i}\epsilon_{i}$,
$\mathcal{S}=\sum_{i=1}^{n}\bd{X}_{i}$
and
$\mathcal{S}_{S}=\sum_{i\in S}\bd{X}_{i}$,
where $S$ is a subset of $\{1,2,\cdots,n\}$.
It is straightforward to show
\begin{align}
\label{eq2.7}
\varphi_{n}(\bd{\beta})
& = \nonumber
(\mathbb{S}_{S_{10}} + \mathbb{S}_{S_{11}}+\mathbb{S}_{S_{12}})(\bd{\beta}-\bd{\beta}^{\star})
 -(\mathbb{S}_{S_{11}}^{\mu} +  \mathbb{S}_{S_{12}}^{\mu}) \\
& -
(\mathbb{S}_{S_{10}}^{\epsilon}+
\mathbb{S}_{S_{11}}^{\epsilon}
+\mathbb{S}_{S_{12}}^{\epsilon})
-
\lambda
(\mathcal{S}_{S_{20}} + \mathcal{S}_{S_{21}} + \mathcal{S}_{S_{22}}
- \mathcal{S}_{S_{30}} - \mathcal{S}_{S_{31}} - \mathcal{S}_{S_{32}}),
\end{align}
where the index sets
$S_{10}  = \{s+1\leq i \leq n: |\bd{X}_{i}^{T}(\bd{\beta}^{\star}-\bd{\beta}) + \epsilon_{i}| \leq \lambda\}$,
$S_{11}  = \{1\leq i \leq s_{1}: |\mu_{i}^{\star}+\bd{X}_{i}^{T}(\bd{\beta}^{\star}-\bd{\beta}) + \epsilon_{i}| \leq \lambda\}$
and
$S_{12}  = \{s_{1}+1\leq i \leq s: |\mu_{i}^{\star}+\bd{X}_{i}^{T}(\bd{\beta}^{\star}-\bd{\beta}) + \epsilon_{i}| \leq \lambda\}$;
$S_{20}$, $S_{21}$ and $S_{22}$  are defined similarly
except that the absolute operation is omitted and ``$\leq$" is replaced by ``$>$";
$S_{30}$, $S_{31}$ and $S_{32}$,  are defined similarly with $S_{20}$, $S_{21}$ and $S_{22}$
except that ``$> \lambda$" is replaced by ``$< -\lambda$".
Note that all these index sets depend on $\bd{\beta}$.

\section{Asymptotic Properties}
\label{sec3}
In this section,
we consider the asymptotic properties of the penalized estimators  $\hat{\bd{\beta}}$ and $\hat{\bd{\mu}}$.
Assumption (A),
together with Assumption (B),
enables the penalized estimation method to
distinguish
the large incidental parameters
from \revTang{others},
and thus simplifies
the asymptotic properties of the index sets $S_{ij}$'s
in (\ref{eq2.7})
\revTang{in the sense that} they become independent of $\bd{\beta}$ wpg1.
Denote a hypercube of $\bd{\beta}^{\star}$ by
$B_{C}(\bd{\beta}^{\star})=\{\bd{\beta}\in\mathbb{R}^{d}:|\beta_{j}-\beta_{j}^{\star}|\leq C, 1\leq j\leq d\}$
with a constant $C>0$.

\begin{lem}[On Index Sets $S_{ij}$'s] \label{lem3.1}
Under Assumptions (A) and (B),
for every $C>0$ and
every $\bd{\beta}\in B_{C}(\bd{\beta}^{\star})$,
wpg1,
$S_{10} = S_{10}^{\star}$,
$S_{11} = \emptyset$,
$S_{12} = S_{12}^{\star}$,
$S_{20} = \emptyset$,
$S_{21} = S_{21}^{\star}$,
$S_{22} = \emptyset$,
$S_{30} = \emptyset$,
$S_{31} = S_{31}^{\star}$
and
$S_{32} = \emptyset$,
where the limit index sets
$S_{10}^{\star}=\{s+1, s+2, \cdots, n\}$,
$S_{12}^{\star}=\{s_{1}+1, s+2, \cdots, s\}$,
$S_{21}^{\star}=\{1 \leq i\leq s_{1}: \mu_{i}^{\star} > 0 \}$ and
$S_{31}^{\star}=\{1 \leq i\leq s_{1}: \mu_{i}^{\star} < 0 \}$.
\end{lem}
By Lemma~\ref{lem3.1}, wpg1, the solution $\hat{\bd{\beta}}$ to (\ref{eq2.6})
has an analytic expression:
\begin{equation}\label{eq3.2}
\hat{\bd{\beta}}
=
\bd{\beta}^{\star}
+
(\mathbb{S}_{S_{10}^{\star}} + \mathbb{S}_{S_{12}^{\star}})^{-1}
[\mathbb{S}_{S_{12}^{\star}}^{\mu}
 +
(\mathbb{S}_{S_{10}^{\star}}^{\epsilon}
+
\mathbb{S}_{S_{12}^{\star}}^{\epsilon})
+
\lambda
(\mathcal{S}_{S_{21}^{\star}}
-\mathcal{S}_{S_{31}^{\star}})],
\end{equation}
from which,
we derive asymptotic properties of $\hat{\bd{\beta}}$.
Some analysis needs the following assumption.
\begin{description}
  \item[Assumption (C):] There exists some constant $\delta>0$ such that $\E\normE{\bd{X}_{0}}^{2+\delta}<\infty$
  and
$\normE{\bd{\mu}_{2}^{\star}}/ \| \bd{\mu}_{2}^{\star}\|_{2+\delta}$ diverges to infinity,
where
$\| \bd{\mu}_{2}^{\star}\|_{2+\delta} = \bigl (\sum_{i=s_{1}+1}^{s}|\mu_{i}^{\star}|^{2+\delta} \bigr )^{1/(2+\delta)}$.
\end{description}
The following result shows the existence of a unique consistent estimator of $\bd{\beta}^{\star}$.
\begin{thm}[Existence and Consistency of $\hat{\bd{\beta}}$]
\label{thm3.2}
Under Assumptions (A) and (B),
if either $s_{2}=o(n/(\kappa_{n}\gamma_{n}))$
or Assumption (C) holds,
then,
for every fixed $C>0$,
wpg1,
there exists a unique estimator $\hat{\bd{\beta}}_{n}\in B_{C}(\bd{\beta}^{\star})$
such that $\psi_{n}(\hat{\bd{\beta}}_{n})=0$ and $\hat{\bd{\beta}}_{n}\ConvProb \bd{\beta}^{\star}$.
\end{thm}

In Theorem
\ref{thm3.2},
there are two different kinds of sufficient conditions:
on is on $s_{2}$, which is the size of bounded incidental parameters $\bd{\mu}_{2}^{\star}$,
and the other is Assumption (C), which is about the norms of $\bd{\mu}_{2}^{\star}$.
They come from different analysis approaches on the term $\mathbb{S}_{S_{12}^{\star}}^{\mu}$ in (\ref{eq3.2}).
One does not imply the other.
For details, see Supplement \ref{PC.Paper.Supplement.Asymptotic.Properties}.
Specially,
if $s_{2}=O(n^{\alpha_{2}})$ for some $\alpha_{2}\in(0,1)$
and $\kappa_{n}\gamma_{n}\ll n^{(1-\alpha_{2})}$,
then $\hat{\bd{\beta}}$ is consistent by Theorem \ref{thm3.2}.


Next, we consider the asymptotic distributions of the consistent estimator $\hat{\bd{\beta}}_{n}$
obtained in Theorem \ref{thm3.2}.
Without loss of generality,
we assume the sizes of index sets
$S_{21}^{\star}=\{1\leq i\leq s_{1}: \mu_{i}^{\star} > 0\}$
and
$S_{31}^{\star}=\{1\leq i\leq s_{1}: \mu_{i}^{\star} < 0\}$
are asymptotically equivalent to $as_{1}$ and $(1-a)s_{1}$
with a constant $a\in(0,1)$.
Similar to Theorem \ref{thm3.2},
there are two different sets of conditions on $\bd{\mu}_{2}^{\star}$
corresponding to two different analysis approaches.
Denote $\sim$ as the asymptotic equivalence and $D_{n} = \normE{\bd{\mu}_{2}^{\star}}$.
\begin{thm}[\revTang{Asymptotic Distributions on $\hat{\bd{\beta}}_{n}$}]
\label{thm3.4}
Under Assumptions (A) and (B), suppose $s_{2} \ll  \sqrt{n}/(\kappa_{n}\gamma_{n})$ holds
or Assumption (C) and $D_{n}^{2}/n=o(1)$ hold.
\begin{enumerate} \itemsep -0.05in
\item [(1)] If $s_{1} \ll  n/\lambda^{2}$,
then
$
\sqrt{n}(\hat{\bd{\beta}}_{n}-\bd{\beta}^{\star})
\ConvDist
N(0, \sigma^{2}\bd{\bd{\Sigma}}_{X}^{-1});
$ {\bf [main case]}

\item [(2)] If $s_{1}\sim b n/\lambda^{2}$,
then
$
\sqrt{n}(\hat{\bd{\beta}}_{n}-\bd{\beta}^{\star})
\ConvDist
N(0, (b+\sigma^{2})\bd{\bd{\Sigma}}_{X}^{-1}),
$
for every constant $b\in\mathbb{R}^{+}$;

\item [(3)] If $s_{1}\gg n/\lambda^{2}$,
then
$
r_{n}(\hat{\bd{\beta}}_{n}-\bd{\beta}^{\star})
\ConvDist
N(0,\bd{\bd{\Sigma}}_{X}^{-1}),
$
where $r_{n} \sim n/(\lambda\sqrt{s_{1}})$.
\end{enumerate}
\end{thm}

When the incidental parameters are really sparse,
the size $s_{1}$ of large incidental parameters is small
and
the size $s_{2}$ or the magnitude $D_{n}$ of bounded incidental parameters is also small
so that the conditions of case \emph{(1)} tends to hold.
This case is of most interest
and we denote it as the \emph{main case}.
The other cases are presented to provide
a relatively complete picture of the asymptotic distributions of $\hat{\bd{\beta}}$.
In fact, Theorem \ref{thm3.6} in Supplement \ref{PC.Paper.Supplement.Asymptotic.Properties}
shows more possible asymptotic distributions.
Note that
the constant $a$ does not appear in the limit distributions of Theorem \ref{thm3.4}
due to cancelation
and that
the sub-$\sqrt{n}$ consistency emerges
in case \emph{(3)} when $s_1$ is large,
because for this case
the impact of the large incidental parameters is too big to be handled efficiently
by the penalized estimation.
For case \emph{(2)}, in one direction,
as $b \to 0$,
its condition and limit distribution become those of case \emph{(1)};
in the other direction,
as $b$ increases, it approaches case \emph{(3)}.
This boundary phenomenon was in spirit similar to that in \cite{Tang2012a}.
Specially,
if
$\lambda \ll  n^{\alpha_{1}}$,
$\kappa_{n}\gamma_{n} \ll  n^{\alpha_{2}}$,
$s_{1} \ll n^{1-\alpha_{1}}$
and  $s_{2} \ll n^{1/2-\alpha_{2}}$,
for some
$\alpha_{1}\in(0,1)$
and
$\alpha_{2}\in(0,1/2)$,
then
$
\sqrt{n}(\hat{\bd{\beta}}_{n}-\bd{\beta}^{\star})
\ConvDist
N(0, \sigma^{2}\bd{\bd{\Sigma}}_{X}^{-1})
$
by the main case of
Theorem \ref{thm3.4}.

\begin{rem}[An Oracle Property]
\label{rem3}
Suppose an oracle tells
the true $\bd{\mu}^{\star}$.
Then, with the adjusted responses $\bd{Y}-\bd{\mu}^{\star}$,
the oracle estimator of $\bd{\beta}^{\star}$ is given by
$
    \hat{\bd{\beta}}^{(O)}
=
(\bd{X}\bd{X}^{T})^{-1}\bd{X}^{T}(\bd{Y}-\bd{\mu}^{\star}).
$
The limiting distribution of
$
\sqrt{n}(\hat{\bd{\beta}}_{n}^{(O)}-\bd{\beta}^{\star})
$
is
$
N(0, \sigma^{2}\bd{\bd{\Sigma}}_{X}^{-1}).
$
Comparing this with
the main case of
Theorems \ref{thm3.4},
it follows that
the penalized estimator $\hat{\bd{\beta}}_{n}$ enjoys an oracle property.
\end{rem}

\revTang{Although mainly interested in the estimation of $\bd{\beta}^{\star}$,
we also obtain  the soft-threshold estimator $\hat{\bd{\mu}}$ of $\bd{\mu}^{\star}$}: for each $i$,
\begin{equation}
\label{PC.Paper.expression.hat.mu}
    \hat\mu_{i} = \mu_{i}(\hat{\bd{\beta}}) = (|Y_{i}-\bd{X}_{i}^{T}\hat{\bd{\beta}} |- \lambda)_{+}
    \mbox{sgn}(Y_{i}-\bd{X}_{i}^{T}\hat{\bd{\beta}}).
\end{equation}
Denote
$
\mathcal{E}
=
\{
\hat\mu_{i}\not=0, \text{ for } i=1,2,\cdots, s_{1};
\text{ and }
\hat\mu_{i}=0, \text{ for } i=s_{1}+1,s_{1}+2,\cdots, n
\}.
$

\begin{thm}[Partial Selection Consistency on $\hat{\bd{\mu}}$]
\label{thm3.7}
Under Assumptions (A) and (B),
\revTang{if $\hat{\bd{\beta}} \ConvProb \bd{\beta}^{\star}$},
then
$
P(\mathcal{E}) \rightarrow 1.
$
\end{thm}
Theorem \ref{thm3.7} shows that,
wpg1,
the indexes of $\bd{\mu_{1}^{\star}}$ and $\bd{\mu_{3}^{\star}}$
are estimated correctly,
but those of $\bd{\mu_{2}^{\star}}$ wrongly.
We call this a partial selection consistency phenomenon.

\subsection{Two-Step Estimation}
\label{subsec3.1}

Theorems \ref{thm3.4}
shows that the penalized estimator $\hat{\bd{\beta}}_{n}$
has multiple different limit distributions,
which complicates the application of these theorems in practice.
In addition, the convergence rate of $\hat{\bd{\beta}}_{n}$
is less than the optimal rate $\sqrt{n}$ in the challenging cases
where the impact of large incidental parameters is substantial.
To address these issues, we propose the following two-step estimation method:
firstly, we apply the penalized estimation
(\ref{eq2.3})
and
let
$
\hat I_{0} = \{1\leq i\leq n: \hat\mu_{i} = 0\};
$
secondly, we define the two-step estimator as
\begin{equation}
\label{eq3.4}
\tilde{\bd{\beta}}
=
(\bd{X}_{\hat I_{0}}^{T}\bd{X}_{\hat I_{0}})^{-1}\bd{X}_{\hat I_{0}}^{T}\bd{Y}_{\hat I_{0}},
\end{equation}
where
$\bd{X}_{\hat I_{0}}$ consists of $\bd{X}_{i}$'s
whose indexes are in $\hat I_{0}$
and
$\bd{Y}_{\hat I_{0}}$ consists of the corresponding $Y_{i}$'s.
The following theorem shows that
$\tilde{\bd{\beta}}$ is consistent and
Asymptotic Gaussian.
\begin{thm}[Consistency and Asymptotic Normality on $\tilde{\bd{\beta}}$]
\label{thm3.8}
Suppose Assumptions (A) and (B) hold.
If either $s_{2}=o(n/(\kappa_{n}\gamma_{n}))$
or Assumption (C) holds,
then $\tilde{\bd{\beta}} \ConvProb \bd{\beta}^{\star}$.
If either $s_{2}=o(\sqrt{n}/(\kappa_{n}\gamma_{n}))$ holds
or Assumptions (C) and $D_{n}^{2}/n=o(1)$ hold,
then
$
\sqrt{n}(\tilde{\bd{\beta}} - \bd{\beta}^{\star} )
\ConvDist
N(0, \sigma^{2}\bd{\bd{\Sigma}}_{X}^{-1})
$.
\end{thm}

Comparing Theorem \ref{thm3.8} with Theorem \ref{thm3.4},
we see that $\hat{\bd{\beta}}$ has three possible asymptotic distributions
but $\tilde{\bd{\beta}}$ has only one since for $\tilde{\bd{\beta}}$ the conditions on $s_{1}$ disappear.
It is because the two-step method identifies and removes large incidental parameters
by exploiting the partial selection consistency property of $\hat{\bd{\mu}}$
in Theorem \ref{thm3.7}.
Further, the two-step estimator improves the convergence rate to the optimal one over the one-step estimator
for the challenging case with $s_{1} \gg  n/\lambda^{2}$.
Because of these advantages,
we suggest to use the two-step method to make statistical inferences.

When the incidental parameters are sparse in the sense that
the size or the magnitude of the bounded incidental parameters are small,
i.e. $s_{2}=o(\sqrt{n}/(\kappa_{n}\gamma_{n}))$ or $D_{n}^{2}/n=o(1)$,
it follows, by Theorem \ref{thm3.8},
$\sqrt{n}(\tilde{\bd{\beta}} - \bd{\beta}^{\star} )
\ConvDist
N(0, \sigma^{2}\bd{\bd{\Sigma}}_{X}^{-1})$,
from which
a confidence region
with asymptotic confidence level $1-\alpha$  is given by
\begin{equation}
\label{eq3.5}
\{\bd{\beta}\in\mathbb{R}^{d}: \sigma^{-1}\sqrt{n}\normE{\bd{\bd{\Sigma}}_{X}^{1/2}(\tilde{\bd{\beta}}-\bd{\beta})} \leq q_{\alpha}(\chi_{d}) \},
\end{equation}
where $q_{\alpha}(\chi_{d})$ is the upper $\alpha$-quantile of $\chi_{d}$,
the square root of the chi-squared distribution with degrees of freedom $d$.
%
For each component $\beta^{\star}_{j}$ of $\bd{\beta}^{\star}$,
an asymptotic $1-\alpha$ confidence interval is given by
\begin{equation}
\label{PC.Paper.PLS.Soft.TwoStage.Beta.i.Confidence.Interval}
[\tilde{\beta}_{j} \pm n^{-1/2}\sigma\bd{\bd{\Sigma}}_{X}^{-1/2}(j,j)z_{\alpha/2}],
\end{equation}
where
$\bd{\bd{\Sigma}}_{X}^{-1/2}(j,j)$ is the square root of the $(j,j)$ entry of $\bd{\bd{\Sigma}}_{X}^{-1}$ and
$z_{\alpha/2}$ is the upper $\alpha/2$-quantile of $N(0,1)$.
%
%
The confidence
region (\ref{eq3.5})
and
interval (\ref{PC.Paper.PLS.Soft.TwoStage.Beta.i.Confidence.Interval})
involve unknown parameters $\bd{\bd{\Sigma}}_{X}$ and $\sigma$.
They can be estimated by
$
    \hat{\bd{\Sigma}}_{X} = (1/n)\bd{X}^{T}\bd{X}
$
and
\begin{equation}
\label{eq.sigma.hat}
\hat{\sigma}
=
\#(\hat I_{0})^{-1/2} \normE{ \bd{Y}_{\hat I_{0}} - \bd{X}_{\hat I_{0}}^{T} \tilde{\bd{\beta}} }.
\end{equation}
By Law of Large Numbers, $\hat{\bd{\Sigma}}_{X}$ is consistent.
By lemma \ref{lem3.9} in Supplement,
$\hat{\sigma}$ is also consistent.
Hence,
after replacing $\bd{\bd{\Sigma}}_{X}$ and $\sigma$ in
the confidence
region (\ref{eq3.5})
and
interval (\ref{PC.Paper.PLS.Soft.TwoStage.Beta.i.Confidence.Interval})
with
$\hat{\bd{\Sigma}}_{X}$ and $\hat{\sigma}$,
the resulting confidence region and interval
have the asymptotic confidence level $1-\alpha$.

\subsection{Theoretical and Data-Driven Regularization Parameters}
\label{subsec:Regularization Parameters}
By Assumption (B),
the theoretical regularization parameter $\lambda$
depends on
$\kappa_{n}$ and $\gamma_{n}$,
which are also crucial
to the boundary conditions
of the asymptotic properties of
the penalized estimators
$\hat{\bd{\beta}}$ and $\tilde{\bd{\beta}}$.
By Assumption (A),
$\kappa_{n}$ and $\gamma_{n}$
is determined by the distributions
of $\bd{X}_{0}$ and $\epsilon_{0}$, respectively.
It is of interest to explicitly derive
$\kappa_{n}$ and $\gamma_{n}$
for some typical cases on
the covariates and errors.
When
the covariates are bounded with $C_{X} > 0$
and the random errors follow $N(0,\sigma^{2})$,
let $\kappa_{n}=\sqrt{d}C_{X}$
and $\gamma_{n}=\sqrt{2\sigma^{2}\log(n)}$.
They satisfy (\ref{eq2.2}) in Assumption (A)
and the specification of $\lambda$ (\ref{eq3.1}) in Assumption (B) becomes
$\alpha\sqrt{2\sigma^{2}\log(n)} \leq \lambda \ll \min\{\mu^{\star},\sqrt{n}\}$.
When
$\bd{X}_{0}$ and $\epsilon_{0}$
follow $N(0, \bd{\bd{\Sigma}}_{X})$
and $N(0,\sigma^{2})$, respectively.
Denote by $\sigma_{X}^{2}$ the maximum of diagonal elements of $\bd{\bd{\Sigma}}_{X}$.
We can take $\kappa_{n}=\sqrt{2d\sigma_{X}^{2}\log(n)}$ and $\gamma_{n}=\sqrt{2\sigma^{2}\log(n)}$.
Then, (\ref{eq3.1}) becomes
$\sqrt{\log(n)} \ll \lambda \ll \min\{\mu^{\star},\sqrt{n}\}$.
A case on exponentially tailed random variables has all been considered
in Supplement
\ref{PC.Paper.Supplement.Typical.Covariates.Errors}.

Although the theoretical specification of $\lambda$ guaranties desired asymptotic properties,
a data-driven specification is of interest in practice.
A popular way to specify $\lambda$ is to use 
multi-fold cross-validation.
The validation set, however, needs to be made as little contaminated as possible.
We propose the following procedure
to identify a data-driven regularization parameter:
\begin{description}
  \item[Step 1:] On the training and testing data sets.
    \begin{enumerate}
      \item Apply ordinary least squares (OLS) to all the data and obtain residuals $\hat\epsilon_{i}^{(OLS)} = Y_{i} - \bd{X}_{i}^{T}\hat{\bd{\beta}}^{(OLS)}$ for each $i$.
      \item Identify the set of ``pure" data
            corresponding to the $n_{pure}$ smallest values in  $\{|\hat\epsilon_{i}^{(OLS)}|\}$.
      \item Compute the updated OLS estimator $\hat{\bd{\beta}}^{(OLS,2)}$ with the ``pure" data
            and obtain updated residuals $\{\hat\epsilon_{i}^{(OLS,2)}\}$
            for each $i$.
      \item Identify the updated ``pure" data set
            corresponding to the $n_{pure}$ smallest $\{|\hat\epsilon_{i}^{(OLS,2)}|\}$
            and label the remaining as the ``contaminated" data set.
      \item Randomly select a subset from the updated ``pure" data set as a testing set
      and merge the remaining ``pure'' data set and the ``contaminated'' one into a training set.
    \end{enumerate}
  \item[Step 2:] On the range $[\lambda_{L}, \lambda_{U}]$
        of the regularization parameter.
    \begin{enumerate}
    \item Compute the standard deviation $\hat\sigma_{pure}$ of the residuals
            of the ``pure" data set.
    \item Set $\lambda_{L}=\alpha_{l}\hat\sigma_{pure}$
            and $\lambda_{U}=\alpha_{u}\hat\sigma_{pure}$,
            where $\alpha_{l}<\alpha_{u}$ are positive constants.
    \end{enumerate}
  \item[Step 3:] On the data-driven regularization parameter.
    \begin{enumerate}
      \item For each grid point of $\lambda$ in the interval $[\lambda_L, \lambda_U$],
            apply a penalized method to the training set
            and obtain the estimator $\hat{\bd{\beta}}_{\lambda, train}$
            and the corresponding test error $\hat{\sigma}_{\lambda,test}^{2} = \sum_{\text{testing set}}(Y_{i}-\bd{X}_{i}^{T}\hat{\bd{\beta}}_{\lambda, train})^{2}$.
      \item Identify the data-driven regularization parameter $\lambda_{opt}$,
            which minimizes $\hat{\sigma}_{\lambda,test}^{2}$
            among the grid points.
    \end{enumerate}
\end{description}
  This simple data-driven procedure can certainly be further improved.
  For example, In Step 1, the sub-steps 3 and 4 can be repeated more times
  to obtain a better ``pure" data set.
  In Step two, the range for $\lambda$ can also be obtained from quantiles
  of $\{|\hat\epsilon_{i}^{(OLS,2)}|\}$.
  We can also hybrid quantities based on $\hat\sigma_{pure}$ and quantiles of $\{|\hat\epsilon_{i}^{(OLS,2)}|\}$
  to determine $[\lambda_{L}, \lambda_{U}]$.

The good performance of this data-driven regularization parameter will be demonstrated in
Subsection \ref{sec6.2}.

\section{Diverging number of structural parameters}
\label{sec:Diverging number of structural parameters}
In Sections
\ref{sec2} and
\ref{sec3},
we have considered model (\ref{eq2.1})
under the assumption that
the number of covariates $d$ is a fixed integer.
However, when there are a moderate or large number of covariates,
it is appropriate to assume that $d$ diverges to infinity with the sample size.
In this section,
we consider model (\ref{eq2.1})
with the assumption that $d\rightarrow \infty$ and $d\ll n$.

Since the number of covariates grows orderly slower than the sample size,
we chose to continue use the penalized estimation
(\ref{eq2.3})
for $(\bd{\mu}^{\star},\bd{\beta}^{\star})$
and the penalized two-step estimation
(\ref{eq3.4})
for $\bd{\beta}^{\star}$.
The corresponding estimators are still denoted as
$(\hat{\bd{\mu}},\hat{\bd{\beta}})$
and
$\tilde{\bd{\beta}}$,
but we should keep it in mind that
their dimensions diverge to infinity with $n$.
The characterizations of $\hat{\bd{\beta}}$ in Lemmas
\ref{lem2.1}
and
\ref{lem2.3}
are still valid since they are finite-sample results.
The iteration algorithm also wpg1 stops at the second iteration,
which is shown by Proposition
\ref{prop2.2:infty}
in Supplement \ref{PC.Paper.Supplement.Extension.d.infty}.

As before, it is critical to properly specify the regularization parameter $\lambda$
for the case with a diverging number of covariates.
\begin{description}
  \item[Assumption (B')]
  The regularization parameter $\lambda$ satisfies
  \begin{equation}
\label{eq5.1}
\sqrt{d}\kappa_{n} \ll  \lambda,~  \alpha\gamma_{n} \leq \lambda, \text{ and } \lambda \ll  \mu^{\star},
\end{equation}
where
$\kappa_{n}$ and $\gamma_{n}$
are defined in (\ref{eq2.2})
and
$\alpha>2$.
\end{description}
Comparing Assumption (B') with Assumption (B),
the main difference in formation is that $\kappa_{n}$ is changed to $\sqrt{d}\kappa_{n}$.
In fact, $\kappa_{n}$ in (\ref{eq5.1})
also depends on $d$, which will be shown
in Supplement
\ref{PC.Paper.Supplement.Typical.Covariates.Errors}.
This difference is caused by the assumption that $d$ diverges to $\infty$.
\begin{lem}[On Index Sets $S_{ij}$'s]
\label{PC.Paper.d.infty.index.sets.consistency}
Under Assumptions (A) and (B'),
the conclusion of Lemma \ref{lem3.1} holds.
\end{lem}
Thus, wpg1,
still valid is the crucial analytic expression of $\hat{\bd{\beta}}$
(\ref{eq3.2}),
from which we derive its theoretical properties.
They are similar to those in the previous section,
with additional technical complexity caused by the diverging dimension $d$.

Denote
$\normFd{\cdot}=d^{-1/2}\normF{\cdot}$,
where $\normF{\cdot}$ is the Frobenius norm,
and $\kappa_{X}=d^{-1} \sum_{j=1}^{d}(\E[X_{0j}^{4}])^{1/2}$,
the average of the square root of the fourth marginal moments of $\bd{X}_{0}$.
We make the following assumptions on $\bd{\Sigma}_{X}$ and $\kappa_{X}$.
\begin{description}
  \item[Assumption (D):] $\normFd{\bd{\Sigma}_{X}^{-1}}$ is bounded.
  \item[Assumption (E):]  $\kappa_{X}$ is bounded.
\end{description}

\begin{thm}[Existence and Consistency on $\hat{\bd{\beta}}$]
\label{thm5.2}
Suppose  Assumptions (A), (B'), (D) and (E) hold.
If there exists $r_d$, a sequence of positive numbers depending on $d$,  such that
$d^{3}/n\rightarrow 0$,
$(r_dd)^{2}/n\rightarrow 0$,
$s_{1}=o(n/(r_d\sqrt{d}\kappa_{n}\lambda))$
and
$s_{2}=o(n/(r_d\sqrt{d}\kappa_{n}\gamma_{n}))$,
then,
for every fixed $C>0$,
wpg1,
there exists a unique estimator $\hat{\bd{\beta}}\in B_{C}(\bd{\beta}^{\star})$
such that
$\psi_{n}(\hat{\bd{\beta}})=0$
and
$r_d \| \hat{\bd{\beta}} -  \bd{\beta}^{\star}\|_{2} \ConvProb 0$.
\end{thm}

Next, we consider the asymptotic distribution on $\hat{\bd{\beta}}$.
Since the dimension of $\hat{\bd{\beta}}$ diverges to infinity,
following \cite{Fan2011b}, it is more appropriate to study its linear maps.
Let
$\bd{A}_{n}$ be a $q\times d$ matrix,
where $q$ is a fixed integer,
$\bd{G}_{n}=\bd{A}_{n}\bd{A}_{n}^{T}$
with the largest eigenvalue
$\lambda_{\max}(\bd{G}_{n})$,
and
$\bd{G}_{X,n} = \bd{A}_{n} \bd{\Sigma}_{X}^{-1} \bd{A}_{n}^{T}$.
Denote by $\lambda_{\min}(\bd{\Sigma}_{X})$
the smallest eigenvalue of $\bd{\Sigma}_{X}$,
$\sigma_{X,\max}^{2}=\max_{1\leq j\leq d}\Var[X_{0j}]$,
$\sigma_{X,\min}^{2}=\min_{1\leq j\leq d}\Var[X_{0j}]$
and
$\gamma_{X,\max}=\max_{1\leq j\leq d}\E|X_{0j}|^{3}$.
Abbreviate ``with respect to" by ``wrt".
We assume further
\begin{description}\itemsep -0.05in
  \item[Assumption (D'):] $\lambda_{\min}(\bd{\Sigma}_{X})$ is bounded away from zero, which implies Assumption (D).
  \item[Assumption (D''):] $\normFd{\bd{\bd{\Sigma}}_{X}}$ is bounded.

   \item[Assumption (F):] $\normF{\bd{A}_{n}}$ and
  $\lambda_{\max}(\bd{G}_{n})$ are bounded and
  $\bd{G}_{X,n}$ converges to a $q\times q$ symmetric matrix $\bd{G}_{X}$ wrt $\normF{\cdot}$.
  \item[Assumption (G):] $\sigma_{X, \max}>0$; $\sigma_{X, \max}$ and $\gamma_{X, \min}$ are bounded from above
and $\sigma_{X,\min}$ is bounded away from zero.
\end{description}

Similar to the main case of
Theorem \ref{thm3.4},
a properly scaled $\hat{\bd{\beta}}_{n}$
is asymptotically Gaussian.

\begin{thm}[Asymptotic Distribution on $\hat{\bd{\beta}}$]
\label{thm3.4:d:infty:alternative:Paper}
Suppose Assumptions (A), (B'), (D'), (D''), (E), (F) and (G) hold.
If $d^{5}\log d = o(n)$, $s_{1}=o(\sqrt{n}/(\lambda \sqrt{d}\kappa_{n}))$ and $s_{2}=o(\sqrt{n}/(\sqrt{d}\kappa_{n}\gamma_{n}))$,
then
$
\sqrt{n}\bd{A}_{n}(\hat{\bd{\beta}} - \bd{\beta}^{\star})
\ConvDist
N(0, \sigma^{2}\bd{G}_{X})
$.
\end{thm}

The penalized estimator $\hat{\bd{\mu}}$ obtained
by (\ref{PC.Paper.expression.hat.mu}) is partially consistent.
\begin{thm}[Partial Selection Consistency on $\hat{\bd{\mu}}$]
\label{thm5.4}
Suppose Assumptions (A) and (B')hold
and $\hat{\bd{\beta}}$ is a consistent estimator of $\bd{\beta}^{\star}$
wrt $r_d\normE{\cdot}$.
If $r_d\geq 1/\sqrt{d}$,
then
$
P(\mathcal{E}) \rightarrow 1.
$
\end{thm}

We can construct the penalized two-step estimator
$\tilde{\bd{\beta}}$
through
(\ref{eq3.4})
with $\hat{\bd{\mu}}$.
This two-step estimator is consistent
by Theorem \ref{thm5.5} in Supplement \ref{PC.Paper.Supplement.Extension.d.infty}
and its asymptotic distribution,
as an extension of the main case in Theorem \ref{thm3.8},
is given by the following theorem.
\begin{thm}
[Asymptotic Distribution on $\tilde{\bd{\beta}}$]
\label{thm3.8:d:infty:alternative:Paper}
Suppose
all the assumptions and conditions
of Theorem
\ref{thm3.4:d:infty:alternative:Paper} hold
except that the condition on $s_{1}$ is not required.
Then
$
\sqrt{n}\bd{A}_{n}(\tilde{\bd{\beta}} - \bd{\beta}^{\star} )
\ConvDist
N(0, \sigma^{2}\bd{G}_{X})
$.
\end{thm}

From Theorems
\ref{thm3.4:d:infty:alternative:Paper}
and
\ref{thm3.8:d:infty:alternative:Paper},
Wald-type asymptotic confidence regions of $\bd{\beta}^{\star}$ are availabe.
For example,
a confidence region based on $\tilde{\bd{\beta}}$
with asymptotic confidence level $1-\alpha$  is given by
\begin{equation}
\label{eq5.2:alternative}
\{\bd{\beta}\in\mathbb{R}^{d}: \sigma^{-1}\sqrt{n}\normE{\bd{G}_{X,n}^{-1/2}\bd{A}_{n}(\tilde{\bd{\beta}}-\bd{\beta})} \leq q_{\alpha}(\chi_{q}) \}.
\end{equation}
Since $\bd{G}_{X,n}$ involves the unknown $\bd{\Sigma}_{X}$,
we estimate it by $\hat{\bd{G}}_{X,n}=\bd{A}_{n}\hat{\bd{\Sigma}}_{X}^{-1}\bd{A}_{n}^{T}$.
On the other hand, $\sigma$ is estimated by $\hat{\sigma}$ in (\ref{eq.sigma.hat}) as before.
After plugging
$\hat{\bd{G}}_{X,n}$ and $\hat{\sigma}$
into (\ref{eq5.2:alternative}),
we obtain
\begin{equation}
\label{eq5.2:alternative:plug:in}
\{\bd{\beta}\in\mathbb{R}^{d}: \hat{\sigma}^{-1}\sqrt{n}\normE{\hat{\bd{G}}_{X,n}^{-1/2}\bd{A}_{n}(\tilde{\bd{\beta}}-\bd{\beta})} \leq q_{\alpha}(\chi_{q}) \}.
\end{equation}
By Lemma \ref{lem3.9:d:infty:Paper} in the appendix, the consistency of $\hat\sigma$ is assured.
Then, Theorem \ref{model:three:multivariate:mean:zero:thm:twostage:limitdistribution.general.covariates.estimated.Sigma.alternative.errors:d:infty:Paper}
in the appendix
guarantees the asymptotic validity of the confidence region
(\ref{eq5.2:alternative:plug:in}).

\section{Numerical Evaluations and Real Data Analysis}
\label{sec6}
In this section,
we evaluate the finite-sample performance of the penalized estimation through simulations
and use it to analyze a real data set.
The model for simulations
is given by, for $i=1, \cdots, n$,
$Y_{i} = \mu_{i}^{\star} + X_{i,1}\beta_{1}^{\star} + \cdots + X_{i,50}\beta_{50}^{\star} + \epsilon_{i}$.
For simplicity,
the deterministic sparse incidental parameters $\{\mu_{i}^{\star}\}$
are generated as i.i.d. copies of $\mu$:
$\mu$ is 0, $U[-c,c]$ and $W(c+\text{Exp}(\tau))$ with probabilities $p_{0}$, $p_{1}$, and $p_{2}$ respectively,
where
$U[-c,c]$ is a uniform random variable in $[-c,c]$,
$W$ takes values $1$ and $-1$ with probabilities $p_{w}$ and $1-p_{w}$, respectively,
and $\text{Exp}(\tau)$ follows an exponential distribution with mean $1/\tau>0$.
Note that $c$ can be viewed as a contamination parameter.
The larger $c$ is, the more contaminated the data.
On the other hand, $p_{w}$ determines the asymmetry of the incidental parameters.
The regression coefficients
$\beta_{1}^{\star} = \cdots = \beta_{50}^{\star} = 1$;
$\{(X_{i,1}, \cdots, X_{i,50})\} \overset{i.i.d.}{\sim} N(0,\bd{\Sigma}_{X})$,
where $\bd{\Sigma}_{X}(i,j) = 2\exp(-\abs{i-j}))$,
which is a Toeplitz matrix and the constant 2 is used to inflate the covariance a little;
the covariates are independent of
$\{\epsilon_{i}\} \overset{i.i.d.}{\sim} N(0,1)$; and $n = 500$;
$p_{0}=0.8$, $p_{1}=0.1$, $p_2 = 0.1$,
$c$ is 0.5, 1, 3 or 5, $p_{w}$ is 0.5 or 0.75 and $\tau=1$.

\subsection{Performance of Penalized Methods} \label{sec6.1}
The following methods for estimating $\bd{\beta}^{\star}$ are evaluated.
(i) Oracle method (O): an oracle knows the index set $S$ of zero $\mu_{i}^{\star}$'s.
Its performance is used as a benchmark.
(ii) Ordinary least squares method (OLS): all $\mu_{i}^{\star}$'s are thought as zeros.
(iii) Four penalized least squares (PLS) methods, namely,
PLS with soft penalty (PLS.Soft or S),
PLS with hard penalty (PLS.Hard or H),
two-step PLS with soft penalty (PLS.Soft.TwoStep or S.TS) and
two-step PLS with hard penalty (PLS.Hard.TwoStep or H.TS).
More specifically,
the oracle estimator of $\bd{\beta}^{\star}$ is given by
$\hat{\bd{\beta}}^{(O)}
=
(\sum_{i\in S}\bd{X}_{i}\bd{X}_{i}^{T})^{-1}\sum_{i\in S}\bd{X}_{i}Y_{i}$.
The hard penalty function is
$p_{\lambda}(|t|) = \lambda^{2} - (|t|-\lambda)^{2}\{|t|<\lambda\}$ (see \cite{Fan2001a}).
Each method is evaluated by
the square root of the empirical mean squared error (RMSE).
Every penalized method
is evaluated with a grid of values for the regularization parameter $\lambda$,
ranging from .5 to 5 by .25.

\begin{figure}
  \centering
  \includegraphics[scale=0.4]{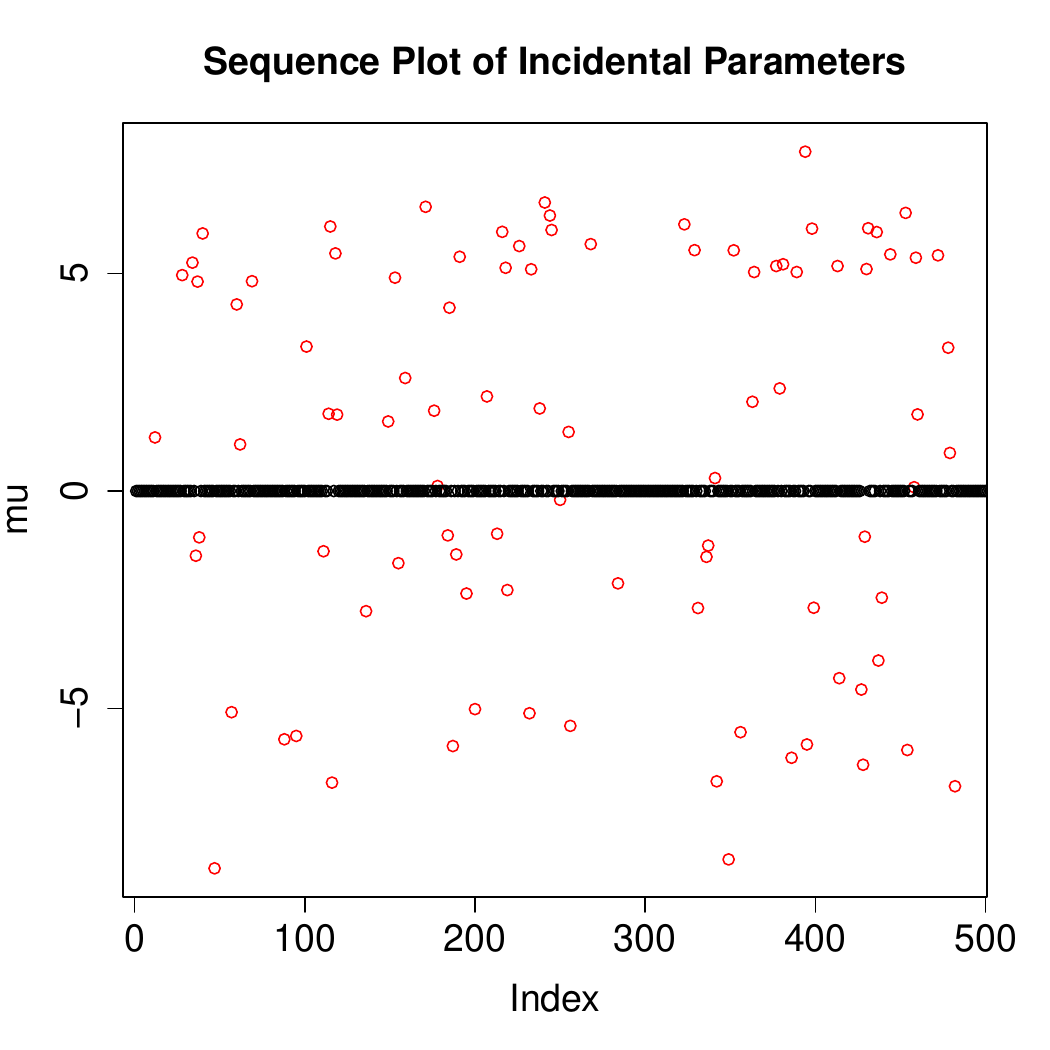}
  \includegraphics[scale=0.4]{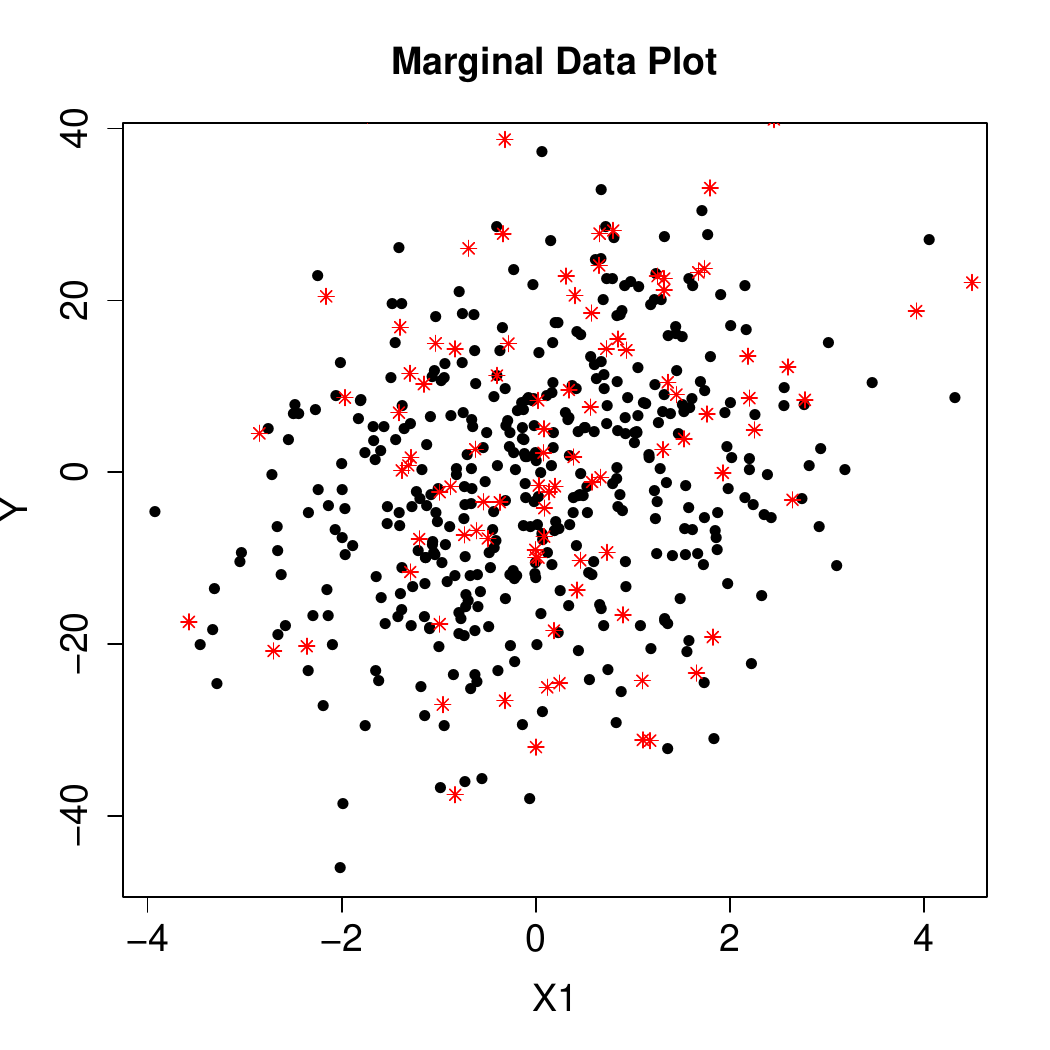}
  \caption
  [The sequence plot of realized intercepts and RMSE of  $\beta_{1}^{\star}$,
    $\beta_{2}^{\star}$ and $\bd{\beta}^{\star}$ for different methods.]
  {\small
    The sequence plot on the left shows 500 incidental parameters $\mu_{i}^{\star}$'s with $c=3$ and $p_{w}=0.75$.
    The non-zero $\mu_{i}^{\star}$'s are in red.
    The scatter plot on the right shows
    the responses $Y_{i}$'s
    against the first covariate $X_{i1}$'s of a data set generated with those 500 incidental parameters.
    The red stars stand for the contaminated sample points, the ones with nonzero $\mu_{i}^{\star}$'s.  }
  \label{fig1}
\end{figure}

\begin{figure}[h]
  \centering
  \includegraphics[scale=0.5]{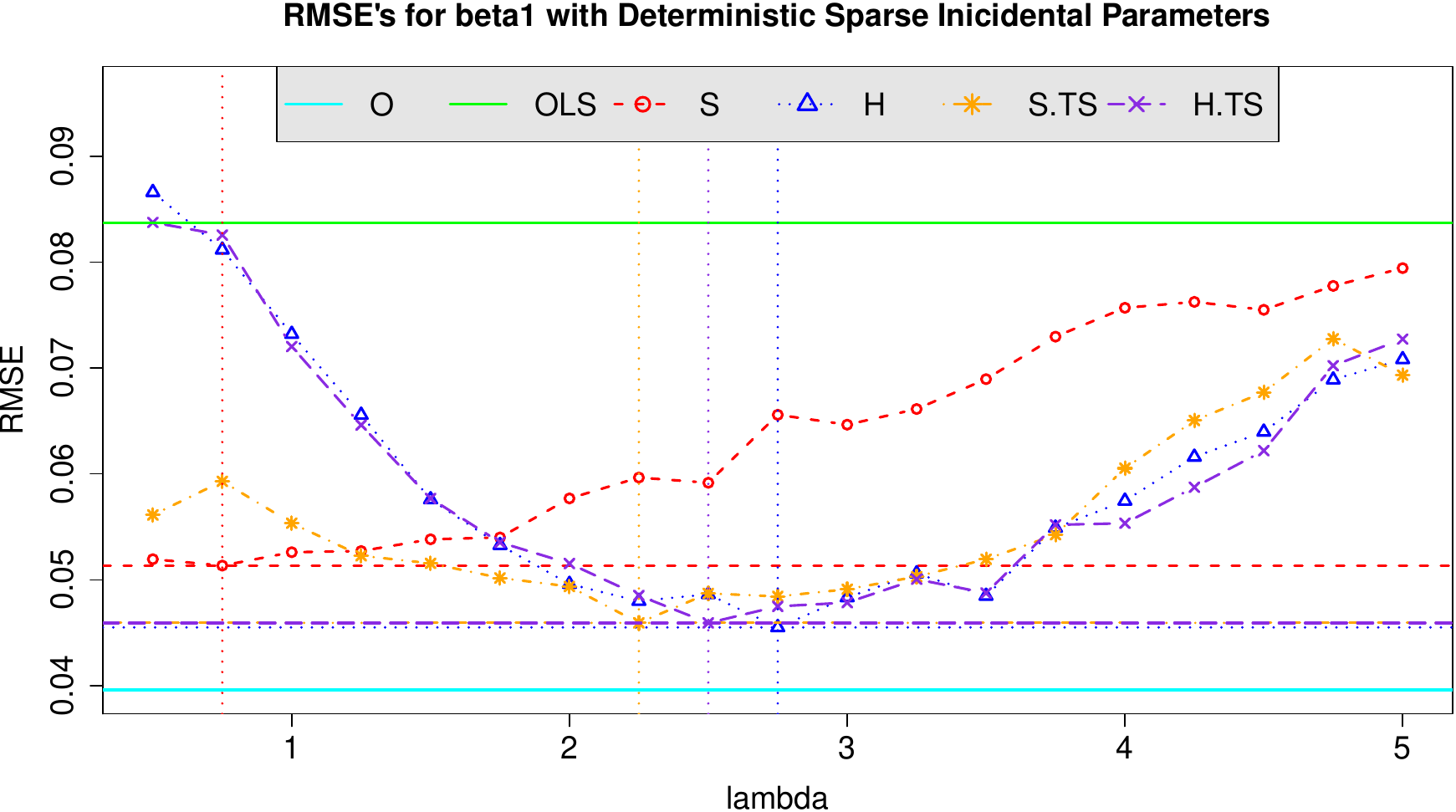}
  \caption[The simulated RMSE of  $\beta_{1}^{\star}$,
    $\beta_{2}^{\star}$ and $\bd{\beta}^{\star}$ for different methods with fixed intercepts.]
  { \small
    RMSE's of the O(Oracle), OLS, S(PLS.Soft), H(PLS.Hard),
    S.TH(PLS.Soft.TwoStep) and H.TH(PLS.Hard.TwoStep)
    estimators of $\beta_{1}^{\star}$
    with the incidental parameters shown in
    Figure \ref{fig1}.
    The top and bottom solid horizontal lines
    show the RMSE's for OLS and O, respectively.
    Other four horizontal lines indicate the minimal RMSE's for those four PLS methods
    and the corresponding four best $\lambda$'s are shown by the vertical lines.}
  \label{fig2}
\end{figure}

The sequence plot of Figure \ref{fig1}
shows $500$ realized incidental parameters $\mu_{i}^{\star}$'s
with $c=3$ and $p_{w}=0.75$.
They are used in data generation of simulations.
The scatter plot of Figure \ref{fig1}
shows the responses $Y_{i}$'s
against the first covariate $X_{i1}$'s of a generated data set
and the red stars stand for the contaminated sample points, the ones with nonzero $\mu_{i}^{\star}$'s.
With fifty covariates, usually it is difficult to graphically identify
the contaminated data points.

With the above incidental parameters,
those six methods are evaluated by simulations
with iteration number 1000.
Since $\bd{\Sigma}_{X}$ is a Toeplitz matrix with equal diagonal elements,
the asymptotic variances of the estimators of $\beta_{1}^{\star}$ and $\beta_{2}^{\star}$ are different
and representative for estimators of other $\beta_{i}^{\star}$'s.
So, we only report simulation results on the estimation of  $\beta_{1}^{\star}$ and $\beta_{2}^{\star}$.

Figure \ref{fig2} shows RMSE's of six estimators for $\beta_{1}^{\star}$.
RMSE's for $\beta_{2}^{\star}$
are similar.
As expected,
the oracle method has the smallest RMSE
and OLS the largest.
RMSE of each PLS method as a function of $\lambda$ forms a convex curve,
which achieves a minimal RMSE significantly below the green line of OLS
and close to the cyan line of O.
More specifically,
RMSE of PLS.Hard achieves the minimal RMSE when $\lambda$ is around 2.75.
On the other hand,
RMSE of PLS.Soft decreases a little till $\lambda$ is around .75,
then increases and stays above RMSE of PLS.Hard.
This reflects the fact that
a large $\lambda$ in a soft-threshold method usually causes bias.
PLS.Hard.TwoStep has very similar performance
with PLS.Hard for all $\lambda$.
PLS.Soft.TwoStep has similar performance with PLS.Soft when $\lambda$ is small.
However,
as $\lambda$ becomes large,
PLS.Soft.TwoStep moves closer to PLS.Hard than PLS.Soft.
It is because PLS.Soft.TwoStep and PLS.Hard.TwoStep have similar estimation
when $\lambda$ is large.
The minimal RMSE of PLS.Soft is slightly larger than those of other PLS Methods.

Table \ref{table 1} depicts
the minimal RMSE's of the estimators for $\beta_{1}^{\star}$ and $\beta_{2}^{\star}$
with the corresponding optimal $\lambda$'s and biases.
The biases are ignorable comparing withe the RMSE's.
The optimal $\lambda$ for PLS.Soft and other PLS methods are around .75 and 2.5, respectively.
This indicates the simple soft threshold method tends to work best with a small $\lambda$
due to the bias issue.
Denote the empirical relative efficiency (ERE) of an estimator $A$
with respect to another estimator $B$
as RMSE($B$)$/$RMSE($A$).
Then, for the estimation of $\beta_{1}^{\star}$,
the ERE's of PLS.Soft, PLS.Hard, PLS.Soft.TwoStep and PLS.Hard.TwoStep
with respect to O
are around $78\%$, $87\%$, $87\%$ and $87\%$, respectively;
the ERE's of the PLS methods
with respect to OLS
are around $176\%$, $196\%$, $196\%$ and $196\%$, respectively.
The ERE's for $\beta_{2}^{\star}$ are similar.
Thus, in terms of ERE (and RMSE),
the PLS methods perform closely to O
and significantly better than OLS.

From Table \ref{table 1},
we can also see that the RMSE's of the estimators for $\beta_{1}^{\star}$
are always smaller than those for $\beta_{1}^{\star}$.
This is because the first covariate is less correlated with others covariates than the second one.

\begin{table}
  \centering
  \begin{tabular}{lccccccccc}
    \hline\hline
         & O & OLS & S & H & S.TS & H.TS & S.P & H.P & LAD  \\ \hline
    Bias($\hat{\beta}_{1}$)$\times 10^{4}$ &  .8 & 17.8 & 1.3 & .6 & 5.9 & 20.4 & 9.5 & 11.5 & 10.4 \\
    RMSE($\hat{\beta}_{1}$)$\times 10^{2}$ &  4.0 & 8.4 & \textbf{5.1} & \textbf{4.6} & \textbf{4.6} & \textbf{4.6} & \textbf{\emph{5.4}}  &  \textbf{\emph{5.0}} & \emph{\textbf{5.4}}\\
    $\lambda$ &   &  & .75 & 2.75 & 2.25 & 2.5 & $\overline{2.45}$  & $\overline{2.47}$ & \\ \hline
    Bias($\hat{\beta}_{2}$)$\times 10^{4}$ &  -2.3 & -46.7 & 24.6 & 43.5 & -21.8 & 10.0 & 32.3 & 13.9 & -7.3\\
    RMSE($\hat{\beta}_{2}$)$\times 10^{2}$ & 4.4 & 9.0 & \textbf{5.3} & \textbf{4.9} & \textbf{5.0} & \textbf{4.9} & \textbf{\emph{5.8}} & \textbf{\emph{5.5}} & \emph{\textbf{5.9}}\\
    $\lambda$ &  &  & .75 & 2.75 & 2.75 & 2.5 & $\overline{2.45}$ & $\overline{2.47}$ & \\ \hline
  \end{tabular}
  \caption[The simulated RMSE of estimators of $\beta_{1}^{\star}$ and $\beta_{1}^{\star}$.]
  {\small
        RMSE's of the Oracle, OLS and LAD estimators and
        The minimal RMSE's of the penalized estimators of $\beta_{1}^{\star}$ and $\beta_{2}^{\star}$
        with the corresponding optimal or data-driven $\lambda$'s and biases
        when the incidental parameters shown in Figure \ref{fig1} are used.
        For a data-driven method,
        the data-driven $\lambda$ is different in each iteration
        so that the reported $\lambda$'s are averages.
        The lines over the numbers emphasize the numbers are averages.
        The standard deviations for the data-driven $\lambda$'s
        of S.P and H.P are
        0.37 and 0.34, respectively.}
  \label{table 1}
\end{table}

\begin{figure}
  \centering
  \includegraphics[scale=0.66]{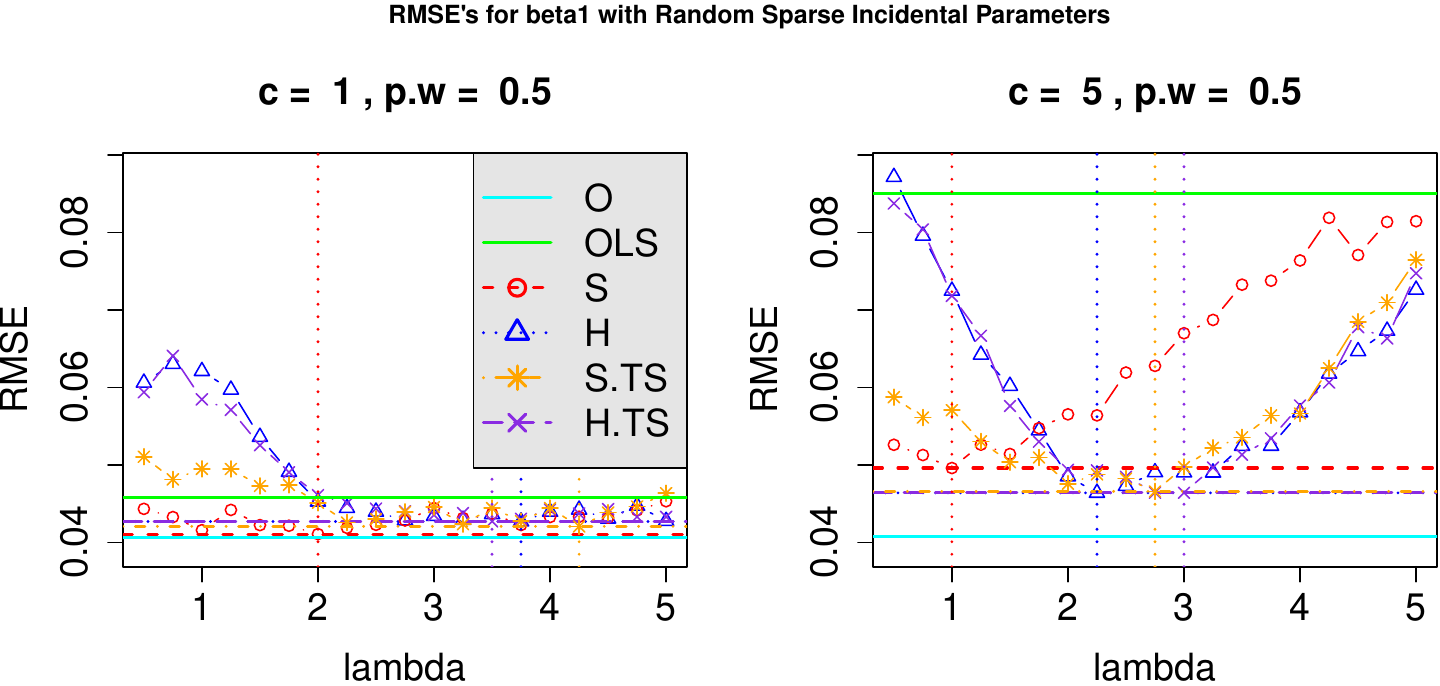}
  \caption[The simulated RMSE of different estimators of $\bd{\beta}^{\star}$ under four model settings.]
    {\small
    Similar to Figure \ref{fig2},
    these plots show the RMSE's of the O(Oracle), OLS,
    S(PLS.Soft), H(PLS.Hard),
    S.TH(PLS.Soft.TwoStep) and H.TH(PLS.Hard.TwoStep)
    estimators of $\beta_{1}^{\star}$
    with randomly generated $\bd{\mu}^{\star}$
    under two settings with
    $p_{w}=0.5$
    and
    $c=1$ or $5$.
    }
  \label{fig3}
\end{figure}

In order to examine the performance of the methods with general incidental parameters
but not just those in Figure \ref{fig1},
we generate $\bd{\mu}^{\star}$ randomly for each iteration.
The iteration number for each simulation is also 1000.

Figure \ref{fig3}
shows the RMSE's of six estimators of $\beta_{1}^{\star}$ under two settings:
$p_{w}=0.5$ and $c=1$ or $5$.
Each plot in Figure \ref{fig3}
presents a similar pattern with Figure \ref{fig2}.
When $p_{w}$ is fixed at 0.5,
the RMSE's of each non-oracle estimator of $\beta_{1}^{\star}$ increases
as the contamination parameter $c$ increases from 1 to 5.
This indicates that
each non-oracle estimator performs worse
as the data becomes more contaminated.
However, the PLS estimators are more robust than OLS,
which is very sensitive to the change of $c$.
We have also done simulations with $p_{w}=0.75$
and the RMSE's of the estimators of $\beta_{1}^{\star}$ are similar to those with $p_{w}=0.75$
so that the corresponding plots are similar to those in Figure \ref{fig3}.
In other words,
the RMSE's of all estimators
are stable with respect to $p_{w}$,
which means the magnitudes of the nonzero incidental parameters matter most but not their signs.
We also note that some penalized methods perform closely to or even outperform
the oracle one when $c$ is small as showed in the plot with $c=1$.
This happens because that O ignores all the contaminated data points,
even those with very light contamination,
but the penalized methods exploit information in such points.

\begin{table}
  \centering
  \begin{tabular}{cccccccccccc}
    \hline \hline
    RMSE$(\hat{\beta_{1}})$$\times 10^{2}$ & O & OLS & S & H & S.TS & H.TS & S.P & H.P & LAD\\ \hline
    $(p_{w},c)=(.5, .5)$ & 4.06 & 4.06 & 3.92 & 3.95 & 3.90 & 3.96 & 4.01 & 4.38 & 4.85 \\
    $\lambda$ &  &  & 1.25 & 3.5 & 2.75 & 2.5 & $\overline{2.08}$ & $\overline{2.15}$  & \\ \hline
    $(p_{w},c)=(.5, 1)$ & 4.06 & 4.58 & 4.11 & 4.27 & 4.21 & 4.27 & 4.01 & 4.59 & 4.75 \\
    $\lambda$ &  &  & 2 & 3.75 & 4.25 & 3.5 & $\overline{2.20}$ &$\overline{2.24}$& \\ \hline
    $(p_{w},c)=(.5, 3)$ & 4.12 & 6.47 & \textbf{4.81} & \textbf{4.99} & \textbf{4.81} & \textbf{4.80} & \emph{\textbf{5.64}} & \emph{\textbf{5.19}} & \emph{\textbf{5.49}}\\
    $\lambda$ &  &  & 1 & 2.5 & 2 & 2.25 & $\overline{2.42}$ &$\overline{2.45}$& \\ \hline
    $(p_{w},c)=(.5, 5)$ & 4.07 & 8.50 & \textbf{4.96} & \textbf{4.63} & \textbf{4.66} & \textbf{4.64} & \emph{\textbf{6.36}} & \emph{\textbf{4.87}} & \emph{\textbf{5.52}}\\
    $\lambda$ &  &  & 1 & 2.25 & 2.75 & 3 & $\overline{2.52}$ &$\overline{2.58}$& \\ \hline
    $(p_{w},c)=(.75, .5)$ & 4.13 & 4.02 & 3.91 & 3.88 & 3.98 & 3.91 & 4.08 & 4.41 & 4.79\\
    $\lambda$ &  &  & 1.5 & 5 & 3.25 & 4.5 & $\overline{2.08}$ &$\overline{2.14}$& \\ \hline
    $(p_{w},c)=(.75, 1)$ & 4.17 & 4.41 & 4.19 & 4.15 & 4.15 & 4.20 & 4.16 & 4.59 & 4.98\\
    $\lambda$ &  &  & 3.25 & 3 & 4 & 3.5 & $\overline{2.15}$ &$\overline{2.23}$& \\ \hline
    $(p_{w},c)=(.75, 3)$ & 4.15 & 5.99 & \textbf{4.91} & \textbf{4.93} & \textbf{4.80} & \textbf{5.02} & \emph{\textbf{5.35}} & \emph{\textbf{5.06}} & \emph{\textbf{5.52}}\\
    $\lambda$ &  &  & 1 & 2.25 & 2 & 2.5 & $\overline{2.44}$  &$\overline{2.47}$& \\ \hline
    $(p_{w},c)=(.75, 5)$ & 3.97 & 8.41 & \textbf{5.01} & \textbf{4.66} & \textbf{4.75} & \textbf{4.66} & \emph{\textbf{6.22}} &  \emph{\textbf{4.90}} & \emph{\textbf{5.76}}\\
    $\lambda$ &  &  & 1.5 & 2.25 & 2.5 & 3 & $\overline{2.55}$ &$\overline{2.60}$& \\ \hline
  \end{tabular}
  \caption[The simulated RMSE of $\hat{\bd{\beta}}$
        for different methods under eight different settings.]
       {\small
       Similar to Table \ref{table 1},
       this one shows
       the RMSE's and minimal RMSE's of nine estimators of $\beta_{1}^{\star}$
       under eight settings on randomly generated $\bd{\mu}^{\star}$ with $p_{w}=0.5, 0.75$ and $c=0.5, 1, 3, 5$.
       The standard deviations of the data-driven $\lambda$'s of S.P and H.P for different settings are between 0.2 and 0.45.
       }
  \label{table 2}
\end{table}

Table \ref{table 2} contains the RMSE's of the estimators of $\beta_{1}^{\star}$ under eight settings with
$p_{w}=0.5$ or $0.75$ and $c=0.5, 1, 3$ or $5$.
For each $p_{w}$, as $c$ increases from 0.5 to 5,
the RMSE's of O with the multiplication factor $10^{2}$ is almost constantly around 4,
those of OLS increases from about 4 to 8.5
and those of PLS ones grow from about 4 to 5,
which confirms the robustness of the PLS estimators.
When $c\leq 1$ is small with respect to the variance of random error $\sigma=1$,
the data points are only slightly contaminated.
OLS and PLS methods perform similar to O.
However, when $c\geq 3$ is large,
which means the data are more contaminated,
the RMSE's of OLS become significantly larger,
but PLS methods perform still closely to O.

\subsection{Performance of Data-Driven Penalized Methods}
\label{sec6.2}
Previous simulations have shown
the PLS methods
with optimal $\lambda$'s
have good RMSE's
comparing with those of the Oracle and OLS ones.
In practice, however,
these optimal $\lambda$'s are unknown.
One approach to obtain a data driven $\lambda$
has been introduced in Subsection \ref{subsec:Regularization Parameters}.
Since, as shown in the previous simulation results,
the two-step PLS methods
perform similarly with the one-step PLS methods, i.e. PLS.Soft and PLS.Hard,
only the latter
are studied by simulations
with data-driven $\lambda$'s
and denoted as PLS.Soft.Prac (S.P) and PLS.Hard.Prac (H.P), respectively.
For estimating data-driven $\lambda$'s,
$\alpha_{l}=2$ and  $\alpha_{u}=7$.
The size of the pure data set $n_{pure}$ is $n/2$
and that of the testing data set is $n_{pure}/2$.

Simulations are first run with the deterministic sparse incidental parameters
as showed in Figure \ref{fig1}.
We can see in Table \ref{table 1} that
the RMSE's of estimators of $\beta_{1}^{\star}$ from PLS.Soft.Prac and PLS.Hard.Prac
are around 5.4 and 5.0,
slightly larger than the optimal values 5.1 and 4.6.,respectively.
However,
they are still significantly smaller than
RMSE of OLS, which is 8.4.
The observations of the estimators of $\beta_{2}^{\star}$ are similar.
As before, we also evaluate
the performance of the data-driven PLS methods
with random sparse incidental parameters.
Table \ref{table 2} shows that, for a given $p_{w}$,
when $c$ is small such as 0.5 and 1,
the RMSE's of PLS.Soft.Prac and PLS.Hard.Prac
are close to those of PLS.Soft and PLS.Hard
with the optimal $\lambda$'s.
In these cases, PLS.Soft.Prac performs
slightly better than PLS.Hard.Prac,
and even better than PLS.Soft with the optimal $\lambda$ and the Oracle method.
On the other hand, for a given $p_{w}$,
when $c$ is large such as 3 and 5,
the RMSE's of PLS.Soft.Prac and PLS.Hard.Prac
are greater than those of PLS.Soft and PLS.Hard, respectively,
but still less than those of OLS.
In these cases, the RMSE's of PLS.Soft.Prac are
larger than those of PLS.Hard.Prac,
which indicates the bias issue
of the soft threshold method.
Thus, the data-driven regularization parameter works well
with penalized estimation.
When the data is slightly contaminated,
the soft penalty is preferred;
otherwise, the hard penalty is recommended.

Tables \ref{table 1} and \ref{table 2} also contain
RMSE of the least absolute deviation regression method (LAD) used in \cite{Fan2012a}
with \emph{all} but not part of the sample points with small residuals.
Generally speaking, in both deterministic and random incidental parameter cases,
LAD performs similarly with the PLS methods with data-driven $\lambda$'s.
More specifically, when $c\leq 1$ is small, PLS.Soft.Prac outperform LAD;
otherwise, LAD performs better.
For all the cases, LAD is dominated by PLS.Hard.Prac.
These observations confirm that LAD is an effective robustness method
and the penalized methods make improvement.


\subsection{Data-Driven Confidence Intervals}
\label{sec:Data-Driven Confidence Intervals}

We next turn to investigate the finite-sample performance of
the asymptotic confidence interval (CI)
(\ref{PC.Paper.PLS.Soft.TwoStage.Beta.i.Confidence.Interval})
for $\beta_{j}^{\star}$ with $j=1,2$
based on PLS two-step methods.
Since CI (\ref{PC.Paper.PLS.Soft.TwoStage.Beta.i.Confidence.Interval})
is based on the properties of the penalized two-step estimator with the soft penalty,
we focus on PLS.TS.Soft
with a data-driven regularization parameter $\lambda$.
The choice of $\lambda$ in Subsection \ref{subsec:Regularization Parameters}
for minimizing RMSE
is usually no longer suitable for
constructing
confidence intervals,
since it is designed to achieve minimal RMSE.
We propose to first obtain $\hat{\sigma}_{pure}$ as in the data-driven procedure in
Subsection \ref{subsec:Regularization Parameters}
and then simply set the data-driven $\lambda$ be five times of $\hat{\sigma}_{pure}$.
Since $\hat{\sigma}_{pure}$ tends to underestimate $\sigma$,
this data-driven $\lambda$ is usually not large
with respect to $\sigma$.
Denote this method as PLS.TwoStage.Soft.Prac or S.TS.P.
After plugging in $\hat{\sigma}$ and $\hat{\sigma}_{j}^{-1}$, the square root of the $(j, j)$th element of $\hat{\bd{\Sigma}}_{X}^{-1}$,
and replace $n$ by $m=\#(\hat{I}_{0})$ in the theoretical CI (\ref{PC.Paper.PLS.Soft.TwoStage.Beta.i.Confidence.Interval}),
we obtain a data-driven CI
$[\tilde{\beta}_{j} \pm m^{-1/2}\hat{\sigma}\hat{\sigma}_{j}^{-1}z_{\alpha/2}]$,
where $\tilde{\beta}_{j}$ is the PLS.TwoStage.Soft.Prac estimator of $\beta_{j}^{\star}$ for each $j$.

This data-driven CI
is compared with CI's based on Oracle and OLS methods.
More specifically,
denote the Oracle and OLS estimators of $\beta_{j}^{\star}$
as $\hat{\beta}_{j}^{(O)}$ and $\hat{\beta}_{j}^{(OLS)}$, respectively.
Then, the corresponding CI's are given by
$[\hat{\beta}_{j}^{(O)} \pm m_{o}^{-1/2}\hat{\sigma}^{(O)}\hat{\sigma}_{j}^{-1}z_{\alpha/2}]$ and
$[\hat{\beta}_{j}^{(OLS)} \pm n^{-1/2}\hat{\sigma}^{(OLS)}\hat{\sigma}_{j}^{-1}z_{\alpha/2}]$,
where $m_{o}$ is the number of zero incidental parameters
and $\hat{\sigma}^{(O)}$ and $\hat{\sigma}^{(OLS)}$ are the estimators of $\sigma$
from O and OLS methods, respectively.

The simulation settings are the same to the previous ones with \emph{deterministic} sparse incidental parameters
except the following changes.
(a) The number of covariates $d$ is reduced to 5 from 50.
This is because when $d=50$ and the nominal level is $95\%$,
even the empirical coverage rate (CR) of the oracle confidence interval for $\beta_{1}^{\star}$ becomes $93.5\%$,
not very close to $95\%$.
(b) The iteration number is increased from 1000 to 10000 to improve the accuracy of CR's.
(c) The probabilities of nonzero incidental parameters
    $(p_{1}, p_{2})$ are set to be $(0.01, 0.01)$, $(0.03, 0.03)$ and $(0.05, 0.05)$;
the contamination parameter $c$ is increased to $10$.
In order to achieve good second order asymptotic approximation,
we can either increase the sample size or enlarge the signal noise ratio.
Here we adopt the latter.

\begin{table}
  \centering
  \begin{tabular}{l||l|ccc||l|ccc}
    \hline \hline
  $(p_{1}, p_{2})$ & $\beta_{1}^{\star}$ & O & OLS & S.TS.P & $\beta_{2}^{\star}$ & O & OLS & S.TS.P\\ \hline
  (.01,.01) &    &  & .950 & .944 &    &  & .948 & .945 \\
  (.03,,03) & CR & .955 & .951 & .948 & CR & .948 & .946 & .947 \\
  (.05,.05) &    &  & .946 & .945 &    &  & .949 & .946 \\ \hline
  (.01,.01) &    &  & .200 & \textbf{.135} &  & & .213 & \textbf{.144} \\
  (.03,,03) & AL & .133 & .279 & \textbf{.137} & AL & .142 & .297 & \textbf{.146} \\
  (.05,.05) &    & & .382 & \textbf{.139} &   & & .407 & \textbf{.149} \\ \hline
  \end{tabular}
  \caption[Coverage rates (CR) and average length (AL) of $95\%$ confidence intervals.]
       {\small Coverage rates (CR) and average length (AL) of $95\%$ confidence intervals
       for $\beta_{1}^{\star}$ and $\beta_{2}^{\star}$
       from O, OLS, PLS.TwoStage.Soft.Prac methods
       under three settings on deterministic sparse incidental parameters.}
  \label{table:CI}
\end{table}

Table \ref{table:CI} reports the empirical coverage rates (CR)
and average lengths (AL)
of the CI's of $\beta_{1}^{\star}$ and $\beta_{2}^{\star}$
from O, OLS and PLS.TS.Soft.Prac methods
under three different settings on the incidental parameters.
For the oracle method,
these three settings are the same
and thus only one set of simulation results are presented.
Table \ref{table:CI} shows that
the CR's of all methods under all settings are close to the nominal level .95.
The OLS treats the deterministic incidental parameters as random ones
and achieves excellent CR's.
However, the AL's of OLS are significantly larger than those of O and PLS.TS.Soft.Prac,
especially when there are more non-zero incidental parameters.
On the other hand, the AL's of PLS.TS.Soft.Prac are only slightly larger than
those of O.
This means PLS.TS.Soft.Prac has excellent efficiency
in terms of AL's given excellent CR's.
Also note that the AL's for $\beta_{1}^{\star}$ are less than those for $\beta_{2}^{\star}$.
This is because the asymptotic variance of $\hat{\beta}_{1}$
is less than that of $\hat{\beta}_{2}$
when the covariance matrix $\bd{\Sigma}_{X}$ is a Toeplitz matrix.
Simulations with random incidental parameters under the same settings
have also been done and
the results are similar to those in Table \ref{table:CI}
with slightly inflated AL's for OLS and PLS.TS.Soft.Prac
due to the randomness of the incidental parameters.

\subsection{Real Data Analysis}
We implement the penalized estimation
with the soft penalty
in the method of estimating false discovery proportion of a multiple testing procedure
proposed by \cite{Fan2012a}
for investigating the association
between the expression level of gene CCT8,
which is closely related to Down Syndrome phenotypes,
and thousands of SNPs.
The data set consists of three populations:
60 Utah residents (CEU),
45 Japanese and 45 Chinese (JPTCHB)
and
60 Yoruba (YRI).
More details on the data set can be found
in
\cite{Fan2012a}.

\begin{figure}
  \centering
  \includegraphics[scale=0.3]{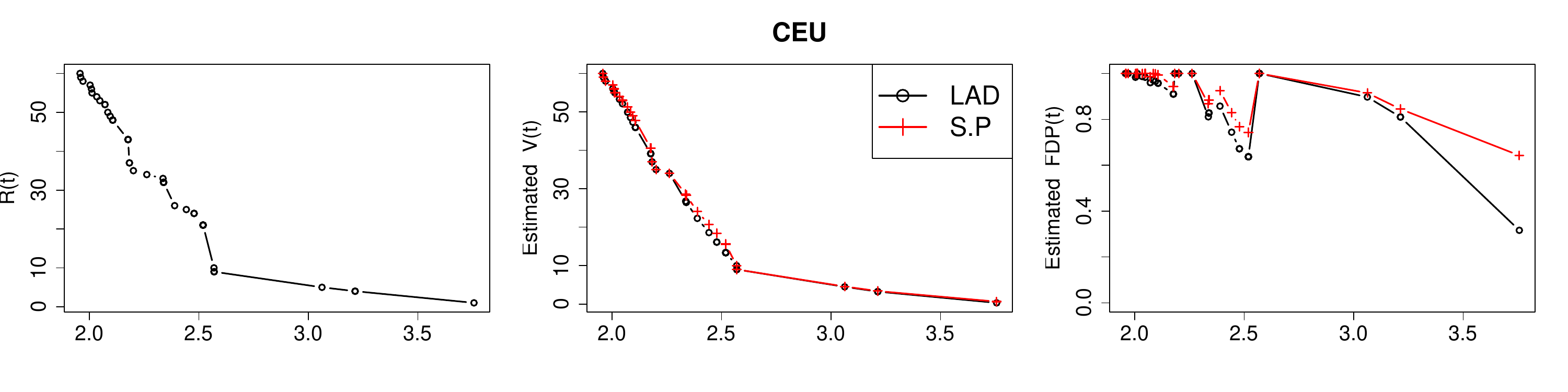} \\
  \includegraphics[scale=0.3]{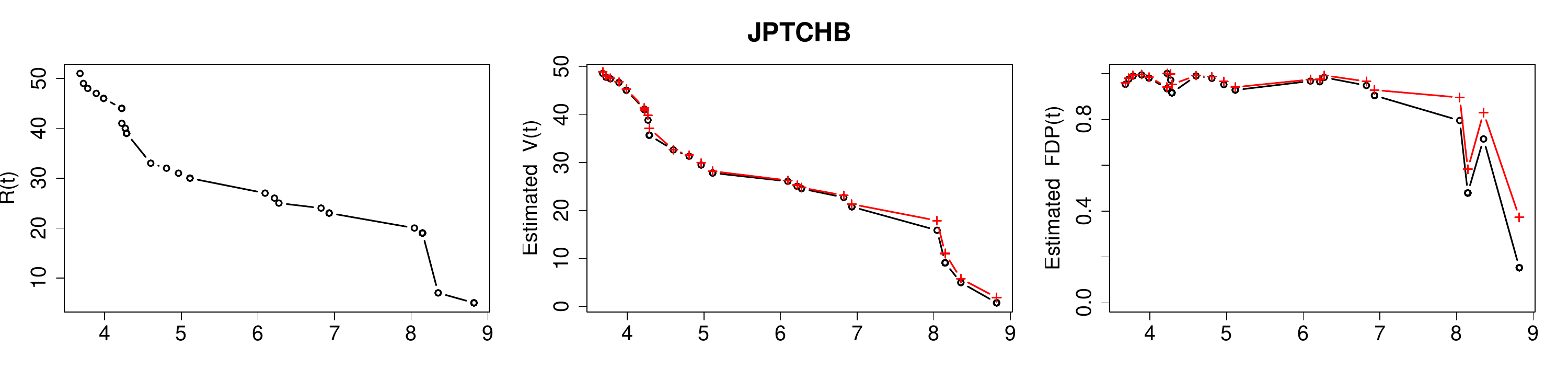} \\
  \includegraphics[scale=0.3]{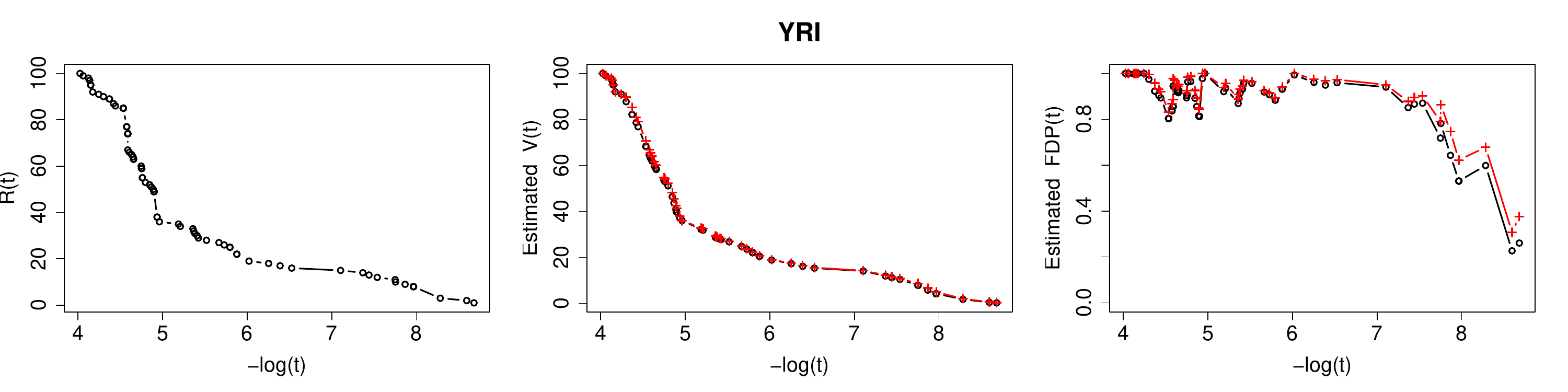}
  \caption
  [Discovery number $R(t)$,
  estimated false discovery number $V(t)$
  and estimated false discovery proportion $\text{FDP}(t)$]
  {\small Discovery number $R(t)$,
  estimated false discovery number $V(t)$
  and estimated false discovery proportion $\text{FDP}(t)$
  as functions of a threshold on $t$
  for populations CEU, JTPCHB and YRI.
  The $x$-axis is $-\log_{10}(t)$.}
  \label{fig.real.data}
\end{figure}

In the testing procedure by \cite{Fan2012a},
a filtered least absolute deviation regression (LAD) is used
to estimate the loading factors
with $90\%$ of the cases (SNPs)
whose test statistics are small
and thus the resulting estimator is statistically biased.
We upgrade this step
with S.P
described in Subsections
\ref{subsec:Regularization Parameters}
and
\ref{sec6.2}
and
re-estimate the number of false discoveries $V(t)$
and
the false discovery proportion $\text{FDP}(t)$
as functions of $-\log_{10}(t)$,
where $t$ is a thresholding value.
Figure \ref{fig.real.data}
shows the number of total discoveries $R(t)$,
$\hat{V}(t)$ and $\widehat{\text{FDP}}(t)$
from procedures using filtered LAD and S.P.
It is clear that
$\hat{V}(t)$ and $\widehat{\text{FDP}}(t)$
with S.P
are uniformly larger than
but reasonably close to
those with filtered LAD.
Table \ref{table.real.data}
contains $R(t)$ and
$\widehat{\text{FDP}}(t)$
with filtered LAD and S.P
for several specific thresholds.
The estimated $\text{FDP}$s
with S.P for CEU and YRI
are slightly larger than
those with LAD
and $\widehat{\text{FDP}}$
for JPTCHB with S.P is more than double of
that with filtered LAD.
This suggests that the estimation of $\text{FDP}$ with filtered LAD
might tend to be optimistic.

\begin{table}
  \centering
  \begin{tabular}{ccccc}
    \hline
    Population & $t$ & $R(t)$ & $\widehat{\text{FDP}}(t)$ with LAD & $\widehat{\text{FDP}}(t)$ with S.P \\ \hline
    CEU & $6.12 \times 10^{-4}$ & 4 & .810 & .845 \\
    JPTCHB & $1.51 \times 10^{-9}$ & 5 &  .153 & .373 \\
    YRI & $2.54 \times 10^{-9}$ & 2 & .227 & .308 \\ \hline
  \end{tabular}
  \caption[The simulated RMSE ]
  {\small Discovery numbers $R(t)$ and estimated false discover proportions $\hat{\text{FDP}}(t)$s from
  methods with LAD and S.P
  for specific values of threshold $t$.}
  \label{table.real.data}
\end{table}

\section{Conclusion and Discussion}
\label{PC.Paper.Conclusion}
This paper considers the estimation
of structural parameters
with a finite or diverging number of covariates
in a linear regression model with the presence of
high-dimensional sparse incidental parameters.
By exploiting the sparsity,
we propose an estimation method
penalizing the incidental parameters.
The penalized estimator of
the structural parameters
is consistent
and asymptotically Gaussian
and achieves
an oracle property.
On the contrary,
the penalized estimator of
the incidental parameters
possesses only partial selection consistency
but not consistency.
Thus, the structural parameters are consistently estimated
while the incidental parameters not,
which presents
a partial consistency phenomenon.
Further, in order to construct better confidence regions for the structural parameters,
we propose
a two-step estimator,
which has fewer possible asymptotic distributions
and can be
asymptotically even more efficient than the one-step penalized estimator
when the size and magnitude of nonzero incidental parameters are substantially large.

Simulation results show
that the penalized methods
with best regularization parameters
achieve significantly smaller mean square errors than
the ordinary least squares method which ignores the incidental parameters.
Also provided
is a data-driven regularization parameter,
with which
the penalized estimators
continue to significantly outperform
ordinary least squares
when the incidental parameters are too large to be neglected.
In terms of average length together with excellent coverage rates,
the advantage of the confidence intervals based on
the two-step estimator with an alternative data-driven regularization parameter
is verified by simulations.
A data set on genome-wide association
is analyzed
with a multiple testing procedure
equipped with a data-driven penalized method
and
false discovery proportions
are estimated.

In econometrics, a fixed effect panel data model is given by,
for $1\leq i\leq n$ and $1\leq t\leq T$,
\begin{equation}\label{model:discussion:fixed effect panel data}
    Y_{it} = \mu_{i}^{\star} + \bd{X}_{it}^{T}\bd{\beta}^{\star} + \epsilon_{it},
\end{equation}
where $\mu_{i}^{\star}$'s are unknown fixed effects.
When $T$ diverges, the fixed effects can be consistently estimated.
When $T$ is finite and greater than or equal to 2,
although the fixed effects can no longer be consistently estimated,
they can be removed by a within-group transformation: for each $i$,
$Y_{it} - \bar{Y}_{i} = (\bd{X}_{it}-\bar{\bd{X}}_{i})^{T}\bd{\beta}^{\star} + \epsilon_{it}-\bar{\epsilon}_{i}$,
where $\bar{Y}_{i}$, $\bar{\bd{X}}_{i}$ and $\bar{\epsilon}_{i}$
are the averages of $Y_{it}$'s, $\bd{X}_{it}$'s and $\epsilon_{it}$'s, respectively.
When $T$ is equal to 1, however, the within-group transformation fails.
Note that, with $T=1$, Model (\ref{model:discussion:fixed effect panel data})
becomes Model (\ref{model:basic})
so that the proposed penalized estimations provide a solution
under the sparsity assumption on the fixed effects.

Although this paper only illustrates
the partial consistency phenomenon
of a penalized estimation method
for a linear regression model,
such a phenomenon
shall universally
exist for a general parametric model,
which contains both
a structural parameter
and a high-dimensional sparse incidental parameter.
For example, consider a panel data logistic regression model:
$P(Y_{it}=1|\bd{X}_{it})=(1+\exp\{-(\mu_{i}^{\star} + \bd{X}_{it}^{T}\bd{\beta}^{\star})\})^{-1}$.
When $T$ is finite, the fixed effects $\mu_{i}^{\star}$'s cannot be removed by
the within-group transformation as in the panel data linear model (\ref{model:discussion:fixed effect panel data}).
However, the proposed penalized estimations can still provide a solution.

Further, if the structural parameter has a dimension diverging
faster than the sample size
and is sparse,
it is expected that the partial consistency phenomenon
will continue to appear
when sparsity penalty is imposed
on both the structural and incidental parameters.

\section*{Acknowledgement}
The authors thank the Editor, an associate editor, and three referees for their many helpful comments that have resulted in significant improvements in the article.
The research was partially supported by NSF grants DMS-1206464 and DMS-0704337 and NIH Grants R01-GM100474-01 and
R01-GM072611-08.

\appendix

\section{Appendix}
In this appendix,
we provide the proofs of
the theoretical results
in Section \ref{sec:Diverging number of structural parameters}.
The proofs of
the results
in Sections
\ref{sec2} and
\ref{sec3}
are in Supplements
\ref{PC.Paper.Supplement.Model.Data.Method} and
\ref{PC.Paper.Supplement.Asymptotic.Properties}.

Denote $\mathbb{S}_{k,l}=\mathbb{S}_{\{k, k+1, \cdots, l\}}$
and $\mathbb{S}_{k,l}^{\epsilon} = \mathbb{S}_{\{k, k+1, \cdots, l\}}^{\epsilon}$.
Let
$ \mathcal{B} =\{\max_{s+1\leq i\leq n}\normE{\bd{X}_{i}}\leq \kappa_{n}\}$ and
$ \mathcal{D}= \bigcap_{i=1}^{n} \{ - \gamma_{n} \leq  \epsilon_{i} \leq \gamma_{n} \}$.
Then $P(\mathcal{B}) \rightarrow 1$
and
$P(\mathcal{D})\rightarrow 1$ by
(\ref{eq2.2}).

\begin{proofName}
\subsection{Proof of Lemma \ref{PC.Paper.d.infty.index.sets.consistency}}
\end{proofName}
\begin{proof}[Proof of Lemma \ref{PC.Paper.d.infty.index.sets.consistency}]
We first consider $S_{i0}$'s, then $S_{i1}$'s, and finally $S_{i2}$'s with $i=1,2,3$.
Consider $S_{10}$, $S_{20}$ and $S_{30}$.
Let
$
\mathcal{\mathcal{A}} =
\{S_{10}=S_{10}^{\star}\}
$.
Note that
$
P(\mathcal{A})
 \geq
P(\mathcal{A} | \mathcal{B})
P(\mathcal{B})
$
and $P(\mathcal{B})\rightarrow 1$.
It suffices to show that $P(\mathcal{A}|\mathcal{B})\rightarrow 1$.
By $\lambda \gg  \sqrt{d}\kappa_{n}$,
it follows
$
P(\mathcal{A} | \mathcal{B})
\geq
P(\{s+1\leq i \leq n: -\lambda + \max_{s+1\leq i\leq n}\normE{\bd{X}_{i}}\sqrt{d}C
\leq \epsilon_{i} \leq \lambda - \max_{s+1\leq i\leq n}\normE{\bd{X}_{i}}\sqrt{d}C \}
\supset S_{10}^{\star}
| \mathcal{B})
\geq
P(\{s+1\leq i \leq n: -\lambda + \kappa_{n}\sqrt{d}C
\leq \epsilon_{i} \leq \lambda - \kappa_{n}\sqrt{d}C \}
\supset S_{10}^{\star})
\geq
P(\mathcal{D}) \rightarrow 1$.
Thus, wpg1,
$S_{10}=S_{10}^{\star}$.
From
$
S_{10} \cup S_{20} \cup S_{30} = S_{10}^{\star}
$,
it follows that, wpg1,
$
S_{20}=S_{30}=\emptyset
$.
Consider $S_{21}$, $S_{31}$ and $S_{11}$.
Recall that
$\mu^{\star}=\min\{|\mu_{i}^{\star}|:1\leq i\leq s_{1}\}$
and note that
$\lambda - \mu^{\star} + \sqrt{d}C\kappa_{n} < -\gamma_{n}$ when $n$ is large.
Let
$S_{211}=S_{21}S_{21}^{\star}$
and
$S_{212}=S_{21}S_{21}^{\star c}$.
We will show
$P(S_{211}=S_{21}^{\star})\rightarrow 1$
and
$P(S_{212}=\emptyset)\rightarrow 1$.
Then
$P(S_{21}=S_{21}^{\star})\rightarrow 1$.
Denote
$\mathcal{A}_{1} = \{S_{211} \supset S_{21}^{\star}\}$.
On the event $\mathcal{B}$,
$S_{211}
\supset
\{1\leq i \leq s_{1}: \epsilon_{i} > \lambda - \mu^{\star} + \sqrt{d}C\kappa_{n} \text{ and } \mu_{i}^{\star}>0\}
 \supset
\{1\leq i \leq s_{1}: \epsilon_{i} > -\gamma_{n} \text{ and } \mu_{i}^{\star}>0\}$.
Then,
$
P(\mathcal{A}_{1})
 \geq
P(\mathcal{A}_{1}|\mathcal{B})P(\mathcal{B})
 \geq
P(\{1\leq i \leq s_{1}: \epsilon_{i} > -\gamma_{n} \text{ and } \mu_{i}^{\star}>0\} \supset S_{21}^{\star} ) P(\mathcal{B})
\rightarrow 1\cdot 1 = 1
$.
It follows that, wpg1, $S_{211}\supset S_{21}^{\star}$.
Note that $S_{211}\subset S_{21}^{\star}$.
Then, wpg1, $S_{211}=S_{21}^{\star}$.
Denote $\mathcal{A}_{2}=\{S_{212} = \emptyset\}$.
On the event $\mathcal{B}$,
$S_{212}
\subset
\{1\leq i \leq s_{1}: \epsilon_{i} > \lambda + \mu^{\star} - \sqrt{d}C\kappa_{n} \text{ and } \mu_{i}^{\star}<0\}$,
which contains
$\{1\leq i \leq s_{1}: \epsilon_{i} > \gamma_{n} \}$.
Then,
$
P(\mathcal{A}_{2})
\geq
P(\mathcal{A}_{2}| \mathcal{B})P(\mathcal{B})
\geq
P(\{1\leq i \leq s_{1}: \epsilon_{i} > \gamma_{n} \} = \emptyset ) P(\mathcal{B})
\rightarrow =1
$.
Then, wpg1, $S_{212}=\emptyset$.
Thus,
$P(S_{21}=S_{21}^{\star})\rightarrow 1$.
Similarly,
we can show, wpg1, $S_{31} =  S_{31}^{\star}$.
Note that $S_{11}$, $S_{21}$ and $S_{31}$
are disjoint and their union is $S_{21}^{\star}\cup S_{31}^{\star}$.
Then, wpg1, $S_{11}=\emptyset$.
Consider $S_{12}$, $S_{22}$ and $S_{32}$.
Denote
$\mathcal{A} = \{S_{12}=S_{12}^{\star}\}$.
Note that
$-\lambda - \mu^{\star}_{i} + \sqrt{d}C\kappa_{n} < -\gamma_{n}$
and
$\lambda - \mu^{\star}_{i} - \sqrt{d}C\kappa_{n}  > \gamma_{n}$
when $n$ is large
for $s_{1}+1\leq i \leq s$.
On the event $\mathcal{B}$,
$
S_{12}
\supset
\{s_{1}+1\leq i \leq s:
-\lambda - \mu^{\star}_{i} + \sqrt{d}C\kappa_{n}
\leq
\epsilon_{i}
\leq
\lambda - \mu^{\star}_{i} - \sqrt{d}C\kappa_{n} \}$,
which contains
$
\{s_{1}+1\leq i \leq s:
-\gamma_{n}
\leq
\epsilon_{i}
\leq
\gamma_{n} \}
$.
Then,
$
P(\mathcal{A})
\geq
P(\mathcal{A}|\mathcal{B})P(\mathcal{B})
\geq
P(\{s_{1}+1\leq i \leq s:
-\gamma_{n}
\leq
\epsilon_{i}
\leq
\gamma_{n} \} = S_{12}^{\star} ) P(\mathcal{B})
\rightarrow 1
$.
Thus, wpg1, $S_{12}=S_{12}^{\star}$.
Note that $S_{12}$, $S_{22}$ and $S_{32}$
are disjoint and their union
is $S_{12}^{\star}$.
Then, wpg1,
$
S_{22}=S_{32}=\emptyset
$.
\end{proof}

Before proceeding to the proofs of
Theorems
\ref{thm5.2}
to
\ref{model:three:multivariate:mean:zero:thm:twostage:limitdistribution.general.covariates.estimated.Sigma.alternative.errors:d:infty:Paper},
we denote
$\bar{\sigma}_{X}^{2} = (1/d)\sum_{j=1}^{d}\Var[X_{0j}]$ and
$\bar{\sigma}_{XX}^{2} = (1/d^{2})\sum_{k=1}^{d}\sum_{l=1}^{d}\Var[X_{0k}X_{0l}]$
and make the following assumptions.
\begin{description}
  \item[Assumption (E1):] $\bar{\sigma}_{X}^{2}$ is bounded.
  \item[Assumption (E2):] $\bar{\sigma}_{XX}^{2}$ is bounded.
\end{description}
Assumption (E) in Section \ref{sec:Diverging number of structural parameters}
implies
Assumptions (E1)
and
(E2)
by Cauchy-Schwarz inequality.
For simplicity, we adopt
the notation $\lesssim$,
which means the left hand side is bounded by a constant times the right,
where the constant does not affect related analysis.
\begin{proofName}
\subsection{Lemma \ref{PC.Paper.Inverse.Matrix.Convergence.Probability.d.Infity}}
\end{proofName}
Below are three lemmas
needed for proving
Theorems
\ref{thm5.2}
to
\ref{model:three:multivariate:mean:zero:thm:twostage:limitdistribution.general.covariates.estimated.Sigma.alternative.errors:d:infty:Paper}.
Their proofs are in Supplement
\ref{PC.Paper.Supplement.Extension.d.infty}.
Suppose that $\bd{M}$ and $\bd{E}$ are matrices
and $\norm{\cdot}$ is a matrix norm and
that $\{\bd{A}_{n}\}$ is a sequence of random $d\times d$ matrices
and $\bd{A}$ a deterministic $d\times d$ matrix,
and denote $\hat{\bd{\Sigma}}_{n} = (1/n)\mathbb{S}_{n}$,
the sample covariance matrix.
\begin{lem}[\cite{Stewart1969}]
\label{lemma:Stewart:1969}
If $\norm{\bd{I}}=1$ and
$\norm{\bd{M}^{-1}}
\norm{\bd{E}} < 1$,
then
\begin{equation*}
\frac
{\norm{(\bd{M}+\bd{E})^{-1} - \bd{M}^{-1}}}
{\norm{\bd{M}^{-1}}}
\leq
\frac
{\norm{\bd{M}^{-1}}\norm{\bd{E}}}
{1-\norm{\bd{M}^{-1}}\norm{\bd{E}}}.
\end{equation*}
\end{lem}
\begin{lem}
\label{PC.Paper.Inverse.Matrix.Convergence.Probability.d.Infity}
If
$\normFd{\bd{A}^{-1}}$ is bounded,
$\bd{A}_{n} \ConvProb \bd{A}$,
and $r_d\geq 1/\sqrt{d}$,
then
$\bd{A}_{n}^{-1} \ConvProb \bd{A}^{-1}$,
where the convergence in probability is wrt $r_d\normF{\cdot}$.
\end{lem}

\begin{proofName}
\subsection{Lemma \ref{PC.Paper.Sample.Covariance.Matrix.Convergence.d.infty}}
\end{proofName}
\begin{lem}
\label{PC.Paper.Sample.Covariance.Matrix.Convergence.d.infty}
If Assumption (E2) holds
and $r_d^{2}d^{4}/n \rightarrow 0$,
then
$
\hat{\bd{\Sigma}}_{n} \ConvProb \bd{\Sigma}_{X}
$
wrt $r_d\normF{\cdot}$.
\end{lem}

\begin{proofLong}
\begin{proof}
[Long Proof of Lemma \ref{PC.Paper.Sample.Covariance.Matrix.Convergence.d.infty}]
For any $\delta>0$, we have
\begin{align*}
P(\normF{\hat{\bd{\Sigma}}_{n} - \bd{\Sigma}}>\delta)
= &
P([\sum_{k=1}^{d}\sum_{l=1}^{d}(\frac{1}{n}\sum_{i=1}^{n}X_{ik}X_{il}-\sigma_{kl})^{2}]^{1/2} > \delta) \\
\leq &
\sum_{k=1}^{d}\sum_{l=1}^{d}
P((\frac{1}{n}\sum_{i=1}^{n}X_{ik}X_{il}-\sigma_{kl})^{2} > \frac{\delta^{2}}{d^{2}}) \\
\leq &
\sum_{k=1}^{d}\sum_{l=1}^{d}\frac{d^{2}}{\delta^{2}}
P(\frac{1}{n}\sum_{i=1}^{n} X_{ik}X_{il}-\sigma_{kl})^{2} \\
= &
\sum_{k=1}^{d}\sum_{l=1}^{d}\frac{d^{2}}{\delta^{2}}
\frac{1}{n^{2}}\sum_{i=1}^{n} P(X_{ik}X_{il}-\sigma_{kl})^{2} \\
= &
\sum_{k=1}^{d}\sum_{l=1}^{d}\frac{d^{2}}{\delta^{2}}
\frac{1}{n^{2}}\sum_{i=1}^{n} P(X_{0k}X_{0l}-\sigma_{kl})^{2} \\
= &
\sum_{k=1}^{d}\sum_{l=1}^{d}\frac{d^{2}}{\delta^{2}}
\frac{1}{n} \Var[X_{0k}X_{0l}] \\
\leq &
\frac{d^{4}}{n}
\frac{1}{\delta^{2}}
\max_{1\leq k,l\leq d}\Var[X_{0k}X_{0l}].
\end{align*}
Thus, if $d^{4}/n\rightarrow 0$,
then $\normF{\hat{\bd{\Sigma}}_{n} - \bd{\Sigma}}\ConvProb 0$,
which means that
$\hat{\bd{\Sigma}}_{n}$ is a consistent estimator of $\bd{\Sigma}$
wrt $\normF{\cdot}$.

Next, change $\delta$ to $\sqrt{d}\delta$.
Then,
\begin{align*}
P(\normFd{\hat{\bd{\Sigma}}_{n} - \bd{\Sigma}}>\delta)
\leq &
\frac{d^{4}}{n}
\frac{1}{d\delta^{2}}
\max_{1\leq k,l\leq d}\Var[X_{0k}X_{0l}] \\
= &
\frac{d^{3}}{n}
\frac{1}{\delta^{2}}
\max_{1\leq k,l\leq d}\Var[X_{0k}X_{0l}].
\end{align*}
Thus, if $d^{3}/n\rightarrow 0$,
then $\normFd{\hat{\bd{\Sigma}}_{n} - \bd{\Sigma}}\ConvProb 0$,
which means that
$\hat{\bd{\Sigma}}_{n}$ is a consistent estimator of $\bd{\Sigma}$
wrt $\normFd{\cdot}$.
\end{proof}
\end{proofLong}

\begin{proofName}
\subsection{Proof of Theorem
\ref{thm5.2}}
\end{proofName}
\begin{proof}
[Proof of Theorem
\ref{thm5.2}]
By
the proof of
Lemma \ref{PC.Paper.d.infty.index.sets.consistency},
wpg1,
the solution $\hat{\bd{\beta}}_{n}$ to $\varphi_{n}(\bd{\beta})=0$ on $\mathcal{B}_{C}(\bd{\beta}^{\star})$ is explicitly given by
$
\hat{\bd{\beta}}_{n}
= \bd{\beta}^{\star} + T_{0}^{-1}(T_{1} + T_{2} + T_{3}-T_{4})
$,
where
$T_{0}=(1/n)\mathbb{S}_{s_{1}+1, n}$,
$T_{1}=(1/n)\mathbb{S}_{S_{12}^{\star}}^{\mu}$,
$T_{2}=(1/n)\mathbb{S}_{s_{1}+1, n}^{\epsilon}$,
$T_{3}=(\lambda/n)\mathcal{S}_{S_{21}^{\star}}$ and
$T_{4}=(\lambda/n)\mathcal{S}_{S_{31}^{\star}}$.
Then,
$
r_d\normE{\hat{\bd{\beta}}_{n} - \bd{\beta}^{\star}}
\leq
\normFd{T_{0}^{-1}} \sum_{i=1}^{4}r_d\sqrt{d}\normE{T_{i}}.
$
We will show that
$\normFd{T_{0}^{-1}}$ is bounded by a positive constant wpg1
and
$r_d\sqrt{d}\normE{T_{i}} \ConvProb 0$ for $i=1,2,3,4$.
Then, $r_d\normE{\hat{\bd{\beta}}_{n} - \bd{\beta}^{\star}}=o_{P}(1)$.
Consider $T_{0}$.
By Lemma \ref{PC.Paper.Sample.Covariance.Matrix.Convergence.d.infty},
$ \normFd{T_{0}-\bd{\Sigma}_{X}}
\ConvProb 0 $
under Assumption (E2)
and the condition $d^{3}/n\rightarrow 0$.
Then, by Lemma \ref{PC.Paper.Inverse.Matrix.Convergence.Probability.d.Infity},
together with Assumption (D),
$\normFd{T_{0}^{-1}-\bd{\Sigma}_{X}^{-1}}
\ConvProb 0$.
This
implies that, wpg1,
$\normFd{T_{0}^{-1}}$
is bounded by a positive constant.
Consider $T_{1}$.
Wpg1,
$
r_d\sqrt{d}\normE{T_{1}}
\leq
r_d\sqrt{d}s_{2}\kappa_{n}\gamma_{n}/n
=o(1)$
for $s_{2}=o(n/(r_d\sqrt{d}\kappa_{n}\gamma_{n}))$.
Consider $T_{2}$.
For any $\delta>0$,
$
P(\normE{T_{2}} > \delta)
\leq
(1/\delta^{2})
P\normE{(1/n)\sum_{i=s_{1}+1}^{n}\bd{X}_{i}\epsilon_{i}}^{2}
\leq
d\sigma^{2}\bar{\sigma}_{X}^{2}/(n\delta^{2})
$,
where
$\bar{\sigma}_{X}^{2}=(1/d)\sum_{j=1}^{d}\sigma_{j}^{2}$.
Thus,
$
P(r_d\sqrt{d}\normE{T_{2}} > \delta)
\leq
r_d^{2}d^{2}\sigma^{2}\bar{\sigma}_{X}^{2}/(n\delta^{2})
\rightarrow 0
$
by Assumption \textbf{(E1)}
and $(r_dd)^{2}/n \rightarrow 0$.
Consider $T_{3}$ and $T_{4}$.
Wpg1,
$
r_d\sqrt{d}\normE{T_{3}}
\leq
r_d\sqrt{d}\lambda
s_{1}
\kappa_{n}/n
=o(1)
$
for $s_{1}=o(n/(r_d\sqrt{d}\lambda \kappa_{n}))$.
Similarly, $r_d\sqrt{d}\normE{T_{4}}=o_{P}(1)$.
%
\end{proof}

\begin{proofLong}
\begin{proof}
[Long Proof of Theorems
\ref{thm5.2}]
By
the proof of Theorem \ref{thm3.2}
and
Lemma \ref{lem3.1}, wpg1,
the solution $\hat{\bd{\beta}}_{n}$ to $\varphi_{n}(\bd{\beta})=0$ on $\mathcal{B}_{C}(\bd{\beta}^{\star})$ is explicitly given by
\begin{equation*}
\hat{\bd{\beta}}_{n}
= \bd{\beta}^{\star} + T_{0}^{-1}(T_{1} + T_{2} + T_{3}-T_{4}),
\end{equation*}
where
\begin{align*}
T_{0}&=(1/n)\mathbb{S}_{s_{1}+1, n}, \\
T_{1}&=(1/n)\mathbb{S}_{S_{12}^{\star}}^{\mu}, \\
T_{2}&=(1/n)\mathbb{S}_{s_{1}+1, n}^{\epsilon}, \\
T_{3}&=(\lambda/n)\mathcal{S}_{S_{21}^{\star}}, \\
T_{4}&=(\lambda/n)\mathcal{S}_{S_{31}^{\star}}.
\end{align*}

Then,
\begin{align*}
\normE{\hat{\bd{\beta}}_{n} - \bd{\beta}^{\star}}
& \leq
\normE{T_{0}^{-1}(T_{1} + T_{2} + T_{3}-T_{4})} \\
& \leq
\normF{T_{0}^{-1}}\normE{T_{1} + T_{2} + T_{3}-T_{4})} \\
& \leq
\normF{T_{0}^{-1}} \sum_{i=1}^{4}\normE{T_{i}}.
\end{align*}
Then,
\begin{align*}
r_d\normE{\hat{\bd{\beta}}_{n} - \bd{\beta}^{\star}}
& \leq
\frac{1}{\sqrt{d}}\normF{T_{0}^{-1}} \sum_{i=1}^{4}r_d\sqrt{d}\normE{T_{i}}.
\end{align*}

We will show the followings:
\begin{itemize}
  \item $(1/\sqrt{d})\normF{T_{0}^{-1}}$ is bounded by a constant wpg1.
  \item $r_d\sqrt{d}\normE{T_{i}} \ConvProb 0$ for $i=1,2,3,4$.
\end{itemize}

Thus, we have the following conclusions:
\begin{itemize}
  \item $r_d\normE{\hat{\bd{\beta}}_{n} - \bd{\beta}^{\star}}=o_{P}(1)$, that is,
$\normE{\hat{\bd{\beta}}_{n} - \bd{\beta}^{\star}}=o_{P}(1/r_d)$.
\end{itemize}

\lineProof

\OnItem{On $(1/\sqrt{d})\normF{T_{0}^{-1}}$.}
We will show that
\begin{gather*}
    (1/\sqrt{d})\normF{T_{0}^{-1}}
    -
    (1/\sqrt{d})\normF{\bd{\Sigma}_{X}^{-1}}
    \ConvProb
    0.
\end{gather*}
By the assumption that
there exists a constant $C>0$ such that
\begin{equation*}
    (1/\sqrt{d})\normF{\bd{\Sigma}_{X}^{-1}} \leq C,
\end{equation*}
we have, wpg1,
\begin{equation*}
    (1/\sqrt{d})\normF{T_{0}^{-1}} \leq 2C.
\end{equation*}

Since
\begin{align*}
|\normFd{T_{0}^{-1}}-\normFd{\bd{\Sigma}_{X}^{-1}}|
\leq
\normFd{T_{0}^{-1}-\bd{\Sigma}_{X}^{-1}},
\end{align*}
it is sufficient to show that
\begin{equation*}
\normFd{T_{0}^{-1}-\bd{\Sigma}_{X}^{-1}}
\ConvProb 0.
\end{equation*}
By Lemma \ref{PC.Paper.Inverse.Matrix.Convergence.Probability.d.Infity},
it is sufficient to show that
\begin{equation*}
\normFd{T_{0}-\bd{\Sigma}_{X}}
\ConvProb 0.
\end{equation*}
In fact,
by
Lemma \ref{PC.Paper.Sample.Covariance.Matrix.Convergence.d.infty},
this holds
under the condition that $d^{3}/n\rightarrow 0$.

\lineProof

Thus, wpg1,
$\normFd{T_{0}^{-1}}$ is bounded by a constant
under the conditions
\begin{itemize}
  \item There exists a constant $C>0$ such that $\normFd{\bd{\Sigma}_{X}^{-1}} \leq C$.
  \item $d^{3}/n\rightarrow 0$.
\end{itemize}

\lineProof

\OnItem{On $T_{1}$.}
We have, wpg1,
\begin{align*}
r_d\sqrt{d}\normE{T_{1}}
& \leq
r_d\sqrt{d}\frac{1}{n}\sum_{i=s_{1}+1}^{s}\normE{\bd{X}_{i}\mu_{i}^{\star}} \\
& =
r_d\sqrt{d}\frac{1}{n}\sum_{i=s_{1}+1}^{s}\normE{\bd{X}_{i}}\cdot|\mu_{i}^{\star}| \\
& \leq
r_d\sqrt{d}s_{2}\kappa_{n}\gamma_{n}/n.
\end{align*}
Thus, if $s_{2}=o(n/(r_d\sqrt{d}\kappa_{n}\gamma_{n}))$,
then $r_d\sqrt{d}\normE{T_{1}}=o_{P}(1)$.

\lineProof

\OnItem{On $T_{2}$.}
We have
\begin{align*}
T_{2}
&=
(1/n)\mathbb{S}_{s_{1}+1, n}^{\epsilon}
=
(1/n)\sum_{i=s_{1}+1}^{n}\bd{X}_{i}^{T}\epsilon_{i}.
\end{align*}

We have
\begin{align*}
P(\normE{T_{2}} > \delta)
& =
P(\normE{(1/n)\sum_{i=s_{1}+1}^{n}\bd{X}_{i}\epsilon_{i}} > \delta) \\
& \leq
\frac{1}{\delta^{2}}
P\normE{(1/n)\sum_{i=s_{1}+1}^{n}\bd{X}_{i}\epsilon_{i}}^{2} \\
& =
\frac{1}{\delta^{2}}
\frac{1}{n^{2}}
P\normE{\sum_{i=s_{1}+1}^{n}\bd{X}_{i}\epsilon_{i}}^{2} \\
& =
\frac{1}{\delta^{2}}
\frac{1}{n^{2}}
P\sum_{i=s_{1}+1}^{n}\sum_{j=s_{1}+1}^{n}\bd{X}_{i}^{T}\bd{X}_{j}\epsilon_{i}\epsilon_{j} \\
& =
\frac{1}{\delta^{2}}
\frac{1}{n^{2}}
\sum_{i=s_{1}+1}^{n}P\normE{\bd{X}_{i}}^{2}\epsilon_{i}^{2} \\
& =
\frac{1}{\delta^{2}}
\frac{1}{n^{2}}
\sum_{i=s_{1}+1}^{n}P\normE{\bd{X}_{i}}^{2}P\epsilon_{i}^{2} \\
& =
\frac{1}{\delta^{2}}
\frac{1}{n^{2}}
(n-s_{1})P\normE{\bd{X}_{0}}^{2}P\epsilon_{0}^{2} \\
& =
\frac{1}{\delta^{2}}
\frac{1}{n^{2}}
(n-s_{1})P\normE{\bd{X}_{0}}^{2}\sigma^{2} \\
& =
\frac{1}{\delta^{2}}
\frac{1}{n^{2}}
(n-s_{1})\sigma^{2}\sum_{j=1}^{d}PX_{0j}^{2} \\
& =
\frac{1}{\delta^{2}}
\frac{1}{n^{2}}
(n-s_{1})\sigma^{2}\sum_{j=1}^{d}\sigma_{j}^{2} \\
& =
\frac{1}{\delta^{2}}
\frac{1}{n^{2}}
d(n-s_{1})\sigma^{2}\frac{1}{d}\sum_{j=1}^{d}\sigma_{j}^{2} \\
& \leq
\frac{1}{\delta^{2}}
\frac{d}{n}
\sigma^{2}\frac{1}{d}\sum_{j=1}^{d}\sigma_{j}^{2} \\
& \leq
\frac{1}{\delta^{2}}
\frac{d}{n}
\sigma^{2}\sigma_{X}^{2},
\end{align*}
where
$\sigma_{X}^{2}=\max_{1\leq j\leq d} \sigma_{j}^{2}$ is bounded.
Thus, we have
\begin{equation*}
P(r_d\sqrt{d}\normE{T_{2}} > \delta)
\leq
\frac{1}{\delta^{2}}
\frac{(r_dd)^{2}}{n}
\sigma^{2}\sigma_{X}^{2}.
\end{equation*}
Thus, if $(r_dd)^{2}/n \rightarrow 0$,
then $r_d\sqrt{d}\normE{T_{2}}\ConvProb 0$.
\lineProof

\OnItem{On $T_{3}$ and $T_{4}$.}
We have, wpg1,
\begin{align*}
r_d\sqrt{d}\normE{T_{3}}
& =
r_d\sqrt{d}\normE{\lambda \frac{1}{n} \mathcal{S}_{S_{21}^{\star}}} \\
& =
r_d\sqrt{d}\lambda \frac{1}{n} \normE{\mathcal{S}_{S_{21}^{\star}}} \\
& \leq
r_d\sqrt{d}\lambda \frac{1}{n}
\#(S_{21}^{\star})
\kappa_{n} \\
& \leq
r_d\sqrt{d}\lambda
s_{1}
\kappa_{n}/n.
\end{align*}
Thus,
if $s_{1}=o(n/(r_d\sqrt{d}\lambda \kappa_{n}))$,
then $r_d\sqrt{d}\normE{T_{3}}=o_{P}(1)$.
In the same way, we can show that $r_d\sqrt{d}\normE{T_{4}}=o_{P}(1)$ under the condition
$s_{1}=o(n/(r_d\sqrt{d}\lambda \kappa_{n}))$.

\lineProof

Therefore,
$\hat{\bd{\beta}}_{n}$ is a consistent estimator of $\bd{\beta}^{\star}$
wrt $r_d\normEd{\cdot}$.\\

\lineProof

Next, we show some special cases with $r_d=1$ and $r_d=1/\sqrt{d}$.
We have,
\begin{align*}
\normE{\hat{\bd{\beta}}_{n} - \bd{\beta}^{\star}}
& \leq
\normE{T_{0}^{-1}(T_{1} + T_{2} + T_{3}-T_{4})} \\
& \leq
\normF{T_{0}^{-1}}\normE{T_{1} + T_{2} + T_{3}-T_{4})} \\
& \leq
\normF{T_{0}^{-1}} \sum_{i=1}^{4}\normE{T_{i}}.
\end{align*}
Then,
\begin{align*}
\frac{1}{\sqrt{d}}\normE{\hat{\bd{\beta}}_{n} - \bd{\beta}^{\star}}
& \leq
\frac{1}{\sqrt{d}}\normF{T_{0}^{-1}} \sum_{i=1}^{4}\normE{T_{i}}.
\end{align*}

We will show the followings:
\begin{itemize}
  \item $(1/\sqrt{d})\normF{T_{0}^{-1}}$ is bounded by a constant wpg1.
  \item $\normE{T_{i}} \ConvProb 0$ for $i=1,2,3,4$.
  \item $\sqrt{d}\normE{T_{i}} \ConvProb 0$ for $i=1,2,3,4$.
\end{itemize}

Thus, we have the following conclusions:
\begin{itemize}
  \item $(1/\sqrt{d})\normE{\hat{\bd{\beta}}_{n} - \bd{\beta}^{\star}}=o_{P}(1)$, that is,
$\normE{\hat{\bd{\beta}}_{n} - \bd{\beta}^{\star}}=o_{P}(\sqrt{d})$.
  \item $\normE{\hat{\bd{\beta}}_{n} - \bd{\beta}^{\star}}=o_{P}(1)$, that is,
$\normE{\hat{\bd{\beta}}_{n} - \bd{\beta}^{\star}}=o_{P}(1)$.
\end{itemize}

\lineProof

\OnItem{On $(1/\sqrt{d})\normF{T_{0}^{-1}}$.}
We will show that
\begin{gather*}
    (1/\sqrt{d})\normF{T_{0}^{-1}}
    -
    (1/\sqrt{d})\normF{\bd{\Sigma}_{X}^{-1}}
    \ConvProb
    0.
\end{gather*}
By the assumption that
there exists a constant $C>0$ such that
\begin{equation*}
    (1/\sqrt{d})\normF{\bd{\Sigma}_{X}^{-1}} \leq C,
\end{equation*}
we have, wpg1,
\begin{equation*}
    (1/\sqrt{d})\normF{T_{0}^{-1}} \leq 2C.
\end{equation*}

Since
\begin{align*}
|\normFd{T_{0}^{-1}}-\normFd{\bd{\Sigma}_{X}^{-1}}|
\leq
\normFd{T_{0}^{-1}-\bd{\Sigma}_{X}^{-1}},
\end{align*}
it is sufficient to show that
\begin{equation*}
\normFd{T_{0}^{-1}-\bd{\Sigma}_{X}^{-1}}
\ConvProb 0.
\end{equation*}
By Lemma \ref{PC.Paper.Inverse.Matrix.Convergence.Probability.d.Infity},
it is sufficient to show that
\begin{equation*}
\normFd{T_{0}-\bd{\Sigma}_{X}}
\ConvProb 0.
\end{equation*}
In fact,
by
Lemma \ref{PC.Paper.Sample.Covariance.Matrix.Convergence.d.infty},
this holds
under the condition that $d^{3}/n\rightarrow 0$.

\lineProof

Thus, wpg1,
$\normFd{T_{0}^{-1}}$ is bounded by a constant
under the conditions
\begin{itemize}
  \item There exists a constant $C>0$ such that $\normFd{\bd{\Sigma}_{X}^{-1}} \leq C$.
  \item $d^{3}/n\rightarrow 0$.
\end{itemize}

\lineProof

\OnItem{On $T_{1}$.}
We have, wpg1,
\begin{align*}
\normE{T_{1}}
& \leq
\frac{1}{n}\sum_{i=s_{1}+1}^{s}\normE{\bd{X}_{i}\mu_{i}^{\star}}
=
\frac{1}{n}\sum_{i=s_{1}+1}^{s}\normE{\bd{X}_{i}}\cdot|\mu_{i}^{\star}|
\leq
s_{2}\kappa_{n}\gamma_{n}/n.
\end{align*}
Thus, if $s_{2}=o(n/(\kappa_{n}\gamma_{n}))$,
then $\normE{T_{1}}=o_{P}(1)$.

\lineProof
Similarly, we have
\begin{align*}
\sqrt{d}\normE{T_{1}}
\leq
s_{2}\kappa_{n}\gamma_{n}\sqrt{d}/n.
\end{align*}
Thus, if $s_{2}=o(n/(\sqrt{d}\kappa_{n}\gamma_{n}))$,
then $\sqrt{d}\normE{T_{1}}=o_{P}(1)$.

\lineProof

\OnItem{On $T_{2}$.}
We have
\begin{align*}
T_{2}
&=
(1/n)\mathbb{S}_{s_{1}+1, n}^{\epsilon}
=
(1/n)\sum_{i=s_{1}+1}^{n}\bd{X}_{i}^{T}\epsilon_{i}.
\end{align*}

We have
\begin{align*}
P(\normE{T_{2}} > \delta)
& =
P(\normE{(1/n)\sum_{i=s_{1}+1}^{n}\bd{X}_{i}\epsilon_{i}} > \delta) \\
& \leq
\frac{1}{\delta^{2}}
P\normE{(1/n)\sum_{i=s_{1}+1}^{n}\bd{X}_{i}\epsilon_{i}}^{2} \\
& =
\frac{1}{\delta^{2}}
\frac{1}{n^{2}}
P\normE{\sum_{i=s_{1}+1}^{n}\bd{X}_{i}\epsilon_{i}}^{2} \\
& =
\frac{1}{\delta^{2}}
\frac{1}{n^{2}}
P\sum_{i=s_{1}+1}^{n}\sum_{j=s_{1}+1}^{n}\bd{X}_{i}^{T}\bd{X}_{j}\epsilon_{i}\epsilon_{j} \\
& =
\frac{1}{\delta^{2}}
\frac{1}{n^{2}}
\sum_{i=s_{1}+1}^{n}P\normE{\bd{X}_{i}}^{2}\epsilon_{i}^{2} \\
& =
\frac{1}{\delta^{2}}
\frac{1}{n^{2}}
\sum_{i=s_{1}+1}^{n}P\normE{\bd{X}_{i}}^{2}P\epsilon_{i}^{2} \\
& =
\frac{1}{\delta^{2}}
\frac{1}{n^{2}}
(n-s_{1})P\normE{\bd{X}_{0}}^{2}P\epsilon_{0}^{2} \\
& =
\frac{1}{\delta^{2}}
\frac{1}{n^{2}}
(n-s_{1})P\normE{\bd{X}_{0}}^{2}\sigma^{2} \\
& =
\frac{1}{\delta^{2}}
\frac{1}{n^{2}}
(n-s_{1})\sigma^{2}\sum_{j=1}^{d}PX_{0j}^{2} \\
& =
\frac{1}{\delta^{2}}
\frac{1}{n^{2}}
(n-s_{1})\sigma^{2}\sum_{j=1}^{d}\sigma_{j}^{2} \\
& =
\frac{1}{\delta^{2}}
\frac{1}{n^{2}}
d(n-s_{1})\sigma^{2}\frac{1}{d}\sum_{j=1}^{d}\sigma_{j}^{2} \\
& \leq
\frac{1}{\delta^{2}}
\frac{d}{n}
\sigma^{2}\frac{1}{d}\sum_{j=1}^{d}\sigma_{j}^{2} \\
& \leq
\frac{1}{\delta^{2}}
\frac{d}{n}
\sigma^{2}\sigma_{X}^{2},
\end{align*}
where
$\sigma_{X}^{2}=\max_{1\leq j\leq d} \sigma_{j}^{2}$ is bounded.
Thus, if $d/n \rightarrow 0$,
then $\normE{T_{2}}\ConvProb 0$.
\lineProof

From the above derivation, we have
\begin{equation*}
P(\sqrt{d}\normE{T_{2}} > \delta)
\leq
\frac{1}{\delta^{2}}
\frac{d^{2}}{n}
\sigma^{2}\sigma_{X}^{2}.
\end{equation*}
Thus, if $d^{2}/n \rightarrow 0$,
then $\sqrt{d}\normE{T_{2}}\ConvProb 0$.
\lineProof

\OnItem{On $T_{3}$ and $T_{4}$.}
We have, wpg1,
\begin{align*}
\normE{T_{3}}
=
\normE{\lambda \frac{1}{n} \mathcal{S}_{S_{21}^{\star}}}
=
\lambda \frac{1}{n} \normE{\mathcal{S}_{S_{21}^{\star}}}
\leq
\lambda \frac{1}{n}
\#(S_{21}^{\star})
\kappa_{n}
\leq
\lambda
s_{1}
\kappa_{n}/n.
\end{align*}
Thus,
if $s_{1}=o(n/(\lambda \kappa_{n}))$,
then $T_{3}=o_{P}(1)$.
In the same way, we can show that $T_{4}=o_{P}(1)$ under the condition
$s_{1}=o(n/(\lambda \kappa_{n}))$.

\lineProof

From the above derivation,
we have
\begin{align*}
\sqrt{d}\normE{T_{3}}
\leq
\lambda
s_{1}
\kappa_{n}\sqrt{d}/n.
\end{align*}
Thus,
if $s_{1}=o(n/(\sqrt{d}\lambda \kappa_{n}))$,
then $T_{3}=o_{P}(1)$.
In the same way, we can show that $T_{4}=o_{P}(1)$ under the condition
$s_{1}=o(n/(\sqrt{d}\lambda \kappa_{n}))$.

\lineProof

Therefore,
$\hat{\bd{\beta}}_{n}$ is a consistent estimator of $\bd{\beta}^{\star}$
wrt $\normEd{\cdot}$.
\end{proof}
\end{proofLong}

\begin{proofName}
\subsection{Lemma \ref{PC.Paper.Scaled.IID.Sum.Convergence.Rate.d.infty}}
\end{proofName}
The next lemma is needed for proving
Theorem \ref{thm3.4:d:infty:alternative:Paper}
and its proof is in Supplement \ref{PC.Paper.Supplement.Extension.d.infty}.
Suppose $\{\bd{\xi}_{i}\}$ are i.i.d. copies of $\bd{\xi}_{0}$,
a $d$-dimensional random vector
with mean zero.
Denote $\sigma_{\xi,\max}^{2}=\max_{1\leq j\leq d}\Var[\xi_{0j}]$,
$\sigma_{\xi,\min}^{2}=\min_{1\leq j\leq d}\Var[\xi_{0j}]$
and
$\gamma_{\xi,\max}=\max_{1\leq j\leq d}\E|\xi_{0j}|^{3}$.
\begin{lem}
\label{PC.Paper.Scaled.IID.Sum.Convergence.Rate.d.infty}
Suppose $\sigma_{\xi, \max}$ and $\gamma_{\xi,\max}$ are bounded from above
and $\sigma_{\xi,\max}$ is bounded from zero.
If $d=o(\sqrt{n})$,
then
$(1/\sqrt{n})\sum_{i=1}^{n}\bd{\xi}_{i} = O_{P}(\sqrt{d \log{d}})$ wrt $\normE{\cdot}$.
\end{lem}

\begin{proofName}
\subsection{Proof of Theorem \ref{thm3.4:d:infty:alternative:Paper}}
\end{proofName}
\begin{proof}
[Proof of Theorem
\ref{thm3.4:d:infty:alternative:Paper}]
We reuse the notations $T_{i}$'s in the proof of Theorems
\ref{thm5.2},
from which,
$
\sqrt{n}\bd{A}_{n}(\hat{\bd{\beta}}_{n} - \bd{\beta}^{\star})
=V_{1}+V_{2}+V_{3}-V_{4}
$,
where
$V_{i}=\bd{B}_{n}T_{i}$
for $i=1,2,3,4$
and $\bd{B}_{n}=\sqrt{n}\bd{A}_{n}
T_{0}^{-1}$.
It is sufficient to show that
$V_{2}\ConvDist
N(0, \sigma^{2}\bd{G}_{X})$
and other $V_{i}$'s are $o_{P}(1)$.
Consider $V_{1}$.
We have
$
\normE{V_{1}}
\leq
\sqrt{n d}
\normF{\bd{A}_{n}}
\normFd{T_{0}^{-1}}
\normE{T_{1}}.
$
By Assumption (F),
$\normF{\bd{A}_{n}}$
is bounded.
By Lemmas
\ref{PC.Paper.Inverse.Matrix.Convergence.Probability.d.Infity}
and
\ref{PC.Paper.Sample.Covariance.Matrix.Convergence.d.infty}
and Assumption (D),
for $d=o(n^{1/3})$, wpg1,
$\normFd{T_{0}^{-1}}$
is bounded.
We have, wpg1,
$
\normE{T_{1}}
\leq
s_{2}\kappa_{n}\gamma_{n}/n$.
Then,
$
\normE{V_{1}}
\lesssim
\sqrt{d/n}s_{2}\kappa_{n}\gamma_{n},
$
Thus,
$\normE{V_{1}}=o_{P}(1)$
for $s_{2}=o(\sqrt{n}/(\sqrt{d}\kappa_{n}\gamma_{n}))$.
Consider $V_{2}$.
We have
$V_{2}=V_{21}+ V_{22}$,
where
$
V_{21}
=
\sqrt{n}\bd{A}_{n}\bd{\bd{\Sigma}}_{X}^{-1}T_{2}$ and
$V_{22}
=
\sqrt{n}\bd{A}_{n}
(T_{0}^{-1} - \bd{\bd{\Sigma}}_{X}^{-1})
T_{2}$.
First, note that
$
V_{21}
=
\sqrt{(n-s_{1})/n}
\sum_{i=s_{1}+1}^{n}
\bd{Z}_{n,i}
$,
where
$
\bd{Z}_{n,i}
=
(1/\sqrt{n-s_{1}})
\bd{A}_{n}\bd{\bd{\Sigma}}_{X}^{-1}\bd{X}_{i}\epsilon_{i}
$.
On one hand,
for every $\delta>0$,
$
\sum_{i=s_{1}+1}^{n}
\E\normE{\bd{Z}_{n,i}}^{2}
\{\normE{\bd{Z}_{n,i}} > \delta\}
\leq
(n-s_{1})
\E\normE{\bd{Z}_{n,0}}^{4}/\delta^{2},
$
and
$\E\normE{\bd{Z}_{n,0}}^{4}
=
\frac{1}{(n-s_{1})^{2}}
\E
\epsilon_{0}^{4}
\E
(
\bd{X}_{0}^{T}
\bd{\bd{\Sigma}}_{X}^{-1}
\bd{A}_{n}^{T}
\bd{A}_{n}
\bd{\bd{\Sigma}}_{X}^{-1}
\bd{X}_{0}
)^{2}$,
which is
$\leq
\frac{d^{2}}{(n-s_{1})^{2}}
\E
\epsilon_{0}^{4}
\lambda_{\max}(\bd{G}_{n})
\lambda_{\min}^{-2}(\bd{\bd{\Sigma}}_{X})
\kappa_{X}^{2}$.
Then,
by Assumptions (D'), (E) and (F) and for $d=o(\sqrt{n})$,
$
\sum_{i=s_{1}+1}^{n}
\E\normE{\bd{Z}_{n,i}}^{2}
\{\normE{\bd{Z}_{n,i}} > \delta\}
\rightarrow 0
$.
On the other hand,
$
\sum_{i=s_{1}+1}^{n}
\Cov(\bd{Z}_{n,i})
=
\sigma^{2}\bd{A}_{n}\bd{\bd{\Sigma}}_{X}^{-1}\bd{A}_{n}^{T}
\rightarrow
\sigma^{2}\bd{G}_{X}
$
by Assumption (F).
Thus, by central limit theorem
(see Proposition 2.27 in \cite{Vaart1998}),
$
V_{21}
\ConvDist
N(0, \sigma^{2}\bd{G}_{X})
$.
Next, consider $V_{22}$.
Note that
$
\normE{V_{22}}
\leq
\normF{\bd{A}_{n}}
(d\log(d))^{1/2}\normF{T_{0}^{-1} - \bd{\bd{\Sigma}}_{X}^{-1}}
(d\log(d))^{-1/2}\normE{\sqrt{n}T_{2}}
$.
By Assumption (F), $\normF{\bd{A}_{n}}$ is $O(1)$;
by Lemmas \ref{PC.Paper.Inverse.Matrix.Convergence.Probability.d.Infity} and
        \ref{PC.Paper.Sample.Covariance.Matrix.Convergence.d.infty},
        $(d\log(d))^{1/2}\normF{T_{0}^{-1} - \bd{\bd{\Sigma}}_{X}^{-1}}$ is $o_{P}(1)$
        for $d^{5}\log(d) = o(n)$;
by Lemma \ref{PC.Paper.Scaled.IID.Sum.Convergence.Rate.d.infty},
        $(d\log(d))^{-1/2}\normE{\sqrt{n}T_{2}}=(d\log(d))^{-1/2}\normE{\frac{1}{\sqrt{n}}\mathbb{S}_{s_{1}+1,n}^{\epsilon}}$ is $O_{P}(1)$ for $d=o(\sqrt{n})$.
Then,
$V_{22}\ConvProb 0$.
Thus, by slutsky's lemma,
$V_{2}\ConvDist
N(0, \sigma^{2}\bd{G}_{X})$.
Consider $V_{3}$ and $V_{4}$.
First consider $V_{3}$.
By noting that $s_{1}=o(\sqrt{n}/(\lambda \sqrt{d}\kappa_{n}))$,
wpg1,
$
\normE{V_{3}}
\leq
\sqrt{n d}
\normF{\bd{A}_{n}}
\normFd{T_{0}^{-1}}
\normE{T_{3}}
\lesssim
\sqrt{d}\lambda s_{1}\kappa_{n}/\sqrt{n}
\rightarrow 0$.
Thus, $\normE{V_{3}}=o_{P}(1)$.
In the same way, $\normE{V_{4}}=o_{P}(1)$.
%
\end{proof}

\begin{proofLong}
\begin{proof}
[Long Proof of Theorem
\ref{thm3.4:d:infty:alternative:Paper}]
We reuse the notations on $T_{i}$'s in the proof of Theorems
\ref{thm5.2},
from which,
$$
\sqrt{n}\bd{A}_{n}(\hat{\bd{\beta}}_{n} - \bd{\beta}^{\star})
=V_{1}+V_{2}+V_{3}-V_{4},
$$
where
\begin{align*}
V_{1} &=
\sqrt{n}\bd{A}_{n}
T_{0}^{-1}T_{1}, \\
V_{2} &=
\sqrt{n}\bd{A}_{n}
T_{0}^{-1}T_{2}, \\
V_{3} &=
\sqrt{n}\bd{A}_{n}
T_{0}^{-1}T_{3}, \\
V_{4} &=
\sqrt{n}\bd{A}_{n}
T_{0}^{-1}T_{4}.
\end{align*}

Next we derive the asymptotic properties of $V_{i}$'s,
from which the desired result follows by applying Slutsky's lemma.

\lineProof

\OnItem{On $V_{1}$.}
We have
\begin{align*}
\normE{V_{1}}
& =
\normE{\sqrt{n}
T_{0}^{-1}T_{1}} \\
& \leq
\sqrt{n}
\normF{\bd{A}_{n}}
\normF{T_{0}^{-1}}
\normE{T_{1}} \\
& =
\sqrt{n d}
\normF{\bd{A}_{n}}
\normFd{T_{0}^{-1}}
\normE{T_{1}}.
\end{align*}

By Assumption (F)
$\normF{\bd{A}_{n}}$
is bounded.

By Lemmas
\ref{PC.Paper.Inverse.Matrix.Convergence.Probability.d.Infity}
and
\ref{PC.Paper.Sample.Covariance.Matrix.Convergence.d.infty}
and Assumption (D),
for $d=o(n^{1/3})$, wpg1,
$\normFd{T_{0}^{-1}}$
is bounded.

We have, wpg1,
\begin{equation*}
\normE{T_{1}}
=
\normE{\frac{1}{n}\sum_{i=s_{1}+1}^{s}\bd{X}_{i}\mu_{i}^{\star}}
\leq
\frac{1}{n}s_{2}\kappa_{n}\gamma_{n}.
\end{equation*}

Then,
\begin{align*}
\normE{V_{1}}
& \lesssim
\sqrt{\frac{d}{n}}s_{2}\kappa_{n}\gamma_{n},
\end{align*}
where $\lesssim$ means that the left side is bounded by a constant times the right side.

Thus,
$\normE{V_{1}}=o_{P}(1)$
for $s_{2}=o(\sqrt{n}/(\sqrt{d}\kappa_{n}\gamma_{n}))$.

\lineProof

\OnItem{On $V_{2}$.}
We have
$V_{2}=V_{21}+ V_{22}$,
where
\begin{align*}
V_{21}
& =
\sqrt{n}\bd{A}_{n}\bd{\bd{\Sigma}}_{X}^{-1}
T_{2}, \\
V_{22}
& =
\sqrt{n}\bd{A}_{n}
(T_{0}^{-1} - \bd{\bd{\Sigma}}_{X}^{-1})
T_{2}. \\
\end{align*}

First, consider $V_{21}$.
We have
\begin{align*}
V_{21}
& =
\sqrt{n}\bd{A}_{n}\bd{\bd{\Sigma}}_{X}^{-1}T_{2} \\
& =
\sqrt{n}\bd{A}_{n}\bd{\bd{\Sigma}}_{X}^{-1}
\frac{1}{n}\sum_{i=s_{1}+1}^{n}\bd{X}_{i}\epsilon_{i} \\
& =
\frac{\sqrt{n-s_{1}}}{\sqrt{n}}
\sum_{i=s_{1}+1}^{n}
\bd{Z}_{n,i},
\end{align*}
where
\begin{equation*}
\bd{Z}_{n,i}
=
\frac{1}{\sqrt{n-s_{1}}}
\bd{A}_{n}\bd{\bd{\Sigma}}_{X}^{-1}\bd{X}_{i}\epsilon_{i}.
\end{equation*}

For every $\delta>0$,
\begin{align*}
\sum_{i=s_{1}+1}^{n}
\E\normE{\bd{Z}_{n,i}}^{2}
\{\normE{\bd{Z}_{n,i}} > \delta\}
& \leq
\sum_{i=s_{1}+1}^{n}
\E\normE{\bd{Z}_{n,i}}^{4}/\delta^{2}
=
(n-s_{1})
\E\normE{\bd{Z}_{n,0}}^{4}/\delta^{2},
\end{align*}
and
\begin{align*}
\E\normE{\bd{Z}_{n,0}}^{4}
& =
\E(\normE{\bd{Z}_{n,0}}^{2})^{2} \\
& =
\E(
\frac{1}{n-s_{1}}
\epsilon_{0}^{2}
\bd{X}_{0}^{T}
\bd{\bd{\Sigma}}_{X}^{-1}
\bd{A}_{n}^{T}
\bd{A}_{n}
\bd{\bd{\Sigma}}_{X}^{-1}
\bd{X}_{0}
)^{2} \\
& =
\frac{1}{(n-s_{1})^{2}}
\E
\epsilon_{0}^{4}
\E
(
\bd{X}_{0}^{T}
\bd{\bd{\Sigma}}_{X}^{-1}
\bd{A}_{n}^{T}
\bd{A}_{n}
\bd{\bd{\Sigma}}_{X}^{-1}
\bd{X}_{0}
)^{2} \\
& \leq
\frac{1}{(n-s_{1})^{2}}
\E
\epsilon_{0}^{4}
\lambda_{\max}(\bd{G}_{n})
\lambda_{\min}(\bd{\bd{\Sigma}}_{X})^{-2}
\E
(
\bd{X}_{0}^{T}
\bd{X}_{0}
)^{2} \\
& \leq
\frac{d^{2}}{(n-s_{1})^{2}}
\E
\epsilon_{0}^{4}
\lambda_{\max}(\bd{G}_{n})
\lambda_{\min}(\bd{\bd{\Sigma}}_{X})^{-2}
\frac{1}{d^{2}}
\E
(
\bd{X}_{0}^{T}
\bd{X}_{0}
)^{2}\\
& \leq
\frac{d^{2}}{(n-s_{1})^{2}}
\E
\epsilon_{0}^{4}
\lambda_{\max}(\bd{G}_{n})
\lambda_{\min}(\bd{\bd{\Sigma}}_{X})^{-2}
\frac{1}{d^{2}}
\E
(
\sum_{j=1}^{d}X_{0j}^{2}
)^{2}\\
& \leq
\frac{d^{2}}{(n-s_{1})^{2}}
\E
\epsilon_{0}^{4}
\lambda_{\max}(\bd{G}_{n})
\lambda_{\min}(\bd{\bd{\Sigma}}_{X})^{-2}
\frac{1}{d^{2}}
\sum_{k=1}^{d}\sum_{l=1}^{d}\E X_{0k}^{2}X_{0l}^{2} \\
& \leq
\frac{d^{2}}{(n-s_{1})^{2}}
\E
\epsilon_{0}^{4}
\lambda_{\max}(\bd{G}_{n})
\lambda_{\min}(\bd{\bd{\Sigma}}_{X})^{-2}
\frac{1}{d^{2}}
\sum_{k=1}^{d}\sum_{l=1}^{d}(\E X_{0k}^{4})^{1/2}(\E X_{0l}^{4})^{1/2} \\
& \leq
\frac{d^{2}}{(n-s_{1})^{2}}
\E
\epsilon_{0}^{4}
\lambda_{\max}(\bd{G}_{n})
\lambda_{\min}(\bd{\bd{\Sigma}}_{X})^{-2}
(\frac{1}{d}\sum_{j=1}^{d}(\E X_{0j}^{4})^{1/2})^{2}.
\end{align*}
Then,
\begin{align*}
\sum_{i=s_{1}+1}^{n}
\E\normE{\bd{Z}_{n,i}}^{2}
\{\normE{\bd{Z}_{n,i}} > \delta\}
& \leq
\frac{d^{2}}{n-s_{1}}
\frac{1}{\delta^{2}}
\E
\epsilon_{0}^{4}
\lambda_{\max}(\bd{G}_{n})
\lambda_{\min}(\bd{\bd{\Sigma}}_{X})^{-2}
(\frac{1}{d}\sum_{j=1}^{d}(\E X_{0j}^{4})^{1/2})^{2}.
\end{align*}
Thus, by Assumption (E) and for $d=o(\sqrt{n})$,
\begin{equation*}
\sum_{i=s_{1}+1}^{n}
\E\normE{\bd{Z}_{n,i}}^{2}
\{\normE{\bd{Z}_{n,i}} > \delta\}
\rightarrow 0.
\end{equation*}

On the other hand,
\begin{align*}
\sum_{i=s_{1}+1}^{n}
\Cov(\bd{Z}_{n,i})
& =
\sum_{i=s_{1}+1}^{n}
\frac{1}{\sqrt{n-s_{1}}}
\bd{A}_{n}\bd{\bd{\Sigma}}_{X}^{-1}
\Cov(\bd{X}_{i}\epsilon_{i})
\bd{\bd{\Sigma}}_{X}^{-1}
\bd{A}_{n}^{T}
\frac{1}{\sqrt{n-s_{1}}} \\
& =
\sum_{i=s_{1}+1}^{n}
\frac{1}{n-s_{1}}
\bd{A}_{n}
\bd{\bd{\Sigma}}_{X}^{-1}
\bd{\bd{\Sigma}}_{X}\sigma^{2}
\bd{\bd{\Sigma}}_{X}^{-1}
\bd{A}_{n}^{T}\\
& =
\sigma^{2}\bd{A}_{n}\bd{\bd{\Sigma}}_{X}^{-1}\bd{A}_{n}^{T}
\rightarrow
\sigma^{2}\bd{G}_{X}.
\end{align*}

Thus, by central limit theorem
(see, for example, Proposition 2.27 in \cite{Vaart1998}),
\begin{align*}
V_{21}
& \ConvDist
N(0, \sigma^{2}\bd{G}_{X}).\\
\end{align*}

Next, consider $V_{22}$.
We have
\begin{align*}
\normE{V_{22}}
& =
\normE{\sqrt{n}\bd{A}_{n}
(T_{0}^{-1} - \bd{\bd{\Sigma}}_{X}^{-1})T_{2}} \\
& \leq
\sqrt{n}
\normF{\bd{A}_{n}}
\normF{T_{0}^{-1} - \bd{\bd{\Sigma}}_{X}^{-1}}
\normE{T_{2}} \\
& =
\normF{\bd{A}_{n}}
(d\log(d))^{1/2}\normF{T_{0}^{-1} - \bd{\bd{\Sigma}}_{X}^{-1}}
(d\log(d))^{-1/2}\normE{\sqrt{n}T_{2}}.
\end{align*}
Note that
\begin{itemize}
  \item By Assumption (F), $\normF{\bd{A}_{n}}$ is $O(1)$;
  \item by Lemmas \ref{PC.Paper.Inverse.Matrix.Convergence.Probability.d.Infity} and
        \ref{PC.Paper.Sample.Covariance.Matrix.Convergence.d.infty},
        $(d\log(d))^{1/2}\normF{T_{0}^{-1} - \bd{\bd{\Sigma}}_{X}^{-1}}$ is $o_{P}(1)$
        for $d^{5}\log(d) = o(n)$;
  \item by Lemma \ref{PC.Paper.Scaled.IID.Sum.Convergence.Rate.d.infty},
        $(d\log(d))^{-1/2}\normE{\sqrt{n}T_{2}}=(d\log(d))^{-1/2}\normE{\frac{1}{\sqrt{n}}\mathbb{S}_{s_{1}+1,n}^{\epsilon}}$ is $O_{P}(1)$ for $d=o(\sqrt{n})$.
\end{itemize}
Then,
$V_{22}\ConvProb 0$.

\lineProof

Thus, by slutsky's lemma,
$V_{2}\ConvDist
N(0, \sigma^{2}\bd{G}_{X})$.

\lineProof

\OnItem{On $V_{3}$ and $V_{4}$.}
First consider $V_{3}$.
By noting that $s_{1}=o(\sqrt{n}/(\lambda \sqrt{d}\kappa_{n}))$,
wpg1,
\begin{align*}
\normE{V_{3}}
& =
\normE{\sqrt{n}\bd{A}_{n}
T_{0}^{-1}T_{3}} \\
& \leq
\sqrt{n}
\normF{\bd{A}_{n}}
\normF{T_{0}^{-1}}
\normE{T_{3}} \\
& =
\sqrt{n d}
\normF{\bd{A}_{n}}
\normFd{T_{0}^{-1}}
\normE{T_{3}} \\
& \lesssim
\sqrt{n d}\lambda s_{1}\kappa_{n}/n
=
\sqrt{d}\lambda s_{1}\kappa_{n}/\sqrt{n}
\rightarrow 0.
\end{align*}
Thus, $\normE{V_{3}}=o_{P}(1)$.
In the same way, $\normE{V_{4}}=o_{P}(1)$.

Therefore, the result of the theorem follows
by Slutsky's lemma.
\end{proof}
\end{proofLong}

\begin{proof}[Proof of Theorem
\ref{thm5.4}]
By the definition of $\mathcal{E}$, we have
$
P(\mathcal{\mathcal{E}})
= T_{1}T_{2}T_{3}$,
where
$T_{1} = P(\bigcap_{i=1}^{s_{1}}\{|\mu_{i}^{\star}+\bd{X}_{i}^{T}(\bd{\beta}^{\star}-\hat{\bd{\beta}})+\epsilon_{i}| > \lambda\})$,
$T_{2} = P(\bigcap_{i=s_{1}+1}^{s}\{|\mu_{i}^{\star}+\bd{X}_{i}^{T}(\bd{\beta}^{\star}-\hat{\bd{\beta}})+\epsilon_{i}| \leq \lambda\})$
and
$T_{3} = P(\bigcap_{i=s+1}^{n}\{|\bd{X}_{i}^{T}(\bd{\beta}^{\star}-\hat{\bd{\beta}})+\epsilon_{i}| \leq \lambda\})$.
We will show that each $T_{i}$ converges to one.
Then, $P(\mathcal{\mathcal{E}})\rightarrow 1$.
Denote $\mathcal{C}=\{r_d\normE{\hat{\bd{\beta}}-\bd{\beta}^{\star}}\leq 1\}$.
Then $P(\mathcal{C})\rightarrow 1$ since $\hat{\bd{\beta}}$ is a consistent estimator of $\bd{\beta}^{\star}$
wrt $r_d\normE{\cdot}$.
Consider $T_{1}$.
We have
$1-T_{1}\leq T_{11} + T_{12},$
where
$T_{11} =
P(\bigcup_{i\in S_{21}^{\star}}\{|\mu_{i}^{\star}+\bd{X}_{i}^{T}(\bd{\beta}^{\star}-\hat{\bd{\beta}})+\epsilon_{i}| \leq \lambda\})$
and
$T_{12} =
P(\bigcup_{i\in S_{31}^{\star}}\{|\mu_{i}^{\star}+\bd{X}_{i}^{T}(\bd{\beta}^{\star}-\hat{\bd{\beta}})+\epsilon_{i}| \leq \lambda\}).$
It is sufficient to show that both $T_{11}$ and $T_{12}$ converge to zero.
By
$\sqrt{d}\kappa_{n} \ll  \lambda \ll  \mu^{\star}$,
$
T_{11}
\leq
P(
\bigcup_{i\in S_{21}^{\star}}\{\epsilon_{i} \leq \lambda - \mu^{\star} + \normE{\bd{X}_{i}}\cdot\normE{\hat{\bd{\beta}} - \bd{\beta}^{\star}}\}, \mathcal{C}) + P(\mathcal{C}^{c})$,
which is
$\leq
P(
\bigcup_{i\in S_{21}^{\star}}\{\epsilon_{i} \leq \lambda - \mu^{\star} + \sqrt{d}\kappa_{n}\}) + P(\mathcal{C}^{c})
\leq
s_{1}
P\{\epsilon_{0} \leq -\gamma_{n}\} + P(\mathcal{C}^{c})
\longrightarrow 0$.
Similarly, $T_{12}\rightarrow 0$.
Thus
$T_{1}\rightarrow 1$.
Consider $T_{2}$ and $T_{3}$.
By
$\alpha\gamma_{n} \leq \lambda$
and $\sqrt{d}\kappa_{n} \ll \lambda$,
$
T_{2}
\geq
P(
\bigcap_{i=s_{1}}^{s}\{-\lambda - \mu_{i}^{\star} + (1/r_d)\kappa_{n} \leq \epsilon_{i} \leq \lambda - \mu_{i}^{\star} - (1/r_d)\kappa_{n} \}, \mathcal{C})$,
which is
$\geq
P(
\bigcap_{i=s_{1}}^{s}\{-\lambda - \mu_{i}^{\star} + \sqrt{d}\kappa_{n} \leq \epsilon_{i} \leq \lambda - \mu_{i}^{\star} - \sqrt{d}\kappa_{n} \}, \mathcal{C})
\geq
P(
\bigcap_{i=s_{1}}^{s}\{-\gamma_{n} \leq \epsilon_{i} \leq \gamma_{n} \}, \mathcal{C})
\rightarrow 1$.
Then
$T_{2}\rightarrow 1$.
Similarly, $T_{3} \rightarrow 1$.
\end{proof}

\begin{proofLong}
\begin{proof}[Long Proof of Theorem
\ref{thm5.4}]
By the definition of $\mathcal{E}$, we have
$$
P(\mathcal{\mathcal{E}})
= T_{1}T_{2}T_{3},$$
where
\begin{align*}
T_{1} &= P(\bigcap_{i=1}^{s_{1}}\{|\mu_{i}^{\star}+\bd{X}_{i}^{T}(\bd{\beta}^{\star}-\hat{\bd{\beta}})+\epsilon_{i}| > \lambda\}), \\
T_{2} &= P(\bigcap_{i=s_{1}+1}^{s}\{|\mu_{i}^{\star}+\bd{X}_{i}^{T}(\bd{\beta}^{\star}-\hat{\bd{\beta}})+\epsilon_{i}| \leq \lambda\}), \\
T_{3} &= P(\bigcap_{i=s+1}^{n}\{|\bd{X}_{i}^{T}(\bd{\beta}^{\star}-\hat{\bd{\beta}})+\epsilon_{i}| \leq \lambda\}).
\end{align*}

We will show that each of $T_{1}$, $T_{2}$ and $T_{3}$ converges to one.
Then $P(\mathcal{\mathcal{E}})\rightarrow 1$.

\lineProof

Denote $\mathcal{C}=\{r_d\normE{\hat{\bd{\beta}}-\bd{\beta}^{\star}}\leq 1\}$.
Then $P(\mathcal{C})\rightarrow 1$ since $\hat{\bd{\beta}}$ is a consistent estimator of $\bd{\beta}^{\star}$
wrt $r_d\normE{\cdot}$.

\OnItem{On $T_{1}$.}
We have
$$1-T_{1}\leq T_{11} + T_{12},$$
where
\begin{align*}
T_{11} &=
P(\bigcup_{i\in S_{21}^{\star}}\{|\mu_{i}^{\star}+\bd{X}_{i}^{T}(\bd{\beta}^{\star}-\hat{\bd{\beta}})+\epsilon_{i}| \leq \lambda\})\\
T_{12} &=
P(\bigcup_{i\in S_{31}^{\star}}\{|\mu_{i}^{\star}+\bd{X}_{i}^{T}(\bd{\beta}^{\star}-\hat{\bd{\beta}})+\epsilon_{i}| \leq \lambda\}).
\end{align*}
It is sufficient to show that both $T_{11}$ and $T_{12}$ converge to zero.

By noting
$\sqrt{d}\kappa_{n} \ll  \lambda \ll  \mu^{\star}$, we have
\begin{align*}
T_{11}
& \leq
P(
\bigcup_{i\in S_{21}^{\star}}\{\epsilon_{i} \leq \lambda - \mu^{\star} + \normE{\bd{X}_{i}}\cdot\normE{\hat{\bd{\beta}} - \bd{\beta}^{\star}}\}, \mathcal{C}) + P(\mathcal{C}^{c})\\
& \leq
P(
\bigcup_{i\in S_{21}^{\star}}\{\epsilon_{i} \leq \lambda - \mu^{\star} + (1/r_d)\kappa_{n}\}) + P(\mathcal{C}^{c}) \\
& \leq
P(
\bigcup_{i\in S_{21}^{\star}}\{\epsilon_{i} \leq \lambda - \mu^{\star} + \sqrt{d}\kappa_{n}\}) + P(\mathcal{C}^{c}) \\
& \leq
s_{1}
P\{\epsilon_{0} \leq -\gamma_{n}\} + P(\mathcal{C}^{c})
\longrightarrow 0.
\end{align*}
Similarly, $T_{12}\rightarrow 0$.
Thus
$T_{1}\rightarrow 1$.

\lineProof

\OnItem{On $T_{2}$ and $T_{3}$.}
By noting
$\alpha\gamma_{n} \leq \lambda$
and $\sqrt{d}\kappa_{n} \ll \lambda$,
we have
\begin{align*}
T_{2}
& \geq
P(
\bigcap_{i=s_{1}}^{s}\{-\lambda - \mu_{i}^{\star} + (1/r_d)\kappa_{n} \leq \epsilon_{i} \leq \lambda - \mu_{i}^{\star} - (1/r_d)\kappa_{n} \}, \mathcal{C}) \\
& \geq
P(
\bigcap_{i=s_{1}}^{s}\{-\lambda - \mu_{i}^{\star} + \sqrt{d}\kappa_{n} \leq \epsilon_{i} \leq \lambda - \mu_{i}^{\star} - \sqrt{d}\kappa_{n} \}, \mathcal{C}) \\
& \geq
P(
\bigcap_{i=s_{1}}^{s}\{-\gamma_{n} \leq \epsilon_{i} \leq \gamma_{n} \}, \mathcal{C})
\rightarrow 1.
\end{align*}
Then
$T_{2}\rightarrow 1$.
Similarly, $T_{3} \rightarrow 1$.
Then, we have
$T_{3}\rightarrow 1$.
This completes the proof.
\end{proof}
\end{proofLong}

\begin{proofName}
\subsection{Proof of Theorem \ref{thm3.8:d:infty:alternative:Paper}}
\end{proofName}
\begin{proof}
[Proof of Theorem \ref{thm3.8:d:infty:alternative:Paper}]
Note that
$ \sqrt{n}\bd{A}_{n}\bd{\bd{\Sigma}}_{X}^{1/2}(\tilde{\bd{\beta}} - \bd{\beta}^{\star})
=
\tilde{R}_{1}
+
\tilde{R}_{2}
+
V_{1}
+
V_{2}
$,
where
$\tilde{R}_{1} = \sqrt{n}\bd{A}_{n}R_{1}$,
$\tilde{R}_{2} = \sqrt{n}\bd{A}_{n}R_{2}$,
$
R_{1}
=
(\bd{X}_{\hat I_{0}}^{T}\bd{X}_{\hat I_{0}})^{-1}\bd{X}_{\hat I_{0}}^{T}\bd{Y}_{\hat I_{0}}\{\hat{I}_{0}\not=I_{0}\}
$,
$
R_{2}
=
-(\bd{X}_{I_{0}}^{T}\bd{X}_{I_{0}})^{-1}\bd{X}_{I_{0}}^{T}\bd{Y}_{I_{0}}\{\hat{I}_{0}\not=I_{0}\}
$,
and $V_{i}$'s are defined in the proof of Theorem
\ref{thm3.4:d:infty:alternative:Paper}.
Since $P(\normE{\tilde{R}_{1}}=0)\geq P\{\hat{I}_{0} = I_{0}\} \rightarrow 1$,
we have $\tilde{R}_{1} = o_{P}(1)$.
Similarly, $\tilde{R}_{2} = o_{P}(1)$.
By the proof of Theorem
\ref{thm3.4:d:infty:alternative:Paper},
$V_{1}=o_{P}(1)$ and $V_{2}\ConvDist N(0, \sigma^{2}\bd{G}_{X})$.
Therefore, the desired result follows by Slutsky's lemma.
\end{proof}

\begin{proofLong}
\begin{proof}
[Long Proof of Theorem \ref{thm3.8:d:infty:alternative:Paper}]
We reuse the notations on $T_{i}$'s in the proof of Theorems
\ref{thm3.4:d:infty:alternative:Paper}.
We have
\begin{align*}
\tilde{\bd{\beta}} - \bd{\beta}^{\star}
& =
R_{1} + R_{2} + T_{0}^{-1}T_{1} + T_{0}^{-1}T_{2},
\end{align*}
where
$
R_{1}
=
(\bd{X}_{\hat I_{0}}^{T}\bd{X}_{\hat I_{0}})^{-1}\bd{X}_{\hat I_{0}}^{T}\bd{Y}_{\hat I_{0}}\{\hat{I}_{0}\not=I_{0}\}
$,
$
R_{2}
=
-(\bd{X}_{I_{0}}^{T}\bd{X}_{I_{0}})^{-1}\bd{X}_{I_{0}}^{T}\bd{Y}_{I_{0}}\{\hat{I}_{0}\not=I_{0}\}
$,
$T_{0} = \mathbb{S}_{s_{1}+1,n}/n$,
$
T_{1} =
\mathbb{S}_{S_{12}^{\star}}^{\mu}/n$
and
$
T_{2} =
\mathbb{S}_{s_{1}+1,n}^{\epsilon}/n$.
Thus,

We have
\begin{align*}
\sqrt{n}\bd{A}_{n}(\tilde{\bd{\beta}} - \bd{\beta}^{\star})
& =
\sqrt{n}\bd{A}_{n}R_{1}
+ \sqrt{n}\bd{A}_{n}R_{2}
+ \sqrt{n}\bd{A}_{n}T_{0}^{-1}T_{1}
+ \sqrt{n}\bd{A}_{n}T_{0}^{-1}T_{2} \\
& =:
\tilde{R}_{1}
+
\tilde{R}_{2}
+
V_{1}
+
V_{2}.
\end{align*}

\OnItem{On $\tilde{R}_{1}$ and $\tilde{R}_{2}$.}
Since $P(\normE{\tilde{R}_{1}}=0)\geq P\{\hat{I}_{0} = I_{0}\} \rightarrow 1$,
we have $\tilde{R}_{1} = o_{P}(1)$.
Similarly, $\tilde{R}_{2} = o_{P}(1)$.

\OnItem{On $V_{1}$.}
See the proof of Theorem
\ref{thm5.3}.

\OnItem{On $V_{2}$.}
See the proof of Theorem
\ref{thm5.3}.

Thus, we obtain the asymptotic distribution of $\tilde{\bd{\beta}}$ by Slutsky's lemma.

\end{proof}
\end{proofLong}

\begin{lem}[Consistency on $\hat{\sigma}$]
\label{lem3.9:d:infty:Paper}
Suppose the assumptions and conditions of
Theorem \ref{thm5.2} hold
with $r_d \geq \sqrt{d}$.
If $s_{2}=o(n/\gamma_{n}^{2})$,
then
$\hat{\sigma} \ConvProb \sigma$.
\end{lem}
\begin{proof}
[Proof of Lemma \ref{lem3.9:d:infty:Paper}]
Since the assumptions and conditions of
Theorem \ref{thm5.2} hold
with $r_d\geq\sqrt{d}$,
the penalized estimators $\hat{\bd{\beta}}$
and $\tilde{\bd{\beta}}$
are consistent estimators of $\bd{\beta}^{\star}$
wrt $\sqrt{d}\normE{\cdot}$
by Theorems \ref{thm5.2}
and
\ref{thm5.5} in Supplement \ref{PC.Paper.Supplement.Extension.d.infty}.
Let $\mathcal{\mathcal{A}}=\{\hat{I}_{0}=I_{0}\}$.
Then
$\mathcal{\mathcal{A}}$ occurs wpg1
by Theorem \ref{thm5.4}.

Note that
$\hat{\sigma}^{2}=T\mathcal{\mathcal{A}}+\hat{\sigma}^{2}\mathcal{\mathcal{A}}^{c}$,
where
$
T
=(n-s_{1})^{-1}\normE{ \bd{Y}_{I_{0}} - \bd{X}_{I_{0}}^{T} \tilde{\bd{\beta}} }^{2}
$.
It suffices to show that
$T\ConvProb \sigma^{2}$.
Note that $T=\sum_{i=1}^{6}T_{i}$, where
$T_{1}  =  (n-s_{1})^{-1}\sum_{i=s_{1}+1}^{n}[\bd{X}_{i}^{T}(\bd{\beta}^{\star} - \tilde{\bd{\beta}})]^{2}$,
$T_{2}  =  (n-s_{1})^{-1}\sum_{i=s_{1}+1}^{n} \epsilon_{i}^{2}$,
$T_{3}  =  2(n-s_{1})^{-1}\sum_{i=s_{1}+1}^{n} \bd{X}_{i}^{T}(\bd{\beta}^{\star} - \tilde{\bd{\beta}})\epsilon_{i}$,
$T_{4}  =  (n-s_{1})^{-1}\sum_{i=s_{1}+1}^{s} \mu_{i}^{\star 2}$,
$T_{5}  =  2(n-s_{1})^{-1}\sum_{i=s_{1}+1}^{s} \mu_{i}\bd{X}_{i}^{T}(\bd{\beta}^{\star} - \tilde{\bd{\beta}})$
and
$T_{6}  =  2(n-s_{1})^{-1}\sum_{i=s_{1}+1}^{s} \mu_{i}^{\star}\epsilon_{i}$.
It is clear that $T_{2}\ConvProb \sigma^{2}$.
Thus, it is sufficient to show
other $T_{i}$'s are $o_{P}(1)$.
For every $\eta>0$,
wpg1,
$\sqrt{d}\normE{\bd{\beta}^{\star} - \tilde{\bd{\beta}}}\leq \eta$.
By Assumption (E1), wpg1,
$
|T_{1}|
\leq
\frac{1}{d}\frac{1}{n-s_{1}}\sum_{i=s_{1}+1}^{n}\normE{\bd{X}_{i}^{T}}^{2}
(\sqrt{d}\normE{\bd{\beta}^{\star} - \tilde{\bd{\beta}}})^{2}
\leq
2\eta^{2}\frac{1}{d}\E\normE{\bd{X}_{0}^{T}}^{2}
=
2\eta^{2}\bar{\sigma}_{X}^{2}
\lesssim \eta^{2}
$.
For every $\eta>0$, wpg1,
$
|T_{3}|
\leq
2\frac{1}{\sqrt{d}}\frac{1}{n-s_{1}}\sum_{i=s_{1}+1}^{n} \normE{\bd{X}_{i}^{T}\epsilon_{i}}\sqrt{d}\normE{\bd{\beta}^{\star} - \tilde{\bd{\beta}}}
\leq
4\eta\frac{1}{\sqrt{d}}\E\normE{\bd{X}_{0}^{T}\epsilon_{0}}
=
4\sigma\eta\bar{\sigma}_{X}
\lesssim \eta
$.
For $s_{2}=o(n/\gamma_{n}^{2})$,
$
|T_{4}|
\leq
(n-s_{1})^{-1} s_{2} \gamma_{n}^{2}
\rightarrow 0
$.
For $s_{2}=o(\sqrt{d}n/(\gamma_{n}\kappa_{n}))$,
$
|T_{5}|
\leq
2\frac{1}{\sqrt{d}}\frac{1}{n-s_{1}}s_{2} \gamma_{n}\kappa_{n}\sqrt{d}\normE{\bd{\beta}^{\star} - \tilde{\bd{\beta}}}
\leq
2\eta\frac{1}{\sqrt{d}}\frac{1}{n-s_{1}}s_{2} \gamma_{n}\kappa_{n}
\ConvProb 0
$.
For $s_{2}=o(n/\gamma_{n})$, wpg1,
$
|T_{6}|
\leq
4\frac{1}{n-s_{1}}\gamma_{n}s_{2}\E|\epsilon_{0}|
\rightarrow 0.
$
%
\end{proof}

\begin{proofLong}
\begin{proof}
[Long Proof of Lemma \ref{lem3.9:d:infty:Paper}]
Since the assumptions and conditions of
Theorem \ref{thm5.2} hold
with $r_d\geq\sqrt{d}$,
the penalized estimators $\hat{\bd{\beta}}$
and $\tilde{\bd{\beta}}$
are consistent estimators of $\bd{\beta}^{\star}$
wrt $\sqrt{d}\normE{\cdot}$
by Theorems \ref{thm5.2}
and
\ref{thm5.6}.
Let $\mathcal{\mathcal{B}}=\{\hat{I}_{0}=I_{0}\}$.
Then
$\mathcal{\mathcal{B}}$ occurs wpg1
by Theorem \ref{thm5.4}.

\lineProof

We have
$\hat{\sigma}^{2}=T\mathcal{\mathcal{B}}+\hat{\sigma}^{2}\mathcal{\mathcal{B}}^{c}$,
where
\begin{align*}
T
& =
\frac{1}{\#(I_{0})} \normE{ \bd{Y}_{\hat I_{0}} - \bd{X}_{\hat I_{0}}^{T} \tilde{\bd{\beta}} }^{2}
=\frac{1}{n-s_{1}} \normE{ \bd{Y}_{I_{0}} - \bd{X}_{I_{0}}^{T} \tilde{\bd{\beta}} }^{2}.
\end{align*}
Since $\hat{\sigma}^{2}\mathcal{\mathcal{B}}^{c}\ConvProb 0$, it is sufficient to show that
$T\ConvProb \sigma^{2}$.

\lineProof

We have $T=\sum_{i=1}^{6}T_{i}$, where
\begin{align*}
T_{1} & =  \frac{1}{n-s_{1}}\sum_{i=s_{1}+1}^{n}[\bd{X}_{i}^{T}(\bd{\beta}^{\star} - \tilde{\bd{\beta}})]^{2}, \\
T_{2} & =  \frac{1}{n-s_{1}}\sum_{i=s_{1}+1}^{n} \epsilon_{i}^{2}, \\
T_{3} & =  2\frac{1}{n-s_{1}}\sum_{i=s_{1}+1}^{n} \bd{X}_{i}^{T}(\bd{\beta}^{\star} - \tilde{\bd{\beta}})\epsilon_{i}, \\
T_{4} & =  \frac{1}{n-s_{1}}\sum_{i=s_{1}+1}^{s} \mu_{i}^{\star 2}, \\
T_{5} & =  2\frac{1}{n-s_{1}}\sum_{i=s_{1}+1}^{s} \mu_{i}\bd{X}_{i}^{T}(\bd{\beta}^{\star} - \tilde{\bd{\beta}}), \\
T_{6} & =  2\frac{1}{n-s_{1}}\sum_{i=s_{1}+1}^{s} \mu_{i}^{\star}\epsilon_{i}.
\end{align*}
\lineProof

We will show that.
Then, we have the desired result.

\OnItem{On $T_{1}$.}
For every $\eta>0$, wpg1,
\begin{align*}
|T_{1}|
& =
|\frac{1}{n-s_{1}}\sum_{i=s_{1}+1}^{n}[\bd{X}_{i}^{T}(\bd{\beta}^{\star} - \tilde{\bd{\beta}})]^{2}| \\
& =
\frac{1}{n-s_{1}}\sum_{i=s_{1}+1}^{n}[\bd{X}_{i}^{T}(\bd{\beta}^{\star} - \tilde{\bd{\beta}})]^{2} \\
& \leq
\frac{1}{n-s_{1}}\sum_{i=s_{1}+1}^{n}\normE{\bd{X}_{i}^{T}}^{2}\normE{\bd{\beta}^{\star} - \tilde{\bd{\beta}}}^{2} \\
& \leq
\frac{1}{d}\frac{1}{n-s_{1}}\sum_{i=s_{1}+1}^{n}\normE{\bd{X}_{i}^{T}}^{2}d\normE{\bd{\beta}^{\star} - \tilde{\bd{\beta}}}^{2} \\
(wpg1)& \leq
\frac{1}{d}\frac{1}{n-s_{1}}\sum_{i=s_{1}+1}^{n}\normE{\bd{X}_{i}^{T}}^{2}\eta \\
(wpg1)& \leq
2\eta\frac{1}{d}\E\normE{\bd{X}_{0}^{T}}^{2} \\
& =
2\eta\frac{1}{d}\E\sum_{j=1}^{d}X_{0,j}^{2} \\
& =
2\eta\frac{1}{d}\sum_{j=1}^{d}\sigma_{j}^{2} \\
(assumption)& \lesssim \eta.
\end{align*}
Thus, $T_{1}=o_{P}(1)$ under the conditions:
\begin{itemize}
  \item $\sqrt{d}\normE{\bd{\beta}^{\star} - \tilde{\bd{\beta}}}\ConvProb 0$.
  \item $\frac{1}{d}\sum_{j=1}^{d}\sigma_{j}^{2}$ is bounded.
\end{itemize}

\lineProof
\OnItem{On $T_{2}$.}
By law of large number,
\begin{align*}
T_{2}
& =
\frac{1}{n-s_{1}}\sum_{i=s_{1}+1}^{n} \epsilon_{i}^{2}
\ConvProb \sigma^{2}.
\end{align*}
Thus, $T_{2} \ConvProb \sigma^{2}$.

\lineProof
\OnItem{On $T_{3}$.}
For every $\eta>0$, wpg1,
\begin{align*}
|T_{3}|
& =
|2\frac{1}{n-s_{1}}\sum_{i=s_{1}+1}^{n} \bd{X}_{i}^{T}(\bd{\beta}^{\star} - \tilde{\bd{\beta}})\epsilon_{i}| \\
& \leq
2\frac{1}{n-s_{1}}\sum_{i=s_{1}+1}^{n} |\bd{X}_{i}^{T}\epsilon_{i}(\bd{\beta}^{\star} - \tilde{\bd{\beta}})| \\
& \leq
2\frac{1}{n-s_{1}}\sum_{i=s_{1}+1}^{n} \normE{\bd{X}_{i}^{T}\epsilon_{i}}\normE{\bd{\beta}^{\star} - \tilde{\bd{\beta}}} \\
& \leq
2\frac{1}{\sqrt{d}}\frac{1}{n-s_{1}}\sum_{i=s_{1}+1}^{n} \normE{\bd{X}_{i}^{T}\epsilon_{i}}\sqrt{d}\normE{\bd{\beta}^{\star} - \tilde{\bd{\beta}}} \\
(wpg1) & \leq
2\frac{1}{\sqrt{d}}\frac{1}{n-s_{1}}\sum_{i=s_{1}+1}^{n} \normE{\bd{X}_{i}^{T}\epsilon_{i}}\eta \\
(wpg1) & \leq
4\eta\frac{1}{\sqrt{d}}\E\normE{\bd{X}_{0}^{T}\epsilon_{0}} \\
& =
4\eta\frac{1}{\sqrt{d}}\E(\sum_{j=1}^{d}X_{0,j}^{2}\epsilon_{0}^{2})^{1/2} \\
& \leq
4\eta\frac{1}{\sqrt{d}}(\E \sum_{j=1}^{d}X_{0,j}^{2}\epsilon_{0}^{2})^{1/2} \\
& =
4\eta(\frac{1}{d}\sum_{j=1}^{d}\sigma_{j}^{2}\sigma^{2})^{1/2} \\
& =
4\sigma\eta(\frac{1}{d}\sum_{j=1}^{d}\sigma_{j}^{2})^{1/2} \\
&\lesssim \eta.
\end{align*}
Thus, $T_{3}=o_{P}(1)$ under the conditions:
\begin{itemize}
  \item $\sqrt{d}\normE{\bd{\beta}^{\star} - \tilde{\bd{\beta}}}\ConvProb 0$.
  \item $\frac{1}{d}\sum_{j=1}^{d}\sigma_{j}^{2}$ is bounded.
\end{itemize}

\lineProof
\OnItem{On $T_{4}$.}
We have
\begin{align*}
|T_{4}|
& =
|\frac{1}{n-s_{1}}\sum_{i=s_{1}+1}^{s} \mu_{i}^{\star 2}| \\
& =
\frac{1}{n-s_{1}}\sum_{i=s_{1}+1}^{s} \mu_{i}^{\star 2} \\
& \leq
\frac{1}{n-s_{1}}\sum_{i=s_{1}+1}^{s} \gamma_{n}^{2} \\
& =
\frac{1}{n-s_{1}} s_{2} \gamma_{n}^{2}
\rightarrow 0.
\end{align*}
Thus, $T_{4}=o(1)$
under the condition:
\begin{itemize}
  \item $s_{2}=o(n/\gamma_{n}^{2})$.
\end{itemize}

\lineProof
\OnItem{On $T_{5}$.}
We have,
\begin{align*}
|T_{5}|
& =
|2\frac{1}{n-s_{1}}\sum_{i=s_{1}+1}^{s} \mu_{i}\bd{X}_{i}^{T}(\bd{\beta}^{\star} - \tilde{\bd{\beta}})| \\
& =
2\frac{1}{n-s_{1}}\sum_{i=s_{1}+1}^{s} |\mu_{i}\bd{X}_{i}^{T}(\bd{\beta}^{\star} - \tilde{\bd{\beta}})| \\
& \leq
2\frac{1}{n-s_{1}}\sum_{i=s_{1}+1}^{s} |\mu_{i}\|\bd{X}_{i}^{T}(\bd{\beta}^{\star} - \tilde{\bd{\beta}})| \\
& \leq
2\frac{1}{n-s_{1}}\sum_{i=s_{1}+1}^{s} \gamma_{n}\normE{\bd{X}_{i}}\normE{\bd{\beta}^{\star} - \tilde{\bd{\beta}}} \\
& \leq
2\frac{1}{n-s_{1}}s_{2} \gamma_{n}\kappa_{n}\normE{\bd{\beta}^{\star} - \tilde{\bd{\beta}}} \\
& \leq
2\frac{1}{\sqrt{d}}\frac{1}{n-s_{1}}s_{2} \gamma_{n}\kappa_{n}\sqrt{d}\normE{\bd{\beta}^{\star} - \tilde{\bd{\beta}}}
\ConvProb 0.
\end{align*}
Thus, $T_{5}=o_{P}(1)$ under the conditions:
\begin{itemize}
  \item $s_{2}=o(\sqrt{d}n/(\gamma_{n}\kappa_{n}))$.
  \item $\sqrt{d}\normE{\bd{\beta}^{\star} - \tilde{\bd{\beta}}}\ConvProb 0$.
\end{itemize}

\lineProof
\OnItem{On $T_{6}$.}
We have
\begin{align*}
|T_{6}|
& =
|2\frac{1}{n-s_{1}}\sum_{i=s_{1}+1}^{s} \mu_{i}^{\star}\epsilon_{i}| \\
& \leq
2\frac{1}{n-s_{1}}\sum_{i=s_{1}+1}^{s} |\mu_{i}^{\star}\epsilon_{i}| \\
& \leq
2\frac{1}{n-s_{1}}\gamma_{n}\sum_{i=s_{1}+1}^{s} |\epsilon_{i}| \\
(wpg1)& \leq
2\frac{1}{n-s_{1}}\gamma_{n}s_{2}2\E|\epsilon_{0}|
\rightarrow 0.
\end{align*}
Thus, we have
$T_{6}=o_{P}(1)$ under the conditions
\begin{itemize}
  \item $s_{2}=o(n/\gamma_{n})$.
\end{itemize}
\end{proof}
\end{proofLong}

\begin{thm}
[Asymptotic Distributions on $\hat{\bd{\beta}}$ and $\tilde{\bd{\beta}}$ with $\hat{\bd{G}}_{X,n}$]
\label{model:three:multivariate:mean:zero:thm:twostage:limitdistribution.general.covariates.estimated.Sigma.alternative.errors:d:infty:Paper}
Under the assumptions and conditions
of Theorem
\ref{thm3.4:d:infty:alternative:Paper},
if $d^{8}(\log(d))^{2}=o(n)$,
then
$\sqrt{n}\hat{\bd{G}}_{X,n}^{-1/2}\bd{A}_{n}(\hat{\bd{\beta}} - \bd{\beta}^{\star} )
\ConvDist
N(0, \sigma^{2}\bd{I}_{q})$.
Similarly, under the assumptions and  conditions
of Theorem
\ref{thm3.8:d:infty:alternative:Paper},
If $d^{8}(\log(d))^{2}=o(n)$,
then
$\sqrt{n}\hat{\bd{G}}_{X,n}^{-1/2}\bd{A}_{n}(\tilde{\bd{\beta}} - \bd{\beta}^{\star} )
\ConvDist
N(0, \sigma^{2}\bd{I}_{q})$.
\end{thm}

Note that a stronger requirement on $d$ is required to handle $\hat{\bd{G}}_{X,n}^{-1/2}$ in
Theorem \ref{model:three:multivariate:mean:zero:thm:twostage:limitdistribution.general.covariates.estimated.Sigma.alternative.errors:d:infty:Paper}..
Below is a lemma needed for proving
Theorem \ref{model:three:multivariate:mean:zero:thm:twostage:limitdistribution.general.covariates.estimated.Sigma.alternative.errors:d:infty:Paper}.
\begin{lem}
[\cite{Wihler2009}]
\label{PC.Paper.Lemma.Inequality.Matrix.Square.Root.Differentce}
Suppose
$\bd{A}$ and $\bd{B}$ are $m\times m$ symmetric positive-semidefinite matrices.
Then, for $p>1$,
$
\normF{\bd{A}^{1/p} -\bd{B}^{1/p}}^{p}
\leq
m^{(p-1)/2}
\normF{\bd{A}-\bd{B}}
$.
Specifically, for $p=2$,
$
\normF{\bd{A}^{1/2} -\bd{B}^{1/2}}
\leq
(m^{1/2}
\normF{\bd{A}-\bd{B}})^{1/2}
$.
\end{lem}

\begin{proofName}
\subsection{Proof of Theorem
\ref{model:three:multivariate:mean:zero:thm:twostage:limitdistribution.general.covariates.estimated.Sigma.alternative.errors:d:infty:Paper}}
\end{proofName}
\begin{proof}
[Proof of Theorem
\ref{model:three:multivariate:mean:zero:thm:twostage:limitdistribution.general.covariates.estimated.Sigma.alternative.errors:d:infty:Paper}]
We only show the result on $\hat{\bd{\beta}}$.
since the result on $\tilde{\bd{\beta}}$ can be obtained in a similar way.
We reuse the notations $T_{i}$'s in the proof of Theorems
\ref{thm5.2},
from which,
$
\sqrt{n}
\hat{\bd{G}}_{X,n}^{-1/2}
\bd{A}_{n}(\hat{\bd{\beta}}_{n} - \bd{\beta}^{\star})
= M + R
$,
where
$M = \sqrt{n}
\bd{G}_{X,n}^{-1/2}
\bd{A}_{n}(\hat{\bd{\beta}}_{n} - \bd{\beta}^{\star})$
and
$R = \sqrt{n}
(\hat{\bd{G}}_{X,n}^{-1/2} - \bd{G}_{X,n}^{-1/2})
\bd{A}_{n}(\hat{\bd{\beta}}_{n} - \bd{\beta}^{\star})$.
By Theorem
\ref{thm3.4:d:infty:alternative:Paper},
$M \ConvDist N(0, \sigma^{2}\bd{G}_{X})$.
Then, it is sufficient to show that $R\ConvProb 0$ wrt $\normE{\cdot}$.
We have
$
R
=R_{1}+R_{2}+R_{3}-R_{4}
$,
where
$R_{i}=\bd{B}_{n}T_{i}$
for $i=1,2,3,4$
and
$\bd{B}_{n}=\sqrt{n}
(\hat{\bd{G}}_{X,n}^{-1/2} - \bd{G}_{X,n}^{-1/2})
\bd{A}_{n}
T_{0}^{-1}$.
We will show each $R_{i}$ converges to zero in probability,
which finishes the proof.
Before that,
we first establish an inequality for
$\normF{\hat{\bd{G}}_{X,n}^{-1/2} - \bd{G}_{X,n}^{-1/2}}$.
By Lemma \ref{PC.Paper.Lemma.Inequality.Matrix.Square.Root.Differentce},
$\normF{\hat{\bd{G}}_{X,n}^{-1/2} - \bd{G}_{X,n}^{-1/2}}
\leq
(\sqrt{q}
\normF{\hat{\bd{G}}_{X,n}^{-1} - \bd{G}_{X,n}^{-1}}
)^{1/2}$.
Note that,
by Lemma \ref{PC.Paper.Sample.Covariance.Matrix.Convergence.d.infty},
$\normF{\hat{\bd{\Sigma}}_{n}-\bd{\Sigma}_{X}}\ConvProb 0$ for $d^{4}=o(n)$.
Then, by Lemma \ref{PC.Paper.Inverse.Matrix.Convergence.Probability.d.Infity},
$
\normF{\hat{\bd{G}}_{X,n} - \bd{G}_{X,n}}
\leq
\normF{\bd{A}_{n}}^{2}
\normF{\hat{\bd{\Sigma}}_{n}^{-1}-\bd{\Sigma}_{X}^{-1}}
\lesssim
\normF{\bd{A}_{n}}^{2}
\normF{\hat{\bd{\Sigma}}_{n}-\bd{\Sigma}_{X}}
\ConvProb 0
$.
Thus, by Lemma \ref{PC.Paper.Inverse.Matrix.Convergence.Probability.d.Infity},
$
\normF{\hat{\bd{G}}_{X,n}^{-1} - \bd{G}_{X,n}^{-1}}
\lesssim
\normF{\hat{\bd{G}}_{X,n} - \bd{G}_{X,n}}
\lesssim
\normF{\bd{A}_{n}}^{2}
\normF{\hat{\bd{\Sigma}}_{n}-\bd{\Sigma}_{X}}
$.
Since $q$ is a fixed integer,
it follows
$
\normF{\hat{\bd{G}}_{X,n}^{-1/2} - \bd{G}_{X,n}^{-1/2}}
\lesssim
\normF{\bd{A}_{n}}
(\sqrt{q}\normF{\hat{\bd{\Sigma}}_{n}-\bd{\Sigma}_{X}})^{1/2}
\lesssim
\normF{\bd{A}_{n}}
(\normF{\hat{\bd{\Sigma}}_{n}-\bd{\Sigma}_{X}})^{1/2}
$.
Consider $R_{1}$.
Note that
$\normE{R_{1}}
\leq
\sqrt{n}
\sqrt{d}\normF{\hat{\bd{G}}_{X,n}^{-1/2} - \bd{G}_{X,n}^{-1/2}}
\normF{\bd{A}_{n}}
\normFd{T_{0}^{-1}}
\normE{T_{1}}$,
which is
$\lesssim
\sqrt{n}
(d\normF{\hat{\bd{\Sigma}}_{n}-\bd{\Sigma}_{X}})^{1/2}
\normF{\bd{A}_{n}}^{2}
\normFd{T_{0}^{-1}}
\normE{T_{1}}$.
By Lemmas
\ref{PC.Paper.Inverse.Matrix.Convergence.Probability.d.Infity}
and
\ref{PC.Paper.Sample.Covariance.Matrix.Convergence.d.infty},
$d\normF{\hat{\bd{\Sigma}}_{n}-\bd{\Sigma}_{X}}=o_{P}(1)$
for $d^{6}=o(n)$.
By Assumption (F),
$\normF{\bd{A}_{n}}$
is bounded.
By Lemmas
\ref{PC.Paper.Inverse.Matrix.Convergence.Probability.d.Infity}
and
\ref{PC.Paper.Sample.Covariance.Matrix.Convergence.d.infty}
and Assumption (D),
for $d=o(n^{1/3})$, wpg1,
$\normFd{T_{0}^{-1}}$
is bounded.
Also note that, wpg1,
$\normE{T_{1}}
\leq
s_{2}\kappa_{n}\gamma_{n}/n$.
Then,
$\normE{R_{1}}
\lesssim
s_{2}\kappa_{n}\gamma_{n}/\sqrt{n}$.
Thus,
$\normE{R_{1}}=o_{P}(1)$
for $s_{2}=o(\sqrt{n}/(\kappa_{n}\gamma_{n}))$.
Consider $R_{2}$.
Note that
$\normE{R_{2}}
\leq
\normF{\hat{\bd{G}}_{X,n}^{-1/2} - \bd{G}_{X,n}^{-1/2}}
\normF{\bd{A}_{n}}
\normF{T_{0}^{-1}}
\normE{\sqrt{n}T_{2}}$,
which is
$\lesssim
(d^{2}\log(d)\normF{\hat{\bd{\Sigma}}_{n}-\bd{\Sigma}_{X}})^{1/2}
\normF{\bd{A}_{n}}^{2}
\normFd{T_{0}^{-1}}
(d\log(d))^{-1/2}
\normE{\sqrt{n}T_{2}}$.
By Lemmas \ref{PC.Paper.Inverse.Matrix.Convergence.Probability.d.Infity} and
        \ref{PC.Paper.Sample.Covariance.Matrix.Convergence.d.infty},
        $d^{2}\log(d)\normF{\hat{\bd{\Sigma}}_{n}-\bd{\Sigma}_{X}}$ is $o_{P}(1)$
        for $d^{8}(\log(d))^{2} = o(n)$.
By Assumption (F), $\normF{\bd{A}_{n}}$ is $O(1)$.
By Lemmas
\ref{PC.Paper.Inverse.Matrix.Convergence.Probability.d.Infity}
and
\ref{PC.Paper.Sample.Covariance.Matrix.Convergence.d.infty}
and Assumption (D),
for $d=o(n^{1/3})$, wpg1,
$\normFd{T_{0}^{-1}}$
is bounded.
By Lemma \ref{PC.Paper.Scaled.IID.Sum.Convergence.Rate.d.infty},
        $(d\log(d))^{-1/2}\normE{\sqrt{n}T_{2}}=(d\log(d))^{-1/2}\normE{\frac{1}{\sqrt{n}}\mathbb{S}_{s_{1}+1,n}^{\epsilon}}$ is $O_{P}(1)$ for $d=o(\sqrt{n})$.
Thus,
$R_{2}\ConvProb 0$.
Consider $R_{3}$ and $R_{4}$.
By $s_{1}=o(\sqrt{n}/(\lambda \kappa_{n}))$,
wpg1,
$\normE{R_{3}}
\leq
\sqrt{n}
\normF{\hat{\bd{G}}_{X,n}^{-1/2} - \bd{G}_{X,n}^{-1/2}}
\normF{\bd{A}_{n}}
\normF{T_{0}^{-1}}
\normE{T_{3}}$,
which is
$\lesssim
\sqrt{n}
(d\normF{\hat{\bd{\Sigma}}_{n}-\bd{\Sigma}_{X}})^{1/2}
\normF{\bd{A}_{n}}^{2}
\normFd{T_{0}^{-1}}
\normE{T_{3}}
\lesssim
\lambda s_{1}\kappa_{n}/\sqrt{n}
\rightarrow 0$.
Thus, $\normE{R_{3}}=o_{P}(1)$.
In the same way, $\normE{R_{4}}=o_{P}(1)$.
%
\end{proof}

\begin{proofLong}
\begin{proof}
[Long Proof of Corollary
\ref{model:three:multivariate:mean:zero:thm:twostage:limitdistribution.general.covariates.estimated.Sigma.alternative.errors:d:infty:Paper}]
We reuse the notations on $T_{i}$'s in the proof of Theorems
\ref{thm5.2},
from which,
$$
\sqrt{n}
\hat{\bd{G}}_{X,n}^{-1/2}
\bd{A}_{n}(\hat{\bd{\beta}}_{n} - \bd{\beta}^{\star})
= M + R,
$$
where
\begin{align*}
M &= \sqrt{n}
\bd{G}_{X,n}^{-1/2}
\bd{A}_{n}(\hat{\bd{\beta}}_{n} - \bd{\beta}^{\star}), \\
R &= \sqrt{n}
(\hat{\bd{G}}_{X,n}^{-1/2} - \bd{G}_{X,n}^{-1/2})
\bd{A}_{n}(\hat{\bd{\beta}}_{n} - \bd{\beta}^{\star}).
\end{align*}
By Theorem
\ref{thm3.4:d:infty:alternative:Paper},
$M \ConvDist N(0, \sigma^{2}\bd{G}_{X})$.
Then, it is sufficient to show that $R\ConvProb 0$ wrt $\normE{\cdot}$.
We have
$$
R
=R_{1}+R_{2}+R_{3}-R_{4},
$$
where
\begin{align*}
R_{1} &=
\sqrt{n}
(\hat{\bd{G}}_{X,n}^{-1/2} - \bd{G}_{X,n}^{-1/2})
\bd{A}_{n}
T_{0}^{-1}T_{1}, \\
R_{2} &=
\sqrt{n}
(\hat{\bd{G}}_{X,n}^{-1/2} - \bd{G}_{X,n}^{-1/2})
\bd{A}_{n}
T_{0}^{-1}T_{2}, \\
R_{3} &=
\sqrt{n}
(\hat{\bd{G}}_{X,n}^{-1/2} - \bd{G}_{X,n}^{-1/2})
\bd{A}_{n}
T_{0}^{-1}T_{3}, \\
R_{4} &=
\sqrt{n}
(\hat{\bd{G}}_{X,n}^{-1/2} - \bd{G}_{X,n}^{-1/2})
\bd{A}_{n}
T_{0}^{-1}T_{4}.
\end{align*}

Next we show each $R_{i}$ converges to zero in probability.

\lineProof
By Lemma \ref{PC.Paper.Lemma.Inequality.Matrix.Square.Root.Differentce},
\begin{align*}
\normF{\hat{\bd{G}}_{X,n}^{-1/2} - \bd{G}_{X,n}^{-1/2}}
& =
\normF{(\hat{\bd{G}}_{X,n}^{-1})^{1/2} - (\bd{G}_{X,n}^{-1})^{1/2}} \\
& \leq
(\sqrt{q}
\normF{\hat{\bd{G}}_{X,n}^{-1} - \bd{G}_{X,n}^{-1}}
)^{1/2}.
\end{align*}
Note that,
by Lemma \ref{PC.Paper.Sample.Covariance.Matrix.Convergence.d.infty},
$\normF{\hat{\bd{\Sigma}}_{n}-\bd{\Sigma}_{X}}\ConvProb 0$ for $d^{4}=o(n)$.
Then,
\begin{align*}
\normF{\hat{\bd{G}}_{X,n} - \bd{G}_{X,n}}
& =
\normF{\bd{A}_{n}
(\hat{\bd{\Sigma}}_{n}^{-1}-\bd{\Sigma}_{X}^{-1})
\bd{A}_{n}^{T}} \\
& \leq
\normF{\bd{A}_{n}}^{2}
\normF{\hat{\bd{\Sigma}}_{n}^{-1}-\bd{\Sigma}_{X}^{-1}} \\
& \lesssim
\normF{\bd{A}_{n}}^{2}
\normF{\hat{\bd{\Sigma}}_{n}-\bd{\Sigma}_{X}}
\ConvProb 0.
\end{align*}
Thus,
by Lemma \ref{PC.Paper.Inverse.Matrix.Convergence.Probability.d.Infity},
\begin{align*}
\normF{\hat{\bd{G}}_{X,n}^{-1} - \bd{G}_{X,n}^{-1}}
& \lesssim
\normF{\hat{\bd{G}}_{X,n} - \bd{G}_{X,n}} \\
& \lesssim
\normF{\bd{A}_{n}}^{2}
\normF{\hat{\bd{\Sigma}}_{n}-\bd{\Sigma}_{X}}.
\end{align*}

Thus,
\begin{align*}
\normF{\hat{\bd{G}}_{X,n}^{-1/2} - \bd{G}_{X,n}^{-1/2}}
& \lesssim
\normF{\bd{A}_{n}}
(\sqrt{q}\normF{\hat{\bd{\Sigma}}_{n}-\bd{\Sigma}_{X}})^{1/2}
\lesssim
\normF{\bd{A}_{n}}
(\normF{\hat{\bd{\Sigma}}_{n}-\bd{\Sigma}_{X}})^{1/2}.
\end{align*}

\lineProof

\OnItem{On $R_{1}$.}
We have
\begin{align*}
\normE{R_{1}}
& =
\normE{\sqrt{n}
(\hat{\bd{G}}_{X,n}^{-1/2} - \bd{G}_{X,n}^{-1/2})
\bd{A}_{n}
T_{0}^{-1}T_{1}} \\
& \leq
\sqrt{n}
\normF{\hat{\bd{G}}_{X,n}^{-1/2} - \bd{G}_{X,n}^{-1/2}}
\normF{\bd{A}_{n}}
\normF{T_{0}^{-1}}
\normE{T_{1}} \\
& =
\sqrt{n}
\sqrt{d}\normF{\hat{\bd{G}}_{X,n}^{-1/2} - \bd{G}_{X,n}^{-1/2}}
\normF{\bd{A}_{n}}
\normFd{T_{0}^{-1}}
\normE{T_{1}} \\
& \lesssim
\sqrt{n}
(d\normF{\hat{\bd{\Sigma}}_{n}-\bd{\Sigma}_{X}})^{1/2}
\normF{\bd{A}_{n}}^{2}
\normFd{T_{0}^{-1}}
\normE{T_{1}}.
\end{align*}

By Lemmas
\ref{PC.Paper.Inverse.Matrix.Convergence.Probability.d.Infity}
and
\ref{PC.Paper.Sample.Covariance.Matrix.Convergence.d.infty},
$d\normF{\hat{\bd{\Sigma}}_{n}^{-1}-\bd{\Sigma}_{X}^{-1}}=o_{P}(1)$
for $d^{6}=o(n)$.

By Assumption (F)
$\normF{\bd{A}_{n}}$
is bounded.

By Lemmas
\ref{PC.Paper.Inverse.Matrix.Convergence.Probability.d.Infity}
and
\ref{PC.Paper.Sample.Covariance.Matrix.Convergence.d.infty}
and Assumption (D),
for $d=o(n^{1/3})$, wpg1,
$\normFd{T_{0}^{-1}}$
is bounded.

We have, wpg1,
\begin{equation*}
\normE{T_{1}}
=
\normE{\frac{1}{n}\sum_{i=s_{1}+1}^{s}\bd{X}_{i}\mu_{i}^{\star}}
\leq
\frac{1}{n}s_{2}\kappa_{n}\gamma_{n}.
\end{equation*}

Then,
\begin{align*}
\normE{R_{1}}
& \lesssim
\frac{1}{\sqrt{n}}s_{2}\kappa_{n}\gamma_{n},
\end{align*}
where $\lesssim$ means that the left side is bounded by a constant times the right side.

Thus,
$\normE{R_{1}}=o_{P}(1)$
for $s_{2}=o(\sqrt{n}/(\kappa_{n}\gamma_{n}))$.

\lineProof

\OnItem{On $R_{2}$.}
We have
\begin{align*}
\normE{R_{2}}
& =
\normE{\sqrt{n}
(\hat{\bd{G}}_{X,n}^{-1/2} - \bd{G}_{X,n}^{-1/2})
\bd{A}_{n}T_{0}^{-1}T_{2}} \\
& \leq
\normF{\hat{\bd{G}}_{X,n}^{-1/2} - \bd{G}_{X,n}^{-1/2}}
\normF{\bd{A}_{n}}
\normF{T_{0}^{-1}}
\normE{\sqrt{n}T_{2}} \\
& \lesssim
d(\log(d))^{1/2}
(\normF{\hat{\bd{\Sigma}}_{n}-\bd{\Sigma}_{X}})^{1/2}
\normF{\bd{A}_{n}}^{2}
\normFd{T_{0}^{-1}}
(d\log(d))^{-1/2}
\normE{\sqrt{n}T_{2}}\\
& =
(d^{2}\log(d)\normF{\hat{\bd{\Sigma}}_{n}-\bd{\Sigma}_{X}})^{1/2}
\normF{\bd{A}_{n}}^{2}
\normFd{T_{0}^{-1}}
(d\log(d))^{-1/2}
\normE{\sqrt{n}T_{2}}.
\end{align*}

Note that
\begin{itemize}
  \item by Lemmas \ref{PC.Paper.Inverse.Matrix.Convergence.Probability.d.Infity} and
        \ref{PC.Paper.Sample.Covariance.Matrix.Convergence.d.infty},
        $d^{2}\log(d)\normF{\hat{\bd{\Sigma}}_{n}-\bd{\Sigma}_{X}}$ is $o_{P}(1)$
        for $d^{8}(\log(d))^{2} = o(n)$;
  \item By Assumption (F), $\normF{\bd{A}_{n}}$ is $O(1)$;
  \item By Lemmas
\ref{PC.Paper.Inverse.Matrix.Convergence.Probability.d.Infity}
and
\ref{PC.Paper.Sample.Covariance.Matrix.Convergence.d.infty}
and Assumption (D),
for $d=o(n^{1/3})$, wpg1,
$\normFd{T_{0}^{-1}}$
is bounded.
  \item by Lemma \ref{PC.Paper.Scaled.IID.Sum.Convergence.Rate.d.infty},
        $(d\log(d))^{-1/2}\normE{\sqrt{n}T_{2}}=(d\log(d))^{-1/2}\normE{\frac{1}{\sqrt{n}}\mathbb{S}_{s_{1}+1,n}^{\epsilon}}$ is $O_{P}(1)$ for $d=o(\sqrt{n})$.
\end{itemize}
Thus,
$R_{2}\ConvProb 0$.

\lineProof

\OnItem{On $R_{3}$ and $R_{4}$.}
First consider $V_{3}$.
By noting that $s_{1}=o(\sqrt{n}/(\lambda \kappa_{n}))$,
wpg1,
\begin{align*}
\normE{R_{3}}
& =
\normE{
\sqrt{n}
(\hat{\bd{G}}_{X,n}^{-1/2} - \bd{G}_{X,n}^{-1/2})
\bd{A}_{n}
T_{0}^{-1}T_{3}} \\
& \leq
\sqrt{n}
\normF{\hat{\bd{G}}_{X,n}^{-1/2} - \bd{G}_{X,n}^{-1/2}}
\normF{\bd{A}_{n}}
\normF{T_{0}^{-1}}
\normE{T_{3}} \\
& \lesssim
\sqrt{n}
(\normF{\hat{\bd{\Sigma}}_{n}-\bd{\Sigma}_{X}})^{1/2}
\normF{\bd{A}_{n}}^{2}
\normF{T_{0}^{-1}}
\normE{T_{3}} \\
& =
\sqrt{n}
(d\normF{\hat{\bd{\Sigma}}_{n}-\bd{\Sigma}_{X}})^{1/2}
\normF{\bd{A}_{n}}^{2}
\normFd{T_{0}^{-1}}
\normE{T_{3}} \\
& \lesssim
\sqrt{n}\lambda s_{1}\kappa_{n}/n
=
\lambda s_{1}\kappa_{n}/\sqrt{n}
\rightarrow 0.
\end{align*}
Thus, $\normE{R_{3}}=o_{P}(1)$.
In the same way, $\normE{R_{4}}=o_{P}(1)$.

Therefore, the result of the theorem follows
by Slutsky's lemma.
\end{proof}
\end{proofLong}



\section{Supplementary Materials}
Additional materials for
Sections \ref{sec1} to \ref{sec:Diverging number of structural parameters}
can be found in the file of supplementary materials.



\begin{thebibliography}{27}
\providecommand{\natexlab}[1]{#1}
\providecommand{\url}[1]{\texttt{#1}}
\expandafter\ifx\csname urlstyle\endcsname\relax
  \providecommand{\doi}[1]{doi: #1}\else
  \providecommand{\doi}{doi: \begingroup \urlstyle{rm}\Url}\fi

\bibitem[Basu(1977)]{Basu1977}
Debabrata Basu.
\newblock On the elimination of nuisance parameters.
\newblock \emph{Journal of the American Statistical Association}, 72:\penalty0
  355--366, 1977.

\bibitem[Bean et~al.(2012)Bean, Bickel, El~Karoui, Lim, and Yu]{Bean2012}
Derek Bean, Peter Bickel, Noureddine El~Karoui, Chinghway Lim, and Bin Yu.
\newblock Penalized robust regression in high-dimension.
\newblock 2012.

\bibitem[Chen et~al.(2010)Chen, Goldstein, and Shao]{Chen2010}
Louis~H.Y. Chen, Larry Goldstein, and Qi-Man Shao.
\newblock \emph{Normal Approximation by Stein's Method}.
\newblock Springer-Verlag, 2010.

\bibitem[Fan and Li(2001)]{Fan2001a}
Jianqing Fan and Runze Li.
\newblock Variable selection via nonconcave penalized likelihood and its oracle
  properties.
\newblock \emph{Journal of the American Statistical Association}, 96\penalty0
  (456):\penalty0 1348--1360, Dec 2001.
\newblock ISSN 0162-1459.

\bibitem[Fan and Lv(2010)]{Fan2010}
Jianqing Fan and Jinchi Lv.
\newblock A selective overview of variable selection in high dimensional
  feature space.
\newblock \emph{Statistica Sinica}, 20:\penalty0 101--148, 2010.

\bibitem[Fan and Peng(2004)]{Fan2004}
Jianqing Fan and Heng Peng.
\newblock On non-concave penalized likelihood with diverging number of
  parameters.
\newblock \emph{The Annals of Statistics}, 32:\penalty0 928--961, 2004.

\bibitem[Fan et~al.(2011)Fan, Liao, and Mincheva]{Fan2011b}
Jianqing Fan, Yuan Liao, and Martina Mincheva.
\newblock High dimensional covariance matrix estimation in approximate factor
  models.
\newblock 2011.

\bibitem[Fan et~al.(2012{\natexlab{a}})Fan, Fan, and Barut]{Fan2012}
Jianqing Fan, Yingying Fan, and Emre Barut.
\newblock Adaptive robust variable selection.
\newblock 2012{\natexlab{a}}.

\bibitem[Fan et~al.(2012{\natexlab{b}})Fan, Feng, and Tong]{Fan2012a}
Jianqing Fan, Yang Feng, and Xin Tong.
\newblock A road to classification in high dimensional space: the regularized
  optimal affine discriminant.
\newblock \emph{Journal of the Royal Statistical Society Series B},
  2012{\natexlab{b}}.

\bibitem[Huber(1964)]{Huber1964}
Peter~J. Huber.
\newblock Robust estimation of a location parameter.
\newblock \emph{The Annals of Mathematical Statistics}, 35:\penalty0 73--101,
  1964.

\bibitem[Huber(1973)]{Huber1973}
Peter~J. Huber.
\newblock Robust regression: Asymptotics, conjectures and monte carlo.
\newblock \emph{The Annals of Statistics}, 1:\penalty0 799--821, 1973.

\bibitem[Jahn(2007)]{Jahn2007}
Johannes Jahn.
\newblock \emph{Introduction to the Theory of Nonlinear Optimization}.
\newblock Springer Berlin Heidelberg, 2007.

\bibitem[Kiefer and Wolfowitz(1956)]{Kiefer1956}
J.~Kiefer and J.~Wolfowitz.
\newblock Consistency of the maximum likelihood estimator in the presence of
  infinitely many incidental parameters.
\newblock \emph{The Annals of Mathematical Statistics}, 27:\penalty0 887--906,
  1956.

\bibitem[Kosorok(2008)]{Kosorok2008}
Michael~R. Kosorok.
\newblock Bootstrapping the grenander estimator.
\newblock In \emph{Beyond Parametrics in Interdisciplinary Research:
  Festschrift in Honor of Professor Pranab K. Sen}, pages 282--292. Institute
  of Mathematical Statistics: Hayward, CA., 2008.

\bibitem[Lambert-Lacroix and Zwald(2011)]{Lambert-Lacroix2011}
S.~Lambert-Lacroix and L.~Zwald.
\newblock Robust regression through the huber's criterion and adaptive lasso
  penalty.
\newblock \emph{Electronic Journal of Statistics}, 5:\penalty0 1015--1053,
  2011.
\newblock ISSN 1935-7524.

\bibitem[Lancaster(2000)]{Lancaster2000}
Tony Lancaster.
\newblock The incidental parameter problem since 1948.
\newblock \emph{Journal of Econometrics}, 95:\penalty0 391--413, 2000.

\bibitem[Moreira(2009)]{Moreira2009}
Marcelo~J. Moreira.
\newblock A maximum likelihood method for the incidental parameter problem.
\newblock \emph{The Annals of Statistics}, 37\penalty0 (6A):\penalty0
  3660--3696, 2009.
\newblock ISSN 0090-5364.

\bibitem[Neyman and Scott(1948)]{Neyman1948}
Jerzy Neyman and Elizabeth~L. Scott.
\newblock Consistent estimates based on partially consistent observations.
\newblock \emph{Econometrica}, 16:\penalty0 1--32, 1948.

\bibitem[Portnoy and He(2000)]{Portnoy2000}
Stephen Portnoy and Xuming He.
\newblock A robust journey in the new millennium.
\newblock \emph{Journal of the American Statistical Association}, 95:\penalty0
  1331--1335, 2000.

\bibitem[Shiryaev(1995)]{Shiryaev1995}
Albert~N. Shiryaev.
\newblock \emph{Probability}.
\newblock Springer-Verlag, second edition, 1995.

\bibitem[Stewart(1969)]{Stewart1969}
G.~W. Stewart.
\newblock On the continuity of the generalized inverse.
\newblock \emph{SIAM Journal on Applied Mathematics}, 17:\penalty0 33--45,
  1969.

\bibitem[Tang et~al.(2012)Tang, Banerjee, and Kosorok]{Tang2012a}
Runlong Tang, Moulinath Banerjee, and Michael~R. Kosorok.
\newblock Likelihood based inference for current status data on a grid: A
  boundary phenomenon and an adaptive inference procedure.
\newblock \emph{The Annals of Statistics}, 40\penalty0 (1):\penalty0 45--72,
  2012.

\bibitem[Tibshirani(1996)]{Tibshirani1996a}
Robert Tibshirani.
\newblock Regression shrinkage and selection via the lasso.
\newblock \emph{J. Royal. Statist. Soc B}, 58:\penalty0 267--288, 1996.

\bibitem[van~der Vaart(1998)]{Vaart1998}
Aad~W. van~der Vaart.
\newblock \emph{Asymptotic Statistics}.
\newblock Cambridge University Press, 1998.

\bibitem[van~der Vaart and Wellner(1996)]{Vaart1996}
Aad~W. van~der Vaart and Jon~A. Wellner.
\newblock \emph{Weak Convergence and Empirical Processes}.
\newblock Springer, 1996.

\bibitem[Wihler(2009)]{Wihler2009}
Thomas~P. Wihler.
\newblock On the holder continuity of matrix functions for normal matrices.
\newblock \emph{Journal of inequalities in pure and applied mathematics}, 10,
  2009.

\bibitem[Zhao and Yu(2006)]{Zhao2006}
Peng Zhao and Bin Yu.
\newblock On model selection consistency of lasso.
\newblock \emph{The Journal of Machine Learning Research}, 7(Nov):\penalty0
  2541--2563, 2006.

\end{thebibliography}

\begin{thebibliography}{26}
\providecommand{\natexlab}[1]{#1}
\providecommand{\url}[1]{\texttt{#1}}
\expandafter\ifx\csname urlstyle\endcsname\relax
  \providecommand{\doi}[1]{doi: #1}\else
  \providecommand{\doi}{doi: \begingroup \urlstyle{rm}\Url}\fi


\bibitem[Fan and Lv(2011)]{Fan2011b}
Jianqing Fan and Jinchi Lv.
\newblock Non-concave penalized likelihood with np-dimensionality.
\newblock \emph{IEEE Transactions On Information Theory}, 57:\penalty0
  5467--5484, 2011.

\bibitem[Fan and Peng(2004)]{Fan2004}
Jianqing Fan and Heng Peng.
\newblock On non-concave penalized likelihood with diverging number of
  parameters.
\newblock \emph{The Annals of Statistics}, 32:\penalty0 928--961, 2004.


\bibitem[Kosorok(2008)]{Kosorok2008}
Michael~R. Kosorok.
\newblock \emph{Introduction to Empirical Processes and Semiparametric
  Inference}.
\newblock Springer New York, 2008.

\bibitem[Neyman and Scott(1948)]{Neyman1948}
Jerzy Neyman and Elizabeth~L. Scott.
\newblock Consistent estimates based on partially consistent observations.
\newblock \emph{Econometrica}, 16:\penalty0 1--32, 1948.

\bibitem[Shiryaev(1995)]{Shiryaev1995}
Albert~N. Shiryaev.
\newblock \emph{Probability}.
\newblock Springer-Verlag, second edition, 1995.

\bibitem[van~der Vaart(1998)]{Vaart1998}
Aad~W. van~der Vaart.
\newblock \emph{Asymptotic Statistics}.
\newblock Cambridge University Press, 1998.


\bibitem[Zhao and Yu(2006)]{Zhao2006}
Peng Zhao and Bin Yu.
\newblock On model selection consistency of lasso.
\newblock \emph{The Journal of Machine Learning Research}, 7(Nov):\penalty0
  2541--2563, 2006.

\end{thebibliography}

\newpage
\setcounter{page}{0}   
\pagenumbering{roman}  

\begin{center}
\Large\emph{\textbf{Supplementary Materials for Paper: }}

\vspace{3mm}
\textbf{``Partial Consistency with Sparse Incidental Parameters"}

\vspace{3mm}
\emph{by} Jianqing Fan, Runlong Tang and Xiaofeng Shi
\end{center}




\section{Supplement for Section \ref{sec1}}
\label{PC.Paper.Supplement.Introduction}
In this supplement, we first show the method proposed by \cite{Neyman1948}
does not work for model (\ref{model:basic})
and then explain which assumptions or conditions for the consistent results of the penalized methods in
\cite{Zhao2006},
\cite{Fan2004}
and \cite{Fan2011b}
are not satisfied for model
(\ref{model:basic}).

Although the modified equations of maximum likelihood method proposed by \cite{Neyman1948}
could handle ``a number of important cases"
with incidental parameters,
unfortunately, it does not work for model (\ref{model:basic}).
More specifically,
consider the simplest case of model  (\ref{model:basic}) with $d=1$:
\begin{equation*}
    Y_{i} = \mu_{i}^{\star} + {X}_{i}\beta^{\star} + \epsilon_{i}, \text{ for } i=1,2,\cdots, n,
\end{equation*}
where $\{\epsilon_{i}\}$ are i.i.d. copies of $N(0,\sigma^{2})$.
Using the notations of \cite{Neyman1948},
the likelihood function for $({X}_{i}, Y_{i})$ is
$
p_{i}
=
p_{i}(\beta,\sigma,\mu_{i}|{X}_{i}, Y_{i})
=
(\sqrt{2\pi}\sigma)^{-1}
\exp\{-(2\sigma^{2})^{-1}(Y_{i}-\mu_{i}^{\star}-{X}_{i}\beta)^{2}\},
$
and the log-likelihood function is
$
\log p_{i}
=
-\log (\sqrt{2\pi}\sigma)
- (2\sigma^{2})^{-1}(Y_{i}-\mu_{i}^{\star}-{X}_{i}\beta)^{2}.
$
Then, the score functions are
\begin{align*}
\phi_{i1}
& =
\frac{\partial \log p_{i}}{\partial \beta}
=
\frac{1}{\sigma^{2}}(Y_{i}-\mu_{i}^{\star}-{X}_{i}\beta){X}_{i}, \\
\phi_{i2}
& =
\frac{\partial \log p_{i}}{\partial \sigma}
=
\frac{1}{\sigma} + \frac{1}{\sigma^{3}}(Y_{i}-\mu_{i}^{\star}-{X}_{i}\beta)^{2}, \\
\omega_{i}
& =
\frac{\partial \log p_{i}}{\partial \mu_{i}}
=
\frac{1}{\sigma^{2}}(Y_{i}-\mu_{i}^{\star}-{X}_{i}\beta).
\end{align*}
From the equation $\omega_{i}=0$,
we have
$
    \hat\mu_{i} = Y_{i}-{X}_{i}\beta.
$
Plugging this $\hat\mu_{i}$ into $\phi_{i1}$ and $\phi_{i2}$
(replacing $\mu_{i}$ with $\hat\mu_{i}$),
we obtain
$\phi_{i1}=0$ and $\phi_{i2}=1/\sigma$.
Then,
$E_{i1} = \E \phi_{i1} = 0$
and
$E_{i2} = \E \phi_{i1} = 1/\sigma$.
Thus,
$E_{i1}$ and $E_{i2}$ do only depend on the structural parameters ($\bd{\beta}^{\star}$ and $\sigma$).
However, we then have
$\Phi_{i1}=\phi_{i1}-E_{i1}=0$ and $\Phi_{i2}=\phi_{i2}-E_{i2}=0$.
This means $F_{n1}=F_{n2}=0$,
independent of structural parameters!
Consequently, the estimation equations degenerate to two $0=0$ equations,
which means the modified equation of maximum likelihood method does not work
for model  (\ref{model:basic}).

Next,
we explicitly explain which assumptions or conditions for the consistent results of the penalized methods in
\cite{Zhao2006},
\cite{Fan2004}
and \cite{Fan2011b}
are not valid for model
(\ref{model:basic}).

\cite{Zhao2006} derive strong sign consistency for lasso estimator.
However, their consistency results Theorems 3 and 4 do not apply to model (\ref{model:basic}),
since the above specific design matrix $\bd{X}$ does not satisfy their regularity condition (6) on page 2546.
More specifically, with model (\ref{model:basic}),
\begin{equation*}
    C_{11}^{n} =
    \frac{1}{n}
    \begin{pmatrix}
  \bd{I}_{s} & \bd{X}_{1,s} \\
  \bd{X}_{1,s}^{T} & \sum_{i=1}^{n}\bd{X}_{i}\bd{X}_{i}^{T} \\
 \end{pmatrix}
 \SC
 \begin{pmatrix}
  \bd{0} & \bd{0} \\
  \bd{0} & \bd{\bd{\Sigma}}_{X} \\
 \end{pmatrix},
\end{equation*}
where $\bd{\bd{\Sigma}}_{X}$ is the covariance matrix of the covariates.
This means that some of the eigenvalues of $C_{11}^{n}$ goes to $0$ as $n\rightarrow \infty$.
Then the regularity condition (6),
which is
\begin{equation*}
    \alpha^{T} C_{11}^{n} \alpha \geq \text{ a positive constant }, \text{ for all } \alpha\in\mathbb{R}^{s+d} \text{ such that } \|\alpha\|_{2}^{2} = 1,
\end{equation*}
does not hold any more.
Thus the consistency results Theorems 3 and 4  in \cite{Zhao2006}
is not applicable for model (\ref{model:basic}).

\cite{Fan2004} show
the consistency with Euclidean metric of a penalized likelihood estimator
when the dimension of the sparse parameter increases
with the sample size
in Theorem 1 on Page 935.
Under their framework,
the log-likelihood function of the data point $V_{i}=(\bd{X}_{i},Y_{i})$ for each $i$
from model (\ref{model:basic})
with random errors being i.i.d. copies of $N(0,\sigma^{2})$
is given by
\begin{equation*}
\log f_{n}(V_{i},\mu_{i},\bd{\beta})
\propto
-\frac{1}{2\sigma^{2}}
(Y_{i} - \mu_{i} - \bd{X}_{i}^{T}\bd{\beta})^{2},
\end{equation*}
where $\propto$ means ``proportional to".
As we can see that log-likelihood functions with different $i$'s
might different since $\mu_{i}$'s might be different for different $i$'s.
This violates a condition
that all the data points are i.i.d. from a structural density
in Assumption (G) on Page 934.

This violation might not be essential,
however,
since we could consider the log-likelihood function
for all the data directly.
That is, we consider
\begin{equation*}
L_{n}(\bd{\mu},\bd{\beta})
=
\sum_{i=1}^{n}
\log  f_{n}(V_{i},\mu_{i},\bd{\beta})
\propto
-\frac{1}{2\sigma^{2}}
\sum_{i=1}^{n}
(Y_{i} - \mu_{i} - \bd{X}_{i}^{T}\bd{\beta})^{2}.
\end{equation*}
Then, the Fisher information matrix for $(\bd{\mu},\bd{\beta})$ is given by
\begin{equation*}
I_{n+d}(\bd{\mu},\bd{\beta})
=
 \begin{pmatrix}
  \sigma^{-2}\bd{I}_{n} & 0 \\
  0 & n\sigma^{-2}\bd{\bd{\Sigma}}_{X}^{2} \\
 \end{pmatrix},
\end{equation*}
where $I_{n}$ is the $n\times n$ identity matrix.
Then, the Fisher information for one data point is
\begin{equation*}
\frac{1}{n}I_{n+d}(\bd{\mu},\bd{\beta})
=
 \begin{pmatrix}
  n^{-1}\sigma^{-2}\bd{I}_{n} & 0 \\
  0 & \sigma^{-2}\bd{\bd{\Sigma}}_{X}^{2}\\
 \end{pmatrix}.
\end{equation*}
It is clear that
the minimal eigenvalue
$\lambda_{\min}(I_{n+d}(\bd{\mu},\bd{\beta})/n)=n^{-1}\sigma^{2}\rightarrow 0$
as $n\rightarrow \infty$.
This violates the condition that the minimal eigenvalue should be lower bounded from $0$
in Assumption (F) on Page 934.
Thus, the consistency result Theorem 1 in \cite{Fan2004}
can not be applied to model (\ref{model:basic}).

\cite{Fan2011b}
``consider the variable selection problem of nonpolynomial dimensionality
in the context of
generalized linear models"
by taking the penalized likelihood approach with folded-concave penalties.
Theorem 3 on page 5472 of \cite{Fan2011b}
shows that there exists a consistent estimator of the unknown parameters
with the Euclidean metric under certain conditions.
In Condition 4 on page 5472,
there is a condition on a minimal eigenvalue
\begin{equation*}
\underset{\bd{\delta}\in N_{0}}{\min}
\lambda_{\min}[\bd{X}_{I}^{T}\bd{\bd{\Sigma}(\bd{X_{I}\bd{\delta}})}\bd{X}_{I}] \geq cn,
\end{equation*}
where $\bd{X}_{I}$ consists of the first $s+d$ columns of
the design matrix $\bd{X}$.
With model (\ref{model:basic}),
this condition becomes
\begin{equation*}
\lambda_{\min}[\bd{X}_{I}^{T}\bd{X}_{I}] \geq c n,
\end{equation*}
which is
\begin{equation*}
\lambda_{\min}[(1/n)\bd{X}_{I}^{T}\bd{X}_{I}]
=
\lambda_{\min}[C_{11}^{n}]
\geq c,
\end{equation*}
where $C_{11}^{n}$ is the matrix defined in \cite{Zhao2006}
and $c$ is a positive constant.
Since the minimal eigenvalue $\lambda_{\min}[C_{11}^{n}]$
converges to 0,
the above condition does not hold.
Thus, the consistency result Theorem 3 of \cite{Fan2011b}
is not applicable for model (\ref{model:basic}).

\section{Supplement for Section \ref{sec2}}
\label{PC.Paper.Supplement.Model.Data.Method}
In this supplement,
we provide
Lemmas
\ref{lem2.1}
and
\ref{lem2.3},
Proposition
\ref{prop2.2}
and their proofs.
Before that,
there are two graphs,
Figures
\ref{model:three:simplest:mean:zero:Fig:modelThreeMultivariateMean0Onmupm:General:Covariates:Errors}
and
\ref{model:three:simplest:mean:zero:Fig:PLS:Soft:2},
illustrating
the incidental parameters
and
the step of updating the responses
in the iteration algorithm
with $d=1$.

\begin{figure}
\begin{center}
  \includegraphics[scale=0.7]{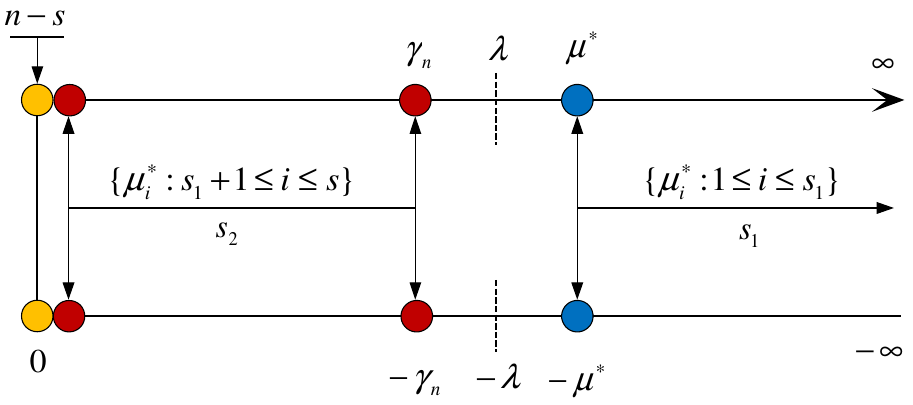}
  \caption[An illustration of $\mu_{i}^{\star}$'s.]
  {An illustration of three types of $\mu_{i}^{\star}$'s, that is,
   large $\bd{\mu}_{1}^{\star}$,  bounded $\bd{\mu}_{2}^{\star}$ and  zero $\bd{\mu}_{3}^{\star}$.
   The negative half of the real line is folded at 0
   under the positive half
   for convenience.
   For the penalized least square method
   with a soft penalty function
   and under the assumption of fixed $d$,
   the specification of the regularization parameter $\lambda$
   is that
   $\kappa_{n} \ll  \lambda,~  \alpha\gamma_{n} \leq \lambda, \text{ and } \lambda \ll  \min\{\mu^{\star},\sqrt{n}\}$.
  }
  \label{model:three:simplest:mean:zero:Fig:modelThreeMultivariateMean0Onmupm:General:Covariates:Errors}
\end{center}
\end{figure}

\begin{figure}
\begin{center}
  \includegraphics[scale=0.6]{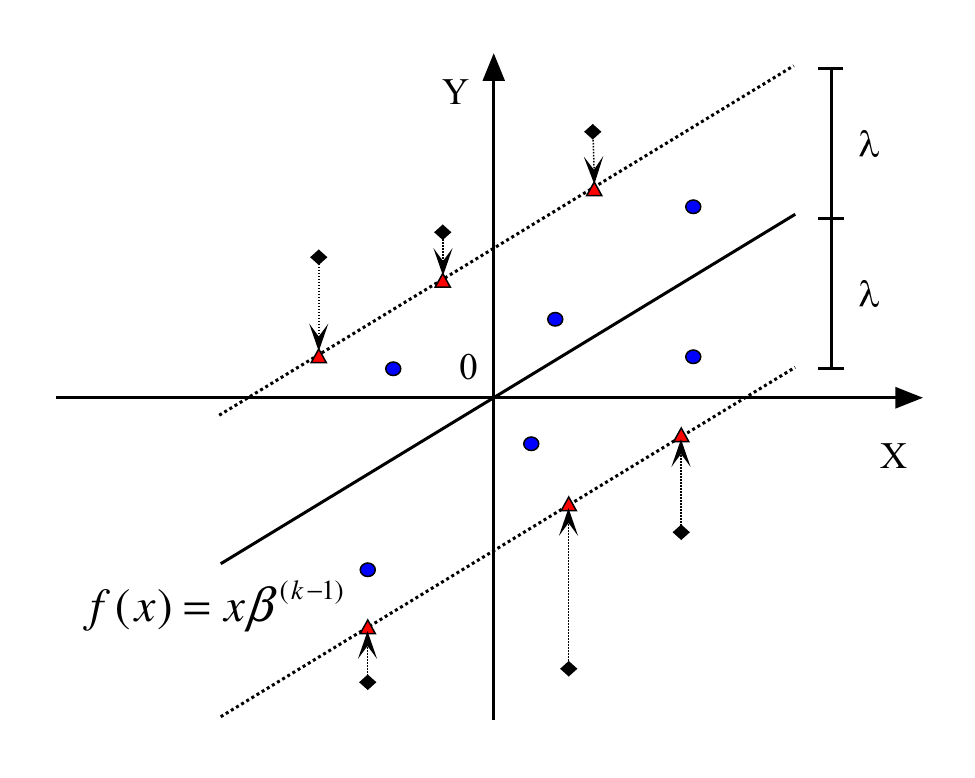}
  \caption[An illustration for the updating of responses with $d=1$.]
  {An illustration for the updating of responses with $d=1$.
  The solid black line is a fitted regression line.
  The dashed black lines are the corresponding shifted regression lines.
  The circle and diamond points are the original data points.
  The circle and triangle points are the updated data points.
  That is, the diamond points are drawn onto the shifted regression lines.
  }
  \label{model:three:simplest:mean:zero:Fig:PLS:Soft:2}
\end{center}
\end{figure}

\begin{lem}
\label{lem2.1}
A necessary and sufficient condition for
$(\hat{\bd{\mu}}, \hat{\bd{\beta}})$
to be a minimizer of $L(\bd{\mu},\bd{\beta})$ is that
\begin{align*}
&  \hat{\bd{\beta}} =
    (\bd{X}^{T}\bd{X})^{-1}\bd{X}^{T}(\bd{Y}-\hat{\bd{\mu}}), \\
& Y_{i} - \hat{\mu}_{i} - \bd{X}_{i}^{T}\hat{\bd{\beta}} = \lambda  \sign(\hat{\mu}_{i}),  ~~ \text{ for } i \in \hat{I}_{0}^{c}, \\
& |Y_{i} - \bd{X}_{i}^{T}\hat{\bd{\beta}}| \leq \lambda, ~~ \text{ for } i \in \hat{I}_{0},
\end{align*}
where $\sign(\cdot)$ is a sign function and $\hat{I}_{0}=\{1\leq i\leq n: \hat{\mu}_{i} = 0\}$.
\end{lem}

\begin{proofName}
\subsection{Proof of Lemma \ref{lem2.1}}
\end{proofName}
\begin{proof}
[Proof of Lemma \ref{lem2.1}]
By subdifferential calculus
(see, for example, Theorem 3.27 in \cite{Jahn2007}),
a necessary and sufficient condition for
$(\hat{\bd{\mu}}, \hat{\bd{\beta}})$
to be a minimizer of $L(\bd{\mu},\bd{\beta})$ is that
zero is in the subdifferential of $L$ at $(\hat{\bd{\mu}}, \hat{\bd{\beta}})$,
which means that, for each $i$,
\begin{align*}
&  \hat{\bd{\beta}} =
    (\bd{X}^{T}\bd{X})^{-1}\bd{X}^{T}(\bd{Y}-\hat{\bd{\mu}}), \\
& Y_{i} - \hat{\mu}_{i} - \bd{X}_{i}^{T}\hat{\bd{\beta}} = \lambda \sign(\hat{\mu}_{i}),  ~~ \text{ if } \hat\mu_{i} \not = 0, \\
& |Y_{i} - \bd{X}_{i}^{T}\hat{\bd{\beta}}| \leq \lambda, ~~ \text{ if } \hat\mu_{i} = 0.
\end{align*}
Thus, the conclusion of Lemma \ref{lem2.1} follows.
\end{proof}

\begin{prop}
\label{prop2.2}
Suppose Assumptions (A) and (B) hold and there exist positive constants $C_{1}$ and $C_{2}$ such that
$\normE{\bd{\beta}^{\star}} < C_{1}$ and
$\normE{\bd{\beta}^{(0)}} < C_{2}$ wpg1.
If
$s_{1}\lambda/n =O(1)$
and $s_{2}\gamma_{n}/n = o(1)$,
then,
for every $K\geq 1$ and $k\leq K$,
wpg1 as $n\rightarrow\infty$,
$$
    \normE{\bd{\beta}^{(K+1)} - \bd{\beta}^{(K)}}\leq O((s_{1}/n)^{K}),
    \quad \mbox{and} \quad
    \normE{\bd{\beta}^{(k)}}\leq 2\sqrt{d}C_{1} + C_{2}.
$$
\end{prop}

\begin{rem}
For any prespecified critical value in the stopping rule,
Proposition~\ref{prop2.2} implies that the algorithm stops at the second iteration wpg1.
In practice, the sample size $n$ might not be large enough for
the \revTang{two-iteration} estimator
to have a decent performance
so that more iterations are usually needed to activate the stopping rule.
By Proposition \ref{prop2.2},
$K$ iterations will make the distance $\normE{\bd{\beta}^{(K+1)} - \bd{\beta}^{(K)}}$ of the small order $(s_{1}/n)^{K}$.
When $s_{1}/n$ is small,
the algorithm converges quickly,
which has been verified by our simulations.
\end{rem}
\begin{proofName}
\subsection{Proof of Proposition \ref{prop2.2}}
\end{proofName}

\begin{proof}
[Proof of Proposition \ref{prop2.2}]
First, we show that, wpg1, $\normE{\bd{\beta}^{(1)}}$ is bounded by $2\sqrt{d}C_{1}+C_{2}$.
For each $k\geq 1$,
we have
\begin{align*}
\mathbb{S}\bd{\beta}^{(k)}
& =
\mathbb{S}_{S_{11}}^{\mu}
+
\mathbb{S}_{S_{12}}^{\mu}
+
\mathbb{S}_{S_{1}}\bd{\beta}^{\star}
+
\mathbb{S}_{S_{1}}^{\epsilon}
+
\mathbb{S}_{S_{2}\cup S_{3}}\bd{\beta}^{(k-1)}
+
\lambda(\mathcal{S}_{S_{2}}-\mathcal{S}_{S_{3}}),
\end{align*}
where $S_{i}=\cup_{j=1}^{3}S_{ij}(\bd{\beta}^{(k-1)})$
for $i=1,2,3$
and $S_{ij}$'s are defined at the end of Section
\ref{sec2}.
Denote $\mathcal{A}_{k-1}$ as the event
\begin{align*}
\{
S_{11}(\bd{\beta}^{(k-1)})=\emptyset,
S_{12}(\bd{\beta}^{(k-1)})=S_{12}^{\star},
S_{1}(\bd{\beta}^{(k-1)})=S_{10}^{\star}\cup S_{12}^{\star};
S_{2}(\bd{\beta}^{(k-1)})=S_{21}^{\star};
S_{3}(\bd{\beta}^{(k-1)})=S_{31}^{\star}
\},
\end{align*}
where $S_{ij}^{\star}$'s are defined at the beginning of
Section  \ref{sec3}.

By Lemma \ref{lem3.1},
$P(\mathcal{A}_{0})\rightarrow 1$.
Thus, wpg1,
$$
\bd{\beta}^{(1)}
=
T_{0}^{-1}T_{1}
+
T_{0}^{-1}T_{2}
+
T_{0}^{-1}T_{3}
+
T_{0}^{-1}T_{4}(\bd{\beta}^{(0)})
+
T_{0}^{-1}T_{5},
$$
where
$T_{0} = \mathbb{S}/n$,
$T_{1} = \mathbb{S}_{S_{12}^{\star}}^{\mu}/n$,
$T_{2} = \mathbb{S}_{s_{1}+1,n}\bd{\beta}^{\star}/n$,
$T_{3} = \mathbb{S}_{s_{1}+1,n}^{\epsilon}/n$,
$T_{4}(\bd{\beta}^{(0)}) = \mathbb{S}_{1,s_{1}}\bd{\beta}^{(0)}/n$
and
$T_{5} = (\mathbb{S}_{S_{21}^{\star}}-\mathbb{S}_{S_{31}^{\star}})\lambda/n$.
We will show that, wpg1,
$ \normE{T_{0}^{-1}T_{1}} \leq C_{2}/4$,
$ \normE{T_{0}^{-1}T_{2}} \leq 2\sqrt{d}C_{1}$,
$ \normE{T_{0}^{-1}T_{3}} \leq C_{2}/4$,
$ \normE{T_{0}^{-1}T_{4}(\bd{\beta}^{(0)})} \leq C_{2}/4$
and
$ \normE{T_{0}^{-1}T_{5}} \leq C_{2}/4$.
Then, wpg1,
\begin{equation*}
\normE{\bd{\beta}^{(1)}}
\leq
\sum_{i=1}^{5}
\normE{T_{0}^{-1}T_{i}}
\leq
2\sqrt{d}C_{1} + C_{2}.
\end{equation*}

\OnItem{On $T_{0}^{-1}T_{1}$.}
For $s_{2}\gamma_{n}/n=o(1)$, wpg1,
\begin{align*}
\normE{T_{0}^{-1}T_{1}}
& \leq
\normF{
(\frac{1}{n}\mathbb{S})^{-1}}
\normE{
\frac{1}{n}\mathbb{S}_{S_{12}^{\star}}^{\mu}
}
\leq
4\normF{
\bd{\Sigma}_{X}^{-1}}
\E
\normE{
\bd{X}_{0}
}
\frac{s_{2}}{n}
\gamma_{n}
\rightarrow 0.
\end{align*}
Thus, wpg1,
$\normE{T_{0}^{-1}T_{1}}
\leq C_{2}/4$.

\OnItem{On $T_{0}^{-1}T_{2}$.}
Wpg1,
\begin{align*}
\normE{T_{0}^{-1}T_{2}}
& \leq
\normF{
(\frac{1}{n}\mathbb{S})^{-1}
\frac{1}{n}\mathbb{S}_{s_{1}+1,n}}
\normE{
\bd{\beta}^{\star}
}
\leq
2 \normF{
\bd{I}_{d}}
C_{1}
=
2 \sqrt{d}
C_{1}.
\end{align*}

\OnItem{On $T_{0}^{-1}T_{3}$.}
Wpg1,
\begin{align*}
\normE{
T_{0}^{-1}T_{3}
}
& \leq
2\normF{
\bd{\Sigma}_{X}^{-1}}
\normE{\frac{1}{n}\mathbb{S}_{s_{1}+1,n}^{\epsilon}
}
\ConvProb 0.
\end{align*}
Thus, wpg1,
$\normE{T_{0}^{-1}T_{3}} \leq C_{2}/4$.

\OnItem{On $T_{0}^{-1}T_{4}(\bd{\beta}^{(0)})$.}
For $s_{1}/n = o(1)$,
\begin{align*}
\normE{T_{0}^{-1}T_{4}(\bd{\beta}^{(0)})}
& \leq
\frac{s_{1}}{n}
\normF{
(\frac{1}{n}\mathbb{S})^{-1}
\frac{1}{s_{1}}\mathbb{S}_{1,s_{1}}
}
\normE{
\bd{\beta}^{(0)}
}
\leq
\frac{s_{1}}{n}
2\sqrt{d}
C_{2}
\ConvProb 0.
\end{align*}
Thus, wpg1,
$\normE{T_{0}^{-1}T_{4}(\bd{\beta}^{(0)})} \leq C_{2}/4$.

\OnItem{On $T_{0}^{-1}T_{5}$.}
For $s_{1}\lambda/n = O(1)$, wpg1,
\begin{align*}
\normE{
T_{0}^{-1}T_{5}
}
& \leq
2\normF{
\bd{\Sigma}_{X}^{-1}}
\frac{s_{1}\lambda}{n}
(\normE{
\frac{1}{s_{1}}\mathcal{S}_{S_{21}^{\star}}}
+\normE{
\frac{1}{s_{1}}\mathcal{S}_{S_{31}^{\star}}
})
\ConvProb 0.
\end{align*}
Thus, wpg1,
$\normE{T_{0}^{-1}T_{5}} \leq C_{2}/4$.

Next, consider
$\normE{\bd{\beta}_{2} - \bd{\beta}_{1}}$.
Since $\bd{\beta}^{(1)}$ is bounded wpg1,
by Lemma \ref{lem3.1},
$\mathcal{A}_{1}$ occurs wpg1.
Then,
$$
\bd{\beta}^{(2)}
=
T_{0}^{-1}T_{1}
+
T_{0}^{-1}T_{2}
+
T_{0}^{-1}T_{3}
+
T_{0}^{-1}T_{4}(\bd{\beta}^{(1)})
+
T_{0}^{-1}T_{5},
$$
where
$
T_{4}(\bd{\beta}^{(1)}) = (1/n)\mathbb{S}_{1,s_{1}}\bd{\beta}^{(1)}.
$
Thus, wpg1,
$$\bd{\beta}^{(2)} - \bd{\beta}^{(1)}
=
\mathbb{S}^{-1}\mathbb{S}_{1,s_{1}}
(\bd{\beta}^{(1)}-\bd{\beta}^{(0)}).$$
It follows that, for $s_{1}=o(n)$, wpg1,
$$
\normE{\bd{\beta}^{(2)}  - \bd{\beta}^{(1)}}
\leq
\normF{\mathbb{S}^{-1}\mathbb{S}_{1,s_{1}}}
\normE{\bd{\beta}^{(1)} -\bd{\beta}^{(0)} }
\leq
(2\sqrt{d}s_{1}/n)
(4\sqrt{d}C_{1} + 2C_{2})
\rightarrow 0.
$$
Then, wpg1,
$\bd{\beta}^{(2)} = \bd{\beta}^{(1)}$,
which means that, wpg1,
the iteration algorithm
stops at the second iteration.

Finally, for any $K\geq 1$,
repeat the above arguments.
Then, with at least probability
$p_{n,K} = P(\bigcap_{k=0}^{K}\mathcal{A}_{k})$,
which increases to one by Lemma  \ref{lem3.1},
we have
\begin{align*}
\normE{\bd{\beta}^{(K+1)}  - \bd{\beta}^{(K)}}
\leq
(2\sqrt{d}s_{1}/n)^{K}
(4\sqrt{d}C_{1} + 2C_{2})
=O((s_{1}/n)^{K})
\rightarrow 0,
\end{align*}
and
$\normE{\bd{\beta}^{(k)}}\leq 2\sqrt{d}C_{1} + C_{2}$ for all $k\leq K$.
\end{proof}

\begin{proofLong}
\begin{proof}
[Long Proof of Lemma \ref{prop2.2}]
First, we show that, wpg1, $\normE{\bd{\beta}^{(1)}}$ is bounded by $2\sqrt{d}C_{1}+C_{2}$.
For each $k\geq 1$,
we have
\begin{align*}
\mathbb{S}\bd{\beta}^{(k)}
& =
\sum_{i\in S_{1}}\bd{X}_{i}Y_{i}
+
\sum_{i\in S_{2}}\bd{X}_{i}(\bd{X}_{i}^{T}\bd{\beta}^{(k-1)} + \lambda)
+
\sum_{i\in S_{3}}\bd{X}_{i}(\bd{X}_{i}^{T}\bd{\beta}^{(k-1)} - \lambda) \\
& =
\sum_{i\in S_{1}}\bd{X}_{i}(\mu_{i}^{\star} + \bd{X}_{i}^{T}\bd{\beta}^{\star} + \epsilon_{i})
+
\sum_{i\in S_{2}}\bd{X}_{i}(\bd{X}_{i}^{T}\bd{\beta}^{(k-1)} + \lambda)
+
\sum_{i\in S_{3}}\bd{X}_{i}(\bd{X}_{i}^{T}\bd{\beta}^{(k-1)} - \lambda) \\
& =
\sum_{i\in S_{1}}\bd{X}_{i}\mu_{i}^{\star}
+
\sum_{i\in S_{1}}\bd{X}_{i}\bd{X}_{i}^{T}\bd{\beta}^{\star}
+
\sum_{i\in S_{1}}\bd{X}_{i}\epsilon_{i}
+
\sum_{i\in S_{2}\cup S_{3}}\bd{X}_{i}\bd{X}_{i}^{T}\bd{\beta}^{(k-1)}
+
\lambda(\sum_{i\in S_{2}}\bd{X}_{i}-\sum_{i\in S_{3}}\bd{X}_{i})\\
& =
\sum_{i\in S_{11}}\bd{X}_{i}\mu_{i}^{\star}
+
\sum_{i\in S_{12}}\bd{X}_{i}\mu_{i}^{\star}
+
\sum_{i\in S_{1}}\bd{X}_{i}\bd{X}_{i}^{T}\bd{\beta}^{\star}
+
\sum_{i\in S_{1}}\bd{X}_{i}\epsilon_{i}
+
\sum_{i\in S_{2}\cup S_{3}}\bd{X}_{i}\bd{X}_{i}^{T}\bd{\beta}^{(k-1)}
+
\lambda(\sum_{i\in S_{2}}\bd{X}_{i}-\sum_{i\in S_{3}}\bd{X}_{i}),
\end{align*}
where the index sets all depend on $\bd{\beta}^{(k-1)}$. Specifically,
\begin{align*}
S_{1} = S_{1}(\bd{\beta}^{(k-1)}) &= \{1\leq i\leq n: |Y_{i}-\bd{X}_{i}^{T}\bd{\beta}^{(k-1)}| \leq \lambda\}, \\
S_{2} = S_{2}(\bd{\beta}^{(k-1)}) &= \{1\leq i\leq n: Y_{i}-\bd{X}_{i}^{T}\bd{\beta}^{(k-1)} > \lambda\}, \\
S_{3} = S_{3}(\bd{\beta}^{(k-1)}) &= \{1\leq i\leq n: Y_{i}-\bd{X}_{i}^{T}\bd{\beta}^{(k-1)} < -\lambda\}, \\
S_{11} = S_{11}(\bd{\beta}^{(k-1)}) &= \{1\leq i\leq s_{1}: |Y_{i}-\bd{X}_{i}^{T}\bd{\beta}^{(k-1)}| \leq \lambda\}, \\
S_{12} = S_{11}(\bd{\beta}^{(k-1)}) &= \{s_{1}+1\leq i\leq s: |Y_{i}-\bd{X}_{i}^{T}\bd{\beta}^{(k-1)}| \leq \lambda\}.
\end{align*}

Denote an event
\begin{align*}
\mathcal{A}_{k-1}
=
\{
S_{11}(\bd{\beta}^{(k-1)})=\emptyset,
S_{12}(\bd{\beta}^{(k-1)})=S_{12}^{\star},
S_{1}(\bd{\beta}^{(k-1)})=S_{10}^{\star}\cup S_{12}^{\star};
S_{2}(\bd{\beta}^{(k-1)})=S_{21}^{\star};
S_{3}(\bd{\beta}^{(k-1)})=S_{31}^{\star}
\}.
\end{align*}

By Lemma \ref{lem3.1},
$P(\mathcal{A}_{0})\rightarrow 1$.

Thus, wpg1, we have
$$
\bd{\beta}^{(1)}
=
T_{0}^{-1}T_{1}
+
T_{0}^{-1}T_{2}
+
T_{0}^{-1}T_{3}
+
T_{0}^{-1}T_{4}(\bd{\beta}^{(0)})
+
T_{0}^{-1}T_{5},
$$
where
\begin{align*}
T_{0} &= \frac{1}{n}\mathbb{S}, \\
T_{1} & =  \frac{1}{n}\mathbb{S}_{S_{12}^{\star}}^{\mu},\\
T_{2} &= \frac{1}{n}\mathbb{S}_{s_{1}+1,n}\bd{\beta}^{\star}, \\
T_{3} &= \frac{1}{n}\mathbb{S}_{s_{1}+1,n}^{\epsilon}, \\
T_{4}(\bd{\beta}^{(0)}) &= \frac{1}{n}\mathbb{S}_{1,s_{1}}\bd{\beta}^{(0)}, \\
T_{5} &= \frac{1}{n}(\mathbb{S}_{S_{21}^{\star}}-\mathbb{S}_{S_{31}^{\star}})\lambda.
\end{align*}

We will show that, wpg1,
\begin{align*}
& \normE{T_{0}^{-1}T_{1}} \leq C_{2}/4, \\
& \normE{T_{0}^{-1}T_{2}} \leq 2\sqrt{d}C_{1}, \\
& \normE{T_{0}^{-1}T_{3}} \leq C_{2}/4, \\
& \normE{T_{0}^{-1}T_{4}(\bd{\beta}^{(0)})} \leq C_{2}/4, \\
& \normE{T_{0}^{-1}T_{5}} \leq C_{2}/4.
\end{align*}

Thus, we have, wpg1,
\begin{equation*}
\normE{\bd{\beta}^{(1)}}
\leq
\sum_{i=1}^{5}
\normE{T_{0}^{-1}T_{i}}
\leq
2\sqrt{d}C_{1} + C_{2}.
\end{equation*}

\OnItem{On $T_{0}^{-1}T_{1}$.}
We have, for $s_{2}\gamma_{n}/n=o(1)$, wpg1,
\begin{align*}
\normE{T_{0}^{-1}T_{1}}
& =
\normE{
(\frac{1}{n}\mathbb{S})^{-1}
\frac{1}{n}\mathbb{S}_{S_{12}^{\star}}^{\mu}} \\
& \leq
\normF{
(\frac{1}{n}\mathbb{S})^{-1}}
\normE{
\frac{1}{n}\mathbb{S}_{S_{12}^{\star}}^{\mu}
} \\
& \leq
\normF{
(\frac{1}{n}\mathbb{S})^{-1}}
\normE{
\frac{1}{n}\sum_{i=s_{1}+1}^{s}\bd{X}_{i}\mu_{i}
} \\
& =
\normF{
(\frac{1}{n}\mathbb{S})^{-1}}
\frac{s_{2}}{n}\normE{
\frac{1}{s_{2}}
\sum_{i=s_{1}+1}^{s}\bd{X}_{i}\mu_{i}
} \\
& \leq
\normF{
(\frac{1}{n}\mathbb{S})^{-1}}
\frac{s_{2}}{n}
\gamma_{n}
\frac{1}{s_{2}}
\sum_{i=s_{1}+1}^{s}
\normE{
\bd{X}_{i}
} \\
(wpg1) & \leq
2\normF{
\bd{\Sigma}_{X}^{-1}}
\frac{s_{2}}{n}
\gamma_{n}
2
\E
\normE{
\bd{X}_{0}
} \\
& =
4\normF{
\bd{\Sigma}_{X}^{-1}}
\E
\normE{
\bd{X}_{0}
}
\frac{s_{2}}{n}
\gamma_{n}
\rightarrow 0.
\end{align*}

Thus, we have, wpg1,
$\normE{T_{0}^{-1}T_{1}}
\leq C_{2}/4$.

\OnItem{On $T_{0}^{-1}T_{2}$.}
We have, wpg1,
\begin{align*}
\normE{T_{0}^{-1}T_{2}}
& =
\normE{
(\frac{1}{n}\mathbb{S})^{-1}
\frac{1}{n}\mathbb{S}_{s_{1}+1,n}\bd{\beta}^{\star}
} \\
& \leq
\normF{
(\frac{1}{n}\mathbb{S})^{-1}
\frac{1}{n}\mathbb{S}_{s_{1}+1,n}}
\normE{
\bd{\beta}^{\star}
} \\
& \leq
\normF{
(\frac{1}{n}\mathbb{S})^{-1}
\frac{1}{n}\mathbb{S}_{s_{1}+1,n}}
C_{1} \\
(wpg1) & \leq
2 \normF{
\bd{\Sigma}_{X}^{-1}
\bd{\Sigma}_{X}}
C_{1} \\
& =
2 \normF{
\bd{I}_{d}}
C_{1} \\
& =
2 \sqrt{d}
C_{1}.
\end{align*}

\OnItem{On $T_{0}^{-1}T_{3}$.}
We have, wpg1,
\begin{align*}
\normE{
T_{0}^{-1}T_{3}
}
& =
\normE{
(\frac{1}{n}\mathbb{S})^{-1}\frac{1}{n}\mathbb{S}_{s_{1}+1,n}^{\epsilon}
} \\
& =
\normF{
(\frac{1}{n}\mathbb{S})^{-1}}
\normE{\frac{1}{n}\mathbb{S}_{s_{1}+1,n}^{\epsilon}
} \\
(wpg1) & \leq
2\normF{
\bd{\Sigma}_{X}^{-1}}
\normE{\frac{1}{n}\mathbb{S}_{s_{1}+1,n}^{\epsilon}
}
\ConvProb 0.
\end{align*}

Thus, wpg1,
$\normE{T_{0}^{-1}T_{3}} \leq C_{2}/4$.

\OnItem{On $T_{0}^{-1}T_{4}(\bd{\beta}^{(0)})$.}
We have,
\begin{align*}
\normE{T_{0}^{-1}T_{4}(\bd{\beta}^{(0)})}
& =
\normE{
(\frac{1}{n}\mathbb{S})^{-1}
\frac{1}{n}\mathbb{S}_{1,s_{1}}\bd{\beta}^{(0)}
} \\
& \leq
\normF{
(\frac{1}{n}\mathbb{S})^{-1}}
\normE{
\frac{1}{n}\mathbb{S}_{1,s_{1}}\bd{\beta}^{(0)}
} \\
& \leq
\normF{
(\frac{1}{n}\mathbb{S})^{-1}}
\normE{
\frac{1}{n}\mathbb{S}_{1,s_{1}}}
\normE{\bd{\beta}^{(0)}
} \\
(wpg1) & \leq
2\normF{
\bd{\Sigma}_{X}^{-1}}
\normE{
\frac{1}{n}\mathbb{S}_{1,s_{1}}}
C_{2}
\ConvProb 0.
\end{align*}

Thus, wpg1,
$\normE{T_{0}^{-1}T_{4}(\bd{\beta}^{(0)})} \leq C_{2}/4$.

\OnItem{On $T_{0}^{-1}T_{5}$.}
For $s_{1}\lambda/n = O(1)$, wpg1,
\begin{align*}
\normE{
T_{0}^{-1}T_{5}
}
& =
\normE{
(\frac{1}{n}\mathbb{S})^{-1}
\frac{1}{n}(\mathcal{S}_{S_{21}^{\star}}-\mathcal{S}_{S_{31}^{\star}})\lambda
} \\
& \leq
\normF{
(\frac{1}{n}\mathbb{S})^{-1}}
\frac{s_{1}\lambda}{n}
(\normE{
\frac{1}{s_{1}}\mathcal{S}_{S_{21}^{\star}}
-
\frac{1}{s_{1}}\mathcal{S}_{S_{31}^{\star}}
}) \\
& \leq
\normF{
(\frac{1}{n}\mathbb{S})^{-1}}
\frac{s_{1}\lambda}{n}
(\normE{
\frac{1}{s_{1}}\mathcal{S}_{S_{21}^{\star}}}
+\normE{
\frac{1}{s_{1}}\mathcal{S}_{S_{31}^{\star}}
}) \\
(wpg1) & \leq
2\normF{
\bd{\Sigma}_{X}^{-1}}
\frac{s_{1}\lambda}{n}
(\normE{
\frac{1}{s_{1}}\mathcal{S}_{S_{21}^{\star}}}
+\normE{
\frac{1}{s_{1}}\mathcal{S}_{S_{31}^{\star}}
})
\ConvProb 0.
\end{align*}
Thus, wpg1,
$\normE{T_{0}^{-1}T_{5}} \leq C_{2}/4$.

\lineProof

Next, consider
$\normE{\bd{\beta}_{2} - \bd{\beta}_{1}}$.
Since $\bd{\beta}^{(1)}$ is bounded wpg1,
by Lemma \ref{lem3.1},
$\mathcal{A}_{1}$ occurs wpg1.

Then,
$$
\bd{\beta}^{(2)}
=
T_{0}^{-1}T_{1}
+
T_{0}^{-1}T_{2}
+
T_{0}^{-1}T_{3}
+
T_{0}^{-1}T_{4}(\bd{\beta}^{(1)})
+
T_{0}^{-1}T_{5},
$$
where
\begin{align*}
T_{4}(\bd{\beta}^{(1)}) &= \frac{1}{n}\mathbb{S}_{1,s_{1}}\bd{\beta}^{(1)}.
\end{align*}

Thus, wpg1,
$$\bd{\beta}^{(2)} - \bd{\beta}^{(1)}
=
\mathbb{S}^{-1}\mathbb{S}_{1,s_{1}}
(\bd{\beta}^{(1)}-\bd{\beta}^{(0)}).$$
Thus, for $s_{1}=o(n)$,
$$
\normE{\bd{\beta}^{(2)}  - \bd{\beta}^{(1)}}
\leq
\normF{\mathbb{S}^{-1}\mathbb{S}_{1,s_{1}}}
\normE{\bd{\beta}^{(1)} -\bd{\beta}^{(0)} }
\leq
(2\sqrt{d}s_{1}/n)
(4\sqrt{d}C_{1} + 2C_{2})
\rightarrow 0.
$$
Thus, wpg1,
$\bd{\beta}^{(2)} = \bd{\beta}^{(1)}$,
which means that, wpg1, the iteration algorithm converges at the first iteration.

\lineProof

For any $K\geq 1$,
Repeat the above arguments,
with at least probability
$p_{n,K} = P(\bigcap_{k=0}^{K}\mathcal{A}_{k})$,
we have
\begin{align*}
\normE{\bd{\beta}^{(K+1)}  - \bd{\beta}^{(K)}}
\leq
(2\sqrt{d}s_{1}/n)^{K}
(4\sqrt{d}C_{1} + 2C_{2})
=O((s_{1}/n)^{K})
\rightarrow 0,
\end{align*}
and
$\normE{\bd{\beta}^{(K+1)} - \bd{\beta}^{(K)}}\leq O((s_{1}/n)^{K})$.
\end{proof}
\end{proofLong}

\begin{lem}
\label{lem2.3}
A necessary and sufficient condition for
$(\hat{\bd{\mu}}, \hat{\bd{\beta}})$
to be a minimizer of $L(\bd{\mu},\bd{\beta})$ is that
it is a solution to equations
(\ref{eq2.4}) and (\ref{eq2.5}).
\end{lem}

\begin{proofName}
\subsection{Proof of Lemma \ref{lem2.3}}
\end{proofName}

\begin{proof}
[Proof of Lemma \ref{lem2.3}]
First, we show
a solution of
(\ref{eq2.4}) and (\ref{eq2.5})
satisfies the necessary and sufficient condition in
Lemma \ref{lem2.1}.
Denote a solution of (\ref{eq2.4}) and (\ref{eq2.5})
as $(\hat{\bd{\mu}}, \hat{\bd{\beta}})$.
Then
$
    \hat{\bd{\beta}}
    =
    (\bd{X}^{T}\bd{X})^{-1}\bd{X}^{T}(\bd{Y}-\hat{\bd{\mu}})
$,
which is exactly the first condition in Lemma \ref{lem2.1},
and, for each $i=1,2,\cdots, n$,
$(\hat{\bd{\mu}}, \hat{\bd{\beta}})$ satisfies
one of three cases:
$|Y_{i}- \bd{X}_{i}^{T}\hat{\bd{\beta}}| \leq \lambda$
and $\hat{\mu}_{i} = 0$;
$Y_{i}-\bd{X}_{i}^{T}\hat{\bd{\beta}} > \lambda$
and $\hat{\mu}_{i} =Y_{i} - \bd{X}_{i}^{T}\hat{\bd{\beta}} - \lambda$;
$Y_{i}-\bd{X}_{i}^{T}\hat{\bd{\beta}} < -\lambda$
and $\hat{\mu}_{i} =Y_{i} - \bd{X}_{i}^{T}\hat{\bd{\beta}} + \lambda$.
If $(\hat{\bd{\mu}}, \hat{\bd{\beta}})$ satisfies the first case,
it satisfies the third condition in Lemma \ref{lem2.1}.
If $(\hat{\bd{\mu}}, \hat{\bd{\beta}})$ satisfies the second case,
then $\hat{\mu}_{i}>0$
and
$Y_{i} - \hat{\mu}_{i} - \bd{X}_{i}^{T}\hat{\bd{\beta}} = \lambda = \lambda \sign(\hat{\mu}_{i})$,
which means that the second case satisfies  the second condition in Lemma \ref{lem2.1}.
Similarly, the third case also satisfies  the second condition in Lemma \ref{lem2.1}.
Thus $(\hat{\bd{\mu}}, \hat{\bd{\beta}})$ satisfies
the necessary and sufficient condition in
Lemma \ref{lem2.1}.

In the other direction,
suppose $(\hat{\bd{\mu}}, \hat{\bd{\beta}})$ satisfies
the necessary and sufficient condition in
Lemma \ref{lem2.1}.
Then, the first condition in Lemma \ref{lem2.1}
exactly $(\ref{eq2.4})$.
For each $i$,
$(\hat{\bd{\mu}}, \hat{\bd{\beta}})$ satisfies
one of three cases:
$\hat{\mu}_{i} = 0$ and
$|Y_{i}- \bd{X}_{i}^{T}\hat{\bd{\beta}}| \leq \lambda$;
$\hat{\mu}_{i} > 0$
and
$Y_{i} - \hat{\mu}_{i} -\bd{X}_{i}^{T}\hat{\bd{\beta}} = \lambda$;
$\hat{\mu}_{i} < 0$
and
$Y_{i} - \hat{\mu}_{i} -\bd{X}_{i}^{T}\hat{\bd{\beta}} = -\lambda$.
If $(\hat{\bd{\mu}}, \hat{\bd{\beta}})$ satisfies the first case,
it satisfies the first case in (\ref{eq2.5}).
If $(\hat{\bd{\mu}}, \hat{\bd{\beta}})$ satisfies the second case,
then $\hat{\mu}_{i} =Y_{i} - \bd{X}_{i}^{T}\hat{\bd{\beta}} - \lambda$ and
$Y_{i}-\bd{X}_{i}^{T}\hat{\bd{\beta}} > \lambda$,
which means that
$(\hat{\bd{\mu}}, \hat{\bd{\beta}})$ satisfies the second case of
(\ref{eq2.5}).
Similarly,
If $(\hat{\bd{\mu}}, \hat{\bd{\beta}})$ satisfies the third case,
then it satisfies the third case of
(\ref{eq2.5}).
Thus, $(\hat{\bd{\mu}}, \hat{\bd{\beta}})$ satisfies
(\ref{eq2.4})
and
(\ref{eq2.5}).
\end{proof}

\section{Supplement for Section \ref{sec3}}
\label{PC.Paper.Supplement.Asymptotic.Properties}
In this supplement,
we provide the proofs of the theoretical results
in Section \ref{sec3}.
Before that, we point out that
those two different sufficient conditions in Theorem
\ref{thm3.2}
come from the different analysis on the term $\mathbb{S}_{S_{12}^{\star}}^{\mu}$.
Each of the two different sufficient conditions
does not imply the other.
Specifically,
on one hand,
suppose the absolute values of $\mu_{i}^{\star}$'s are all equal
for $i=s_{1}+1, s_{2}+2, \cdots, s$.
Then,
$
\normE{\bd{\mu}_{2}^{\star}}^{2+\delta}
=s_{2}^{(2+\delta)/2}|\mu_{s}^{\star}|^{2+\delta}
$
and
$
\sum_{i=s_{1}+1}^{s}|\mu^{\star}_{i}|^{2+\delta}
=s_{2}|\mu_{s}^{\star}|^{2+\delta}
$.
Thus Assumption (C) holds automatically
since $s_{2}\rightarrow \infty$.
This means that
Assumption (C) holds
at least when the absolute magnitudes of $\mu_{i}^{\star}$'s are similar to each other.
For this case,
there still exists a consistent estimator
even if $n/(\kappa_{n}\gamma_{n}) \ll  s_{2} \ll  n$.
On the other hand,
suppose $\mu_{s}^{\star}=\gamma_{n}$
and the other $\mu_{i}^{\star}$'s are all equal to a constant $c>0$.
Then,
$
\normE{\bd{\mu}_{2}^{\star}}^{2+\delta}
=[\gamma_{n}^{2} + (s_{2}-1)c^{2}]^{(2+\delta)/2}
$
and
$
\sum_{i=s_{1}+1}^{s}|\mu^{\star}_{i}|^{2+\delta}
=\gamma_{n}^{2+\delta} + (s_{2}-1)c^{2+\delta}
$.
If $s_{2} \ll \gamma_{n}^{2} \ll n/(\kappa_{n}\gamma_{n})$,
the previous two terms are both
asymptotically equivalent to $\gamma_{n}^{2+\delta}$.
Thus Assumption (C) fails
but the other sufficient condition holds.

\begin{proofName}
\subsection{Proof of Lemma
\ref{lem3.1}}
\end{proofName}
\begin{proof}[Proof of Lemma \ref{lem3.1}]
The proof is the similar to that of Lemma \ref{PC.Paper.d.infty.index.sets.consistency} and omitted.
\end{proof}

\begin{proofName}
\subsection{Proof of Theorems
\ref{thm3.2}}
\end{proofName}
\begin{proof}[Proof of Theorems
\ref{thm3.2}]
By Lemma \ref{lem3.1},
wpg1,
the solution $\hat{\bd{\beta}}_{n}$ to $\varphi_{n}(\bd{\beta})=0$ on $\mathcal{B}_{C}(\bd{\beta}^{\star})$ is explicitly given by
\begin{equation*}
\hat{\bd{\beta}}_{n}
= \bd{\beta}^{\star} + T_{0}^{-1}(T_{1} + T_{2} + T_{3}-T_{4}),
\end{equation*}
where
$T_{0}=(1/n)\mathbb{S}_{s_{1}+1, n}$,
$T_{1}=(1/n)\mathbb{S}_{S_{12}^{\star}}^{\mu}$,
$T_{2}=(1/n)\mathbb{S}_{s_{1}+1, n}^{\epsilon}$,
$T_{3}=(\lambda/n)\mathcal{S}_{S_{21}^{\star}}$
and
$T_{4}=(\lambda/n)\mathcal{S}_{S_{31}^{\star}}$.
We will show that
$T_{0} \ConvProb \bd{\bd{\Sigma}}_{X}^{-1}>0$ with the Frobenius norm
and
$T_{i} \ConvProb 0$ with the Euclidean norm for $i=1,2,3,4$.
Thus,
by Slutsky's lemma (see, for example, Lemma 2.8 on page 11 of \cite{Vaart1998}),
$\hat{\bd{\beta}}_{n}$
is a consistent estimator of $\bd{\beta}^{\star}$.

\lineProof

\OnItem{On $T_{0}^{-1}$.}
By law of large number,
$T_{0} \ConvProb \bd{\bd{\Sigma}}_{X}>0$.
Then, by continuous mapping theorem,
$T_{0}^{-1} \ConvProb \bd{\bd{\Sigma}}_{X}^{-1}>0$.

\lineProof

\OnItem{On $T_{1}$: Approach One.}
Suppose $s_{2}=o(n/(\kappa_{n}\gamma_{n}))$.
Then,
\begin{align*}
\normE{T_{1}}
& \leq
\frac{1}{n}\sum_{i=s_{1}+1}^{s}\normE{\bd{X}_{i}\mu_{i}^{\star}}
=
\frac{1}{n}\sum_{i=s_{1}+1}^{s}\normE{\bd{X}_{i}}\cdot|\mu_{i}^{\star}|
\leq
s_{2}\kappa_{n}\gamma_{n}/n
= o(1).
\end{align*}

\lineProof

\OnItem{On $T_{1}$: Approach Two.}
Under Assumption (C), it follows
$
(\bd{\bd{\Sigma}}_{X}\sum_{i=s_{1}+1}^{s}\mu^{\star 2}_{i})^{-1/2}
\mathbb{S}_{S_{12}^{\star}}^{\mu}
\ConvDist N(0,I_{d})
$.
In fact, Assumption (C) implies the Lyapunov condition for sequence of random vectors
(see, e.g. Proposition 2.27 on page 332 of \citep{Vaart1998}).
More specifically,
recall the Lyapunov condition is that there exists some constant $\delta>0$
such that
$$
\sum_{i=s_{1}+1}^{s}
\E
\normE{
(\bd{\bd{\Sigma}}_{X}\sum_{j=s_{1}+1}^{s}\mu^{\star 2}_{j})^{-1/2}
\bd{X}_{i}\mu_{i}^{\star}
}^{2+\delta}
\rightarrow 0.
$$
Then, by Assumption (C),
\begin{align*}
\sum_{i=s_{1}+1}^{s}
\E
\normE{
(\bd{\bd{\Sigma}}_{X}\sum_{j=s_{1}+1}^{s}\mu^{\star 2}_{j})^{-\frac{1}{2}}
\bd{X}_{i}\mu_{i}^{\star}
}^{2+\delta}
\leq
(\sum_{j=s_{1}+1}^{s}\mu^{\star 2}_{j})^{-\frac{2+\delta}{2}}
\sum_{i=s_{1}+1}^{s}
|\mu_{i}^{\star}|^{2+\delta}
\lambda_{\min}^{-\frac{2+\delta}{2}}
\E
\normE{X_{0}}^{2+\delta}
\longrightarrow  0,
\end{align*}
where $\lambda_{\min}>0$ is the minimum eigenvalue of $\bd{\bd{\Sigma}}_{X}$.
Then,
\begin{align*}
\normE{T_{1}}
& =
\normE{\frac{1}{n}\mathbb{S}_{S_{12}^{\star}}^{\mu}}
 \leq
\frac{1}{n}
\normF{(\bd{\bd{\Sigma}}_{X}\sum_{i=s_{1}+1}^{s}\mu^{\star 2}_{i})^{1/2}}
\normE{(\bd{\bd{\Sigma}}_{X}\sum_{i=s_{1}+1}^{s}\mu^{\star 2}_{i})^{-1/2}
\mathbb{S}_{S_{12}^{\star}}^{\mu}}\\
& =
\frac{1}{n}
(\sum_{i=s_{1}+1}^{s}\mu^{\star 2}_{i})^{1/2}
\normF{\bd{\bd{\Sigma}}_{X}^{1/2}}
O_{P}(1)
 \leq
\frac{1}{n}
(s_{2}\gamma_{n}^{2})^{1/2}
O_{P}(1)
\leq
\frac{1}{\sqrt{n}}\gamma_{n}
O_{P}(1)
= o_{P}(1).
\end{align*}

\lineProof

\OnItem{On $T_{2}$.}
By law of large number, $T_{2}=o_{P}(1)$.
\lineProof

\OnItem{On $T_{3}$ and $T_{4}$.}
By noting $\lambda \ll  \sqrt{n}$,
\begin{align*}
\normE{T_{3}}
=
\normE{\lambda \frac{1}{n} \mathcal{S}_{S_{21}^{\star}}}
=
\lambda \frac{\sqrt{s_{1}}}{n} \normE{\frac{1}{\sqrt{s_{1}}} \mathcal{S}_{S_{21}^{\star}}}
\leq
\frac{\lambda}{\sqrt{n}} O_{P}(1)
= o_{P}(1).
\end{align*}

\lineProof
Thus $T_{3}=o_{P}(1)$.
In the same way, we can show that $T_{4}=o_{P}(1)$ holds.
%
%
%
\end{proof}

In Theorem,
one condition is $D_{n}/n=o(1)$.
It turns out we can consider other conditions on $D_{n}$
and derive more possible asymptotic distributions for $\hat{\bd{\beta}}_{n}$.
\begin{thm}[\revTang{Asymptotic Distributions on $\hat{\bd{\beta}}_{n}$}: more cases]
\label{thm3.6}
Under Assumptions (A), (B) and (C),
for all constants
$b, c\in\mathbb{R}^{+}$,
\begin{enumerate}\itemsep -0.05 in
\item [(1)]
when $s_{1} \ll  n/\lambda^{2}$ and $D_{n}^{2}/n=o(1)$,
$
\sqrt{n}(\hat{\bd{\beta}}_{n}-\bd{\beta}^{\star})
\ConvDist
N(0, \sigma^{2}\bd{\bd{\Sigma}}_{X}^{-1})
$; {\bf [main case]}

\item [(2)] when $s_{1} \ll  n/\lambda^{2}$ and $D_{n}^{2}/n \sim c$,
$
\sqrt{n}(\hat{\bd{\beta}}_{n}-\bd{\beta}^{\star})
\ConvDist
N(0, (c+\sigma^{2})\bd{\bd{\Sigma}}_{X}^{-1});
$

\item [(3)] when $s_{1} \ll  n/\lambda^{2}$ and $D_{n}^{2}/n \rightarrow \infty$,
$
r_{n}(\hat{\bd{\beta}}_{n}-\bd{\beta}^{\star})
\ConvDist
N(0, \bd{\bd{\Sigma}}_{X}^{-1}),
$
where $r_{n} \sim n/D_{n}\ll \sqrt{n}$;

\item [(4)] when $s_{1}\sim bn/\lambda^{2}$ and $D_{n}^{2}/n=o(1)$,
$\sqrt{n}(\hat{\bd{\beta}}_{n}-\bd{\beta}^{\star})
\ConvDist
N(0, (b+\sigma^{2})\bd{\bd{\Sigma}}_{X}^{-1});
$

\item [(5)] when $s_{1}\sim bn/\lambda^{2}$ and $D_{n}^{2}/n \sim c$,
$
\sqrt{n}(\hat{\bd{\beta}}_{n}-\bd{\beta}^{\star})
\ConvDist
N(0, (b + c +\sigma^{2})\bd{\bd{\Sigma}}_{X}^{-1});
$

\item [(6)] when $s_{1}\sim bn/\lambda^{2}$ and $D_{n}^{2}/n \rightarrow \infty$,
$
r_{n}(\hat{\bd{\beta}}_{n}-\bd{\beta}^{\star})
\ConvDist
N(0, \bd{\bd{\Sigma}}_{X}^{-1}),
$
where $r_{n} \sim n/D_{n}\ll \sqrt{n}$;

\item [(7)] when $s_{1} \gg  n/\lambda^{2}$ and $D_{n}^{2}/n=o(1)$
or $D_{n}^{2}/n\sim c$,
$
r_{n}(\hat{\bd{\beta}}_{n}-\bd{\beta}^{\star})
\ConvDist
N(0, \bd{\bd{\Sigma}}_{X}^{-1}),
$
where $r_{n}\sim n/(\lambda\sqrt{s_{1}})\ll \sqrt{n}$;


\item [(8)] when $s_{1} \gg  n/\lambda^{2}$ and $D_{n}^{2}/n \rightarrow \infty$,
letting $r_{n}\sim \min\{\sqrt{b}n/(\lambda\sqrt{s_{1}}),n/D_{n}\}\ll \sqrt{n}$,
\begin{enumerate}
\item [(8a)]
  if $\sqrt{b}n/(\lambda\sqrt{s_{1}}) \gg n/D_{n}$,
then
$
r_{n}(\hat{\bd{\beta}}_{n}-\bd{\beta}^{\star})
\ConvDist
N(0,\bd{\bd{\Sigma}}_{X}^{-1});
$

\item [(8b)] if $\sqrt{b}n/(\lambda\sqrt{s_{1}})\sim n/D_{n}$,
then
$
r_{n}(\hat{\bd{\beta}}_{n}-\bd{\beta}^{\star})
\ConvDist
N(0,(1+b)\bd{\bd{\Sigma}}_{X}^{-1});
$

\item [(8c)] if $\sqrt{b}n/(\lambda\sqrt{s_{1}}) \ll n/D_{n}$,
then
$
r_{n}(\hat{\bd{\beta}}_{n}-\bd{\beta}^{\star})
\ConvDist
N(0,b\bd{\bd{\Sigma}}_{X}^{-1}).
$
\end{enumerate}
\end{enumerate}
\end{thm}

  Theorem \ref{thm3.4}
  groups the results according to the asymptotic magnitude of $s_{1}$ given an upper bound
  of the diverging speed of $s_{2}$.
  Alternatively,
  Theorem \ref{thm3.6}
  groups the results according to the asymptotic magnitudes of $s_{1}$ and $D_{n}^{2}$.
  Since both $s_{1}$ and $D_{n}^{2}$ have three cases,
  Theorem \ref{thm3.6} basically contains nine cases.
  For the last case,
  there are further three cases on the relationship
  between $\sqrt{b}n/(\lambda\sqrt{s_{1}})$ and $n/D_{n}$.
  As in Theorem \ref{thm3.4},
  the first case of Theorem \ref{thm3.6} is denoted as the \emph{main case}
  since for this case the incidental parameters are sparse in the sense that
  the size and magnitude of
  the nonzero incidental parameters
  $\bd{\mu}_{1}^{\star}$ and $\bd{\mu}_{2}^{\star}$ are well controlled.
  Note that
  $s_{2}=o(\sqrt{n}/(\kappa_{n}\gamma_{n}))$
  implies
  $D_{n}^{2}/n=o(1)$.
  which means that, under Assumption (C),
  the cases \emph{(1)}, \emph{(4)} and \emph{(7)} of Theorem \ref{thm3.6}
  actually imply
  the three results of Theorem \ref{thm3.4}.
  As in Theorem \ref{thm3.4},
  the convergence rate of $\hat{\bd{\beta}}_{n}$ becomes less than $\sqrt{n}$
  when $s_{1} \gg  n/\lambda^{2}$ or $D_{n}^{2}/n \rightarrow \infty$,
  that is, when the size and magnitude of the nonzero incidental parameters are large;
  the boundary phenomenon also appears.

\begin{proofName}
\subsection{Proof of Theorems
\ref{thm3.4}
and \ref{thm3.6}}
\end{proofName}
\begin{proof}[Proof of Theorems
\ref{thm3.4}
and \ref{thm3.6}]
It is sufficient to provide the proof for the case
where the sizes of index sets
$S_{21}^{\star}=\{1\leq i\leq s_{1}: \mu_{i}^{\star} > 0\}$
and
$S_{31}^{\star}=\{1\leq i\leq s_{1}: \mu_{i}^{\star} < 0\}$
are both asymptotically $s_{1}/2$
and $b=2$.

From the proof of Theorems
\ref{thm3.2},
wpg1,
\begin{equation*}
\hat{\bd{\beta}}
=
\bd{\beta}^{\star}
+
\mathbb{S}_{s_{1}+1,n}^{-1}
[\mathbb{S}_{S_{12}^{\star}}^{\mu}
 +
\mathbb{S}_{s_{1}+1,n}^{\epsilon}
+
\lambda
(\mathcal{S}_{S_{21}^{\star}}
-\mathcal{S}_{S_{31}^{\star}})].
\end{equation*}
Let $r_{n}$ be a sequence going to infinity.
Then,
$
r_{n}(\hat{\bd{\beta}}_{n} - \bd{\beta}^{\star})
=T_{0}^{-1}(V_{1}+V_{2}+V_{3}-V_{4})
$,
where
$V_{1} = r_{n}T_{1}$,
$V_{2} = r_{n}T_{2}$,
$V_{3} = r_{n}T_{3}$,
$V_{4} = r_{n}T_{4}$
and
$T_{i}$'s are defined in
the proof of
Theorem \ref{thm3.2}.
Next we derive the asymptotic properties of $T_{0}$ and $V_{i}$'s,
from which the desired results follow by Slutsky's lemma.

\OnItem{On $T_{0}$.}
By the proof of
Theorem \ref{thm3.2},
$T_{0}^{-1}\ConvProb \bd{\bd{\Sigma}}_{X}^{-1}$

\lineProof

\OnItem{On $V_{1}$: Approach One.}
If $r_{n}=\sqrt{n}$
and
$s_{2}=o(\sqrt{n}/(\kappa_{n}\gamma_{n}))$,
then
\begin{align*}
\normE{T_{1}}
=
\normE{r_{n}\frac{1}{n}\mathbb{S}_{S_{12}^{\star}}^{\mu}}
\leq
r_{n}\frac{1}{n}\sum_{i=s_{1}+1}^{s}\normE{\bd{X}_{i}}\cdot|\mu_{i}^{\star}|
\leq
r_{n}\frac{1}{n}s_{2}\kappa_{n}\gamma_{n}
=
\frac{1}{\sqrt{n}}s_{2}\kappa_{n}\gamma_{n}
=
o(1).
\end{align*}
Thus, if $r_{n}=\sqrt{n}$ or $r_{n}\ll \sqrt{n}$
and
$s_{2}=o(\sqrt{n}/(\kappa_{n}\gamma_{n}))$,
then
$T_{1}=o_{P}(1)$.

\lineProof

\OnItem{On $V_{1}$: Approach Two.}
If $r_{n}=\sqrt{n}$,
then
\begin{align*}
T_{1}
& =
r_{n}\frac{1}{n}\mathbb{S}_{S_{12}^{\star}}^{\mu}
=
r_{n}\frac{D_{n}}{n}\frac{1}{D_{n}}\mathbb{S}_{S_{12}^{\star}}^{\mu}
=
\frac{D_{n}}{\sqrt{n}}\frac{1}{D_{n}}\mathbb{S}_{S_{12}^{\star}}^{\mu},
\end{align*}
where
$
D_{n}= \normE{\bd{\mu}_{2}^{\star}} =(\sum_{i=s_{1}+1}^{s}\mu^{\star 2}_{i})^{1/2}.
$
There are three cases on $D_{n}/\sqrt{n}$ or $D_{n}^{2}/n$.
If $D_{n}^{2}/n \rightarrow 0$,
then $T_{1}\ConvProb 0$.
If $D_{n}^{2}/n \rightarrow 1$,
then $T_{1}\ConvDist N(0,\bd{\bd{\Sigma}}_{X})$.
If $D_{n}^{2}/n \rightarrow \infty$,
it means that $r_{n}=\sqrt{n}$ is too fast.
Let $r_{n}\sim n/D_{n} = \sqrt{n}\sqrt{n/D_{n}^{2}}\ll \sqrt{n}$.
Then $T_{1}\ConvDist N(0,\bd{\bd{\Sigma}}_{X})$;

\lineProof

\OnItem{On $V_{2}$.}
If $r_{n}=\sqrt{n}$, then
$T_{2}\ConvDist N(0,\sigma^{2}\bd{\bd{\Sigma}}_{X})$.
Thus,
if $r_{n}\ll \sqrt{n}$,
$T_{2}\ConvProb 0$;
if $r_{n}\gg \sqrt{n}$;
$T_{2}\ConvProb \infty$.

\lineProof

\OnItem{On $V_{3}$ and $V_{4}$.}
First consider $T_{3}$.
Denote $\#(\cdot)$ as the size function.
If $r_{n}=\sqrt{n}$, then
\begin{align*}
T_{3}
& =
\lambda r_{n} \frac{1}{n} \mathcal{S}_{S_{21}^{\star}}
=
\lambda \sqrt{\frac{s_{1}/2}{n}} \frac{1}{\sqrt{\#(S_{21}^{\star})}} \mathcal{S}_{S_{21}^{\star}}.
\end{align*}
Note that $\#(S_{21}^{\star})= s_{1}/2$.
There are three cases on $\lambda \sqrt{s_{1}/(2n)}$.
If $\lambda \sqrt{s_{1}/(2n)}\rightarrow 0$,
then $T_{3}\ConvProb 0$.
Note that $\lambda \sqrt{s_{1}/(2n)}\rightarrow 0$
is equivalent to  $s_{1}=o(2n/\lambda^{2})$.
If $\lambda \sqrt{s_{1}/(2n)}\rightarrow 1$,
then
$T_{3}\ConvDist N(0, \bd{\bd{\Sigma}}_{X})$.
Note that $\lambda \sqrt{s_{1}/(2n)}\rightarrow 1$
is equivalent to  $s_{1}\sim 2n/\lambda^{2}$.
If $\lambda \sqrt{s_{1}/(2n)}\rightarrow \infty$,
it means $r_{n}=\sqrt{n}$ is too large.
Let
$r_{n}\sim n/(\lambda\sqrt{(s_{1}/2)})=\sqrt{n}\sqrt{2n}/(\lambda\sqrt{s_{1}})\ll \sqrt{n}$.
With this rate $r_{n}$,
$T_{3}\ConvDist N(0, \bd{\bd{\Sigma}}_{X})$.
Note that $\lambda \sqrt{s_{1}/2n}\rightarrow \infty$
is equivalent to  $s_{1}\gg O(2n/\lambda^{2})$.
In the same way, $T_{4}$ can be analyzed and
parallel results can be obtained.
%
\end{proof}

\begin{proofName}
\subsection{Proof of Theorem
\ref{thm3.7}}
\end{proofName}
\begin{proof}[Proof of Theorem
\ref{thm3.7}]
The proof is similar to that of
Theorem
\ref{thm5.4}
and omitted.
\end{proof}

\begin{proofLong}
\begin{proof}[Long Proof of Theorem
\ref{thm3.7}]
By the definition of $\mathcal{E}$, we have
$
P(\mathcal{\mathcal{E}})
= T_{1}T_{2}T_{3}
$,
where
$
T_{1} = P(\bigcap_{i=1}^{s_{1}}\{|\mu_{i}^{\star}+\bd{X}_{i}^{T}(\bd{\beta}^{\star}-\hat{\bd{\beta}})+\epsilon_{i}| > \lambda\})
$,
$
T_{2} = P(\bigcap_{i=s_{1}+1}^{s}\{|\mu_{i}^{\star}+\bd{X}_{i}^{T}(\bd{\beta}^{\star}-\hat{\bd{\beta}})+\epsilon_{i}| \leq \lambda\})
$,
and
$
T_{3} = P(\bigcap_{i=s+1}^{n}\{|\bd{X}_{i}^{T}(\bd{\beta}^{\star}-\hat{\bd{\beta}})+\epsilon_{i}| \leq \lambda\}).
$
We will show that each of $T_{1}$, $T_{2}$ and $T_{3}$ converges to one.
Then $P(\mathcal{\mathcal{E}})\rightarrow 1$.

\lineProof

Denote $\mathcal{C}=\{\normE{\hat{\bd{\beta}}-\bd{\beta}^{\star}}\leq 1\}$.
Then $P(\mathcal{C})\rightarrow 1$ since $\hat{\bd{\beta}}$ is a consistent estimator of $\bd{\beta}^{\star}$.

\OnItem{On $T_{1}$.}
Note that $s_{1}/n \rightarrow 0$,
$\mu^{\star}=\min\{|\mu_{i}^{\star}|:1\leq i\leq s_{1}\}$
and
$\alpha\gamma_{n} \leq \lambda$ and $\kappa_{n} \ll  \lambda \ll  \sqrt{n}$.
Then $1-T_{1}\leq T_{11} + T_{12}$, where
\begin{align*}
T_{11} =
P(\bigcup_{i\in S_{21}^{\star}}\{|\mu_{i}^{\star}+\bd{X}_{i}^{T}(\bd{\beta}^{\star}-\hat{\bd{\beta}})+\epsilon_{i}| \leq \lambda\})
\text{ and }
T_{12} =
P(\bigcup_{i\in S_{31}^{\star}}\{|\mu_{i}^{\star}+\bd{X}_{i}^{T}(\bd{\beta}^{\star}-\hat{\bd{\beta}})+\epsilon_{i}| \leq \lambda\}).
\end{align*}
For $T_{11}$, we have
\begin{align*}
T_{11}
& \leq
P(
\bigcup_{i\in S_{21}^{\star}}\{\epsilon_{i} \leq \lambda - \mu^{\star} + \normE{\bd{X}_{i}}\cdot\normE{\hat{\bd{\beta}} - \bd{\beta}^{\star}}\}, \mathcal{C}) + P(\mathcal{C}^{c})
\leq
P(
\bigcup_{i\in S_{21}^{\star}}\{\epsilon_{i} \leq \lambda - \mu^{\star} + \kappa_{n}\}) + P(\mathcal{C}^{c}) \\
& \leq
s_{1}
P\{\epsilon_{0} \leq -\gamma_{n}\} + P(\mathcal{C}^{c})
\longrightarrow 0.
\end{align*}
Similarly, $T_{12}\rightarrow 0$.
Thus
$T_{1}\rightarrow 1$.

\lineProof

\OnItem{On $T_{2}$ and $T_{3}$.}
By noting that
$s_{2}/n\rightarrow 0$,
$|\mu_{i}^{\star}|\leq \gamma_{n}$ for $s_{1}+1\leq i\leq s$
and
$\lambda \geq 3\gamma_{n}$,
we have
\begin{align*}
T_{2}
\geq
P(
\bigcap_{i=s_{1}}^{s}\{-\lambda - \mu_{i}^{\star} + \kappa_{n} \leq \epsilon_{i} \leq \lambda - \mu_{i}^{\star} - \kappa_{n} \}, \mathcal{C})
\geq
P(
\bigcap_{i=s_{1}}^{s}\{-\gamma_{n} \leq \epsilon_{i} \leq \gamma_{n} \}, \mathcal{C})
\rightarrow 1.
\end{align*}
Then
$T_{2}\rightarrow 1$.
Similarly, $T_{3} \rightarrow 1$.
This completes the proof.
\end{proof}
\end{proofLong}

\subsection{Supplement for Subsection \ref{subsec3.1}}
The following Theorem implies Theorem \ref{thm3.8} since it contains more details.
\begin{thm}[Consistency and Asymptotic Normality on $\tilde{\bd{\beta}}$]
\label{thm3.8.old}
Suppose Assumptions (A) and (B) hold.
If either $s_{2}=o(n/(\kappa_{n}\gamma_{n}))$
or Assumption (C) holds,
then $\tilde{\bd{\beta}} \ConvProb \bd{\beta}^{\star}$.
If $s_{2}=o(\sqrt{n}/(\kappa_{n}\gamma_{n}))$,
then
$
\sqrt{n}(\tilde{\bd{\beta}} - \bd{\beta}^{\star} )
\ConvDist
N(0, \sigma^{2}\bd{\bd{\Sigma}}_{X}^{-1}).
$
On the other hand, under  Assumption (C),
\begin{enumerate}
\item [(1)] if $D_{n}^{2}/n=o(1)$,
then $
\sqrt{n}(\tilde{\bd{\beta}}-\bd{\beta}^{\star})
\ConvDist N(0, \sigma^{2}\bd{\bd{\Sigma}}_{X}^{-1});
$ {\bf [main case]}

\item [(2)] if $D_{n}^{2}/n \sim c$,
then $
\sqrt{n}(\tilde{\bd{\beta}}-\bd{\beta}^{\star})
\ConvDist
N(0, (c +\sigma^{2})\bd{\bd{\Sigma}}_{X}^{-1})
$, for every constant $c\in\mathbb{R}^{+}$;

\item [(3)] if $D_{n}^{2}/n \rightarrow \infty$,
then $
r_{n}(\tilde{\bd{\beta}}-\bd{\beta}^{\star})
\ConvDist
N(0, \bd{\bd{\Sigma}}_{X}^{-1})
$
where $r_{n} \sim n/D_{n}\ll \sqrt{n}$.
\end{enumerate}
\end{thm}

\begin{proof}[Proof of Theorem \ref{thm3.8.old}]
Denote $I_{0}=\{s_{1}+1, s_{1}+2, \cdots, s=s_{1}+s_{2}, s+1, \cdots, n\}$.
Note that $s_{2}=o(\sqrt{n}/(\kappa_{n}\gamma_{n}))$ ensures that $\hat{\bd{\beta}}$ is consistent by Theorem
\ref{thm3.2}.
By Theorem \ref{thm3.7},
$P\{\hat{I}_{0}=I_{0}\}$ goes to 1.
Then,
\begin{equation*}
\tilde{\bd{\beta}}
=
R_{1} + R_{2}
+
T_{0}^{-1}(T_{1} + T_{2}),
\end{equation*}
where
$
R_{1}
=
(\bd{X}_{\hat I_{0}}^{T}\bd{X}_{\hat I_{0}})^{-1}\bd{X}_{\hat I_{0}}^{T}\bd{Y}_{\hat I_{0}}\{\hat{I}_{0}\not=I_{0}\}
$
and
$
R_{2}
=
-(\bd{X}_{I_{0}}^{T}\bd{X}_{I_{0}})^{-1}\bd{X}_{I_{0}}^{T}\bd{Y}_{I_{0}}\{\hat{I}_{0}\not=I_{0}\}
$
and $T_{i}$'s are defined in the proof of
Theorem \ref{thm3.2}.
The proof for the consistency is similar to that of Theorem
\ref{thm3.2}
and is omitted.
Next we show the asymptotic normality.
We have,
\begin{equation*}
r_{n}(\tilde{\bd{\beta}} - \bd{\beta}^{\star})
=
r_{n}R_{1} + r_{n}R_{2} + T_{0}^{-1}(V_{1} + V_{2}),
\end{equation*}
where $V_{i}$'s are defined
in the proof of Theorem
\ref{thm3.6}.
Since $P(\sqrt{n}R_{1}=0)\geq P\{\hat{I}_{0} = I_{0}\} \rightarrow 1$,
we have $\sqrt{n}R_{1} = o_{P}(1)$.
Similarly, $\sqrt{n}R_{2} = o_{P}(1)$.
From the analysis on $V_{i}$'s
in the proof of Theorem
\ref{thm3.6},
the asymptotic distributions follows by Slutsky's lemma.
\end{proof}

\begin{proofLong}
\begin{proof}
[Long Proof of Theorem
\ref{thm3.8}]
First, we provide a decomposition of $\tilde{\bd{\beta}}$.
Denote $I_{0}=\{s_{1}+1, s_{1}+2, \cdots, s=s_{1}+s_{2}, s+1, \cdots, n\}$.
By Theorems
\ref{thm3.2}
and
\ref{thm3.7},
$P\{\hat{I}_{0}=I_{0}\}\rightarrow 1$
if Assumption (A) holds or $s_{2}=o(n/(\kappa_{n}\gamma_{n}))$.
Then,
\begin{equation*}
\tilde{\bd{\beta}}
=
(\bd{X}_{\hat I_{0}}^{T}\bd{X}_{\hat I_{0}})^{-1}\bd{X}_{\hat I_{0}}^{T}\bd{Y}_{\hat I_{0}}
=
(\bd{X}_{I_{0}}^{T}\bd{X}_{I_{0}})^{-1}\bd{X}_{I_{0}}^{T}\bd{Y}_{I_{0}}
+
T_{1} + T_{2},
\end{equation*}
where
$
T_{1}
=
(\bd{X}_{\hat I_{0}}^{T}\bd{X}_{\hat I_{0}})^{-1}\bd{X}_{\hat I_{0}}^{T}\bd{Y}_{\hat I_{0}}\{\hat{I}_{0}\not=I_{0}\}
$
and
$
T_{2}
=
-(\bd{X}_{I_{0}}^{T}\bd{X}_{I_{0}})^{-1}\bd{X}_{I_{0}}^{T}\bd{Y}_{I_{0}}\{\hat{I}_{0}\not=I_{0}\}
$.
Note that
\begin{align*}
(\bd{X}_{I_{0}}^{T}\bd{X}_{I_{0}})^{-1}\bd{X}_{I_{0}}^{T}\bd{Y}_{I_{0}}
& =
\bd{\beta}^{\star}
+
T_{0}^{-1}T_{3}
+
T_{0}^{-1}T_{4},
\end{align*}
where
$T_{0} = \mathbb{S}_{s_{1}+1,n}/n$,
$
T_{3} =
\mathbb{S}_{s_{1}+1,n}^{\epsilon}/n$ and
$
T_{4} =
\mathbb{S}_{S_{12}^{\star}}^{\mu}/n$.
Thus,
\begin{equation*}
\tilde{\bd{\beta}} - \bd{\beta}^{\star}
=
T_{1} + T_{2} + T_{0}^{-1}T_{3} + T_{0}^{-1}T_{4}.
\end{equation*}

Next, we consider the consistency of $\tilde{\bd{\beta}}$.
We will show that
$T_{0}^{-1} \ConvProb \bd{\bd{\Sigma}}_{X}^{-1}$
and
$T_{i}=o_{P}(1)$ for $i=1,2, 3, 4$.
Then, $\tilde{\bd{\beta}} \ConvProb \bd{\beta}^{\star}$.

\OnItem{On $T_{0}$.}
See the proof of Theorem
\ref{thm3.2}.

\OnItem{On $T_{1}$ and $T_{2}$.}
Since $P(T_{1}=0)\geq P\{\hat{I}_{0} = I_{0}\} \rightarrow 1$,
we have $T_{1} = o_{P}(1)$.
Similarly, $T_{2} = o_{P}(1)$.

\OnItem{On $T_{3}$.}
See the proof of Theorem
\ref{thm3.2}.

\OnItem{On $T_{4}$.}
See the proof of Theorem
\ref{thm3.2}.
(Condition: Assumption (A) holds or $s_{2}=o(n/(\kappa_{n}\gamma_{n}))$.)

Thus, if Assumption (A) holds or $s_{2}=o(n/(\kappa_{n}\gamma_{n}))$,
then $\tilde{\bd{\beta}}$ is a consistent estimator of $\bd{\beta}^{\star}$.\\

Next we show the asymptotic normality.
Let $r_{n}$ be a sequence diverging to infinity.
We have
\begin{equation*}
r_{n}(\tilde{\bd{\beta}} - \bd{\beta}^{\star})
=
r_{n}T_{1} + r_{n}T_{2} + T_{0}^{-1}r_{n}T_{3} + T_{0}^{-1}r_{n}T_{4}.
\end{equation*}
First, we analyze $r_{n}T_{i}$ individually.

\OnItem{On $r_{n}T_{1}$ and $r_{n}T_{2}$.}
Since $P(r_{n}T_{1}=0)\geq P\{\hat{I}_{0} = I_{0}\} \rightarrow 1$,
we have $r_{n}T_{1} = o_{P}(1)$.
Similarly, $\sqrt{n}T_{2} = o_{P}(1)$.

\OnItem{On $r_{n}T_{3}$.}
See the proof of Theorems
\ref{thm3.4}
and \ref{thm3.6}.

\OnItem{On $r_{n}T_{4}$.}
See the proof of Theorems
\ref{thm3.4}
and \ref{thm3.6}.

Then we obtain asymptotic distributions of $\tilde{\bd{\beta}}$ by Slutsky's lemma.

\end{proof}
\end{proofLong}

\begin{lem}[Consistency on $\hat{\sigma}$]
\label{lem3.9}
Suppose Assumptions (A) and (B) hold and either $s_{2}=o(n/(\kappa_{n}\gamma_{n}))$ or Assumption (C) holds.
If $s_{2}=o(n/\gamma_{n}^{2})$,
then
$\hat{\sigma} \ConvProb \sigma$.
\end{lem}
\begin{proofName}
\subsection{Proof of Lemma \ref{lem3.9}}
\end{proofName}
\begin{proof}
[Proof of Lemma \ref{lem3.9}]
When Assumption (C) or $s_{2}=o(n/(\kappa_{n}\gamma_{n}))$ holds,
the penalized estimators $\hat{\bd{\beta}}$
and $\tilde{\bd{\beta}}$
are consistent estimators of $\bd{\beta}^{\star}$
by Theorems \ref{thm3.2}
and
\ref{thm3.8}.
Denote $\mathcal{\mathcal{C}}=\{\hat{I}_{0}=I_{0}\}$.
By Theorem \ref{thm3.7},
$\mathcal{\mathcal{C}}$ occurs wpg1.
Then,
$
\hat{\sigma}^{2}
=
T\mathcal{\mathcal{C}}
+ \hat{\sigma}^{2}\mathcal{\mathcal{C}}^{c}
$,
where
$
T= a_{n}\normE{ \bd{Y}_{I_{0}} - \bd{X}_{I_{0}}^{T} \tilde{\bd{\beta}} }^{2}
$
and
$a_{n}=1/(n-s_{1})$.
It is sufficient to show
$
T \ConvProb \sigma^{2}
$.
We have $T=\sum_{i=1}^{6}T_{i}$, where
$T_{1} =  a_{n}\sum_{i=s_{1}+1}^{n}[\bd{X}_{i}^{T}(\bd{\beta}^{\star} - \tilde{\bd{\beta}})]^{2}$,
$T_{2} =  a_{n}\sum_{i=s_{1}+1}^{n} \epsilon_{i}^{2}$,
$T_{3} =  2a_{n}\sum_{i=s_{1}+1}^{n} \bd{X}_{i}^{T}(\bd{\beta}^{\star} - \tilde{\bd{\beta}})\epsilon_{i}$,
$T_{4} =  a_{n}\sum_{i=s_{1}+1}^{s} \mu_{i}^{\star 2}$,
$T_{5} =  2a_{n}\sum_{i=s_{1}+1}^{s} \mu_{i}\bd{X}_{i}^{T}(\bd{\beta}^{\star} - \tilde{\bd{\beta}})$
and
$T_{6} =  2a_{n}\sum_{i=s_{1}+1}^{s} \mu_{i}^{\star}\epsilon_{i}$.
It is straightforward to show that $T_{2}\ConvProb \sigma^{2}$ and each other $T_{i}\ConvProb 0$
under the condition $s_{2}=o(n/\gamma_{n}^{2})$
and by noting that $\tilde{\bd{\beta}}\ConvProb \bd{\beta}^{\star}$.
Then $\hat\sigma$ is a consistent estimator of $\sigma$.
\end{proof}

\begin{proofLong}
\begin{proof}
[Long Proof of Lemma \ref{lem3.9}]
Under the condition that Assumption (A) holds or $s_{2}=o(n/(\kappa_{n}\gamma_{n}))$,
the penalized estimators $\hat{\bd{\beta}}$
and $\tilde{\bd{\beta}}$
are consistent estimators of $\bd{\beta}^{\star}$
by Theorems \ref{thm3.2}
and
\ref{thm3.8}.
Denote $\mathcal{\mathcal{B}}=\{\hat{I}_{0}=I_{0}\}$.
Then
$\mathcal{\mathcal{B}}$ occurs wpg1
by Theorem \ref{thm3.7}.

\lineProof

We have
\begin{align*}
\hat{\sigma}^{2}
& =
\frac{1}{\#(\hat I_{0})} \normE{ \bd{Y}_{\hat I_{0}} - \bd{X}_{\hat I_{0}}^{T} \tilde{\bd{\beta}} }^{2}\mathcal{\mathcal{B}}
+ \frac{1}{\#(\hat I_{0})} \normE{ \bd{Y}_{\hat I_{0}} - \bd{X}_{\hat I_{0}}^{T} \tilde{\bd{\beta}} }^{2}\mathcal{\mathcal{B}}^{c} \\
& =
\frac{1}{\#(I_{0})} \normE{ \bd{Y}_{I_{0}} - \bd{X}_{I_{0}}^{T} \tilde{\bd{\beta}} }^{2}\mathcal{\mathcal{B}}
+ \frac{1}{\#(\hat I_{0})} \normE{ \bd{Y}_{\hat I_{0}} - \bd{X}_{\hat I_{0}}^{T} \tilde{\bd{\beta}} }^{2}\mathcal{\mathcal{B}}^{c}. \\
\end{align*}
Then, it is sufficient to show that
\begin{equation*}
T:=\frac{1}{\#(I_{0})} \normE{ \bd{Y}_{I_{0}} - \bd{X}_{I_{0}}^{T} \tilde{\bd{\beta}} }^{2}
=\frac{1}{n-s_{1}} \normE{ \bd{Y}_{I_{0}} - \bd{X}_{I_{0}}^{T} \tilde{\bd{\beta}} }^{2}
\ConvProb \sigma^{2}.
\end{equation*}

\lineProof

We have $T=\sum_{i=1}^{6}T_{i}$, where
\begin{align*}
T_{1} & =  \frac{1}{n-s_{1}}\sum_{i=s_{1}+1}^{n}[\bd{X}_{i}^{T}(\bd{\beta}^{\star} - \tilde{\bd{\beta}})]^{2}, \\
T_{2} & =  \frac{1}{n-s_{1}}\sum_{i=s_{1}+1}^{n} \epsilon_{i}^{2}, \\
T_{3} & =  2\frac{1}{n-s_{1}}\sum_{i=s_{1}+1}^{n} \bd{X}_{i}^{T}(\bd{\beta}^{\star} - \tilde{\bd{\beta}})\epsilon_{i}, \\
T_{4} & =  \frac{1}{n-s_{1}}\sum_{i=s_{1}+1}^{s} \mu_{i}^{\star 2}, \\
T_{5} & =  2\frac{1}{n-s_{1}}\sum_{i=s_{1}+1}^{s} \mu_{i}\bd{X}_{i}^{T}(\bd{\beta}^{\star} - \tilde{\bd{\beta}}), \\
T_{6} & =  2\frac{1}{n-s_{1}}\sum_{i=s_{1}+1}^{s} \mu_{i}^{\star}\epsilon_{i}.
\end{align*}
\lineProof

\OnItem{On $T_{1}$.}
We have
\begin{align*}
|T_{1}|
& =
|\frac{1}{n-s_{1}}\sum_{i=s_{1}+1}^{n}[\bd{X}_{i}^{T}(\bd{\beta}^{\star} - \tilde{\bd{\beta}})]^{2}| \\
& =
\frac{1}{n-s_{1}}\sum_{i=s_{1}+1}^{n}[\bd{X}_{i}^{T}(\bd{\beta}^{\star} - \tilde{\bd{\beta}})]^{2} \\
& \leq
\frac{1}{n-s_{1}}\sum_{i=s_{1}+1}^{n}\normE{\bd{X}_{i}^{T}}^{2}\normE{\bd{\beta}^{\star} - \tilde{\bd{\beta}}}^{2} \\
(wpg1) & \leq
2\E\normE{\bd{X}_{0}^{T}}^{2}
\normE{\bd{\beta}^{\star} - \tilde{\bd{\beta}}}^{2}
\ConvProb 0.
\end{align*}
Thus, $T_{1}=o_{P}(1)$ under the conditions:
\begin{itemize}
  \item $\normE{\bd{\beta}^{\star} - \tilde{\bd{\beta}}}\ConvProb 0$.
\end{itemize}

\lineProof
\OnItem{On $T_{2}$.}
We have
\begin{align*}
T_{2}
& =
\frac{1}{n-s_{1}}\sum_{i=s_{1}+1}^{n} \epsilon_{i}^{2}
\ConvProb \sigma^{2}.
\end{align*}
Thus, $T_{2} \ConvProb \sigma^{2}$.

\lineProof
\OnItem{On $T_{3}$.}
We have
\begin{align*}
|T_{3}|
& =
|2\frac{1}{n-s_{1}}\sum_{i=s_{1}+1}^{n} \bd{X}_{i}^{T}(\bd{\beta}^{\star} - \tilde{\bd{\beta}})\epsilon_{i}| \\
& \leq
2\frac{1}{n-s_{1}}\sum_{i=s_{1}+1}^{n} |\bd{X}_{i}^{T}\epsilon_{i}(\bd{\beta}^{\star} - \tilde{\bd{\beta}})| \\
& \leq
2\frac{1}{n-s_{1}}\sum_{i=s_{1}+1}^{n} \normE{\bd{X}_{i}^{T}\epsilon_{i}}\normE{\bd{\beta}^{\star} - \tilde{\bd{\beta}}} \\
(wpg1) & \leq
4\E\normE{\bd{X}_{0}^{T}\epsilon_{0}}\normE{\bd{\beta}^{\star} - \tilde{\bd{\beta}}}
\ConvProb 0.
\end{align*}
Thus, $T_{3}=o_{P}(1)$ under the conditions:
\begin{itemize}
  \item $\normE{\bd{\beta}^{\star} - \tilde{\bd{\beta}}}\ConvProb 0$.
\end{itemize}

\lineProof
\OnItem{On $T_{4}$.}
We have
\begin{align*}
|T_{4}|
& =
|\frac{1}{n-s_{1}}\sum_{i=s_{1}+1}^{s} \mu_{i}^{\star 2}| \\
& =
\frac{1}{n-s_{1}}\sum_{i=s_{1}+1}^{s} \mu_{i}^{\star 2} \\
& \leq
\frac{1}{n-s_{1}}\sum_{i=s_{1}+1}^{s} \gamma_{n}^{2} \\
& =
\frac{1}{n-s_{1}} s_{2} \gamma_{n}^{2}
\rightarrow 0.
\end{align*}
Thus, $T_{4}=o(1)$
under the condition:
\begin{itemize}
  \item $s_{2}=o(n/\gamma_{n}^{2})$.
\end{itemize}

\lineProof
\OnItem{On $T_{5}$.}
We have
\begin{align*}
|T_{5}|
& =
|2\frac{1}{n-s_{1}}\sum_{i=s_{1}+1}^{s} \mu_{i}\bd{X}_{i}^{T}(\bd{\beta}^{\star} - \tilde{\bd{\beta}})| \\
& =
2\frac{1}{n-s_{1}}\sum_{i=s_{1}+1}^{s} |\mu_{i}\bd{X}_{i}^{T}(\bd{\beta}^{\star} - \tilde{\bd{\beta}})| \\
& \leq
2\frac{1}{n-s_{1}}\sum_{i=s_{1}+1}^{s} |\mu_{i}\|\bd{X}_{i}^{T}(\bd{\beta}^{\star} - \tilde{\bd{\beta}})| \\
& \leq
2\frac{1}{n-s_{1}}\sum_{i=s_{1}+1}^{s} \gamma_{n}\normE{\bd{X}_{i}}\normE{\bd{\beta}^{\star} - \tilde{\bd{\beta}}} \\
(wpg1)& \leq
2\frac{1}{n-s_{1}}\gamma_{n}s_{2}2\E\normE{\bd{X}_{0}}\normE{\bd{\beta}^{\star} - \tilde{\bd{\beta}}}
\ConvProb
0.
\end{align*}
Thus, $T_{5}=o_{P}(1)$ under the conditions.
\begin{itemize}
  \item  $s_{2}=o(n/\gamma_{n})$.
\end{itemize}

\lineProof
\OnItem{On $T_{6}$.}
We have
\begin{align*}
|T_{6}|
& =
|2\frac{1}{n-s_{1}}\sum_{i=s_{1}+1}^{s} \mu_{i}^{\star}\epsilon_{i}| \\
& \leq
2\frac{1}{n-s_{1}}\sum_{i=s_{1}+1}^{s} |\mu_{i}^{\star}\epsilon_{i}| \\
& \leq
2\frac{1}{n-s_{1}}\gamma_{n}\sum_{i=s_{1}+1}^{s} |\epsilon_{i}| \\
(wpg1)& \leq
2\frac{1}{n-s_{1}}\gamma_{n}s_{2}2\E|\epsilon_{0}|
\rightarrow 0.
\end{align*}
Thus, we have
$T_{6}=o_{P}(1)$ under the conditions
\begin{itemize}
  \item $s_{2}=o(n/\gamma_{n})$.
\end{itemize}
\end{proof}
\end{proofLong}

\subsection{Supplement for Subsection \ref{subsec:Regularization Parameters}}
\label{PC.Paper.Supplement.Typical.Covariates.Errors}

In this supplement,
we consider a special case with exponentially tailed covariates and errors.
For convenience,
we  first introduce the definition
of Orlicz norm and related inequalities.
For a strictly increasing and convex function $\psi$ with $\psi(0)=0$,
the Orlicz norm of a random variable $Z$ with respect to $\psi$ is defined as
\begin{equation*}
\|Z\|_{\psi}
=
\inf
\{C>0: \E \psi(|Z|/C) \leq 1\}.
\end{equation*}
Then, for each $x>0$,
\begin{equation}
\label{eq4.1}
P(|Z|>x) \leq {1}/{\psi(x/\|Z\|_{\psi})}.
\end{equation}
(See Page 96 of \cite{Vaart1996}).
Next, we introduce a lemma on Orlicz norm with $\psi_{1}$.
Suppose
$\{Z_{i}\}_{i=1}^{n}$ is a sequence of random variables
and
$\{\bd{Z}_{i}\}_{i=1}^{n}$ is a sequence of $d$-dimensional random vectors
with $\bd{Z}_{i}=(Z_{i1}, Z_{i2}, \cdots, Z_{id})^{T}$.
From Lemma 8.3 on Page 131 of \cite{Kosorok2008},
we have the following extension.
\begin{lem}
\label{PC.Paper.Lemma.Max.Variable.Vector.Orlicz.psi1}
If for each $1\leq i\leq n$ and $1\leq j \leq d$,
\begin{align*}
P(|Z_{i}|> x ) \leq c\exp\{-\frac{1}{2}\cdot\frac{x^{2}}{ax+b}\} \text{ and }
P(|Z_{ij}|> x ) \leq c\exp\{-\frac{1}{2}\cdot\frac{x^{2}}{ax+b}\},
\end{align*}
with $a, b \geq 0$ and $c>0$,
then
\begin{align*}
& \|\max_{1\leq i\leq n}|Z_{i}\||_{\psi_{1}}
\leq
K
\{a(1+c)\log(1+n) + \sqrt{b(1+c)}\sqrt{\log(1+n)}\}, \\
& \|\max_{1\leq i\leq n}\normE{\bd{Z}_{i}}\|_{\psi_{1}}
\leq
K
\{a\sqrt{d}(1+cd)\log(1+n) + \sqrt{bd(1+cd)}\sqrt{\log(1+n)}\}.
\end{align*}
where $K$ is a universal constant which is independent of $a, b, c$, $\{Z_{i}\}$ and $\{\bd{Z}_{i}\}$.
\end{lem}

\begin{proofName}
\subsection{Proof of Lemma
\ref{PC.Paper.Lemma.Max.Variable.Vector.Orlicz.psi1}}
\end{proofName}
\begin{proof}
[Proof of Lemma
\ref{PC.Paper.Lemma.Max.Variable.Vector.Orlicz.psi1}]
The proof for random variables $\{Z_{i}\}$
is the same to the proof of
Lemma 8.3 on Page 131 of \cite{Kosorok2008}.
For random vectors $\{\bd{Z}_{i}\}$,
\begin{align*}
P(\normE{\bd{Z}_{i}} \geq x)
\leq
P( \max_{1\leq j\leq d} |Z_{ij}| > x/\sqrt{d})
\leq
\sum_{j=1}^{d}P( |Z_{ij}| > x/\sqrt{d})
\leq
c'
\exp\{-\frac{1}{2}\frac{x^{2}}{a'x+b'}\},
\end{align*}
where
$a' = a\sqrt{d}$,
$b' = bd$
and $c' = cd$.
Then, by the result on random variables,
the desired result on random vectors follows.
\end{proof}

Now,
suppose, for every $x>0$,
\begin{equation}
\label{PC.Paper.Covariate.Error.General.Heavy.Tails}
P(|\epsilon_{i}|>x)
\leq
c_{1}\exp\{-\frac{1}{2}\cdot\frac{x^{2}}{a_{1}x + b_{1}}\}
\text{ and }
P(|X_{ij}|>x)
\leq
c_{2}\exp\{-\frac{1}{2}\cdot\frac{x^{2}}{a_{2}x + b_{2}}\},
\end{equation}
with $a_{i}, b_{i} \geq 0$ and $c_{i}>0$  for $i=1,2$.
By Lemma \ref{PC.Paper.Lemma.Max.Variable.Vector.Orlicz.psi1}, it follows
\begin{align*}
& \|\max_{1\leq i\leq n}|\epsilon_{i}\||_{\psi_{1}}
\leq
K
\{a_{1}(1+c_{2})\log(1+n) + \sqrt{b_{1}(1+c_{1})}\sqrt{\log(1+n)}\}, \\
& \|\max_{1\leq i\leq n}\normE{\bd{X}_{i}}\|_{\psi_{1}}
\leq
K
\{a_{2}\sqrt{d}(1+c_{2}d)\log(1+n) + \sqrt{b_{2}d(1+c_{2}d)}\sqrt{\log(1+n)}\}.
\end{align*}

Thus, from the inequality (\ref{eq4.1}),
if $a_{1} > 0$,
let $\gamma_{n}\gg\log(n)$;
otherwise,
let $\gamma_{n}\gg \sqrt{\log(n)}$.
Similarly,
if $a_{2} > 0$,
let $\kappa_{n} \gg \log(n)$;
otherwise,
let $\kappa_{n} \gg \sqrt{\log(n)}$.
Then, such $\gamma_{n}$ and $\kappa_{n}$
satisfy the condition (\ref{eq2.2}).
Suppose
both $a_{1}$ and $a_{2}$ are positive,
which means both $\epsilon_{i}$ and $X_{ij}$'s
have exponential tails.
As before, set $\kappa_{n}=\gamma_{n}=\log(n)\tau_{n}$.
For this case,
the regularization parameter specification (\ref{eq3.1}) becomes
$
   \log(n)\tau_{n} \ll \lambda \ll \min\{\mu^{\star},\sqrt{n}\}.
$

At the end of this supplement,
we simply list explicit expressions of $\kappa_{n}$
under different assumptions on the covariates
for the case with a diverging number of covariates,
which are the extension of the results in
Section \ref{subsec:Regularization Parameters}.
The magnitude of $\kappa_{n}$ becomes larger
than that for the case with $d$ fixed
while $\gamma_{n}$ keeps the same.
Specifically,
if $\bd{X}_{0}$ is bounded with $C_{X}>0$,
then $\kappa_{n}=\sqrt{d}C_{X}$.
If $\bd{X}_{0}$ follows a Gaussian distribution $N(0,\bd{\Sigma}_{X})$,
then $\kappa_{n}=\sqrt{2d\sigma_{X}^{2}[(3/2)\log(d)+\log(n)]}$.
If the Orlicz norm $\norm{X_{0j}}_{\psi}$ exists for $1\leq j\leq d$
and their average $(1/d)\sum_{j=1}^{d}\norm{X_{0j}}_{\psi}$ is bounded,
then $\kappa_{n}\gg d\psi^{-1}(n)$;
for instance, if $\psi=\psi_{p}$ with $p\geq 1$,
then $\kappa_{n}\gg d(\log(n))^{1/p}$.
Finally,
if the data $\{\bd{X}_{i}\}$ satisfies the right inequality of
(\ref{PC.Paper.Covariate.Error.General.Heavy.Tails})
with $a_{2}>0$,
that is, each component of $\bd{X}_{i}$ is sub-exponentially tailed,
then $\kappa_{n}\gg d^{3/2}\log(n)$.
It is worthwhile to note that
these expressions of $\kappa_{n}$ depend on a factor involving the diverging number of covariates $d$,
which will influence the specification of the regularization parameter and the sufficient conditions
of all the theoretical results in Section
\ref{sec:Diverging number of structural parameters}.

\section{Supplement for Section \ref{sec:Diverging number of structural parameters}}
\label{PC.Paper.Supplement.Extension.d.infty}

In this supplement,
we provide Proposition \ref{prop2.2} and its proof,
the proofs of the lemmas
in the appendix
and some additional results.

We first
extend Proposition \ref{prop2.2}
to the case with $d\rightarrow\infty$ and $d\ll n$.
Before that,
we list two simple lemmas for a diverging $d$.
Suppose $\{\bd{\xi}_{i}\}$ is a sequence of i.i.d. copies of $\bd{\xi}_{0}$,
a $d$-dimensional random vector
with mean zero.
Denote $\bar{\sigma}_{\xi}^{2}=(1/d)\sum_{j=1}^{d}\Var[\xi_{0j}]$.
\begin{lem}
\label{PC.Paper.Lemma.Weak.LLN.d.infty}
Suppose $\bar{\sigma}_{\xi}^{2}$ is bounded.
If $d/n=o(1)$,
then
\begin{equation*}
\normE{\frac{1}{n}\sum_{i=1}^{n}\bd{\xi}_{i}}
\ConvProb 0.
\end{equation*}
\end{lem}

\begin{proofLong}
\begin{proof}
For every $\delta>0$,
\begin{align*}
P(\normE{\frac{1}{n}\sum_{i=1}^{n}\bd{\xi}_{i}} > \delta)
& \leq
\frac{1}{\delta^{2}}
P\normE{\frac{1}{n}\sum_{i=1}^{n}\bd{\xi}_{i}}^{2} \\
& =
\frac{1}{\delta^{2}}
\frac{1}{n^{2}}\sum_{i=1}^{n}P\normE{\bd{\xi}_{i}}^{2} \\
& =
\frac{1}{\delta^{2}}
\frac{1}{n}P\normE{\bd{\xi}_{0}}^{2} \\
& =
\frac{1}{\delta^{2}}
\frac{d}{n}
\frac{1}{d}
\sum_{j=1}^{d}
P\bd{\xi}_{0j}^{2} \\
& =
\frac{1}{\delta^{2}}
\frac{d}{n}
\bar{\sigma}_{\xi}^{2}\\
(d/n=o(1))&\rightarrow 0.
\end{align*}
\end{proof}
\end{proofLong}

\begin{proofName}
\subsection{Lemma \ref{PC.Paper.Lemma.Weak.LLN.d.infty.2}}
\end{proofName}
\begin{lem}
\label{PC.Paper.Lemma.Weak.LLN.d.infty.2}
Suppose $\bar{\sigma}_{\xi}^{2}$ is bounded.
If $d/n=o(1)$,
then
\begin{equation*}
\frac{1}{n}
\sum_{i=1}^{n}
\normE{\bd{\xi}_{i}}
-
P\normE{\bd{\xi}_{0}}
\ConvProb
0.
\end{equation*}
\end{lem}

\begin{proofLong}
\begin{proof}
[Long Proof of Lemma \ref{PC.Paper.Lemma.Weak.LLN.d.infty.2}]
For every $\delta>0$,
\begin{align*}
P(
\abs{
\frac{1}{n}
\sum_{i=1}^{n}
\normE{\bd{\xi}_{i}}
-
P\normE{\bd{\xi}_{0}}}
> \delta
)
& =
P(
\abs{
\frac{1}{n}
\sum_{i=1}^{n}
(\normE{\bd{\xi}_{i}}
-
P\normE{\bd{\xi}_{0}})}
> \delta )\\
& \leq
P(
\abs{
\frac{1}{n}
\sum_{i=1}^{n}
(\normE{\bd{\xi}_{i}}
-
P\normE{\bd{\xi}_{0}})}
> \delta)\\
& \leq
\frac{1}{\delta^{2}}
P
\abs{
\frac{1}{n}
\sum_{i=1}^{n}
(\normE{\bd{\xi}_{i}}
-
P\normE{\bd{\xi}_{0}})}^{2}\\
& \leq
\frac{1}{\delta^{2}}
\frac{1}{n}
P
(\normE{\bd{\xi}_{0}}
-
P\normE{\bd{\xi}_{0}})^{2}\\
& =
\frac{1}{\delta^{2}}
\frac{1}{n}
(P
\normE{\bd{\xi}_{0}}^{2}
-
(P\normE{\bd{\xi}_{0}})^{2})\\
& \leq
\frac{1}{\delta^{2}}
\frac{1}{n}
P
\normE{\bd{\xi}_{0}}^{2}\\
& =
\frac{1}{\delta^{2}}
\frac{d}{n}
\bar{\sigma}_{\xi}^{2}
\rightarrow 0.
\end{align*}
\end{proof}
\end{proofLong}

\begin{proofName}
\subsection{Lemma \ref{PC.Paper.Lemma.NormF.NormE.Equal}}
\begin{lem}
\label{PC.Paper.Lemma.NormF.NormE.Equal}
Suppose $\bd{a}$ is a vector.
Then
\begin{align*}
\normF{\bd{a}\bd{a}^{T}}
=
\normE{\bd{a}}^{2}.
\end{align*}
\end{lem}
\end{proofName}
\begin{proofLong}
\begin{proof}
We have
\begin{align*}
\normF{\bd{a}\bd{a}^{T}}
& =
(
\sum_{i=1}^{d}
\sum_{j=1}^{d}
a_{i}^{2}a_{j}^{2}
)^{1/2}  \\
& =
(
\sum_{i=1}^{d}
a_{i}^{2}
)^{1/2}
(\sum_{j=1}^{d}
a_{j}^{2}
)^{1/2}  \\
& =
\sum_{i=1}^{d}
a_{i}^{2}\\
& =
\normE{\bd{a}}^{2}.
\end{align*}
\end{proof}
\end{proofLong}

\begin{proofName}
\subsection{Proposition \ref{prop2.2:infty}}
\end{proofName}

Suppose the specification of the regularization parameter is given by
\begin{equation}
\label{eq5.1:algorithm.convergence}
d\kappa_{n} \ll  \lambda,~  \alpha\gamma_{n} \leq \lambda, \text{ and } \lambda \ll  \mu^{\star},
\end{equation}
where $\alpha$ is a constant greater than 2.

\begin{prop}
\label{prop2.2:infty}
Suppose assumptions (D) and (G) hold
and the regularization parameter satisfies (\ref{eq5.1:algorithm.convergence}).
Suppose there exist constants $C_{1}$ and $C_{2}$ such that
$\normE{\bd{\beta}^{\star}} < C_{1}\sqrt{d}$ and
$\normE{\bd{\beta}^{(0)}} < C_{2}\sqrt{d}$ wpg1.
If the regularization parameter satisfies (\ref{eq3.1}),
$s_{1}\lambda\kappa_{n}/(n\sqrt{d}) =o(1)$
and $s_{2}\kappa_{n}\gamma_{n}/(n\sqrt{d}) = o(1)$,
then,
for every $K\geq 1$,
with at least probability $p_{n,K}$ which increases to one as $n\rightarrow\infty$,
$\normE{\bd{\beta}^{(K+1)} - \bd{\beta}^{(K)}}\leq O((\sqrt{d}s_{1}\kappa_{n}^{2}/n)^{K}d )$
and $\normE{\bd{\beta}^{(k)}}\leq (2C_{1}+C_{2})d$ for all $k\leq K$.
Specifically,
wpg1,
the iterative algorithm stops at the second iteration.
\end{prop}

\begin{proof}
[Proof of Proposition \ref{prop2.2:infty}]
Reuse the notations in the proof of
Lemma \ref{prop2.2}.
First, we show that, wpg1, $\normE{\bd{\beta}^{(1)}}\leq (2C_{1}+C_{2})d$.
For each $k\geq 1$,
\begin{align*}
\mathbb{S}\bd{\beta}^{(k)}
& =
\mathbb{S}_{S_{11}}^{\mu}
+
\mathbb{S}_{S_{12}}^{\mu}
+
\mathbb{S}_{S_{1}}\bd{\beta}^{\star}
+
\mathbb{S}_{S_{1}}^{\epsilon}
+
\mathbb{S}_{S_{2}\cup S_{3}}\bd{\beta}^{(k-1)}
+
\lambda(\mathcal{S}_{S_{2}}-\mathcal{S}_{S_{3}}),
\end{align*}
Since the regularization parameter satisfies (\ref{eq5.1:algorithm.convergence}),
it is easy to check that
the conclusion of
Lemma \ref{PC.Paper.d.infty.index.sets.consistency} continues to hold,
which implies
$P(\mathcal{A}_{0})\rightarrow 1$.

Thus, wpg1,
$$
\bd{\beta}^{(1)}
=
T_{0}^{-1}T_{1}
+
T_{0}^{-1}T_{2}
+
T_{0}^{-1}T_{3}
+
T_{0}^{-1}T_{4}(\bd{\beta}^{(0)})
+
T_{0}^{-1}T_{5}.
$$
We will show that, wpg1,
\begin{align*}
& \normE{T_{0}^{-1}T_{1}} \leq (C_{2}/4)d, \\
& \normE{T_{0}^{-1}T_{2}} \leq 2C_{1}d, \\
& \normE{T_{0}^{-1}T_{3}} \leq (C_{2}/4)d, \\
& \normE{T_{0}^{-1}T_{4}(\bd{\beta}^{(0)})} \leq (C_{2}/4)d, \\
& \normE{T_{0}^{-1}T_{5}} \leq (C_{2}/4)d.
\end{align*}
Thus, wpg1,
\begin{equation*}
\normE{\bd{\beta}^{(1)}}
\leq
\sum_{i=1}^{5}
\normE{T_{0}^{-1}T_{i}}
\leq
(2C_{1}+C_{2})d.
\end{equation*}

\OnItem{On $T_{0}^{-1}T_{1}$.}
Under Assumption (D),
for $s_{2}\kappa_{n}\gamma_{n}/(n\sqrt{d})=o(1)$, wpg1,
\begin{align*}
\normE{T_{0}^{-1}T_{1}}
& \leq
\normF{
(\frac{1}{n}\mathbb{S})^{-1}}
\normE{
\frac{1}{n}\mathbb{S}_{S_{12}^{\star}}^{\mu}
}
\leq
2\normFd{
\bd{\Sigma}_{X}^{-1}}
\frac{s_{2}}{n\sqrt{d}}
\kappa_{n}\gamma_{n}d
\rightarrow 0.
\end{align*}
Thus, wpg1,
$\normE{T_{0}^{-1}T_{1}}
\leq C_{2}d/4$.

\OnItem{On $T_{0}^{-1}T_{2}$.}
Wpg1,
\begin{align*}
\normE{T_{0}^{-1}T_{2}}
& \leq
\normF{
(\frac{1}{n}\mathbb{S})^{-1}
\frac{1}{n}\mathbb{S}_{s_{1}+1,n}}
\normE{
\bd{\beta}^{\star}
} \\
& \leq
\normF{
\bd{I}_{d}}
C_{1}\sqrt{d}
+
\normF{
(\frac{1}{n}\mathbb{S})^{-1}
\frac{1}{n}\mathbb{S}_{1,s_{1}}}
C_{1}\sqrt{d} \\
& \leq
C_{1}d
+
\normF{
(\frac{1}{n}\mathbb{S})^{-1}}
\normF{
\frac{1}{n}\mathbb{S}_{1,s_{1}}}
C_{1}\sqrt{d},
\end{align*}
and
\begin{align*}
\normF{
\frac{1}{n}\mathbb{S}_{1,s_{1}}}
& =
\frac{1}{n}
\sum_{i=1}^{s_{1}}
\normE{
\bd{X}_{i}}^{2}
\leq
\frac{s_{1}}{n}
\kappa_{n}^{2}.
\end{align*}
Thus, Under Assumption (D),
for $s_{1}\kappa_{n}^{2}/n=o(1)$, wpg1,
\begin{align*}
\normE{T_{0}^{-1}T_{2}}
& \leq
C_{1}d
+
2
\normFd{\bd{\Sigma}_{X}^{-1}}
\frac{\sqrt{d}s_{1}}{n}
\kappa_{n}^{2}
C_{1}\sqrt{d}
\leq
2
C_{1}d.
\end{align*}

\OnItem{On $T_{0}^{-1}T_{3}$.}
Under assumptions (D)  and (G),
for $\log(d)/n=o(1)$, wpg1,
\begin{align*}
\normE{
T_{0}^{-1}T_{3}
}
& =
\sqrt{d}\frac{1}{\sqrt{n}}\sqrt{d\log(d)}
\normFd{
(\frac{1}{n}\mathbb{S})^{-1}}
(d\log(d))^{-1/2}
\normE{\frac{1}{\sqrt{n}}\mathbb{S}_{s_{1}+1,n}^{\epsilon}
} \\
& \leq
\frac{d\sqrt{\log(d)}}{\sqrt{n}}
2\normFd{
\bd{\Sigma}_{X}^{-1}}
O_{P}(1)
\ConvProb 0.
\end{align*}
Thus, wpg1,
$\normE{T_{0}^{-1}T_{3}} \leq C_{2}d/4$.

\OnItem{On $T_{0}^{-1}T_{4}(\bd{\beta}^{(0)})$.}
Under Assumption (D),
for $s_{1}
\kappa_{n}^{2}/n$, wpg1,
\begin{align*}
\normE{T_{0}^{-1}T_{4}(\bd{\beta}^{(0)})}
& \leq
\sqrt{d}
\normFd{
(\frac{1}{n}\mathbb{S})^{-1}}
\normF{
\frac{1}{n}\mathbb{S}_{1,s_{1}}}
\normE{\bd{\beta}^{(0)}
} \\
& \leq
\sqrt{d}
2\normFd{
\bd{\Sigma}_{X}^{-1}}
\frac{s_{1}}{n}
\kappa_{n}^{2}
C_{2}\sqrt{d}
\ConvProb 0.
\end{align*}
Thus, wpg1,
$\normE{T_{0}^{-1}T_{4}(\bd{\beta}^{(0)})} \leq C_{2}d/4$.

\OnItem{On $T_{0}^{-1}T_{5}$.}
Under Assumption (D),
for $s_{1}\kappa_{n}\lambda/(n\sqrt{d}) = o(1)$, wpg1,
\begin{align*}
\normE{
T_{0}^{-1}T_{5}
}
& \leq
\sqrt{d}
\normFd{
(\frac{1}{n}\mathbb{S})^{-1}}
\frac{\lambda}{n}
(\normE{
\mathcal{S}_{S_{21}^{\star}}}
+
\normE{
\mathcal{S}_{S_{31}^{\star}}
}) \\
& \leq
\sqrt{d}
2\normFd{
\bd{\Sigma}_{X}^{-1}}
\frac{\lambda}{n}
s_{1}\kappa_{n}
\leq C_{2}d/4.
\end{align*}

\lineProof

Next, consider
$\normE{\bd{\beta}_{2} - \bd{\beta}_{1}}$.
Since $\bd{\beta}^{(1)}\leq (2C_{1}+C_{2})d$ wpg1,
the conclusion of
Lemma \ref{PC.Paper.d.infty.index.sets.consistency} holds,
which implies
$\mathcal{A}_{1}$ occurs wpg1.

Then,
$$
\bd{\beta}^{(2)}
=
T_{0}^{-1}T_{1}
+
T_{0}^{-1}T_{2}
+
T_{0}^{-1}T_{3}
+
T_{0}^{-1}T_{4}(\bd{\beta}^{(1)})
+
T_{0}^{-1}T_{5},
$$
where
\begin{align*}
T_{4}(\bd{\beta}^{(1)}) &= \frac{1}{n}\mathbb{S}_{1,s_{1}}\bd{\beta}^{(1)}.
\end{align*}

Thus, wpg1,
$$\bd{\beta}^{(2)} - \bd{\beta}^{(1)}
=
\mathbb{S}^{-1}\mathbb{S}_{1,s_{1}}
(\bd{\beta}^{(1)}-\bd{\beta}^{(0)}).$$
Thus, for $d^{3/2}s_{1} \kappa_{n}^{2}/n = o(1)$, wpg1,
\begin{align*}
\normE{\bd{\beta}^{(2)}  - \bd{\beta}^{(1)}}
& \leq
\sqrt{d}
\normFd{\frac{1}{n}\mathbb{S}^{-1}}
\normF{\frac{1}{n}\mathbb{S}_{1,s_{1}}}
\normE{\bd{\beta}^{(1)} -\bd{\beta}^{(0)} } \\
& \leq
2
\normFd{\bd{\Sigma}_{X}^{-1}}
\sqrt{d}
\frac{s_{1}}{n}
\kappa_{n}^{2}
(2C_{1}+C_{2})d
\lesssim
d^{3/2} s_{1} \kappa_{n}^{2}/n
\rightarrow 0.
\end{align*}
Thus, wpg1,
$\bd{\beta}^{(2)} = \bd{\beta}^{(1)}$,
which means that, wpg1, the iteration algorithm stops at the second iteration.

\lineProof

For any $K\geq 1$,
repeating the above arguments,
with at least probability
$p_{n,K} = P(\bigcap_{k=0}^{K}\mathcal{A}_{k})$,
which increases to one,
we have
$\bd{\beta}^{(k)}\leq (2C_{1}+C_{2})d$ for $k\leq K$
and
\begin{align*}
\normE{\bd{\beta}^{(K+1)}  - \bd{\beta}^{(K)}}
& \leq
(2
\normFd{\bd{\Sigma}_{X}^{-1}}
\sqrt{d}
\frac{s_{1}}{n}
\kappa_{n}^{2})^{K}
(2C_{1}+C_{2})d
\lesssim
(\sqrt{d}s_{1}\kappa_{n}^{2}/n)^{K}
d
\rightarrow 0.
\end{align*}
This completes the proof.
\end{proof}

\begin{proofLong}
\begin{proof}
[Long Proof of Lemma \ref{prop2.2:infty}]
First, we show that, wpg1, $\normE{\bd{\beta}^{(1)}}\lesssim d$.
For each $k\geq 1$,
we have
\begin{align*}
\mathbb{S}\bd{\beta}^{(k)}
& =
\sum_{i\in S_{1}}\bd{X}_{i}Y_{i}
+
\sum_{i\in S_{2}}\bd{X}_{i}(\bd{X}_{i}^{T}\bd{\beta}^{(k-1)} + \lambda)
+
\sum_{i\in S_{3}}\bd{X}_{i}(\bd{X}_{i}^{T}\bd{\beta}^{(k-1)} - \lambda) \\
& =
\sum_{i\in S_{1}}\bd{X}_{i}(\mu_{i}^{\star} + \bd{X}_{i}^{T}\bd{\beta}^{\star} + \epsilon_{i})
+
\sum_{i\in S_{2}}\bd{X}_{i}(\bd{X}_{i}^{T}\bd{\beta}^{(k-1)} + \lambda)
+
\sum_{i\in S_{3}}\bd{X}_{i}(\bd{X}_{i}^{T}\bd{\beta}^{(k-1)} - \lambda) \\
& =
\sum_{i\in S_{1}}\bd{X}_{i}\mu_{i}^{\star}
+
\sum_{i\in S_{1}}\bd{X}_{i}\bd{X}_{i}^{T}\bd{\beta}^{\star}
+
\sum_{i\in S_{1}}\bd{X}_{i}\epsilon_{i}
+
\sum_{i\in S_{2}\cup S_{3}}\bd{X}_{i}\bd{X}_{i}^{T}\bd{\beta}^{(k-1)}
+
\lambda(\sum_{i\in S_{2}}\bd{X}_{i}-\sum_{i\in S_{3}}\bd{X}_{i})\\
& =
\sum_{i\in S_{11}}\bd{X}_{i}\mu_{i}^{\star}
+
\sum_{i\in S_{12}}\bd{X}_{i}\mu_{i}^{\star}
+
\sum_{i\in S_{1}}\bd{X}_{i}\bd{X}_{i}^{T}\bd{\beta}^{\star}
+
\sum_{i\in S_{1}}\bd{X}_{i}\epsilon_{i}
+
\sum_{i\in S_{2}\cup S_{3}}\bd{X}_{i}\bd{X}_{i}^{T}\bd{\beta}^{(k-1)}
+
\lambda(\sum_{i\in S_{2}}\bd{X}_{i}-\sum_{i\in S_{3}}\bd{X}_{i}),
\end{align*}
where the index sets all depend on $\bd{\beta}^{(k-1)}$. Specifically,
\begin{align*}
S_{1} = S_{1}(\bd{\beta}^{(k-1)}) &= \{1\leq i\leq n: |Y_{i}-\bd{X}_{i}^{T}\bd{\beta}^{(k-1)}| \leq \lambda\}, \\
S_{2} = S_{2}(\bd{\beta}^{(k-1)}) &= \{1\leq i\leq n: Y_{i}-\bd{X}_{i}^{T}\bd{\beta}^{(k-1)} > \lambda\}, \\
S_{3} = S_{3}(\bd{\beta}^{(k-1)}) &= \{1\leq i\leq n: Y_{i}-\bd{X}_{i}^{T}\bd{\beta}^{(k-1)} < -\lambda\}, \\
S_{11} = S_{11}(\bd{\beta}^{(k-1)}) &= \{1\leq i\leq s_{1}: |Y_{i}-\bd{X}_{i}^{T}\bd{\beta}^{(k-1)}| \leq \lambda\}, \\
S_{12} = S_{11}(\bd{\beta}^{(k-1)}) &= \{s_{1}+1\leq i\leq s: |Y_{i}-\bd{X}_{i}^{T}\bd{\beta}^{(k-1)}| \leq \lambda\}.
\end{align*}

Denote an event
\begin{align*}
\mathcal{A}_{k-1}
=
\{
S_{11}(\bd{\beta}^{(k-1)})=\emptyset,
S_{12}(\bd{\beta}^{(k-1)})=S_{12}^{\star},
S_{1}(\bd{\beta}^{(k-1)})=S_{10}^{\star}\cup S_{12}^{\star};
S_{2}(\bd{\beta}^{(k-1)})=S_{21}^{\star};
S_{3}(\bd{\beta}^{(k-1)})=S_{31}^{\star}
\}.
\end{align*}

Since the regularization parameter satisfies (\ref{eq5.1:algorithm.convergence}),
it is easy to check that
the conclusion of
Lemma \ref{PC.Paper.d.infty.index.sets.consistency} continues to hold,
which implies
$P(\mathcal{A}_{0})\rightarrow 1$.

Thus, wpg1, we have
$$
\bd{\beta}^{(1)}
=
T_{0}^{-1}T_{1}
+
T_{0}^{-1}T_{2}
+
T_{0}^{-1}T_{3}
+
T_{0}^{-1}T_{4}(\bd{\beta}^{(0)})
+
T_{0}^{-1}T_{5},
$$
where
\begin{align*}
T_{0} &= \frac{1}{n}\mathbb{S}, \\
T_{1} & =  \frac{1}{n}\mathbb{S}_{S_{12}^{\star}}^{\mu},\\
T_{2} &= \frac{1}{n}\mathbb{S}_{s_{1}+1,n}\bd{\beta}^{\star}, \\
T_{3} &= \frac{1}{n}\mathbb{S}_{s_{1}+1,n}^{\epsilon}, \\
T_{4}(\bd{\beta}^{(0)}) &= \frac{1}{n}\mathbb{S}_{1,s_{1}}\bd{\beta}^{(0)}, \\
T_{5} &= \frac{1}{n}(\mathbb{S}_{S_{21}^{\star}}-\mathbb{S}_{S_{31}^{\star}})\lambda.
\end{align*}

We will show that, wpg1,
\begin{align*}
& \normE{T_{0}^{-1}T_{1}} \leq (C_{2}/4)d, \\
& \normE{T_{0}^{-1}T_{2}} \leq 2C_{1}d, \\
& \normE{T_{0}^{-1}T_{3}} \leq (C_{2}/4)d, \\
& \normE{T_{0}^{-1}T_{4}(\bd{\beta}^{(0)})} \leq (C_{2}/4)d, \\
& \normE{T_{0}^{-1}T_{5}} \leq (C_{2}/4)d.
\end{align*}

Thus, we have, wpg1,
\begin{equation*}
\normE{\bd{\beta}^{(1)}}
\leq
\sum_{i=1}^{5}
\normE{T_{0}^{-1}T_{i}}
\leq
(2C_{1}+C_{2})d.
\end{equation*}

\OnItem{On $T_{0}^{-1}T_{1}$.}
Under Assumption (D),
for $s_{2}\kappa_{n}\gamma_{n}/(n\sqrt{d})=o(1)$, wpg1,
\begin{align*}
\normE{T_{0}^{-1}T_{1}}
& =
\normE{
(\frac{1}{n}\mathbb{S})^{-1}
\frac{1}{n}\mathbb{S}_{S_{12}^{\star}}^{\mu}} \\
& \leq
\normF{
(\frac{1}{n}\mathbb{S})^{-1}}
\normE{
\frac{1}{n}\mathbb{S}_{S_{12}^{\star}}^{\mu}
} \\
& \leq
\normF{
(\frac{1}{n}\mathbb{S})^{-1}}
\normE{
\frac{1}{n}\sum_{i=s_{1}+1}^{s}\bd{X}_{i}\mu_{i}
} \\
& \leq
\normF{
(\frac{1}{n}\mathbb{S})^{-1}}
\frac{s_{2}}{n}
\kappa_{n}\gamma_{n} \\
& \leq
\normFd{
(\frac{1}{n}\mathbb{S})^{-1}}
\frac{s_{2}\sqrt{d}}{n}
\kappa_{n}\gamma_{n} \\
(wpg1) & \leq
2\normFd{
\bd{\Sigma}_{X}^{-1}}
\frac{s_{2}}{n\sqrt{d}}
\kappa_{n}\gamma_{n}d
\rightarrow 0.
\end{align*}

Thus, we have, wpg1,
$\normE{T_{0}^{-1}T_{1}}
\leq C_{2}d/4$.

\OnItem{On $T_{0}^{-1}T_{2}$.}
We have, wpg1,
\begin{align*}
\normE{T_{0}^{-1}T_{2}}
& =
\normE{
(\frac{1}{n}\mathbb{S})^{-1}
\frac{1}{n}\mathbb{S}_{s_{1}+1,n}\bd{\beta}^{\star}
} \\
& \leq
\normF{
(\frac{1}{n}\mathbb{S})^{-1}
\frac{1}{n}\mathbb{S}_{s_{1}+1,n}}
\normE{
\bd{\beta}^{\star}
} \\
& \leq
\normF{
(\frac{1}{n}\mathbb{S})^{-1}
\frac{1}{n}\mathbb{S}_{s_{1}+1,n}}
C_{1}\sqrt{d} \\
& =
\normF{
\bd{I}_{d}-
(\frac{1}{n}\mathbb{S})^{-1}
\frac{1}{n}\mathbb{S}_{1,s_{1}}}
C_{1}\sqrt{d} \\
& \leq
\normF{
\bd{I}_{d}}
C_{1}\sqrt{d}
+
\normF{
(\frac{1}{n}\mathbb{S})^{-1}
\frac{1}{n}\mathbb{S}_{1,s_{1}}}
C_{1}\sqrt{d} \\
& \leq
C_{1}d
+
\normF{
(\frac{1}{n}\mathbb{S})^{-1}}
\normF{
\frac{1}{n}\mathbb{S}_{1,s_{1}}}
C_{1}\sqrt{d},
\end{align*}
and
\begin{align*}
\normF{
\frac{1}{n}\mathbb{S}_{1,s_{1}}}
& =
\frac{1}{n}
\normF{
\mathbb{S}_{1,s_{1}}} \\
& =
\frac{1}{n}
\sum_{i=1}^{s_{1}}
\normF{
\bd{X}_{i}\bd{X}_{i}^{T}} \\
& =
\frac{1}{n}
\sum_{i=1}^{s_{1}}
\normE{
\bd{X}_{i}}^{2} \\
(wpg1) & \leq
\frac{s_{1}}{n}
\kappa_{n}^{2}.
\end{align*}

Thus, Under Assumption (D),
for $s_{1}\kappa_{n}^{2}/n=o(1)$, wpg1,
\begin{align*}
\normE{T_{0}^{-1}T_{2}}
& \leq
C_{1}d
+
2
\normFd{\bd{\Sigma}_{X}^{-1}}
\frac{\sqrt{d}s_{1}}{n}
\kappa_{n}^{2}
C_{1}\sqrt{d} \\
& \leq
2
C_{1}d.
\end{align*}

\OnItem{On $T_{0}^{-1}T_{3}$.}
Under assumptions (D)  and (G),
for $\log(d)/n=o(1)$, wpg1,
\begin{align*}
\normE{
T_{0}^{-1}T_{3}
}
& =
\normE{
(\frac{1}{n}\mathbb{S})^{-1}\frac{1}{n}\mathbb{S}_{s_{1}+1,n}^{\epsilon}
} \\
& =
\normF{
(\frac{1}{n}\mathbb{S})^{-1}}
\normE{\frac{1}{n}\mathbb{S}_{s_{1}+1,n}^{\epsilon}
} \\
& =
\sqrt{d}\frac{1}{\sqrt{n}}\sqrt{d\log(d)}
\normFd{
(\frac{1}{n}\mathbb{S})^{-1}}
(d\log(d))^{-1/2}
\normE{\frac{1}{\sqrt{n}}\mathbb{S}_{s_{1}+1,n}^{\epsilon}
} \\
(wpg1) & \leq
\frac{d\sqrt{\log(d)}}{\sqrt{n}}
2\normFd{
\bd{\Sigma}_{X}^{-1}}
O_{P}(1)
\ConvProb 0.
\end{align*}

Thus, wpg1,
$\normE{T_{0}^{-1}T_{3}} \leq C_{2}d/4$.

\OnItem{On $T_{0}^{-1}T_{4}(\bd{\beta}^{(0)})$.}
Under Assumption (D),
for $s_{1}
\kappa_{n}^{2}/n$, wpg1,
\begin{align*}
\normE{T_{0}^{-1}T_{4}(\bd{\beta}^{(0)})}
& =
\normE{
(\frac{1}{n}\mathbb{S})^{-1}
\frac{1}{n}\mathbb{S}_{1,s_{1}}\bd{\beta}^{(0)}
} \\
& \leq
\normF{
(\frac{1}{n}\mathbb{S})^{-1}}
\normE{
\frac{1}{n}\mathbb{S}_{1,s_{1}}\bd{\beta}^{(0)}
} \\
& \leq
\normF{
(\frac{1}{n}\mathbb{S})^{-1}}
\normF{
\frac{1}{n}\mathbb{S}_{1,s_{1}}}
\normE{\bd{\beta}^{(0)}
} \\
& =
\sqrt{d}
\normFd{
(\frac{1}{n}\mathbb{S})^{-1}}
\normF{
\frac{1}{n}\mathbb{S}_{1,s_{1}}}
\normE{\bd{\beta}^{(0)}
} \\
(wpg1) & \leq
\sqrt{d}
2\normFd{
\bd{\Sigma}_{X}^{-1}}
\frac{s_{1}}{n}
\kappa_{n}^{2}
C_{2}\sqrt{d} \\
& =
\frac{s_{1}}{n}
\kappa_{n}^{2}
2\normFd{
\bd{\Sigma}_{X}^{-1}}
C_{2}d
\ConvProb 0.
\end{align*}

Thus, wpg1,
$\normE{T_{0}^{-1}T_{4}(\bd{\beta}^{(0)})} \leq C_{2}d/4$.

\OnItem{On $T_{0}^{-1}T_{5}$.}
Under Assumption (D),
for $s_{1}\kappa_{n}\lambda/(n\sqrt{d}) = o(1)$, wpg1,
\begin{align*}
\normE{
T_{0}^{-1}T_{5}
}
& =
\normE{
(\frac{1}{n}\mathbb{S})^{-1}
\frac{1}{n}(\mathcal{S}_{S_{21}^{\star}}-\mathcal{S}_{S_{31}^{\star}})\lambda
} \\
& \leq
\normF{
(\frac{1}{n}\mathbb{S})^{-1}}
\frac{\lambda}{n}
(\normE{
\mathcal{S}_{S_{21}^{\star}}
-
\mathcal{S}_{S_{31}^{\star}}
}) \\
& \leq
\normF{
(\frac{1}{n}\mathbb{S})^{-1}}
\frac{\lambda}{n}
(\normE{
\mathcal{S}_{S_{21}^{\star}}}
+
\normE{
\mathcal{S}_{S_{31}^{\star}}
}) \\
& \leq
\sqrt{d}
\normFd{
(\frac{1}{n}\mathbb{S})^{-1}}
\frac{\lambda}{n}
(\normE{
\mathcal{S}_{S_{21}^{\star}}}
+
\normE{
\mathcal{S}_{S_{31}^{\star}}
}) \\
(wpg1) & \leq
\sqrt{d}
2\normFd{
\bd{\Sigma}_{X}^{-1}}
\frac{\lambda}{n}
s_{1}\kappa_{n} \\
& =
d
\frac{s_{1}\kappa_{n}\lambda}{n\sqrt{d}}
2\normFd{
\bd{\Sigma}_{X}^{-1}}.
\end{align*}
Thus, wpg1,
$\normE{T_{0}^{-1}T_{5}} \leq C_{2}d/4$.

\lineProof

Next, consider
$\normE{\bd{\beta}_{2} - \bd{\beta}_{1}}$.
Since $\bd{\beta}^{(1)}\leq (2C_{1}+C_{2})d$ wpg1,
the conclusion of
Lemma \ref{PC.Paper.d.infty.index.sets.consistency} holds,
which implies
$\mathcal{A}_{1}$ occurs wpg1.

Then,
$$
\bd{\beta}^{(2)}
=
T_{0}^{-1}T_{1}
+
T_{0}^{-1}T_{2}
+
T_{0}^{-1}T_{3}
+
T_{0}^{-1}T_{4}(\bd{\beta}^{(1)})
+
T_{0}^{-1}T_{5},
$$
where
\begin{align*}
T_{4}(\bd{\beta}^{(1)}) &= \frac{1}{n}\mathbb{S}_{1,s_{1}}\bd{\beta}^{(1)}.
\end{align*}

Thus, wpg1,
$$\bd{\beta}^{(2)} - \bd{\beta}^{(1)}
=
\mathbb{S}^{-1}\mathbb{S}_{1,s_{1}}
(\bd{\beta}^{(1)}-\bd{\beta}^{(0)}).$$
Thus, for $d^{3/2}s_{1} \kappa_{n}^{2}/n = o(1)$,
\begin{align*}
\normE{\bd{\beta}^{(2)}  - \bd{\beta}^{(1)}}
& \leq
\normF{\mathbb{S}^{-1}\mathbb{S}_{1,s_{1}}}
\normE{\bd{\beta}^{(1)} -\bd{\beta}^{(0)} } \\
& \leq
\normF{\mathbb{S}^{-1}}
\normF{\mathbb{S}_{1,s_{1}}}
\normE{\bd{\beta}^{(1)} -\bd{\beta}^{(0)} } \\
& =
\sqrt{d}
\normFd{\frac{1}{n}\mathbb{S}^{-1}}
\normF{\frac{1}{n}\mathbb{S}_{1,s_{1}}}
\normE{\bd{\beta}^{(1)} -\bd{\beta}^{(0)} } \\
(wpg1) & \leq
2
\normFd{\bd{\Sigma}_{X}^{-1}}
\sqrt{d}
\frac{s_{1}}{n}
\kappa_{n}^{2}
\normE{\bd{\beta}^{(1)} -\bd{\beta}^{(0)} } \\
& \leq
2
\normFd{\bd{\Sigma}_{X}^{-1}}
\sqrt{d}
\frac{s_{1}}{n}
\kappa_{n}^{2}
(2C_{1}+C_{2})d \\
& \lesssim
d^{3/2} s_{1} \kappa_{n}^{2}/n
\rightarrow 0.
\end{align*}
Thus, wpg1,
$\bd{\beta}^{(2)} = \bd{\beta}^{(1)}$,
which means that, wpg1, the iteration algorithm converges at the first iteration.

\lineProof

For any $K\geq 1$,
repeating the above arguments,
with at least probability
$p_{n,K} = P(\bigcap_{k=0}^{K}\mathcal{A}_{k})$,
we have
$\bd{\beta}^{(k)}\leq (2C_{1}+C_{2})d$ for $k\leq K$
and
\begin{align*}
\normE{\bd{\beta}^{(K+1)}  - \bd{\beta}^{(K)}}
& \leq
(2
\normFd{\bd{\Sigma}_{X}^{-1}}
\sqrt{d}
\frac{s_{1}}{n}
\kappa_{n}^{2})^{K}
(2C_{1}+C_{2})d \\
& \lesssim
(\sqrt{d}s_{1}\kappa_{n}^{2}/n)^{K}
d
\rightarrow 0.
\end{align*}
\end{proof}
\end{proofLong}

Next, we provide the proofs of Lemmas
\ref{PC.Paper.Inverse.Matrix.Convergence.Probability.d.Infity},
\ref{PC.Paper.Sample.Covariance.Matrix.Convergence.d.infty}
and
\ref{PC.Paper.Scaled.IID.Sum.Convergence.Rate.d.infty}
in the appendix.

\begin{proofName}
\subsection{Lemma \ref{PC.Paper.Inverse.Matrix.Convergence.Probability.d.Infity}}
\end{proofName}

\begin{proof}
[Proof of Lemma \ref{PC.Paper.Inverse.Matrix.Convergence.Probability.d.Infity}]
Let $\bd{E}=\bd{A}_{n}-\bd{A}$.
Note that $r_d\geq 1/\sqrt{d}$.
Then, $r_d\normF{\bd{E}}\ConvProb 0$
implies $\normFd{\bd{E}}\ConvProb 0$.
Thus, wpg1,
$\normFd{\bd{E}}$ is bounded by a constant $C>0$.
By Lemma \ref{lemma:Stewart:1969},
\begin{align*}
\normFd{\bd{A}_{n}^{-1} - \bd{A}^{-1}}
\leq
\normFd{\bd{A}^{-1}}
\frac{\normFd{\bd{A}^{-1}}\normFd{\bd{E}}}
{1-\normFd{\bd{A}^{-1}}\normFd{\bd{E}}}
\leq
C^{2}
\frac{\normFd{\bd{E}}}
{1-C\normFd{\bd{E}}}.
\end{align*}
Therefore,
\begin{align*}
r_d\normF{\bd{A}_{n}^{-1} - \bd{A}^{-1}}
\leq
C^{2}
\frac{r_d\normF{\bd{E}}}
{1-C\normFd{\bd{E}}}
\ConvProb  0.
\end{align*}
This completes the proof.
\end{proof}

\begin{proofName}
\subsection{Lemma \ref{PC.Paper.Sample.Covariance.Matrix.Convergence.d.infty}}
\end{proofName}
\begin{proof}
[Proof of Lemma \ref{PC.Paper.Sample.Covariance.Matrix.Convergence.d.infty}]
For any $\delta>0$, we have
\begin{align*}
P(\normF{\hat{\bd{\Sigma}}_{n} - \bd{\Sigma}_{X}}>\delta)
\leq &
\sum_{k=1}^{d}\sum_{l=1}^{d}\frac{d^{2}}{\delta^{2}}
P(\frac{1}{n}\sum_{i=1}^{n} X_{ik}X_{il}-\sigma_{kl})^{2}
\leq
\frac{d^{4}}{n}
\frac{1}{\delta^{2}}
\bar{\sigma}_{XX}^{2}.
\end{align*}
Thus,
$P(r_d\normF{\hat{\bd{\Sigma}}_{n} - \bd{\Sigma}_{X}}>\delta)
\leq
\bar{\sigma}_{XX}^{2}
r_d^{2}d^{4}/(n\delta^{2})
=o(1)$
by Assumption (E2)
and for
$r_d^{2}d^{4}/n \rightarrow 0$.
Thus, $\hat{\bd{\Sigma}}_{n}$ is a consistent estimator of $\bd{\Sigma}_{X}$
wrt $r_d\normF{\cdot}$.
\end{proof}

\begin{proofName}
\subsection{Lemma \ref{PC.Paper.Scaled.IID.Sum.Convergence.Rate.d.infty}}
\end{proofName}

\begin{proof}
[Proof of Lemma \ref{PC.Paper.Scaled.IID.Sum.Convergence.Rate.d.infty}]
Let $\alpha_{d}=\sqrt{d \log{d}}$ and $C_{1}\geq \sqrt{2}\sigma_{\xi,\max}$. Then
\begin{align*}
P(\normE{\frac{1}{\sqrt{n}}\sum_{i=1}^{n}\bd{\xi}_{i}} > \alpha_{d}C_{1})
\leq &
\sum_{j=1}^{d}
P(|\frac{1}{\sqrt{n}}\sum_{i=1}^{n}\frac{\xi_{ij}}{\sigma_{j}}| > \frac{\alpha_{d}C_{1}}{\sigma_{j}\sqrt{d}}),
\end{align*}
where $\sigma_{j}$ is the standard deviation of $\xi_{0j}$.
By Berry and Esseen Theorem
(see, for example, P375 in \cite{Shiryaev1995}),
there exists a constant $C_{2}>0$ such that
$
P(\normE{(1/\sqrt{n})\sum_{i=1}^{n}\bd{\xi}_{i}} > \alpha_{d}C_{1})
\leq T_{1} + 2T_{2},
$
where
\begin{align*}
T_{1}=
\sum_{j=1}^{d}
P(|N(0,1)| > \frac{\alpha_{d}C_{1}}{\sigma_{j}\sqrt{d}}), ~~~
T_{2}=
\sum_{j=1}^{d}\frac{C_{2}\E|\xi_{0j}|^{3}}{\sigma_{j}^{3}\sqrt{n}}.
\end{align*}
By noting $d^{2}=o(n)$,
\begin{align*}
T_{1}
\leq & \sum_{j=1}^{d}
P(|N(0,1)| > \frac{\alpha_{d}C_{1}}{\sigma_{\xi,\max}\sqrt{d}})
<
2d
\frac{\sigma_{\xi,\max}\sqrt{d}}{\alpha_{d}C_{1}}
\phi(\frac{\alpha_{d}C_{1}}{\sigma_{\xi,\max}\sqrt{d}})
\rightarrow 0, \\
T_{2}
\leq &
\sum_{j=1}^{d}
\frac{C_{2}\gamma_{\xi,\max}}{\sigma_{\xi,\min}^{3}\sqrt{n}}
=
d\frac{C_{2}\gamma_{\xi,\max}}{\sigma_{\min}^{3}\sqrt{n}}
\rightarrow 0.
\end{align*}
Therefore,
$\normE{(1/\sqrt{n})\sum_{i=1}^{n}\bd{\xi}_{i}}=O_{P}(\alpha_{d})$.
\end{proof}

\begin{proofLong}
\begin{proof}
[Long Proof of Lemma \ref{PC.Paper.Scaled.IID.Sum.Convergence.Rate.d.infty}]
Let $C_{1}\geq \sqrt{2}\sigma_{\xi,\max}$. Then
\begin{align*}
P(\normE{\frac{1}{\sqrt{n}}\sum_{i=1}^{n}\bd{\xi}_{i}} > \alpha_{d}C_{1})
= &
P((\frac{1}{\sqrt{n}}\sum_{i=1}\xi_{i1})^{2}
+ (\frac{1}{\sqrt{n}}\sum_{i=1}^{n}\xi_{i1})^{2}
+ \cdots
+ (\frac{1}{\sqrt{n}}\sum_{i=1}^{n}\xi_{id})^{2} > \alpha_{d}^{2}C_{1}^{2}) \\
\leq &
\sum_{j=1}^{d}
P((\frac{1}{\sqrt{n}}\sum_{i=1}^{n}\xi_{ij})^{2} > \frac{\alpha_{d}^{2}C_{1}^{2}}{d}) \\
= &
\sum_{j=1}^{d}
P(|\frac{1}{\sqrt{n}}\sum_{i=1}^{n}\xi_{ij}| > \frac{\alpha_{d}C_{1}}{\sqrt{d}}) \\
= &
\sum_{j=1}^{d}
P(|\frac{1}{\sqrt{n}}\sum_{i=1}^{n}\frac{\xi_{ij}}{\sigma_{j}}| > \frac{\alpha_{d}C_{1}}{\sigma_{j}\sqrt{d}}),
\end{align*}
where $\sigma_{j}$ is the standard deviation of $\xi_{0j}$.
By Berry and Esseen Theorem
(see, for example, P375 in \cite{Shiryaev1995}),
there exists a constant $C_{2}>0$ such that
\begin{align*}
P(|\frac{1}{\sqrt{n}}\sum_{i=1}^{n}\frac{\xi_{ij}}{\sigma_{j}}| > \frac{\alpha_{d}C_{1}}{\sigma_{j}\sqrt{d}})
& \leq
P(|N(0,1)| > \frac{\alpha_{d}C_{1}}{\sigma_{j}\sqrt{d}})
+
2\frac{C_{2}\E|\xi_{0j}|^{3}}{\sigma_{j}^{3}\sqrt{n}}.
\end{align*}
Then,
\begin{align*}
P(\normE{\frac{1}{\sqrt{n}}\sum_{i=1}^{n}\bd{\xi}_{i}} > \alpha_{d}C_{1})
\leq &
\sum_{j=1}^{d}
P(|N(0,1)| > \frac{\alpha_{d}C_{1}}{\sigma_{j}\sqrt{d}})
+
2\sum_{j=1}^{d}\frac{C_{2}\E|\xi_{0j}|^{3}}{\sigma_{j}^{3}\sqrt{n}}.
\end{align*}
On one hand,
\begin{align*}
\sum_{j=1}^{d}
P(|N(0,1)| > \frac{\alpha_{d}C_{1}}{\sigma_{j}\sqrt{d}})
\leq & \sum_{j=1}^{d}
P(|N(0,1)| > \frac{\alpha_{d}C_{1}}{\sigma_{\xi,\max}\sqrt{d}})\\
= &
d
P(|N(0,1)| > \frac{\alpha_{d}C_{1}}{\sigma_{\xi,\max}\sqrt{d}})\\
< &
2d
\frac{\sigma_{\xi,\max}\sqrt{d}}{\alpha_{d}C_{1}}
\phi(\frac{\alpha_{d}C_{1}}{\sigma_{\xi,\max}\sqrt{d}})\\
= &
2d
\frac{\sigma_{\xi,\max}\sqrt{d}}{\alpha_{d}C_{1}}
\frac{1}{\sqrt{2\pi}}
\exp\{-\frac{1}{2}\frac{\alpha_{d}^{2}C_{1}^{2}}{d\sigma_{\xi,\max}^{2}}\}\\
= &
2d
\frac{\sigma_{\xi,\max}}{C_{1}\sqrt{\log(d)}}
\frac{1}{\sqrt{2\pi}}
\exp\{-\frac{1}{2}\frac{C_{1}^{2}\log(d)}{\sigma_{\xi,\max}^{2}}\}\\
= &
2d
\frac{\sigma_{\xi,\max}}{C_{1}\sqrt{\log(d)}}
\frac{1}{\sqrt{2\pi}}
d^{-\frac{1}{2}\frac{C_{1}^{2}}{\sigma_{\xi,\max}^{2}}} \\
= &
2
\frac{\sigma_{\xi,\max}}{C_{1}\sqrt{\log(d)}}
\frac{1}{\sqrt{2\pi}}
d^{1-\frac{1}{2}\frac{C_{1}^{2}}{\sigma_{\xi,\max}^{2}}} \\
\leq &
2
\frac{\sigma_{\xi,\max}}{C_{1}\sqrt{\log(d)}}
\frac{1}{\sqrt{2\pi}} \rightarrow 0.
\end{align*}
On the other hand,
for $d^{2}=o(n)$,
\begin{align*}
\sum_{j=1}^{d}
\frac{C_{2}\E|\xi_{0j}|^{3}}{\sigma_{j}^{3}\sqrt{n}}
\leq &
\sum_{j=1}^{d}
\frac{C_{2}\gamma_{\xi,\max}}{\sigma_{\xi,\min}^{3}\sqrt{n}} \\
= &
d\frac{C_{2}\gamma_{\xi,\max}}{\sigma_{\min}^{3}\sqrt{n}}
\rightarrow 0.
\end{align*}

Therefore,
$\normE{(1/\sqrt{n})\sum_{i=1}^{n}\bd{\xi}_{i}}=O_{P}(\alpha_{d})$.
\end{proof}
\end{proofLong}

Next result is on the consistency of the penalized two-step estimator
$\tilde{\bd{\beta}}$.

\begin{thm}
[Consistency on $\tilde{\bd{\beta}}$]
\label{thm5.5}
Suppose the assumptions and conditions of
Theorem \ref{thm5.2} hold.
If $r_d\geq 1/\sqrt{d}$,
then $\tilde{\bd{\beta}} \ConvProb \bd{\beta}^{\star}$
wrt $r_d\normE{\cdot}$.
\end{thm}

\begin{proofName}
\subsection{Proof of Theorem
\ref{thm5.5}}
\end{proofName}
\begin{proof}[Proof of Theorems
\ref{thm5.5}]
By
Theorem \ref{thm5.2},
$\hat{\bd{\beta}}\ConvProb \bd{\beta}^{\star}$ wrt $r_d\normE{\cdot}$.
By Theorem
\ref{thm5.4},
$P\{\hat{I}_{0}=I_{0}\}\rightarrow 1$ for $r_d\geq 1/\sqrt{d}$,
where $I_{0}=\{s_{1}+1, s_{1}+2, \cdots, s=s_{1}+s_{2}, s+1, \cdots, n\}$.
Then, wpg1,
\begin{align*}
\tilde{\bd{\beta}} - \bd{\beta}^{\star}
& =
R_{1} + R_{2} + T_{0}^{-1}T_{1} + T_{0}^{-1}T_{2},
\end{align*}
where
$
R_{1}
=
(\bd{X}_{\hat I_{0}}^{T}\bd{X}_{\hat I_{0}})^{-1}\bd{X}_{\hat I_{0}}^{T}\bd{Y}_{\hat I_{0}}\{\hat{I}_{0}\not=I_{0}\}
$,
$
R_{2}
=
-(\bd{X}_{I_{0}}^{T}\bd{X}_{I_{0}})^{-1}\bd{X}_{I_{0}}^{T}\bd{Y}_{I_{0}}\{\hat{I}_{0}\not=I_{0}\}
$
and $T_{i}$'s are defined in the proof of Theorem
\ref{thm5.2}.
Then,
\begin{align*}
r_d\normE{\tilde{\bd{\beta}} - \bd{\beta}^{\star}}
& \leq
r_d\normE{R_{1}}
+
r_d\normE{R_{2}}
+
\normFd{T_{0}^{-1}}
r_d\sqrt{d}\normE{T_{1}}
+
\normFd{T_{0}^{-1}}
r_d\sqrt{d}\normE{T_{2}}.
\end{align*}
Since $P(\normEd{R_{1}}=0)\geq P\{\hat{I}_{0} = I_{0}\} \rightarrow 1$,
we have $R_{1} = o_{P}(1)$.
Similarly, $R_{2} = o_{P}(1)$.
By the proof of Theorem
\ref{thm5.2},
$\normFd{T_{0}^{-1}}$ is bounded and
$r_d\sqrt{d}\normE{T_{i}}\ConvProb 0$
for $i=1,2$.
Thus, $\tilde{\bd{\beta}} \ConvProb \bd{\beta}^{\star}$
wrt $r_d\normE{\cdot}$
and  $r_d\geq 1/\sqrt{d}$.
\end{proof}

Finally,
we provide some additional results on
the asymptotic distributions of
$\hat{\bd{\beta}}$ and $\tilde{\bd{\beta}}$
with a different scaling.
Specifically,
the scaling in Section \ref{sec:Diverging number of structural parameters}
is $\sqrt{n}\bd{A}_{n}$.
Next, we consider
another natural scaling
$\sqrt{n}\bd{A}_{n}\bd{\Sigma}_{X}^{1/2}$.

\begin{thm}[Asymptotic Distribution on $\hat{\bd{\beta}}$]
\label{thm5.3}
Suppose assumptions (D'), (D''), (E), (F) and (G) hold.
If $d^{6}\log d = o(n)$, $s_{1}=o(\sqrt{n}/(\lambda d\kappa_{n}))$ and $s_{2}=o(\sqrt{n}/(d\kappa_{n}\gamma_{n}))$,
then
$$
\sqrt{n}\bd{A}_{n}\bd{\Sigma}_{X}^{1/2}(\hat{\bd{\beta}}_{n} - \bd{\beta}^{\star})
\ConvDist
N(0, \sigma^{2}\bd{G}).
$$
\end{thm}

\begin{thm}
[Asymptotic Distribution on $\tilde{\bd{\beta}}$]
\label{thm5.6}
Suppose
the assumptions and conditions
of Theorem
\ref{thm5.3} hold
except the condition $s_{1}=o(\sqrt{n}/(\lambda d\kappa_{n}))$.
Then
$$
\sqrt{n}\bd{A}_{n}\bd{\bd{\Sigma}}_{X}^{1/2}(\tilde{\bd{\beta}} - \bd{\beta}^{\star} )
\ConvDist
N(0, \sigma^{2}\bd{G}).
$$
\end{thm}

By Theorems
\ref{thm5.3}
and
\ref{thm5.6},
Wald-type confidence regions can be constructed.
In order to validate these confidence regions
with estimated
$\sigma$
and
$\bd{\Sigma}_{X}$,
we need Lemma \ref{lem3.9:d:infty:Paper}
and the following result.
\begin{thm}
[Asymptotic Distributions on $\hat{\bd{\beta}}$ and $\tilde{\bd{\beta}}$ with $\hat{\bd{\Sigma}}_{n}$]
\label{model:three:multivariate:mean:zero:thm:twostage:limitdistribution.general.covariates.estimated.Sigma.errors:d:infty:Paper}
Suppose
the assumptions and conditions
of Theorem
\ref{thm5.3} hold.
If $d^{9}(\log(d))^{2}=o(n)$,
then
$$
\sqrt{n}\bd{A}_{n}\hat{\bd{\Sigma}}_{n}^{1/2}(\hat{\bd{\beta}} - \bd{\beta}^{\star} )
\ConvDist
N(0, \sigma^{2}\bd{G}).
$$
Similarly,
suppose
the assumptions and conditions
of Theorem
\ref{thm5.6} hold.
If $d^{9}(\log(d))^{2}=o(n)$,
then
$$
\sqrt{n}\bd{A}_{n}\hat{\bd{\Sigma}}_{n}^{1/2}(\tilde{\bd{\beta}} - \bd{\beta}^{\star} )
\ConvDist
N(0, \sigma^{2}\bd{G}).$$
\end{thm}

\begin{rem}
A comparison of the assumptions and conditions of
Theorem
\ref{model:three:multivariate:mean:zero:thm:twostage:limitdistribution.general.covariates.estimated.Sigma.errors:d:infty:Paper}
with those of
Theorems
\ref{thm5.3}
and
\ref{thm5.6}
reveals that
a much stronger requirement on $d$ is needed to ensure
$\hat{\bd{\Sigma}}_{n}$ is a good estimator of $\bd{\Sigma}_{X}$.
Precisely,
the former require that
$d^{9}(\log(d))^{2}=o(n)$
and the latter
$d^{6}\log(d)=o(n)$.
This stronger requirement on $d$ is a price paid for estimating $\bd{\Sigma}_{X}$.
\end{rem}

\begin{rem}
The condition on the dimension $d$ in
Theorems
\ref{thm3.4:d:infty:alternative:Paper}
and
\ref{thm3.8:d:infty:alternative:Paper}
is $d^{5}\log(d)=o(n)$,
slightly weaker than
the condition
$d^{6}\log(d)=o(n)$ in
Theorems
\ref{thm5.3}
and
\ref{thm5.6}.
Accordingly,
The condition on the dimension $d$ in
Theorem
\ref{model:three:multivariate:mean:zero:thm:twostage:limitdistribution.general.covariates.estimated.Sigma.alternative.errors:d:infty:Paper}
is $d^{8}(\log(d))^{2}=o(n)$,
slightly weaker than the condition $d^{9}(\log(d))^{2}=o(n)$
in
Theorem
\ref{model:three:multivariate:mean:zero:thm:twostage:limitdistribution.general.covariates.estimated.Sigma.errors:d:infty:Paper}.
This means that the scaling $\sqrt{n}\bd{A}_{n}$
is slightly better than the scaling $\sqrt{n}\bd{A}_{n}\bd{\Sigma}_{X}^{1/2}$
in terms of the condition on $d$.
Further, the former scaling is more suitable
for constructing confidence regions for
some entries of $\bd{\beta}^{\star}$.
\end{rem}

At the end of this supplement,
we provide the proofs of the above theorems.

\begin{proofName}
\subsection{Proof of Theorem \ref{thm5.3}}
\end{proofName}
\begin{proof}[Proof of Theorems
\ref{thm5.3}]
Reuse the notations $T_{i}$'s in the proof of Theorems
\ref{thm5.2},
from which,
$$
\sqrt{n}\bd{A}_{n}\bd{\Sigma}_{X}^{1/2}(\hat{\bd{\beta}}_{n} - \bd{\beta}^{\star})
=V_{1}+V_{2}+V_{3}-V_{4},
$$
where
$V_{i}=\bd{B}_{n}T_{i}$ for $i=1,2,3,4$
and
$\bd{B}_{n}=\sqrt{n}\bd{A}_{n}\bd{\Sigma}_{X}^{1/2}
T_{0}^{-1}$.
We will show $V_{2}\ConvDist N(0, \sigma^{2}\bd{G})$
and other $V_{i}$'s are $o_{P}(1)$,
from which the desired result follows by applying Slutsky's lemma.

\OnItem{On $V_{1}$.}
We have
$
\normE{V_{1}}
\leq
\sqrt{n}d
\normF{\bd{A}_{n}}
\normFd{\bd{\Sigma}_{X}^{1/2}}
\normFd{T_{0}^{-1}}
\normE{T_{1}}
$.
By Assumption (F),
$\normF{\bd{A}_{n}}$
is bounded.
By Assumption (D''),
$\normFd{\bd{\Sigma}_{X}^{1/2}}$
is bounded.
By Lemmas
\ref{PC.Paper.Inverse.Matrix.Convergence.Probability.d.Infity}
and
\ref{PC.Paper.Sample.Covariance.Matrix.Convergence.d.infty}
and Assumption (D),
for $d=o(n^{1/3})$, wpg1,
$\normFd{T_{0}^{-1}}$
is bounded.
Further, wpg1,
$\normE{T_{1}}
\leq
\frac{1}{n}s_{2}\kappa_{n}\gamma_{n}$.
Then,
$
\normE{V_{1}}
\lesssim
\frac{1}{\sqrt{n}}s_{2}d\kappa_{n}\gamma_{n},
$
where $\lesssim$ means that the left side is bounded by a constant times the right side,
as noted at the beginning of the appendix.
Thus,
$\normE{V_{1}}=o_{P}(1)$
for $s_{2}=o(\sqrt{n}/(d\kappa_{n}\gamma_{n}))$.

\OnItem{On $V_{2}$.}
We have
$V_{2}=V_{21}+ V_{22}$,
where
\begin{align*}
V_{21}
& =
\sqrt{n}\bd{A}_{n}\bd{\bd{\Sigma}}_{X}^{-1/2}
T_{2}, ~~~
V_{22}
=
\sqrt{n}\bd{A}_{n}\bd{\bd{\Sigma}}_{X}^{1/2}
(T_{0}^{-1} - \bd{\bd{\Sigma}}_{X}^{-1})
T_{2}.
\end{align*}

First, consider $V_{21}$.
We have
$
V_{21}
=
\sqrt{(n-s_{1})/n}
\sum_{i=s_{1}+1}^{n}
\bd{Z}_{n,i}
$,
where
\begin{equation*}
\bd{Z}_{n,i}
=
\frac{1}{\sqrt{n-s_{1}}}
\bd{A}_{n}\bd{\bd{\Sigma}}_{X}^{-1/2}\bd{X}_{i}\epsilon_{i}.
\end{equation*}

On one hand, for every $\delta>0$,
$
\sum_{i=s_{1}+1}^{n}
\E\normE{\bd{Z}_{n,i}}^{2}
\{\normE{\bd{Z}_{n,i}} > \delta\}
\leq
(n-s_{1})
\E\normE{\bd{Z}_{n,0}}^{4}/\delta^{2}
$
and
\begin{align*}
\E\normE{\bd{Z}_{n,0}}^{4}
& =
\frac{1}{(n-s_{1})^{2}}
\E
\epsilon_{0}^{4}
\E
(
\bd{X}_{0}^{T}
\bd{\bd{\Sigma}}_{X}^{-1/2}
\bd{A}_{n}^{T}
\bd{A}_{n}
\bd{\bd{\Sigma}}_{X}^{-1/2}
\bd{X}_{0}
)^{2} \\
& \leq
\frac{1}{(n-s_{1})^{2}}
\E
\epsilon_{0}^{4}
\lambda_{\max}(\bd{G}_{n})
\lambda_{\min}(\bd{\bd{\Sigma}}_{X})^{-1}
\E
(
\bd{X}_{0}^{T}
\bd{X}_{0}
)^{2} \\
& \leq
\frac{d^{2}}{(n-s_{1})^{2}}
\E
\epsilon_{0}^{4}
\lambda_{\max}(\bd{G}_{n})
\lambda_{\min}(\bd{\bd{\Sigma}}_{X})^{-1}
(\frac{1}{d}\sum_{j=1}^{d}(\E X_{0j}^{4})^{1/2})^{2}.
\end{align*}
Thus, by assumptions (D'), (E) and (F),
$
\sum_{i=s_{1}+1}^{n}
\E\normE{\bd{Z}_{n,i}}^{2}
\{\normE{\bd{Z}_{n,i}} > \delta\}
\rightarrow 0
$
for $d=o(\sqrt{n})$.
On the other hand,
$
\sum_{i=s_{1}+1}^{n}
\Cov(\bd{Z}_{n,i})
=
\sigma^{2}\bd{A}_{n}\bd{A}_{n}^{T}
\rightarrow
\sigma^{2}\bd{G}.
$
Thus, by central limit theorem
(see, for example, Proposition 2.27 in \cite{Vaart1998}),
$
V_{21}
\ConvDist
N(0, \sigma^{2}\bd{G})
$.
Next, consider $V_{22}$.
We have
$$
\normE{V_{22}}
\leq
\normF{\bd{A}_{n}}
\normFd{\bd{\bd{\Sigma}}_{X}^{1/2}}
d(\log(d))^{1/2}\normF{T_{0}^{-1} - \bd{\bd{\Sigma}}_{X}^{-1}}
(d\log(d))^{-1/2}\normE{\sqrt{n}T_{2}}.
$$
By Assumption (F), $\normF{\bd{A}_{n}}$ is $O(1)$;
By Assumption (D''), $\normFd{\bd{\bd{\Sigma}}_{X}^{1/2}}$ is $O(1)$;
by Lemmas \ref{PC.Paper.Inverse.Matrix.Convergence.Probability.d.Infity} and
        \ref{PC.Paper.Sample.Covariance.Matrix.Convergence.d.infty},
        $d(\log(d))^{1/2}\normF{T_{0}^{-1} - \bd{\bd{\Sigma}}_{X}^{-1}}$ is $o_{P}(1)$
        for $d^{6}\log(d) = o(n)$;
By Lemma \ref{PC.Paper.Scaled.IID.Sum.Convergence.Rate.d.infty},
        together with Assumption (G),
        $(d\log(d))^{-1/2}\normE{\sqrt{n}T_{2}}=(d\log(d))^{-1/2}\normE{\frac{1}{\sqrt{n}}\mathbb{S}_{s_{1}+1,n}^{\epsilon}}$ is $O_{P}(1)$ for $d=o(\sqrt{n})$.
Thus,
$V_{22}\ConvProb 0$.
By slutsky's lemma,
$V_{2}\ConvDist
N(0, \sigma^{2}\bd{G})$.

\OnItem{On $V_{3}$ and $V_{4}$.}
First consider $V_{3}$.
By noting that $s_{1}=o(\sqrt{n}/(\lambda d\kappa_{n}))$,
wpg1,
$
\normE{V_{3}}
\leq
d\sqrt{n}
\normF{\bd{A}_{n}}
\normFd{\bd{\Sigma}_{X}^{1/2}}
\normFd{T_{0}^{-1}}
\normE{T_{3}}
\lesssim
d\lambda s_{1}\kappa_{n}/\sqrt{n}
\rightarrow 0.
$
Thus, $\normE{V_{3}}=o_{P}(1)$.
In the same way, $\normE{V_{4}}=o_{P}(1)$.
This completes the proof.
\end{proof}

\begin{proofLong}
\begin{proof}[Long Proof of Theorems
\ref{thm5.3}]
We reuse the notations on $T_{i}$'s in the proof of Theorems
\ref{thm5.2},
from which,
$$
\sqrt{n}\bd{A}_{n}\bd{\Sigma}_{X}^{1/2}(\hat{\bd{\beta}}_{n} - \bd{\beta}^{\star})
=V_{1}+V_{2}+V_{3}-V_{4},
$$
where
\begin{align*}
V_{1} &=
\sqrt{n}\bd{A}_{n}\bd{\Sigma}_{X}^{1/2}
T_{0}^{-1}T_{1}, \\
V_{2} &=
\sqrt{n}\bd{A}_{n}\bd{\Sigma}_{X}^{1/2}
T_{0}^{-1}T_{2}, \\
V_{3} &=
\sqrt{n}\bd{A}_{n}\bd{\Sigma}_{X}^{1/2}
T_{0}^{-1}T_{3}, \\
V_{4} &=
\sqrt{n}\bd{A}_{n}\bd{\Sigma}_{X}^{1/2}
T_{0}^{-1}T_{4}.
\end{align*}

Next we derive the asymptotic properties of $V_{i}$'s,
from which the desired result follows by applying Slutsky's lemma.

\lineProof

\OnItem{On $V_{1}$.}
We have
\begin{align*}
\normE{V_{1}}
& =
\normE{\sqrt{n}\bd{A}_{n}\bd{\Sigma}_{X}^{1/2}
T_{0}^{-1}T_{1}} \\
& \leq
\sqrt{n}
\normF{\bd{A}_{n}}
\normF{\bd{\Sigma}_{X}^{1/2}}
\normF{T_{0}^{-1}}
\normE{T_{1}} \\
& =
\sqrt{n}d
\normF{\bd{A}_{n}}
\normFd{\bd{\Sigma}_{X}^{1/2}}
\normFd{T_{0}^{-1}}
\normE{T_{1}}.
\end{align*}

By Assumption (F),
$\normF{\bd{A}_{n}}$
is bounded.

By Assumption (D''),
$\normFd{\bd{\Sigma}_{X}^{1/2}}$
is bounded.

By Lemmas
\ref{PC.Paper.Inverse.Matrix.Convergence.Probability.d.Infity}
and
\ref{PC.Paper.Sample.Covariance.Matrix.Convergence.d.infty}
and Assumption (D),
for $d=o(n^{1/3})$, wpg1,
$\normFd{T_{0}^{-1}}$
is bounded.

We have, wpg1,
\begin{equation*}
\normE{T_{1}}
=
\normE{\frac{1}{n}\sum_{i=s_{1}+1}^{s}\bd{X}_{i}\mu_{i}^{\star}}
\leq
\frac{1}{n}s_{2}\kappa_{n}\gamma_{n}.
\end{equation*}

Then,
\begin{align*}
\normE{V_{1}}
& \lesssim
\frac{1}{\sqrt{n}}s_{2}d\kappa_{n}\gamma_{n},
\end{align*}
where $\lesssim$ means that the left side is bounded by a constant times the right side.

Thus,
$\normE{V_{1}}=o_{P}(1)$
for $s_{2}=o(\sqrt{n}/(d\kappa_{n}\gamma_{n}))$.

\lineProof

\OnItem{On $V_{2}$.}
We have
$V_{2}=V_{21}+ V_{22}$,
where
\begin{align*}
V_{21}
& =
\sqrt{n}\bd{A}_{n}\bd{\bd{\Sigma}}_{X}^{-1/2}
T_{2}, \\
V_{22}
& =
\sqrt{n}\bd{A}_{n}\bd{\bd{\Sigma}}_{X}^{1/2}
(T_{0}^{-1} - \bd{\bd{\Sigma}}_{X}^{-1})
T_{2}. \\
\end{align*}

First, consider $V_{21}$.
We have
\begin{align*}
V_{21}
& =
\sqrt{n}\bd{A}_{n}\bd{\bd{\Sigma}}_{X}^{-1/2}T_{2} \\
& =
\sqrt{n}\bd{A}_{n}\bd{\bd{\Sigma}}_{X}^{-1/2}
\frac{1}{n}\sum_{i=s_{1}+1}^{n}\bd{X}_{i}\epsilon_{i} \\
& =
\frac{\sqrt{n-s_{1}}}{\sqrt{n}}
\sum_{i=s_{1}+1}^{n}
\bd{Z}_{n,i},
\end{align*}
where
\begin{equation*}
\bd{Z}_{n,i}
=
\frac{1}{\sqrt{n-s_{1}}}
\bd{A}_{n}\bd{\bd{\Sigma}}_{X}^{-1/2}\bd{X}_{i}\epsilon_{i}.
\end{equation*}

For every $\delta>0$,
\begin{align*}
\sum_{i=s_{1}+1}^{n}
\E\normE{\bd{Z}_{n,i}}^{2}
\{\normE{\bd{Z}_{n,i}} > \delta\}
& \leq
\sum_{i=s_{1}+1}^{n}
\E\normE{\bd{Z}_{n,i}}^{4}/\delta^{2}
=
(n-s_{1})
\E\normE{\bd{Z}_{n,0}}^{4}/\delta^{2},
\end{align*}
and
\begin{align*}
\E\normE{\bd{Z}_{n,0}}^{4}
& =
\E(\normE{\bd{Z}_{n,0}}^{2})^{2} \\
& =
\E(
\frac{1}{n-s_{1}}
\epsilon_{0}^{2}
\bd{X}_{0}^{T}
\bd{\bd{\Sigma}}_{X}^{-1/2}
\bd{A}_{n}^{T}
\bd{A}_{n}
\bd{\bd{\Sigma}}_{X}^{-1/2}
\bd{X}_{0}
)^{2} \\
& =
\frac{1}{(n-s_{1})^{2}}
\E
\epsilon_{0}^{4}
\E
(
\bd{X}_{0}^{T}
\bd{\bd{\Sigma}}_{X}^{-1/2}
\bd{A}_{n}^{T}
\bd{A}_{n}
\bd{\bd{\Sigma}}_{X}^{-1/2}
\bd{X}_{0}
)^{2} \\
& \leq
\frac{1}{(n-s_{1})^{2}}
\E
\epsilon_{0}^{4}
\lambda_{\max}(\bd{G}_{n})
\lambda_{\min}(\bd{\bd{\Sigma}}_{X})^{-1}
\E
(
\bd{X}_{0}^{T}
\bd{X}_{0}
)^{2} \\
& \leq
\frac{d^{2}}{(n-s_{1})^{2}}
\E
\epsilon_{0}^{4}
\lambda_{\max}(\bd{G}_{n})
\lambda_{\min}(\bd{\bd{\Sigma}}_{X})^{-1}
\frac{1}{d^{2}}
\E
(
\bd{X}_{0}^{T}
\bd{X}_{0}
)^{2}\\
& \leq
\frac{d^{2}}{(n-s_{1})^{2}}
\E
\epsilon_{0}^{4}
\lambda_{\max}(\bd{G}_{n})
\lambda_{\min}(\bd{\bd{\Sigma}}_{X})^{-1}
\frac{1}{d^{2}}
\E
(
\sum_{j=1}^{d}X_{0j}^{2}
)^{2}\\
& \leq
\frac{d^{2}}{(n-s_{1})^{2}}
\E
\epsilon_{0}^{4}
\lambda_{\max}(\bd{G}_{n})
\lambda_{\min}(\bd{\bd{\Sigma}}_{X})^{-1}
\frac{1}{d^{2}}
\sum_{k=1}^{d}\sum_{l=1}^{d}\E X_{0k}^{2}X_{0l}^{2} \\
& \leq
\frac{d^{2}}{(n-s_{1})^{2}}
\E
\epsilon_{0}^{4}
\lambda_{\max}(\bd{G}_{n})
\lambda_{\min}(\bd{\bd{\Sigma}}_{X})^{-1}
\frac{1}{d^{2}}
\sum_{k=1}^{d}\sum_{l=1}^{d}(\E X_{0k}^{4})^{1/2}(\E X_{0l}^{4})^{1/2} \\
& \leq
\frac{d^{2}}{(n-s_{1})^{2}}
\E
\epsilon_{0}^{4}
\lambda_{\max}(\bd{G}_{n})
\lambda_{\min}(\bd{\bd{\Sigma}}_{X})^{-1}
(\frac{1}{d}\sum_{j=1}^{d}(\E X_{0j}^{4})^{1/2})^{2}.
\end{align*}
Then,
\begin{align*}
\sum_{i=s_{1}+1}^{n}
\E\normE{\bd{Z}_{n,i}}^{2}
\{\normE{\bd{Z}_{n,i}} > \delta\}
& \leq
\frac{d^{2}}{n-s_{1}}
\frac{1}{\delta^{2}}
\E
\epsilon_{0}^{4}
\lambda_{\max}(\bd{G}_{n})
\lambda_{\min}(\bd{\bd{\Sigma}}_{X})^{-1}
(\frac{1}{d}\sum_{j=1}^{d}(\E X_{0j}^{4})^{1/2})^{2}.
\end{align*}
Thus, by assumptions (D'), (E) and (F) and for $d=o(\sqrt{n})$,
\begin{equation*}
\sum_{i=s_{1}+1}^{n}
\E\normE{\bd{Z}_{n,i}}^{2}
\{\normE{\bd{Z}_{n,i}} > \delta\}
\rightarrow 0.
\end{equation*}

On the other hand,
\begin{align*}
\sum_{i=s_{1}+1}^{n}
\Cov(\bd{Z}_{n,i})
& =
\sum_{i=s_{1}+1}^{n}
\frac{1}{\sqrt{n-s_{1}}}
\bd{A}_{n}\bd{\bd{\Sigma}}_{X}^{-1/2}
\Cov(\bd{X}_{i}\epsilon_{i})
\bd{\bd{\Sigma}}_{X}^{-1/2}
\bd{A}_{n}^{T}
\frac{1}{\sqrt{n-s_{1}}} \\
& =
\sum_{i=s_{1}+1}^{n}
\frac{1}{n-s_{1}}
\bd{A}_{n}
\bd{\bd{\Sigma}}_{X}^{-1/2}
\bd{\bd{\Sigma}}_{X}\sigma^{2}
\bd{\bd{\Sigma}}_{X}^{-1/2}
\bd{A}_{n}^{T}\\
& =
\sigma^{2}\bd{A}_{n}\bd{A}_{n}^{T}
\rightarrow
\sigma^{2}\bd{G}.
\end{align*}

Thus, by central limit theorem
(see, for example, Proposition 2.27 in \cite{Vaart1998}),
\begin{align*}
V_{21}
& \ConvDist
N(0, \sigma^{2}\bd{G}).\\
\end{align*}

Next, consider $V_{22}$.
We have
\begin{align*}
\normE{V_{22}}
& =
\normE{\sqrt{n}\bd{A}_{n}\bd{\bd{\Sigma}}_{X}^{1/2}
(T_{0}^{-1} - \bd{\bd{\Sigma}}_{X}^{-1})T_{2}} \\
& \leq
\sqrt{n}
\normF{\bd{A}_{n}}
\normF{\bd{\bd{\Sigma}}_{X}^{1/2}}
\normF{T_{0}^{-1} - \bd{\bd{\Sigma}}_{X}^{-1}}
\normE{T_{2}} \\
& =
\normF{\bd{A}_{n}}
d^{1/2}d^{-1/2}\normF{\bd{\bd{\Sigma}}_{X}^{1/2}}
d^{-1}(\log(d))^{-1/2}d(\log(d))^{1/2}\normF{T_{0}^{-1} - \bd{\bd{\Sigma}}_{X}^{-1}}
(d\log(d))^{1/2}(d\log(d))^{-1/2}\normE{\sqrt{n}T_{2}} \\
& =
\normF{\bd{A}_{n}}
d^{-1/2}\normF{\bd{\bd{\Sigma}}_{X}^{1/2}}
d(\log(d))^{1/2}\normF{T_{0}^{-1} - \bd{\bd{\Sigma}}_{X}^{-1}}
(d\log(d))^{-1/2}\normE{\sqrt{n}T_{2}}\\
& =
\normF{\bd{A}_{n}}
\normFd{\bd{\bd{\Sigma}}_{X}^{1/2}}
d(\log(d))^{1/2}\normF{T_{0}^{-1} - \bd{\bd{\Sigma}}_{X}^{-1}}
(d\log(d))^{-1/2}\normE{\sqrt{n}T_{2}}.
\end{align*}
Note that
\begin{itemize}
  \item By Assumption (F), $\normF{\bd{A}_{n}}$ is $O(1)$;
  \item By Assumption (D''), $\normFd{\bd{\bd{\Sigma}}_{X}^{1/2}}$ is $O(1)$;
  \item by Lemmas \ref{PC.Paper.Inverse.Matrix.Convergence.Probability.d.Infity} and
        \ref{PC.Paper.Sample.Covariance.Matrix.Convergence.d.infty},
        $d(\log(d))^{1/2}\normF{T_{0}^{-1} - \bd{\bd{\Sigma}}_{X}^{-1}}$ is $o_{P}(1)$
        for $d^{6}\log(d) = o(n)$;
  \item by Lemma \ref{PC.Paper.Scaled.IID.Sum.Convergence.Rate.d.infty},
        together with Assumption (G),
        $(d\log(d))^{-1/2}\normE{\sqrt{n}T_{2}}=(d\log(d))^{-1/2}\normE{\frac{1}{\sqrt{n}}\mathbb{S}_{s_{1}+1,n}^{\epsilon}}$ is $O_{P}(1)$ for $d=o(\sqrt{n})$.
\end{itemize}
Then,
$V_{22}\ConvProb 0$.

\lineProof

Thus, by slutsky's lemma,
$V_{2}\ConvDist
N(0, \sigma^{2}\bd{G})$.

\lineProof

\OnItem{On $V_{3}$ and $V_{4}$.}
First consider $V_{3}$.
By noting that $s_{1}=o(\sqrt{n}/(\lambda d\kappa_{n}))$,
wpg1,
\begin{align*}
\normE{V_{3}}
& =
\normE{\sqrt{n}\bd{A}_{n}\bd{\Sigma}_{X}^{1/2}
T_{0}^{-1}T_{3}} \\
& \leq
\sqrt{n}
\normF{\bd{A}_{n}}
\normF{\bd{\Sigma}_{X}^{1/2}}
\normF{T_{0}^{-1}}
\normE{T_{3}} \\
& =
d\sqrt{n}
\normF{\bd{A}_{n}}
\normFd{\bd{\Sigma}_{X}^{1/2}}
\normFd{T_{0}^{-1}}
\normE{T_{3}} \\
& \lesssim
d\sqrt{n}\lambda s_{1}\kappa_{n}/n
=
d\lambda s_{1}\kappa_{n}/\sqrt{n}
\rightarrow 0.
\end{align*}
Thus, $\normE{V_{3}}=o_{P}(1)$.
In the same way, $\normE{V_{4}}=o_{P}(1)$.

Therefore, the result of the theorem follows
by Slutsky's lemma.
\end{proof}
\end{proofLong}

\begin{proofName}
\subsection{Proof of Theorem \ref{thm5.4}}
\end{proofName}

\begin{proofName}
\subsection{Proof of Theorem
\ref{thm5.6}}
\end{proofName}
\begin{proof}[Proof of Theorem
\ref{thm5.6}]
From the proof of Theorem
\ref{thm5.5},
we have
$ \sqrt{n}\bd{A}_{n}\bd{\bd{\Sigma}}_{X}^{1/2}(\tilde{\bd{\beta}} - \bd{\beta}^{\star})
=
\tilde{R}_{1}
+
\tilde{R}_{2}
+
V_{1}
+
V_{2}
$,
where
$\tilde{R}_{1} = \sqrt{n}\bd{A}_{n}\bd{\bd{\Sigma}}_{X}^{1/2}R_{1}$,
$\tilde{R}_{2} = \sqrt{n}\bd{A}_{n}\bd{\bd{\Sigma}}_{X}^{1/2}R_{2}$,
and $R_{i}$'s
and $V_{i}$'s are defined in the proofs of Theorems
\ref{thm5.5} and
\ref{thm5.3}.
Since $P(\normE{\tilde{R}_{1}}=0)\geq P\{\hat{I}_{0} = I_{0}\} \rightarrow 1$,
we have $\tilde{R}_{1} = o_{P}(1)$.
Similarly, $\tilde{R}_{2} = o_{P}(1)$.
By the proof of Theorem
\ref{thm5.3},
$V_{1}=o_{P}(1)$ and $V_{2}\ConvDist N(0, \sigma^{2}\bd{G})$.
Thus, the asymptotic distribution of $\tilde{\bd{\beta}}$ is Gaussian by Slutsky's lemma.
\end{proof}

\begin{proofLong}
\begin{proof}[Proof of Theorems
\ref{thm5.5}
and \ref{thm5.6}]
By
Theorem \ref{thm5.2},
$\hat{\bd{\beta}}\ConvProb \bd{\beta}^{\star}$ wrt $r_d\normE{\cdot}$.
By Theorem
\ref{thm5.4},
$P\{\hat{I}_{0}=I_{0}\}\rightarrow 1$ for $r_d\geq 1/\sqrt{d}$,
where $I_{0}=\{s_{1}+1, s_{1}+2, \cdots, s=s_{1}+s_{2}, s+1, \cdots, n\}$.
Then,
\begin{equation*}
\tilde{\bd{\beta}}
=
(\bd{X}_{\hat I_{0}}^{T}\bd{X}_{\hat I_{0}})^{-1}\bd{X}_{\hat I_{0}}^{T}\bd{Y}_{\hat I_{0}}
=
(\bd{X}_{I_{0}}^{T}\bd{X}_{I_{0}})^{-1}\bd{X}_{I_{0}}^{T}\bd{Y}_{I_{0}}
+
R_{1} + R_{2},
\end{equation*}
where
$
R_{1}
=
(\bd{X}_{\hat I_{0}}^{T}\bd{X}_{\hat I_{0}})^{-1}\bd{X}_{\hat I_{0}}^{T}\bd{Y}_{\hat I_{0}}\{\hat{I}_{0}\not=I_{0}\}
$
and
$
R_{2}
=
-(\bd{X}_{I_{0}}^{T}\bd{X}_{I_{0}})^{-1}\bd{X}_{I_{0}}^{T}\bd{Y}_{I_{0}}\{\hat{I}_{0}\not=I_{0}\}
$.
Note that
\begin{align*}
(\bd{X}_{I_{0}}^{T}\bd{X}_{I_{0}})^{-1}\bd{X}_{I_{0}}^{T}\bd{Y}_{I_{0}}
& =
\bd{\beta}^{\star}
+
T_{0}^{-1}T_{1}
+
T_{0}^{-1}T_{2},
\end{align*}
where
$T_{0} = \mathbb{S}_{s_{1}+1,n}/n$,
$
T_{1} =
\mathbb{S}_{S_{12}^{\star}}^{\mu}/n$
and
$
T_{2} =
\mathbb{S}_{s_{1}+1,n}^{\epsilon}/n$.
Thus,
\begin{align*}
\tilde{\bd{\beta}} - \bd{\beta}^{\star}
& =
R_{1} + R_{2} + T_{0}^{-1}T_{1} + T_{0}^{-1}T_{2}.
\end{align*}

\lineProof

Next, we consider the consistency of $\tilde{\bd{\beta}}$.
We have
\begin{align*}
\normE{\tilde{\bd{\beta}} - \bd{\beta}^{\star}}
& \leq
\normE{R_{1}}
+
\normE{R_{2}}
+
\normF{(d^{1/2}T_{0})^{-1}}
\normE{d^{1/2}T_{1}}
+
\normF{(d^{1/2}T_{0})^{-1}}
\normE{d^{1/2}T_{2}},
\end{align*}
which implies
\begin{align*}
r_d\normE{\tilde{\bd{\beta}} - \bd{\beta}^{\star}}
& \leq
r_d\normE{R_{1}}
+
r_d\normE{R_{2}}
+
\normFd{T_{0}^{-1}}
r_d\sqrt{d}\normE{T_{1}}
+
\normFd{T_{0}^{-1}}
r_d\sqrt{d}\normE{T_{2}}.
\end{align*}

We will show that
$r_d\normEd{R_{1}}\ConvProb 0$,
$r_d\normEd{R_{2}}\ConvProb 0$,
$\normFd{T_{0}^{-1}}$ is bounded,
$r_d\sqrt{d}\normE{T_{1}}\ConvProb 0$
and
$r_d\sqrt{d}\normE{T_{2}}\ConvProb 0$.
Then, $\tilde{\bd{\beta}} \ConvProb \bd{\beta}^{\star}$
wrt $r_d\normE{\cdot}$.

\lineProof

\OnItem{On $R_{1}$ and $R_{2}$.}
Since $P(\normEd{R_{1}}=0)\geq P\{\hat{I}_{0} = I_{0}\} \rightarrow 1$,
we have $R_{1} = o_{P}(1)$.
Similarly, $R_{2} = o_{P}(1)$.

\OnItem{On $T_{0}$.}
See the proof of Theorem
\ref{thm5.2}.

\OnItem{On $T_{1}$.}
See the proof of Theorem
\ref{thm5.2}.

\OnItem{On $T_{2}$.}
See the proof of Theorem
\ref{thm5.2}.

Thus, $\tilde{\bd{\beta}}$ is a consistent estimator of $\bd{\beta}^{\star}$
wrt $r_d\normE{\cdot}$ and $r_d\geq 1/\sqrt{d}$.

\lineProof

Next we show the asymptotic normality.
We have
\begin{align*}
\sqrt{n}\bd{A}_{n}\bd{\bd{\Sigma}}_{X}^{1/2}(\tilde{\bd{\beta}} - \bd{\beta}^{\star})
& =
\sqrt{n}\bd{A}_{n}\bd{\bd{\Sigma}}_{X}^{1/2}R_{1}
+ \sqrt{n}\bd{A}_{n}\bd{\bd{\Sigma}}_{X}^{1/2}R_{2}
+ \sqrt{n}\bd{A}_{n}\bd{\bd{\Sigma}}_{X}^{1/2}T_{0}^{-1}T_{1}
+ \sqrt{n}\bd{A}_{n}\bd{\bd{\Sigma}}_{X}^{1/2}T_{0}^{-1}T_{2} \\
& =:
\tilde{R}_{1}
+
\tilde{R}_{2}
+
V_{1}
+
V_{2}.
\end{align*}

\OnItem{On $\tilde{R}_{1}$ and $\tilde{R}_{2}$.}
Since $P(\normE{\tilde{R}_{1}}=0)\geq P\{\hat{I}_{0} = I_{0}\} \rightarrow 1$,
we have $\tilde{R}_{1} = o_{P}(1)$.
Similarly, $\tilde{R}_{2} = o_{P}(1)$.

\OnItem{On $V_{1}$.}
See the proof of Theorem
\ref{thm5.3}.

\OnItem{On $V_{2}$.}
See the proof of Theorem
\ref{thm5.3}.

Thus, we obtain the asymptotic distribution of $\tilde{\bd{\beta}}$ by Slutsky's lemma.
\end{proof}
\end{proofLong}

\begin{proofName}
\subsection{Proof of Theorem \ref{model:three:multivariate:mean:zero:thm:twostage:limitdistribution.general.covariates.estimated.Sigma.errors:d:infty:Paper}}
\end{proofName}
\begin{proof}
[Proof of Theorem \ref{model:three:multivariate:mean:zero:thm:twostage:limitdistribution.general.covariates.estimated.Sigma.errors:d:infty:Paper}]
We only show the result on $\hat{\bd{\beta}}$.
since the result on $\tilde{\bd{\beta}}$ can be obtained in a similar way.
We reuse the definitions of $T_{i}$'s in the proof of Theorems
\ref{thm5.2},
from which,
$$
\sqrt{n}\bd{A}_{n}\hat{\bd{\Sigma}}_{n}^{1/2}(\hat{\bd{\beta}}_{n} - \bd{\beta}^{\star})
= M + R,
$$
where
$ M = \sqrt{n}\bd{A}_{n}\bd{\Sigma}_{X}^{1/2}(\hat{\bd{\beta}}_{n} - \bd{\beta}^{\star})$ and
$R = \sqrt{n}\bd{A}_{n}(\hat{\bd{\Sigma}}_{n}^{1/2}-\bd{\Sigma}_{X}^{1/2})(\hat{\bd{\beta}}_{n} - \bd{\beta}^{\star})$.
By Theorem
\ref{thm5.3},
$M \ConvDist N(0, \sigma^{2}\bd{G})$.
Then, it is sufficient to show that $R\ConvProb 0$ wrt $\normE{\cdot}$.
We have
$$
R
=R_{1}+R_{2}+R_{3}-R_{4},
$$
where
$R_{i} =
\bd{B}_{n}T_{i}$ for $i=1,2,3,4$
and
$\bd{B}_{n}=\sqrt{n}\bd{A}_{n}(\hat{\bd{\Sigma}}_{n}^{1/2}-\bd{\Sigma}_{X}^{1/2})T_{0}^{-1}$.
We will show each $R_{i}$ converges to zero in probability,
which finishes the proof.

\OnItem{On $R_{1}$.}
By Lemma \ref{PC.Paper.Lemma.Inequality.Matrix.Square.Root.Differentce},
$
\normF{\hat{\bd{\Sigma}}_{n}^{1/2}-\bd{\Sigma}_{X}^{1/2}}
\leq
(d^{1/2}\normF{\hat{\bd{\Sigma}}_{n}-\bd{\Sigma}_{X}})^{1/2}
$.
Then,
\begin{align*}
\normE{R_{1}}
& \leq
\sqrt{n}
\normF{\bd{A}_{n}}
\normF{\hat{\bd{\Sigma}}_{n}^{1/2}-\bd{\Sigma}_{X}^{1/2}}
\normF{T_{0}^{-1}}
\normE{T_{1}} \\
& \leq
\sqrt{n} d
\normF{\bd{A}_{n}}
(\normFd{\hat{\bd{\Sigma}}_{n}-\bd{\Sigma}_{X}})^{1/2}
\normFd{T_{0}^{-1}}
\normE{T_{1}}.
\end{align*}
By Assumption (F),
$\normF{\bd{A}_{n}}$
is bounded.
By Lemma \ref{PC.Paper.Sample.Covariance.Matrix.Convergence.d.infty},
$\normFd{\hat{\bd{\Sigma}}_{n}-\bd{\Sigma}_{X}}=o_{P}(1)$
for $d=o(n^{1/3})$.
By Lemmas
\ref{PC.Paper.Inverse.Matrix.Convergence.Probability.d.Infity}
and
\ref{PC.Paper.Sample.Covariance.Matrix.Convergence.d.infty}
and Assumption (D),
for $d=o(n^{1/3})$, wpg1,
$\normFd{T_{0}^{-1}}$
is bounded.
We have, wpg1,
$
\normE{T_{1}}
\leq
\frac{1}{n}s_{2}\kappa_{n}\gamma_{n}.
$
Then,
$
\normE{R_{1}}
\lesssim
\frac{1}{\sqrt{n}}s_{2}d\kappa_{n}\gamma_{n}$.
Thus,
$\normE{R_{1}}=o_{P}(1)$
for $s_{2}=o(\sqrt{n}/(d\kappa_{n}\gamma_{n}))$.

\OnItem{On $R_{2}$.}
We have
$$
\normE{R_{2}}
\leq
\normF{\bd{A}_{n}}
d(\log(d))^{1/2}
\normF{\hat{\bd{\Sigma}}_{n}^{1/2}-\bd{\Sigma}_{X}^{1/2}}
\normFd{T_{0}^{-1}}
(d\log(d))^{-1/2}
\normE{\sqrt{n}T_{2}},
$$
and
\begin{align*}
d(\log(d))^{1/2}
\normF{\hat{\bd{\Sigma}}_{n}^{1/2}-\bd{\Sigma}_{X}^{1/2}}
& \leq
(d^{5/2}\log(d)\normF{\hat{\bd{\Sigma}}_{n}-\bd{\Sigma}_{X}})^{1/2}.
\end{align*}
By Assumption (F), $\normF{\bd{A}_{n}}$ is $O(1)$;
by Lemma \ref{PC.Paper.Sample.Covariance.Matrix.Convergence.d.infty},
        $d^{5/2}\log(d)\normF{\hat{\bd{\Sigma}}_{n}-\bd{\Sigma}_{X}}=o_{P}(1)$
        for $d^{9}(\log(d))^{2} = o(n)$;
by Lemmas \ref{PC.Paper.Inverse.Matrix.Convergence.Probability.d.Infity} and
        \ref{PC.Paper.Sample.Covariance.Matrix.Convergence.d.infty},
        $d(\log(d))^{1/2}\normF{T_{0}^{-1} - \bd{\bd{\Sigma}}_{X}^{-1}}$ is $o_{P}(1)$
        for $d^{6}\log(d) = o(n)$;
by Lemma \ref{PC.Paper.Scaled.IID.Sum.Convergence.Rate.d.infty},
        $(d\log(d))^{-1/2}\normE{\sqrt{n}T_{2}}=(d\log(d))^{-1/2}\normE{\frac{1}{\sqrt{n}}\mathbb{S}_{s_{1}+1,n}^{\epsilon}}$ is $O_{P}(1)$ for $d=o(\sqrt{n})$.
Thus,
$R_{2}\ConvProb 0$.

\OnItem{On $R_{3}$ and $R_{4}$.}
First consider $R_{3}$.
By noting that $s_{1}=o(\sqrt{n}/(\lambda d\kappa_{n}))$,
wpg1,
\begin{align*}
\normE{R_{3}}
& \leq
d\sqrt{n}
\normF{\bd{A}_{n}}
(\normFd{\hat{\bd{\Sigma}}_{n}^{1/2}-\bd{\Sigma}_{X}^{1/2}})^{1/2}
\normFd{T_{0}^{-1}}
\normE{T_{3}}
\lesssim
d\lambda s_{1}\kappa_{n}/\sqrt{n}
\rightarrow 0.
\end{align*}
Thus, $\normE{R_{3}}=o_{P}(1)$.
In the same way, $\normE{R_{4}}=o_{P}(1)$.
%
\end{proof}

\begin{proofLong}
\begin{proof}
[Long Proof of Corollary \ref{model:three:multivariate:mean:zero:thm:twostage:limitdistribution.general.covariates.estimated.Sigma.errors:d:infty:Paper}]
We reuse the notations on $T_{i}$'s in the proof of Theorems
\ref{thm5.2},
from which,
$$
\sqrt{n}\bd{A}_{n}\hat{\bd{\Sigma}}_{n}^{1/2}(\hat{\bd{\beta}}_{n} - \bd{\beta}^{\star})
= M + R,
$$
where
\begin{align*}
M &= \sqrt{n}\bd{A}_{n}\bd{\Sigma}_{X}^{1/2}(\hat{\bd{\beta}}_{n} - \bd{\beta}^{\star}), \\
R &= \sqrt{n}\bd{A}_{n}(\hat{\bd{\Sigma}}_{n}^{1/2}-\bd{\Sigma}_{X}^{1/2})(\hat{\bd{\beta}}_{n} - \bd{\beta}^{\star}).
\end{align*}
By Theorem
\ref{thm5.3},
$M \ConvDist N(0, \sigma^{2}\bd{G})$.
Then, it is sufficient to show that $R\ConvProb 0$ wrt $\normE{\cdot}$.
We have
$$
R
=R_{1}+R_{2}+R_{3}-R_{4},
$$
where
\begin{align*}
R_{1} &=
\sqrt{n}\bd{A}_{n}(\hat{\bd{\Sigma}}_{n}^{1/2}-\bd{\Sigma}_{X}^{1/2})
T_{0}^{-1}T_{1}, \\
R_{2} &=
\sqrt{n}\bd{A}_{n}(\hat{\bd{\Sigma}}_{n}^{1/2}-\bd{\Sigma}_{X}^{1/2})
T_{0}^{-1}T_{2}, \\
R_{3} &=
\sqrt{n}\bd{A}_{n}(\hat{\bd{\Sigma}}_{n}^{1/2}-\bd{\Sigma}_{X}^{1/2})
T_{0}^{-1}T_{3}, \\
R_{4} &=
\sqrt{n}\bd{A}_{n}(\hat{\bd{\Sigma}}_{n}^{1/2}-\bd{\Sigma}_{X}^{1/2})
T_{0}^{-1}T_{4}.
\end{align*}

Next we show each $R_{i}$ converges to zero in probability.

\lineProof

\OnItem{On $R_{1}$.}
We have
\begin{align*}
\normE{R_{1}}
& =
\normE{\sqrt{n}\bd{A}_{n}(\hat{\bd{\Sigma}}_{n}^{1/2}-\bd{\Sigma}_{X}^{1/2})
T_{0}^{-1}T_{1}} \\
& \leq
\sqrt{n}
\normF{\bd{A}_{n}}
\normF{\hat{\bd{\Sigma}}_{n}^{1/2}-\bd{\Sigma}_{X}^{1/2}}
\normF{T_{0}^{-1}}
\normE{T_{1}}.
\end{align*}

By Lemma,
\begin{align*}
\normF{\hat{\bd{\Sigma}}_{n}^{1/2}-\bd{\Sigma}_{X}^{1/2}}
\leq
(d^{1/2}\normF{\hat{\bd{\Sigma}}_{n}-\bd{\Sigma}_{X}})^{1/2}.
\end{align*}

Thus,
\begin{align*}
\normE{R_{1}}
& \leq
\sqrt{n}
\normF{\bd{A}_{n}}
(d^{1/2}\normF{\hat{\bd{\Sigma}}_{n}-\bd{\Sigma}_{X}})^{1/2}
\normF{T_{0}^{-1}}
\normE{T_{1}} \\
& =
\sqrt{n} d
\normF{\bd{A}_{n}}
(\normFd{\hat{\bd{\Sigma}}_{n}-\bd{\Sigma}_{X}})^{1/2}
\normFd{T_{0}^{-1}}
\normE{T_{1}}.
\end{align*}

By Assumption (F)
$\normF{\bd{A}_{n}}$
is bounded.

By Lemma \ref{PC.Paper.Sample.Covariance.Matrix.Convergence.d.infty},
$\normFd{\hat{\bd{\Sigma}}_{n}-\bd{\Sigma}_{X}}=o_{P}(1)$
for $d=o(n^{1/3})$.

By Lemmas
\ref{PC.Paper.Inverse.Matrix.Convergence.Probability.d.Infity}
and
\ref{PC.Paper.Sample.Covariance.Matrix.Convergence.d.infty}
and Assumption (D),
for $d=o(n^{1/3})$, wpg1,
$\normFd{T_{0}^{-1}}$
is bounded.

We have, wpg1,
\begin{equation*}
\normE{T_{1}}
=
\normE{\frac{1}{n}\sum_{i=s_{1}+1}^{s}\bd{X}_{i}\mu_{i}^{\star}}
\leq
\frac{1}{n}s_{2}\kappa_{n}\gamma_{n}.
\end{equation*}

Then,
\begin{align*}
\normE{V_{1}}
& \lesssim
\frac{1}{\sqrt{n}}s_{2}d\kappa_{n}\gamma_{n},
\end{align*}
where $\lesssim$ means that the left side is bounded by a constant times the right side.

Thus,
$\normE{R_{1}}=o_{P}(1)$
for $s_{2}=o(\sqrt{n}/(d\kappa_{n}\gamma_{n}))$.

\lineProof

\OnItem{On $R_{2}$.}
We have
\begin{align*}
\normE{R_{2}}
& =
\normE{\sqrt{n}\bd{A}_{n}(\hat{\bd{\Sigma}}_{n}^{1/2}-\bd{\Sigma}_{X}^{1/2})T_{0}^{-1}T_{2}} \\
& \leq
\normF{\bd{A}_{n}}
\normF{\hat{\bd{\Sigma}}_{n}^{1/2}-\bd{\Sigma}_{X}^{1/2}}
\normF{T_{0}^{-1}}
\normE{\sqrt{n}T_{2}} \\
& =
\normF{\bd{A}_{n}}
d(\log(d))^{1/2}
\normF{\hat{\bd{\Sigma}}_{n}^{1/2}-\bd{\Sigma}_{X}^{1/2}}
\normFd{T_{0}^{-1}}
(d\log(d))^{-1/2}
\normE{\sqrt{n}T_{2}},
\end{align*}
and
\begin{align*}
d(\log(d))^{1/2}
\normF{\hat{\bd{\Sigma}}_{n}^{1/2}-\bd{\Sigma}_{X}^{1/2}}
& \leq
d(\log(d))^{1/2}
(d^{1/2}\normF{\hat{\bd{\Sigma}}_{n}-\bd{\Sigma}_{X}})^{1/2} \\
& =
(d^{5/2}\log(d)\normF{\hat{\bd{\Sigma}}_{n}-\bd{\Sigma}_{X}})^{1/2}.
\end{align*}

Note that
\begin{itemize}
  \item By Assumption (F), $\normF{\bd{A}_{n}}$ is $O(1)$;
  \item By Lemma \ref{PC.Paper.Sample.Covariance.Matrix.Convergence.d.infty},
        $d^{5/2}\log(d)\normF{\hat{\bd{\Sigma}}_{n}-\bd{\Sigma}_{X}}=o_{P}(1)$
        for $d^{9}(\log(d))^{2} = o(n)$;
  \item by Lemmas \ref{PC.Paper.Inverse.Matrix.Convergence.Probability.d.Infity} and
        \ref{PC.Paper.Sample.Covariance.Matrix.Convergence.d.infty},
        $d(\log(d))^{1/2}\normF{T_{0}^{-1} - \bd{\bd{\Sigma}}_{X}^{-1}}$ is $o_{P}(1)$
        for $d^{6}\log(d) = o(n)$;
  \item by Lemma \ref{PC.Paper.Scaled.IID.Sum.Convergence.Rate.d.infty},
        $(d\log(d))^{-1/2}\normE{\sqrt{n}T_{2}}=(d\log(d))^{-1/2}\normE{\frac{1}{\sqrt{n}}\mathbb{S}_{s_{1}+1,n}^{\epsilon}}$ is $O_{P}(1)$ for $d=o(\sqrt{n})$.
\end{itemize}
Thus,
$R_{2}\ConvProb 0$.

\lineProof

\OnItem{On $R_{3}$ and $R_{4}$.}
First consider $V_{3}$.
By noting that $s_{1}=o(\sqrt{n}/(\lambda d\kappa_{n}))$,
wpg1,
\begin{align*}
\normE{R_{3}}
& =
\normE{\sqrt{n}\bd{A}_{n}(\hat{\bd{\Sigma}}_{n}^{1/2}-\bd{\Sigma}_{X}^{1/2})
T_{0}^{-1}T_{3}} \\
& \leq
\sqrt{n}
\normF{\bd{A}_{n}}
\normF{\hat{\bd{\Sigma}}_{n}^{1/2}-\bd{\Sigma}_{X}^{1/2}}
\normF{T_{0}^{-1}}
\normE{T_{3}} \\
& \leq
\sqrt{n}
\normF{\bd{A}_{n}}
(d^{1/2}\normF{\hat{\bd{\Sigma}}_{n}-\bd{\Sigma}_{X}})^{1/2}
\normF{T_{0}^{-1}}
\normE{T_{3}} \\
& =
d\sqrt{n}
\normF{\bd{A}_{n}}
(\normFd{\hat{\bd{\Sigma}}_{n}^{1/2}-\bd{\Sigma}_{X}^{1/2}})^{1/2}
\normFd{T_{0}^{-1}}
\normE{T_{3}} \\
& \lesssim
d\sqrt{n}\lambda s_{1}\kappa_{n}/n
=
d\lambda s_{1}\kappa_{n}/\sqrt{n}
\rightarrow 0.
\end{align*}
Thus, $\normE{R_{3}}=o_{P}(1)$.
In the same way, $\normE{R_{4}}=o_{P}(1)$.

Therefore, the result of the theorem follows
by Slutsky's lemma.
\end{proof}
\end{proofLong}


\end{document}